\documentclass[a4paper]{book}
\usepackage[latin1]{inputenc}
\usepackage[greek,english]{babel} 
\usepackage{amsmath}
\usepackage{amsthm}
\usepackage{amsfonts}
\usepackage{mathrsfs}
\usepackage{amssymb}
\usepackage{stmaryrd}
\usepackage{comment}
\usepackage[all,cmtip]{xy}
\usepackage{graphicx}
\usepackage{tikz-cd}
\usepackage{tikz}
\usetikzlibrary{shapes,arrows}

\usepackage[nottoc,notlot,notlof]{tocbibind}
\usepackage{tocloft}
\usepackage{imakeidx}
\tikzset{
	symbol/.style={
		draw=none,
		every to/.append style={
			edge node={node [sloped, allow upside down, auto=false]{$#1$}}}
	}
}

\usepackage[width=15.00cm, height=22.00cm]{geometry}
\usepackage{color}

\usepackage{multirow}

\usepackage{mathrsfs}

\usepackage{enumerate}
\usepackage[perpage]{footmisc}
\usepackage{hyperref}
\hypersetup{
	colorlinks,
	citecolor=black,
	filecolor=black,
	linkcolor=black,
	urlcolor=black
}

\usepackage{setspace}

\usepackage[numbers]{natbib}
\usepackage[titletoc]{appendix}

\usepackage{afterpage}

\newcommand\blankpage{%
	\null
	\thispagestyle{empty}%
	\addtocounter{page}{-1}%
	\newpage}

\allowdisplaybreaks

\title{$\mathbb{Z}$-graded supergeometry: Differential graded modules, higher algebroid representations, and linear structures}
\author{Theocharis Papantonis} \date{}

\theoremstyle{definition}
\newtheorem{theorem}{Theorem}[section]
\newtheorem{lemma}[theorem]{Lemma}
\newtheorem*{lemma*}{Lemma}
\newtheorem{proposition}[theorem]{Proposition}
\newtheorem*{proposition*}{Proposition}
\newtheorem{corollary}[theorem]{Corollary}
\newtheorem{definition}[theorem]{Definition}
\newtheorem*{definition*}{Definition}
\newtheorem*{question*}{Question}
\newtheorem*{conjecture*}{Conjecture}
\theoremstyle{remark}
\newtheorem*{remark*}{Remark}
\newtheorem{example}[theorem]{Example}
\newtheorem{remark}[theorem]{Remark}
\DeclareMathOperator{\Hom}{Hom}
\DeclareMathOperator{\End}{End}
\newcommand{\Sym}{\underline{S}}
\DeclareMathOperator{\T}{T}
\DeclareMathOperator{\sgn}{sgn}
\newcommand{\diff}{\mathrm{d}}
\DeclareMathOperator{\Der}{Der}
\DeclareMathOperator{\Jac}{Jac}
\DeclareMathOperator{\ad}{ad}
\DeclareMathOperator{\Ad}{Ad}
\DeclareMathOperator{\id}{id}
\DeclareMathOperator{\Pont}{Pont}
\DeclareMathOperator{\gtr}{str}
\DeclareMathOperator{\tr}{tr}

\newcommand{\CM}{\mathcal{C}^\infty(\mathcal{M})}
\newcommand{\cin}{\mathcal{C}^\infty} \newcommand{\Z}{\mathbb{Z}}
\newcommand{\R}{\mathbb{R}} \newcommand{\M}{\mathcal{M}}
 \newcommand{\Em}{\mathcal{E}}
\newcommand{\Fm}{\mathcal{F}} \newcommand{\N}{\mathcal{N}}
\newcommand{\Q}{\mathcal{Q}} \newcommand{\D}{\mathcal{D}}
\newcommand{\n}{[n]}

\newcommand{\lb}{\llbracket} \newcommand{\rb}{\rrbracket}
\newcommand{\E}{\underline{E}} \newcommand{\A}{\underline{A}}
\newcommand{\F}{\underline{F}}
\DeclareMathOperator{\GHom}{\underline{Hom}}
\DeclareMathOperator{\GEnd}{\underline{End}}
\DeclareMathOperator{\GL}{GL}
\DeclareMathOperator{\rank}{rank} \DeclareMathOperator{\pr}{pr}
\newcommand{\mx}{\mathfrak{X}}
\newcommand{\dr}{\text{d}}
\newcommand{\ldr}[1]{{{\pounds}}_{#1}}

\newcommand{\X}{\mathcal{X}}
\newcommand{\Y}{\mathcal{Y}}

\makeatletter
\newcommand\xleftrightarrow[2][]{%
	\ext@arrow 9999{\longleftrightarrowfill@}{#1}{#2}}
\newcommand\longleftrightarrowfill@{%
	\arrowfill@\leftarrow\relbar\rightarrow}
      \makeatother

\begin{document}
		
		\renewcommand\indexname{Alphabetical Index}
		
		\frontmatter
		
	\thispagestyle{empty}
	\begin{center}
	{\setstretch{2.0}
		\hspace{0pt}
		\vfill
		
		{\LARGE  \textbf{$\mathbb{Z}$-graded supergeometry: Differential graded modules, higher algebroid representations, and linear structures}}
		
	}
		
		\vspace{3cm}
		
		Dissertation \\
		for the award of the degree \\
		``Doctor rerum naturalium" (Dr.rer.nat.) \\
		of the Georg-August-Universität Göttingen 
		
		\vspace{2cm}
		
		within the doctoral program ``Mathematical Sciences" \\
		of the Georg-August University School of Science (GAUSS)
		
		\vspace{3cm}
		
		submitted by \\
		Theocharis Papantonis
		
		\vspace{2cm}
		
		from Athens, Greece\\
		Göttingen, 2021
		
		\vfill
		\hspace{0pt}

	\end{center}
	
	\pagebreak
	
	\thispagestyle{empty}
	\noindent
	\textbf{Thesis Committee}

	\vspace{.5cm}
	\noindent
	Prof.~Dr.~Madeleine Jotz Lean \\
	Mathematisches Institut, Georg-August-Universit\"at G\"ottingen
	
	\vspace{.5cm}
	\noindent
	Prof.~Dr.~Chenchang Zhu \\
	Mathematisches Institut, Georg-August-Universit\"at G\"ottingen
	
	\vspace{.5cm}
	\noindent
	Prof.~Dr.~Thomas Schick \\
	Mathematisches Institut, Georg-August-Universit\"at G\"ottingen
	
	\vspace{1cm}
	\noindent
	\textbf{Members of the Examination Board}
	
	\vspace{.5cm}
	\noindent
	Reviewer: Prof.~Dr.~Madeleine Jotz Lean \\
	Mathematisches Institut, Georg-August-Universit\"at G\"ottingen
    
    \vspace{.5cm}
    \noindent
	Second Reviewer: Prof.~Dr.~Chenchang Zhu \\
	Mathematisches Institut, Georg-August-Universit\"at G\"ottingen
	
	\vspace{.5cm}
	\noindent
	Further members of the Examination Board:
	
	\vspace{.5cm}
	\noindent
	Prof.~Dr.~Thomas Schick \\
	Mathematisches Institut, Georg-August-Universit\"at G\"ottingen
	
	\vspace{.5cm}
	\noindent
	Prof.~Dr.~Frank Gounelas \\
	Mathematisches Institut, Georg-August-Universit\"at G\"ottingen
	
	\vspace{.5cm}
	\noindent
	Prof.~Dr.~Rajan Amit Mehta \\
	Department of Mathematics and Statistics, Smith College, Northampton, MA USA
	
	\vspace{.5cm}
	\noindent
	Prof.~Dr.~Christoph Lehrenfeld \\
	Institut f\"ur Numerische und Angewandte Mathematik, Georg-August-Universit\"at G\"ottingen
	
	\vspace{.5cm}
	\noindent

	\vspace{1cm}
	\noindent
	\textbf{Date of the oral examination:} May 19, 2021

	\pagebreak
	
	\hspace{0pt}
	\vfill
	\hfill \textit{To my family and Julia...}
	
	\hfill \textit{...for their encouragement and support}
	\vfill
	\hspace{0pt}
	
	\pagebreak
	\blankpage	

\phantomsection
\addcontentsline{toc}{chapter}{Abstract}
\begin{center}
	{\large \textbf{Abstract}}
\end{center}
The purpose of this thesis is to present a self-standing review of $\mathbb{Z}$- (respectively $\mathbb{N}$-)graded supergeometry with emphasis in the development and study of two particular structures therein. Namely, representation theory and linear structures of $\Q$-manifolds and higher Lie algebroids (also known in the mathematics and physics literature as $\mathbb{Z}\Q$- and $\mathbb{N}\Q$-manifolds, respectively). 

Regarding the first notion, we introduce differential graded modules (or for short DG-modules) of $\Q$-manifolds and the equivalent notion of representations up to homotopy in the case of Lie $n$-algebroids ($n\in\mathbb{N}$). These are generalisations of the homonymous structures of \cite{Vaintrob97,GrMe10,ArCr12} that exist already in the case of ordinary Lie algebroids, i.e. when $n=1$. The adjoint and coadjoint modules are described, and
the corresponding split versions of the adjoint and coadjoint representations up to homotopy of Lie $n$-algebroids are
explained. In particular, the case of Lie $2$-algebroids is analysed in detail. The compatibility of a graded
Poisson bracket with the homological vector field on a $\mathbb{Z}$-graded manifold is shown to be equivalent
to an (anti-)morphism from the coadjoint module to the adjoint module, leading to an alternative
characterisation of non-degeneracy of graded Poisson structures. Applying this result to symplectic Lie $2$-algebroids, gives another algebraic characterisation of Courant algebroids in terms of their adjoint and coadjoint representations. In addition, the Weil algebra of a general $\Q$-manifold is defined and is computed explicitly in the case of Lie $n$-algebroids over a base (smooth) manifold $M$ together with a choice of a splitting and linear $TM$-connections. Similarly to the work of Abad and Crainic in \cite{ArCr12}, our computation involves the coadjoint representation of the Lie $n$-algebroid and the induced $2$-term representations up to homotopy of the tangent bundle $TM$ on the vector bundles of the underlying complex of the Lie $n$-algebroid given by the choice of the linear connections.

The second object that we define and explore in this work is the linear structures on $\mathbb{Z}$-graded manifolds, for which we see the connection with DG-modules and representations up to homotopy. In the world of split Lie $n$-algebroids, this leads to the notion of higher VB-algebroids, which we call VB-Lie $n$-algebroids; that is, Lie $n$-algebroids that are in some sense linear over another Lie $n$-algebroid. We prove that there is an equivalence between the category of VB-Lie $n$-algebroids over a Lie $n$-algebroid $\A$ and the category of $(n+1)$-term representations up to homotopy of $\A$, generalising a well-known result from the theory of ordinary VB-algebroids over Lie algebroids, i.e., in our setting, VB-Lie $1$-algebroids over Lie $1$-algebroids.

\vspace{1cm}
\noindent
\textbf{Keywords:} graded manifold, supermanifold, $\Q$-manifold, Lie $n$-algebroid, VB-algebroid, representation up to homotopy, differential graded module, adjoint representation, adjoint module, Weil algebra.

    \newpage
	\blankpage
	
	\phantomsection
	\addcontentsline{toc}{chapter}{Acknowledgements}
	\begin{center}
		{\large \textbf{Acknowledgements}}
	\end{center}
	First of all, I would like to thank my supervisor in G\"ottingen, Prof.~Madeleine Jotz Lean, without whom the completion of this thesis would have been impossible. Through these few lines, I would like to express my sincere gratitude to her for trusting me to be her PhD student as well as for helping and guiding me through these three years of my studies in G\"ottingen. Because of her, my studying time has successfully come to an end and my long-standing goal of pursuing a PhD in pure mathematics is finally achieved. 
	
	Next, I would like to thank my second supervisor in G\"ottingen, Prof.~Chenchang Zhu. The discussions, seminars, and courses that we had together, provided me with ideas and knowledge that were extremely valuable to my mathematical development.
	
	Many thanks to Prof.~Thomas Schick for accepting to be my third supervisor in G\"ottingen. Attending his lectures at the university expanded my knowledge in new topics outside my research field and served for me as a nice break from the intensive focus on my projects.
	
	I would like to express the deepest appreciation to Prof.~Rajan Amit Mehta from the Smith College, Northampton, MA USA, for our great collaboration which had very positive impact on my research. Our discussions and exchanging of ideas were extremely helpful to me for completing this journey and are highly appreciated.
	
	Special thanks go also to my good friend in G\"ottingen, Miquel Cueca Ten. Our endless chatting about mathematical and everyday topics were crucial to me for both my research and my life in G\"ottingen. His comments on the early version of this thesis have made the final result much more readable.
	
	I would like to express my gratitude to the University of G\"ottingen and especially the Mathematical Institute for providing me with a nice and inspiring place to complete my research and expand my knowledge in many other topics. It was a great honour to be part of the miraculous mathematical history of this department.
	
	Nobody has been more important to me than my family and Julia. My sincere gratefulness for their constant love and support throughout all these years. This thesis is dedicated to them.
	
	Last but not least, I would like to thank all my friends from our research group in G\"ottingen, Miquel, David, Ilias (especially for correcting my spelling mistakes), and many others. You made my three years in G\"ottingen such a great time and I wish you all the best!
	
	\newpage
	
	\phantomsection
	\addcontentsline{toc}{chapter}{Contents}
	\tableofcontents
	
	\mainmatter
	
	\chapter{Introduction}
	
	Throughout history mathematics and physics have always been complementary to each other in our way of understanding and explaining our universe. They have constantly influenced one another, with the exchanging of ideas and results going both ways: On the one hand physicists have applied theoretical results of mathematics to obtain physical and practical outcomes, while on the other hand, many physical ideas have served as inspiration for mathematicians to develop abstract mathematical theories. Their notable interaction in modern research can be seen by two of the most elegant and closely related branches of modern mathematics and theoretical physics, namely \textit{Differential Geometry} and \textit{Mathematical Physics}; in short, the former studies \textit{(differentiable) manifolds} which is the idea of ``space" in its most general form and the latter deals with the mathematical methods suitable for physical theories.
	
	There are two fundamental concepts of Differential Geometry that are needed in physics and are relevant for this thesis. The first is the notion of \textit{Poisson} and \textit{symplectic} structures on manifolds which serve as phase spaces in the Hamiltonian description of classical mechanics. The second comes from the spaces known as \textit{Lie groups} together with their infinitesimal counterparts called \textit{Lie algebras}, or more generally \textit{Lie groupoids} and \textit{Lie algebroids}, which provide a rigorous mathematical description of the concept of \textit{symmetry}. A symmetry of a physical system is a characteristic of the system that is preserved under some transformation and its physical importance can be deduced from the famous Noether's Theorem which states that each physical symmetry has a corresponding conservation law. Some remarkable examples of symmetry groups in physics are the Poincaré group $\mathbb{R}^{1,3}\rtimes\operatorname{O}(1,3)$ in Special Relativity, the unitary group $\operatorname{U}(1)$ in quantum electrodynamics and the product group $\operatorname{U}(1)\times \operatorname{SU}(2) \times \operatorname{SU}(3)$ in the Standard Model of particle physics. Roughly, the symmetries of the Poincar\'e group correspond to conservation of mass-energy, conservation of linear momentum, and conservation of angular momentum; the conservation law of the unitary group is the conservation of electric charge; the conservation laws of $\operatorname{U}(1)\times \operatorname{SU}(2) \times \operatorname{SU}(3)$ are the conservation of electric charge, the conservation of weak isospin and the conservation of color charge.
	
	The present work studies a common subfield of Differential Geometry and Mathematical Physics known as \textit{$\mathbb{Z}$-graded supergeometry}. Our main mathematical objects will be a generalisation of manifolds called \textit{graded manifolds} as well as \textit{higher Poisson structures} and \textit{$\Q$-manifolds}, with emphasis on the special case of \textit{higher (split) Lie algebroids} for which we will develop a well-behaved \textit{Representation Theory}. These terms will be explained briefly in the rest of the introduction where we provide some historical background and sketch of what will be analysed later in great detail.
	
	\paragraph*{Poisson and symplectic structures -- A historical overview}
	
	The first appearance of the Poisson bracket was considered already by Sim\'eon Denis Poisson himself more than 200 years ago in his pioneering improvement of Lagrangian mechanics \cite{Poisson1809} as an operation on physical observables, i.e.~functions of the phase space of classical mechanical systems. In this work, Poisson introduced the bracket\index{Poisson bracket}
	\[
	\{f,g\} := \sum_{i=1}^m\left( \frac{\partial f}{\partial q^i}\frac{\partial g}{\partial p^i} - \frac{\partial g}{\partial q^i}\frac{\partial f}{\partial p^i} \right),
	\]
	where $f$ and $g$ are functions of the variables $q^i$ and $p^i:=\frac{\partial L}{\partial \dot{q}^i}$ for a mechanical system with Lagrangian $L=L(q^i,\dot{q}^i)$, and proved that if the functions $f$ and $g$ are first integrals of the system, i.e.~they are preserved under the dynamics of the system, then so is the function $\{f,g\}$. The Poisson bracket was then used by William Rowan Hamilton in 1835 \cite{Hamilton1835} to express his equations of motion\index{Hamilton's equations of motion}
	\[
	\dot{p}^i = - \frac{\partial H}{\partial q^i}
	\qquad \text{and} \qquad
	\dot{q}^i = \frac{\partial H}{\partial p^i},
	\]
	with the function $H$ given by $H = \sum_{i=1}^{m}\dot{q}^ip^i - L$
	and called \textit{Hamiltonian} function\index{Hamiltonian function} of the system. Thus the theory that is now called Hamiltonian mechanics was established as an improvement of Newtonian mechanics. Seven to eight years later, Carl Jacobi \cite{Jacobi1884} proved that the Poisson bracket satisfies what is now known as the \textit{Jacobi equation}\index{Jacobi identity}:
	\[
	\{f,\{g,h\}\} = \{\{f,g\},h\} + \{g,\{f,h\}\}. 
	\]
	
	In a modern terminology, a \textit{Poisson structure} on a smooth manifold $M$ is a bivector field\index{Poisson bivector field} $\pi\in\mathfrak{X}^2(M):=\Gamma(\wedge^2TM)$ such that the corresponding bracket $\{f,g\} := \left\langle \diff f\wedge\diff g,\pi \right\rangle$
	satisfies the Jacobi identity. The prototype example is $M = \mathbb{R}^{2m}$ with coordinates $(q^1,\ldots,q^m,p^1,\ldots,p^m)$ and
	\begin{equation}\label{Basic Poisson tensor}
	\pi = \sum_{i=1}^{m}\frac{\partial}{\partial p^i}\wedge\frac{\partial}{\partial q^i}
	\end{equation}
	with the corresponding bracket being the one discovered by Poisson. Given a Poisson structure on the manifold $M$, there is an induced $C^\infty(M)$-linear map $\pi^\sharp:\Omega^1(M)\to \mathfrak{X}(M)$ which sends the $1$-form $\alpha$ to the vector field $i_\alpha\pi$, and in particular, the exact $1$-form $\diff f$ to the vector field $X_f:=\{f,\cdot\}$ called the \textit{Hamiltonian vector field of $f$}\index{Hamiltonian vector field}.
	If the map $\pi^\sharp$ is non-degenerate (i.e.~it is an isomorphism), the Poisson structure is called \textit{symplectic}\index{symplectic structure} and is equivalent to a 2-form $\omega\in\Omega^2(M)$ (called \textit{symplectic form})\index{symplectic form} which is closed and non-degenerate, i.e.~$\diff\omega = 0$ and the map $(\pi^\sharp)^{-1}=\omega^\flat:\mathfrak{X}(M)\to\Omega^1(M),X\mapsto i_X\omega=\omega(X,\cdot)$ is an isomorphism. The 2-form that corresponds to the Poisson bracket on $M=\mathbb{R}^{2m}$ is defined by $\omega = \sum_{i=1}^{m}\diff q^i\wedge\diff p^i$.
	Using the above formalism, Hamilton's equations of motion on a symplectic manifold $(M,\omega)$\index{symplectic manifold} with Hamiltonian $H$ are just the integral curves of the vector field $\{\cdot\,,H\}=-X_H$. 
	
	Although not all Poisson structures come from a symplectic form and thus cannot be written in the form of equation (\ref{Basic Poisson tensor}), symplectic structures are still present in their local structure. This breakthrough was achieved in 80's with the famous paper of Weinstein \cite{Weinstein83} which proved the \textit{local splitting of Poisson manifolds}: Locally around every point, a Poisson structure is a product of a symplectic structure with a degenerate Poisson structure that vanishes at that point. The theorem is essentially a generalisation of Darboux's theorem in symplectic geometry. In the next years, many splitting theorems were proved for other similar structures such as Lie algebroids \cite{Dufour01,Fernandes02,Weinstein00}, generalised complex manifolds \cite{AbBo06}, Dirac structures \cite{Blohmann17}, Courant algebroids and $L_\infty$-algebroids \cite{BiBuMe20}, etc.
	
	Nowadays, Poisson geometry is a very active field of research in mathematics with connections to numerous areas, such as non-commutative geometry, topological field theories, representation theory, etc. The interested reader may find many graduate texts written on the topic, for example \cite{DuZu05,GePiVa12,Vaisman94,daSiWe99,Waldmann07}.
	
	\paragraph*{Lie groupoids and Lie algebroids}
	
	The reader who has taken a standard undergraduate course in differential geometry may be familiar with Lie groups and Lie algebras. A Lie group is a group (in the algebraic sense) which is also a smooth manifold and its elements can be smoothly inverted and multiplied in pairs. A Lie algebra is a vector space equipped with a bracket operation $[\cdot\,,\cdot]$ that is bilinear, anti-symmetric and satisfies the Jacobi identity. As it was mentioned earlier, the importance of Lie groups and Lie algebras is that they offer a mathematical treatment for simple symmetries in physics. A Lie groupoid, denoted $\mathcal{G}\rightrightarrows M$, is in some sense a ``smooth" collection of Lie groups and hence provides a systematic way of describing more complicated symmetries that a single group would fail to capture \cite{Weinstein01}. As a first short definition we give the following: a \textit{Lie groupoid}\index{Lie groupoid} is a small category such that every morphism is an isomorphism with all the objects being smooth. The detailed description of its definition is long and so we postpone it to Chapter \ref{Chapter: Conclusion, open problems and further research}. The situation gets a little bit easier if one considers Lie algebroids, which can be thought of as a ``linear approximation" of Lie groupoids and so may reflect some of their properties. They are in some sense infinite dimensional Lie algebras over a possible ``curved'' space, and serve as a common generalisation of Lie algebras and tangent vector bundles. Their close relation to Poisson manifolds comes from the fact that a Lie algebroid structure on a vector bundle $A\to M$ is equivalent to a Poisson bracket on the manifold $A^*$ that is fibre-wise linear. The precise mathematical definition is as follows.
	\begin{definition*}
		A \textit{Lie algebroid}\index{Lie algebroid} is a smooth vector bundle $A\to M$ together with a smooth vector bundle map $\rho:A\to TM$ over the identity on $M$, called \textit{anchor}\index{anchor map}, and a Lie bracket\index{Lie algebroid!Lie bracket} $[\cdot\,,\cdot]$ on its space of sections $\Gamma(A)$ which are compatible: $\rho([a,b]) = [\rho(a),\rho(b)]$ and satisfy the \textit{Leibniz identity}\index{Leibniz identity}
		\[
		[a,fb] = \rho(a)f\cdot b + f[a,b]
		\]
		for all $f\in C^\infty(M),a,b\in\Gamma(A)$.
		Equivalently, a Lie algebroid can be defined as a vector bundle $A\to M$ together with an operator $\diff_{A}:\Omega^\bullet(A):=\Gamma(\wedge^\bullet A^*) \to \Omega^{\bullet+1}(A)$\index{Lie algebroid!differential} that squares to zero: $\diff_{A}^2 = 0$ and satisfies
		\[
		\diff_{A}(\alpha\wedge\beta) = \diff_{A}\alpha\wedge\beta + (-1)^{\deg(\alpha)}\alpha\wedge\diff_{A}\beta.
		\]
	\end{definition*}
	\noindent One obtains a Lie algebroid from a Lie groupoid pretty much the same way that a Lie algebra is obtained by a Lie group, i.e.~by differentiation. The interesting problem was the integration of Lie algebroids, i.e.~under what circumstances a Lie algebroid comes in this manner from a Lie groupoid. This was solved in 2001 in the work of Cattaneo and Felder \cite{CaFe01}, and in 2003 in the paper by Crainic and Fernandes \cite{CrFe03}.
	
	\paragraph*{Courant algebroids}
	
	In the 90's the works of Courant \cite{CourantTheo90} and Dorfman \cite{Dorfman93} in mechanical systems with constraints led them to the discovery of what is now known as the \textit{Courant bracket}\index{Courant bracket} on sections of the vector bundle $\mathbb{T}M:=TM\oplus T^*M$:
	\[
	\llbracket X + \omega, Y + \eta  \rrbracket = [X,Y] + \ldr{X}\eta - i_Y\diff\omega
	\]
	for all $X,Y\in\mathfrak{X}(M),\omega,\eta\in\Omega^1(M)$. The abstraction of this idea was done a few years later by Liu, Weinstein and Xu \cite{LiWeXu97}, and gave rise to Lie bialgebroids and consequently to Courant algebroids.
	\begin{definition*}
		A \textit{Courant algebroid}\index{Courant algebroid} is a smooth vector bundle $E\to M$
		equipped with a fibre-wise non-degenerate symmetric bilinear form $\langle \cdot\,,\cdot \rangle:E\times_M E\to\mathbb{R}$, a bilinear
		bracket $\llbracket \cdot\,,\cdot \rrbracket$ on the smooth sections $\Gamma(E)$, and an anchor map $\rho: E \to TM$, which satisfy
		the following conditions: for $f\in C^\infty(M),e_1,e_2,e_3\in\Gamma(E)$
		\begin{enumerate}
			\item $\rho(\llbracket e_1,e_2 \rrbracket) = [\rho(e_1),\rho(e_2)]$
			\item $\llbracket e_1,fe_2 \rrbracket = \rho(e_1)f\cdot e_2 + f\llbracket e_1,e_2 \rrbracket$
			\item $\llbracket e_1, \llbracket e_2,e_3 \rrbracket \rrbracket =
			\llbracket \llbracket e_1,e_2 \rrbracket, e_3 \rrbracket + 
			\llbracket e_2, \llbracket e_1,e_3 \rrbracket \rrbracket$
			\item $\rho(e_1)\langle e_2,e_3 \rangle = 
			\langle \llbracket e_1,e_2 \rrbracket, e_3 \rangle + 
			\langle e_2, \llbracket e_1,e_3 \rrbracket \rangle$
			\item $\llbracket e_1,e_2 \rrbracket + \llbracket e_2,e_1 \rrbracket = \mathcal{D}\langle e_1,e_2 \rangle$.
		\end{enumerate}
	Here, we use the notation $\mathcal{D}:=\rho^*\circ\diff$ and identify $E$ with $E^*$ via the pairing: $\langle \mathcal{D}(f),e \rangle = \rho(e)f$ for all $f\in C^\infty(M),e\in\Gamma(E)$. Later it was shown that the first two axioms above can be deduced from the other three.
	\end{definition*}
	
	Introducing \textit{Dirac structures}\index{Dirac structure}, which were also the original insight of Courant and Weinstein's work, leads one to view Courant algebroids as a suitable framework for simultaneous treatment of pre-symplectic and Poisson structures. A Dirac structure of a Courant algebroid $E\to M$ is a vector subbundle $D\to M$ which is maximally isotropic with respect to $\langle\cdot\,,\cdot\rangle$ and integrable:
	\[
	\langle D,D \rangle = 0,
	\qquad
	\rank(D) = \frac{1}{2}\rank(E),
	\qquad
	\llbracket \Gamma(D),\Gamma(D) \rrbracket \subset \Gamma(D).
	\]
	Then the graphs of pre-symplectic forms $\omega\in\Omega^2(M)$ (i.e. $\diff\omega = 0$) and Poisson tensors $\pi\in\mathfrak{X}^2(M)$ are Dirac structures of the standard Courant algebroid $\mathbb{T}M$:
	\[
	\text{Graph}(\omega) := \{ v + i_v\omega\ |\ v\in TM \}
	\qquad\text{and}\qquad
	\text{Graph}(\pi) := \{ i_\alpha\pi + \alpha\ |\ \alpha\in T^*M \}.
	\]
	
	In addition, Courant algebroids and Dirac structures are important in modern mathematical physics due to the fact that they are the fundamental objects in generalised complex geometry and hence in mirror symmetry.
	
	\paragraph*{Origin of supergeometry}
	
	Supermanifolds with a $\mathbb{Z}_2$-grading were introduced already in the 60's in the pioneering work of Berezin \cite{Berezin66}. They appear often in the physics literature, as they have applications in superstring theory due to their success in providing a formal mathematical description of supersymmetric field theories \cite{SaSt74}. Intuitively, one should think of them as spaces with two kinds of coordinate functions: \textit{even} and \textit{odd}, or in physicists' language \textit{bosonic} and \textit{fermionic}. The characteristic difference of these two coordinates is that the former are commutative, while the latter are anti-commutative; that is, a \textit{$\mathbb{Z}_2$-graded supermanifold} is a space which locally admits a coordinate representation of the form $(x^i,\theta^j)$, where $x^i$ are ordinary real-valued spacetime coordinates and $\theta^j$ are ``formal" Grassmann-valued coordinates satisfying
	\[
	x^ix^j = x^jx^i
	\qquad \text{and} \qquad
	\theta^\alpha\theta^\beta = -\, \theta^\beta\theta^\alpha.
	\]
	This setting encodes the quantum viewpoint of the cosmos in which particles are divided into two categories depending on their spin: \textit{bosons} and \textit{fermions}; the former commute, while the latter anti-commute.
	
	In a more mathematical terminology, the definition of a $\mathbb{Z}_2$-graded supermanifold\index{$\mathbb{Z}_2$-graded supermanifold} of dimension $p|q$ is sheaf-theoretic, defining it as an ordinary $p$-dimensional manifold equipped with a sheaf of $\mathbb{Z}_2$-graded algebras that locally looks like the algebra of functions of the $p|q$-dimensional \textit{superspace}: 	$C^\infty(\mathbb{R}^p)\otimes\Lambda^\bullet(\xi^1,\ldots,\xi^q)$;
	here, $\Lambda^\bullet(\xi^1,\ldots,\xi^q)$ is the Grassmann algebra on $q$ generators.
	Locally, the functions (or physical observables) on supermanifolds take the form of a ``power series'' in the odd coordinates $\theta^j$ with coefficients in the ring of smooth functions over $\mathbb{R}^p$ with coordinates $x=(x^1,\ldots x^p)$: 
	\[
	f(x,\theta) = f_0(x) + \sum_{r=1}^{q}\sum_{i_1<\ldots<i_r} f_{i_1\ldots i_r}(x)\,\theta^{i_1}\ldots\theta^{i_r}.
	\]
	Some mathematical works on this subject were written by Berezin \cite{Berezin87}, Kostant \cite{Kostant77}, Deligne and Morgan \cite{DeMo99}, Manin \cite{Manin97}, Tuynman \cite{Tuynman04}, Varadarajan \cite{Varadarajan04}, and Carmeli, Caston and Fioresi \cite{CaCaFi11}.
	
	\paragraph*{Introduction of $\mathbb{Z}$- and $\mathbb{N}$-gradings}
	
	The possibility of more general gradings was mentioned already in the works of Kostant \cite{Kostant77} and Tuynman \cite{Tuynman04} but was not studied there. The use of a $\mathbb{Z}$-grading was treated properly in the PhD thesis of Mehta \cite{Mehta06}, who was inspired by previous works on graded geometry from Kontsevich \cite{Kontsevich03},
	Roytenberg \cite{Roytenberg02}, \v{S}evera \cite{Severa05}, and Voronov \cite{Voronov02}. Similarly to the $\mathbb{Z}_2$-case, a $\mathbb{Z}$- (respectively $\mathbb{N}$-)graded manifold\index{$\mathbb{Z}$-graded manifold}\index{$\mathbb{N}$-graded manifold} can be thought of as a space whose coordinates admit a $\mathbb{Z}$- (respectively $\mathbb{N}$-)grading and may commute or anti-commute depending on their degree:
	\[
	\xi_1\xi_2 = (-1)^{\text{deg}(\xi_1)\text{deg}(\xi_2)}\xi_2\xi_1.
	\]
	In sheaf-theoretic terminology, the structure sheaf of the graded manifold is now a sheaf of $\mathbb{Z}$- (respectively $\mathbb{N}$-)graded algebras and locally the functions look like a sum of the form
	\[
	f(x,\xi) = f_0(x) + \sum_{r=1}^\infty \sum_{i_1\leq\ldots\leq i_r} f_{i_1\ldots i_r}(x)\,\xi^{i_1}\ldots\xi^{i_r}.
	\]
	The important result of the works mentioned above was the realisation that many complicated definitions of classical differential geometric objects can be equivalently described in very simple terms in the language of graded geometry. The table at the end of this chapter offers a summary of the equivalences between graded and classical geometries.
	
	\paragraph*{$\Q$-structures, higher (split) Lie algebroids and Poisson brackets}
	
	A \textit{$\Q$-manifold}\index{$\Q$-manifold} (also known in the literature as differential graded manifold\index{differential graded manifold}, or DG-manifold\index{DG-manifold}) is a $\mathbb{Z}$-graded manifold equipped with a degree 1 vector field $\Q$ that squares to zero: $\Q^2 = 0$. In the special case where the underlying manifold is $\mathbb{N}$-graded of degree $n\in\mathbb{N}$ (i.e.~the highest degree of its coordinates is $n$), it is also called $\mathbb{N}\Q$-manifold\index{$\mathbb{N}\Q$-manifold} of degree $n$ or Lie $n$-algebroid\index{Lie $n$-algebroid}. Alternatively, a Lie $n$-algebroid can be described by means of a graded vector bundle together with a family of higher brackets on its space of sections \cite{ShZh17}. In this case, we refer to them as \textit{split} Lie $n$-algebroids\index{Lie $n$-algebroid!split}. The name comes from the equivalence of Lie 1-algebroids and ordinary Lie algebroids as in the definition above \cite{Vaintrob97}. Other correspondences of graded structures with classical objects can be realised by extending Poisson geometry to the graded world. A Poisson structure on a graded manifold, known as $\mathcal{P}$-manifold\index{$\mathcal{P}$-manifold}, is a bracket on its space of (graded) functions which satisfies graded versions of anti-symmetry, Leibniz and Jacobi identities. The compatibility of the vector field $\Q$ and the Poisson bracket is known as a $\mathcal{PQ}$-structure\index{$\mathcal{PQ}$-manifold} or Poisson Lie $n$-algebroid\index{Poisson Lie $n$-algebroid} in the case of an $\mathbb{N}$-grading.
	
	\paragraph*{Representations of Lie groupoids and Lie algebroids}
	
	The goal of Representation Theory\index{representation theory} in mathematics is to study general abstract structures by ``representing" them as simpler and more concrete objects. Informally, a representation\index{representation} is an action of a mathematical structure which is in some sense ``linear". A crucial representation theory is that of Lie groups and Lie algebras, whose objects are represented as matrices over a linear vector space. Formally, a representation of a Lie group $G$\index{representation!of a Lie group} or a Lie algebra $\mathfrak{g}$\index{representation!of a Lie algebra} is a vector space $V$ together with a Lie group homomorphism $G\to \GL(V)$, respectively a Lie algebra homomorphism $\mathfrak{g}\to \mathfrak{gl}(V)$, where $\GL(V)$ is the general linear group of all invertible linear transformations of $V$ under their composition and $\mathfrak{gl}(V)$ is the space of endomorphisms of $V$ equipped with the Lie bracket given by the commutator of endomorphisms 
	$[\phi,\psi] := \phi\circ\psi - \psi\circ\phi$, for $\phi,\psi\in\mathfrak{gl}(V)$.
	The basic example of a representation of a Lie group or a Lie algebra is the \textit{adjoint representation}\index{adjoint representation}: The adjoint representation of a Lie algebra $\mathfrak{g}$ is given by the vector space $\mathfrak{g}$ together with the Lie algebra homomorphism $\mathfrak{g}\to\mathfrak{gl}(\mathfrak{g}),\ x\mapsto \ad_x:=[x,\cdot]$.
	Given a Lie group $G$ with corresponding Lie algebra $\mathfrak{g}$, its adjoint representation is the vector space $\mathfrak{g}$ together with the Lie group homomorphism $G\to \mathfrak{gl}(\mathfrak{g}),\ g\mapsto \Ad_g:=\diff\Psi_g|_e$, where $e\in G$ is the identity element of $G$ and $\Psi_g:G\to G,\Psi_g(h):=ghg^{-1}$ is the conjugation with respect to the element $g\in G$. In fact, one can differentiate a Lie group representation to obtain a representation of its Lie algebra and this process links the two adjoint representations defined above. The converse procedure of integrating a Lie algebra representation to obtain a representation of its Lie group is also possible under the extra condition that $G$ is a simply connected Lie group.
	
	Passing to representations of Lie groupoids and Lie algebroids\index{representation!of a Lie groupoid}\index{representation!of a Lie algebroid}, the role of the vector space $V$ is played by a smooth vector bundle $E\to M$. Precisely, a representation of a Lie groupoid $\mathcal{G}\rightrightarrows M$ is a vector bundle $E\to M$ together with a groupoid homomorphism $\mathcal{G} \to \GL(E)$ covering the identity map on $M$; here, $\GL(E)\rightrightarrows M$ denotes the \textit{general linear groupoid} of $E$. Similarly, a representation of a Lie algebroid $A\to M$ is a vector bundle $E\to M$ together with a Lie algebroid homomorphism $A\to \Der(E)$, where $\Der(E)$ is the \textit{derivation Lie algebroid}\index{derivation Lie algebroid} of $E$ equipped with the commutator bracket; explicitly, $\Der(E)$ consists of $\mathbb{R}$-linear operators $\delta:\Gamma(E)\to\Gamma(E)$ such that there exists a vector field $X\in\mathfrak{X}(M)$ with the property that $\delta(f\sigma) = X(f)\sigma + f\delta(\sigma)$ for all $f\in C^\infty(M),\sigma\in\Gamma(E)$. That is, a representation of a Lie algebroid $A\to M$ can be defined as a vector bundle $E\to M$ together with an $A$-connection $\nabla$ on the sections of $E$ which is \textit{flat}: $R_{\nabla} = 0$, i.e.~$\nabla_{[a,b]} = \nabla_a\nabla_b - \nabla_b\nabla_a =: [\nabla_a,\nabla_b]$ for all $a,b\in\Gamma(A)$. Equivalently, it can be described by an exterior derivative $\diff_{\nabla}:\Omega^\bullet(A,E)\to \Omega^{\bullet+1}(A,E)$ that squares to zero, called the \textit{(action) differential}.
	From a more geometric point of view, a Lie algebroid representation of $A\to M$ on $E\to M$ is a Lie algebra map from the sections of $A$ to the vector fields over the manifold $E$ that are fibre-wise linear, i.e.~their (local) flows consist of vector bundle automorphisms of $E$. 
	
	\paragraph*{Representations up to homotopy -- DG-modules\footnote{This paragraph coincides with the introduction written for \cite{JoMePa19}}.}
	
	Although the above notion of representations of a Lie algebroid seems to generalise well the notion of a Lie algebra representation, it suffers from a non-trivial problem: It does not include a well-defined notion of the adjoint representation of the Lie algebroid. Some early attempts to define the adjoint representation of a Lie algebroid can be traced back to Evens, Lu and Weinstein \cite{EvLuWe99}. The nowadays accepted solution to this problem was done by Gracia-Saz and Mehta \cite{GrMe10}, and
	independently by Abad and Crainic \cite{ArCr12}, who showed that the notion of \textit{representation up to homotopy of a Lie algebroid}\index{representation!up to homotopy} is
	a good notion of representations of Lie algebroids which includes the adjoint representation. Roughly, the idea is to
	let the Lie algebroid $A$ act via a differential on Lie algebroid forms $\Omega^\bullet(A,\E)$ which take values on a cochain
	complex of vector bundles $\E = \bigoplus_i E_i[i]$ instead of just a single vector bundle $E$. This notion is essentially a $\mathbb{Z}$-graded analogue of Quillen's super-representations \cite{Quillen85}.
	
	After their discovery, representations up
	to homotopy have been extensively studied in other works, see e.g.~\cite{Mehta09,ArCrDh11,ArSc11,ArSc13,DrJoOr15,Mehta15,TrZh16,GrJoMaMe18,CaBrOr18,Jotz19b,BrOr19,Vitagliano15b}. In particular, the adjoint representation up to homotopy of a Lie algebroid is
	proving to be as fundamental in the study of Lie algebroids, as the adjoint representation of a
	Lie algebra is in the study of Lie algebras. In that context, the adjoint representation controls
	deformations of Lie algebras (see e.g.~\cite{CrScSt14} and references therein), symmetries of a Lie algebra,
	and it is a key to the classification and the algebraic integration of Lie algebras \cite{vEs62a,vEs62b}. The
	adjoint representation up to homotopy of a Lie algebroid appears naturally in each attempt to
	understand a feature of Lie algebras in the more general context of Lie algebroids. In particular,
	the deformations of a Lie algebroid are controlled by the cohomology with values in its adjoint
	representation up to homotopy \cite{Vaintrob97,Mehta14}, and an ideal in a Lie algebroid is a subrepresentation
	of the adjoint representation up to homotopy \cite{DrJoOr15}. While a Lie bialgebra $(\mathfrak{g}, \mathfrak{g}^*)$ is a pair of Lie algebras on dual vector spaces whose coadjoint representations form a matched pair, a Lie bialgebroid  is a pair of Lie algebroids on vector bundles in duality, whose coadjoint 2-representations form a matched pair
	(see Theorem 3.11 in \cite{GrJoMaMe18}). From another point of view, 2-term representations up to homotopy, which are equivalent to decompositions of VB-algebroids \cite{GrMe10}, have proved to be
	a powerful tool in the study of multiplicative structures on Lie groupoids (se e.g.~\cite{JoOr14,DrJoOr15,BuCaOr09,ArCr11}),
	which, at the infinitesimal level, can be described as linear structures on algebroids. An algebraic viewpoint of representations up to homotopy of Lie algebroids was achieved by Mehta in \cite{Mehta09}, where he showed that representations up to homotopy of a Lie algebroid $A\to M$
	are equivalent, up to isomorphism, to \textit{Lie algebroid modules over $A$}\index{Lie algebroid!module} in the sense of \cite{Vaintrob97}; that is, differential graded modules over the differential graded algebra $(\Omega(A),\diff_{A})$.
	
	In this work, we extend the above notions of modules, and consequently of representations up to
	homotopy, to the context of higher Lie algebroids. The definition is the natural generalisation of
	the case of usual Lie algebroids explained above, i.e.~differential graded modules over the space of
	smooth functions of the underlying graded manifold. The obtained notion is analysed in detail,
	including the two most important examples of representations, namely, the adjoint and the coadjoint representations (up to homotopy).
	
	In addition to the impact of representations up to homotopy in
	the study of Lie algebroids in the last ten years, our general
	motivation for studying representations up to homotopy of
	higher Lie $n$-algebroids comes from the case $n=2$, and in
	particular from Courant algebroids. As it was pointed out by Mehta, in light of AKSZ theory,
	it seems reasonable to expect that the category of
	representations (up to homotopy) of Courant algebroids might
	have interesting connections to 3-dimensional topology. The
	results in this thesis should be useful in the study of such
	representations.  The first step is the search for a good
	notion not only of the adjoint representation of a Courant
	algebroid, but also of its ideals, similar to the work done in
	\cite{JoOr14}. Since Courant algebroids are equivalent to Lie $2$-algebroids
	with a compatible symplectic structure
	\cite{Roytenberg02,Severa05}, the natural question that arises is the following:
	\begin{question*}
		Is a compatible Poisson or symplectic structure on a Lie $n$-algebroid encoded in its adjoint representation?
	\end{question*}
	The answer to this question is positive, since it turns out that a Poisson bracket on a Lie $n$-algebroid gives rise to a natural map from the coadjoint to the adjoint representation which is an
	anti-morphism of left representations and a morphism of right representations (see Theorem \ref{thm_poisson},
	Corollary \ref{cor_poisson} and Section \ref{Section: Poisson Lie algebroids of low degree}), i.e.~it anti-commutes with the differentials of their structure as
	left representations and commutes with the differentials of their structure as right representations.
	Further, the Poisson structure is symplectic if and only if this map is in fact a left anti-isomorphism
	and right isomorphism. This result is already known in some special cases, including Poisson Lie
	0-algebroids, i.e.~ordinary Poisson manifolds $(M,\{\cdot\,,\cdot\})$, and Courant algebroids over a point,
	i.e.~quadratic Lie algebras $(\mathfrak{g},[\cdot\,,\cdot],\langle\cdot\,,\cdot\rangle)$. In the former case the map reduces to the natural map
	$\sharp:T^*M\to TM$ obtained from the Poisson bracket on $M$, and in the latter case it is the inverse of
	the map defined by the non-degenerate pairing $\mathfrak{g}\to\mathfrak{g^*},x\mapsto\langle x,\cdot\rangle$.

	\paragraph*{Linear structures}
	
	Similarly to the case of ordinary manifolds, one can consider vector bundles over $\mathbb{Z}$-manifolds \cite{Mehta06}. Geometrically, they can be viewed as graded manifolds that admit two kinds of (graded) local coordinates: smooth and linear. The structures explained before may be defined to be fibre-wise linear leading to \textit{linear $\Q$-manifolds}\index{$\Q$-manifold!linear}, \textit{linear $\mathcal{P}$-manifolds}\index{$\mathcal{P}$-manifold!linear}, etc. In this setting, a linear split Lie $n$-algebroid gives rise to the notion of \textit{split VB-Lie $n$-algebroids}\index{VB-Lie $n$-algebroid}\index{VB-Lie $n$-algebroid!split} (Chapter \ref{Chapter: Higher split VB-algebroid structures}); that is, double vector bundles in the sense of Mackenzie \cite{Mackenzie05} with a linear split Lie $n$-algebroid structure over a split Lie $n$-algebroid side bundle. In fact, a linear $\Q$-structure on a vector bundle in the category of graded manifolds is the same as a differential graded module in the sense described above. From a classical geometric point of view, a representation up to homotopy of a split Lie $n$-algebroid is equivalent to a split VB-Lie $n$-algebroid (Theorem \ref{Theorem RUTHS and higher VB-Lie algebroids}, Proposition \ref{Correspondence of morphisms} and Theorem \ref{Equivalence of categories}).
	
	\paragraph*{Remark on collaboration}
	
	Parts of the present work (differential graded modules, representations up to homotopy of Lie $n$-algebroids, and some results on VB-Lie $n$-algebroids) were done jointly with the author's PhD supervisor Madeleine Jotz Lean and Rajan Mehta. The idea of studying representations up to homotopy of higher Lie $n$-algebroids was proposed by Mehta, whose approach was more algebraic via module theory, while Jotz Lean's approach was differential geometric via split structures and representations up to homotopy. For example, the two view points of the adjoint module/representation of a Lie $n$-algebroid $(\M,\Q)$ are given by $(\mathfrak{X}(\M),\ldr{\Q})$ (algebraic definition proposed by Mehta -- Chapter \ref{Chapter: Differential graded modules}) and the complex $TM[0]\oplus A_1[1]\oplus\ldots\oplus A_n[n]$ (geometric, or split, definition proposed by Jotz Lean -- Chapter \ref{Chapter: Representations up to homotopy}), where $A_1[1]\oplus\ldots\oplus A_n[n]\simeq \M$ is a splitting of the underlying $[2]$-manifold over the base smooth manifold $M$. As it was claimed by Mehta, it should be possible that the two approaches can be unified. The present thesis achieves this result by giving the precise isomorphism connecting the adjoint module and the adjoint representation up to homotopy (Section \ref{adjoint_module_adjoint_representation_isomorphism} for the case of Lie $2$-algebroids and Section \ref{Adjoint representation of a Lie n-algebroid} for the general case of Lie $n$-algebroids). At the time of submission of this thesis, the results of the joint work with Madeleine Jotz Lean and Rajan Mehta are submitted for publication and can also been found online in the following preprint \cite{JoMePa19}:
	
	\vspace{.2cm}
	\noindent
	\begin{center}
		\textit{``Modules and representations up to homotopy of Lie $n$-algebroids''}, arXiv:2001.01101, 2020
	\end{center}
	\vspace{.2cm}
	
	\noindent
	A more detailed list of the results found (in a more concise form) in the above preprint is the following: 
	
	\begin{itemize}
		\item Differential graded modules (Chapter \ref{Chapter: Differential graded modules}): Basic definitions, adjoint and coadjoint modules, map between coadjoint and adjoint modules of $\mathcal{PQ}$-manifolds (sections \ref{Section: The category of differential graded modules}, \ref{Section: Adjoint and coadjoint modules}, and \ref{Section PQ-manifolds: coadjoint vs adjoint modules}).
		
		\item Representations up to homotopy (Chapter \ref{Chapter: Representations up to homotopy}): Basic definitions, detailed analysis of the case of a Lie $2$-algebroid and its adjoint representation, comparison of adjoint module with adjoint representation (sections \ref{Section: The category of representations up to homotopy}, \ref{Section: The case of (split) Lie 2-algebroids}, \ref{Section: Adjoint representation of a split Lie $2$-algebroid}, \ref{adjoint_module_adjoint_representation_isomorphism}, \ref{Section: Coordinate transformation of the adjoint representation}, \ref{Adjoint representation of a Lie n-algebroid}, \ref{Section: The Weil algebra of a Lie n-algebroid}, \ref{Section: Poisson Lie algebroids of low degree}).
		
		\item Higher VB-algebroid structures (Chapter \ref{Chapter: Higher split VB-algebroid structures}): Differential geometric definition of VB-Lie $n$-algebroids and correspondence of VB-Lie $n$-algebroids with $(n+1)$-representations (parts of sections \ref{Section: Classical interpretation} and \ref{Section: Split VB-Lie n-algebroids and (n+1)-representations}).
	\end{itemize}
	
	\paragraph*{New results and main achievements of the thesis}
	
	\begin{itemize}
		\item The development of a well-behaved representation theory of Lie $n$-algebroids for general $n\in\mathbb{N}$ -- Chapter \ref{Chapter: Differential graded modules} and Chapter \ref{Chapter: Representations up to homotopy}.
		
		\item The precise connection between the adjoint module and the adjoint representation of Lie $n$-algebroids. In particular, a recipe for computing explicitly the structure objects of the adjoint representation of any Lie $n$-algebroid -- Chapter \ref{Chapter: Representations up to homotopy} and more precisely Section \ref{adjoint_module_adjoint_representation_isomorphism} and Section \ref{Adjoint representation of a Lie n-algebroid}.
		
		\item The explicit formulae for the structure objects that make up 3-term representations of a split Lie $2$-algebroid and in particular the thorough analysis of its adjoint representation -- Chapter \ref{Chapter: Representations up to homotopy} and more precisely Section \ref{Section: The case of (split) Lie 2-algebroids}, Section \ref{Section: Adjoint representation of a split Lie $2$-algebroid}, and Section \ref{Section: Coordinate transformation of the adjoint representation}.
		
		\item An alternative description of Courant algebroids in terms of the adjoint and coadjoint modules of the underlying Lie $2$-algebroid -- Chapter \ref{Chapter: Differential graded modules} and more precisely Section \ref{Section PQ-manifolds: coadjoint vs adjoint modules}.
		
		\item The description of linear structures on vector bundles in the category of graded manifolds and, in particular, the introduction of the notion of higher VB-Lie algebroids in the split setting -- Chapter \ref{Chapter: Linear structures on vector bundles} and Chapter \ref{Chapter: Higher split VB-algebroid structures}. 
		
		\item The equivalence between the category of representations up to homotopy of higher Lie algebroids and the category of higher VB-Lie algebroids -- Chapter \ref{Chapter: Higher split VB-algebroid structures} and more precisely Section \ref{Section: Split VB-Lie n-algebroids and (n+1)-representations} and Section \ref{Section: Change of decomposition}.
	\end{itemize}
	
	\noindent
	A list of important propositions/theorems of the thesis grouped by chapter is the following:
	
	\begin{itemize}
		\item \underline{Chapter \ref{Chapter: Differential graded modules}:} Theorem \ref{thm_poisson}, Theorem \ref{thm_Weil_Poisson-Weil}.
		
		\item \underline{Chapter \ref{Chapter: Representations up to homotopy}:} Proposition \ref{Adjoint_representation_of_Lie_2-algebroid} (together with Remark \ref{Adjoint representation of Courant algebroid}), Proposition \ref{Isomorphism with change of connections} and Proposition \ref{Isomorphism with change of splitting}.
		
		\item \underline{Chapter \ref{Chapter: Higher split VB-algebroid structures}:} Theorem \ref{Theorem RUTHS and higher VB-Lie algebroids}, Proposition \ref{Correspondence of morphisms} and Theorem \ref{Equivalence of categories}.
	\end{itemize}
	
	\paragraph*{Outline of the thesis}
	
	\begin{itemize}
		\item \underline{Chapter \ref{Chapter: Preliminaries}:} This chapter sets some notation and conventions, and introduces the fundamental constructions needed for the rest of the work. In particular, it recalls graded vector bundles and complexes, different notions of algebroids, double vector bundles, and the language of sheaves.
		
		\item \underline{Chapter \ref{Chapter: Z- and N-graded supergeometry}:} In this chapter, we recall the basic notions from graded geometry. Specifically, we define $\mathbb{Z}$- and $\mathbb{N}$-graded manifolds, graded vector fields and $\Q$-structures, Lie $n$-algebroids, and the graded generalisation of Poisson brackets.
		
		\item \underline{Chapter \ref{Chapter: Graded tangent and cotangent bundles}:} This chapter is devoted to the category of vector bundles over graded manifolds and in particular to the tangent and cotangent bundles. We explain various geometric constructions such as (pseudo)differential forms and (pseudo)multivector fields, the Weil and the Poisson-Weil algebras of a $\Q$-manifold, homotopy Poisson structures, and graded symplectic forms. Moreover, we recall the correspondence of Poisson manifolds and Courant algebroids with degree 1 and 2 symplectic $\Q$-manifolds.
		
		\item \underline{Chapter \ref{Chapter: Differential graded modules}:} Here we introduce the notion of differential graded modules for $\Q$-manifolds and construct some important examples including the adjoint and coadjoint modules. We investigate the relation between adjoint and coadjoint modules, and Weil and Poisson-Weil algebras for the case of $\mathcal{PQ}$-manifolds and then the result is applied to Courant algebroids yielding the description in terms of its adjoint and coadjoint modules.
		
		\item \underline{Chapter \ref{Chapter: Representations up to homotopy}:} In this chapter, we generalise the notion of representations up to homotopy to higher Lie $n$-algebroids and spell out the structure objects for the case of a split Lie 2-algebroid. We give an explicit description of the adjoint and coadjoint representations of split Lie 2-algebroids, and explain how to compute the adjoint representation of any split Lie $n$-algebroid using its identification with the adjoint module. In addition, we describe the Weil algebra of a split Lie $n$-algebroid and compute the map between the coadjoint and the adjoint representation of a Poisson Lie $n$-algebroid for $n=0,1,2$.
		
		\item \underline{Chapter \ref{Chapter: Linear structures on vector bundles}:} This chapter studies linear structures on vector bundles over graded manifolds. In particular, it recalls the space of linear multivector fields and defines linear $\Q$-manifolds and higher linear Poisson structures.
		
		\item \underline{Chapter \ref{Chapter: Higher split VB-algebroid structures}:} In this chapter, we introduce higher split VB-Lie algebroids and define their Weil algebra and the induced higher fat Lie algebroid. Additionally, we prove the equivalence between higher split VB-Lie algebroids and representations up to homotopy.
		
		\item \underline{Chapter \ref{Chapter: Conclusion, open problems and further research}:} This chapter summarises the results of the thesis and suggests some further open research topics that are related to this work.
		
		\item \underline{Appendices:} In Appendix \ref{Chapter: More on graded geometry}, there are some extra minor results on graded geometry, namely the geometrisation of $\mathbb{N}$-graded vector bundles over $\mathbb{N}$-manifolds, characteristic classes of $1^\text{st}$ order for general Lie $n$-algebroids, and a brief reminder of the Cartan calculus for $\mathbb{Z}$-manifolds. In Appendix \ref{Chapter: Computations}, we provide some long computations that were skipped in the main body of the thesis. 
	\end{itemize}
	
	\paragraph*{Relation to other work}
	
	During the preparation of this work, the author learnt that Caseiro and
	Laurent-Gengoux also consider representations up to homotopy of Lie $n$-algebroids, in particular
	the adjoint representation, in a work in preparation.
	In \cite{Vitagliano15b}, Vitagliano considers representations of strongly homotopy of Lie-Rinehart algebras.
	This kind of algebras are the purely algebraic versions of graded vector bundles
	over graded manifolds equipped with a homological vector field that is tangent to the zero section.
	If the base manifold has grading concentrated in degree 0 and the vector bundle is negatively graded,
	the notion recovers the one of split Lie $n$-algebroids. In that case, Vitagliano's representations
	correspond to the representations up to homotopy considered in this work. In addition, since the DG $\M$-modules considered in this thesis are the sheaves of sections of
	$\Q$-vector bundles, they are themselves also special cases of Vitagliano's strongly homotopy Lie-Rinehart algebras.
	
	Our study of linear structures is related to the great detailed works of La Pastina and Vitagliano \cite{LaPastina20,LaVi18,LaVi19}, where they consider linear structures of Lie algebroids and Lie groupoids (called there vector bundles over Lie algebroids and Lie groupoids) and study their deformations. Most of the language they use is differential geometric, in contrast to our approach which is algebraic and more general since it covers Lie $n$-algebroids for general $n\in\mathbb{N}$.
	
	\paragraph*{Notation and conventions}
	Ordinary letters such as $M,N,\ldots, A,E,V,F,\ldots, X,Y,Z,\ldots$ denote classical differential geometric objects, i.e.~(smooth) manifolds, (smooth) vector bundles, vector fields over smooth manifolds, etc. Underlined ordinary symbols denote graded objects of classical differential geometry, e.g., $\E\to M$ is a smooth graded vector bundle over the smooth manifold $M$, $\underline{\Hom}(\E,\F)$ denotes the space of graded homomorphisms between $\E\to M$ and $\F\to M$, $\underline{S}(\E)$ denotes the graded symmetric algebra of $\E\to M$, etc. Calligraphic letters such as $\M,\N,\ldots,\mathcal{A},\Em,\mathcal{F},\ldots,\Q,\X,\Y,\mathcal{Z},\ldots$ are used for objects of graded geometry, e.g., graded manifolds, graded vector bundles and modules, graded vector fields on graded manifolds, etc. 
	
	We work in the smooth category, unless otherwise stated. Our bundles are always assumed to be vector bundles and of finite rank, even if it is not stated explicitly. For many of our formulae, we adopt Einstein's summation convention, i.e.~repeated upper and lower indices are assumed to be summed over all possible values.
	
	\newpage

	\vspace*{\fill}
	\begin{center}
		{\Large Summary of equivalences between graded and classical geometries}
	\end{center}
	
	\begin{center}
		\begin{tabular}{|| c | c | c ||} 
			\hline
			\textbf{Graded geometry} & \textbf{Classical geometry} & \textbf{Source} \\ 
			\hline\hline
			Degree 2 $\mathbb{N}$-manifolds & metric VB-algebroids & \cite{delCarpio-Marek15,Jotz18b} \\
			\hline
			Degree $n$ $\mathbb{N}\Q$-manifolds/Lie $n$-algebroids & split Lie $n$-algebroids & \cite{BoPo13,ShZh17} \\ 
			\hline
			Lie 1-algebroids & Lie algebroids & \cite{Vaintrob97} \\ 
			\hline
			Lie 2-algebroids & VB-Courant algebroids & \cite{Li-Bland12,Jotz19b} \\
			\hline
			Poisson Lie 1-algebroids & Lie bialgebroids & \cite{Voronov02} \\
			\hline
			Poisson Lie 2-algebroids & LA-Courant algebroids & \cite{Li-Bland12,Jotz18d} \\
			\hline
			Symplectic Lie 1-algebroids & Poisson manifolds & \cite{Severa05} \\
			(Lagrangian submanifolds) & (coisotropic submanifolds) & (\cite{CaFe04}) \\
			\hline
			Symplectic Lie 2-algebroids & Courant algebroids & \cite{Roytenberg02,Severa05} \\
			(Lagrangian submanifolds) & (Dirac structures) & (\cite{Roytenberg02,Li-Bland12}) \\
			\hline
			Symplectic $\Q$-manifolds $T^*[k]A[1]$ & higher Courant algebroids $A\oplus\wedge^{k-1}A^*$ & \cite{Cueca19a,Cueca19} \\ 
			\hline
			Degree 1 $\Q$-groupoids & LA-groupoids & \cite{Mehta09a} \\ 
			\hline
		\end{tabular}
	\end{center}
	\vspace*{\fill}
	
	\chapter{Preliminaries}\label{Chapter: Preliminaries}
	
	In this chapter, we provide all the necessary notation and background that is required for understanding the mathematical content of this thesis. In what follows here and in the next chapters, objects like manifolds, maps, vector bundles, etc., are considered to be smooth even if it is not explicitly mentioned.

	\section{Graded vector bundles and complexes}
	
	Given two ordinary vector bundles $E\to M$ and $F\to N$, there
	is a bijection between vector bundle morphisms $\phi\colon E\to F$
	covering $\phi_0\colon M\to N$ and morphisms of modules
	$\phi^\star\colon \Gamma(F^*)\to \Gamma(E^*)$
	over the pull-back
	$\phi_0^*\colon C^\infty(N)\to C^\infty(M)$. Explicitly, the
	map $\phi^\star$ is defined by
	$\phi^\star(f)(m)=\phi_m^*f_{\phi_0(m)}$ for
	$f\in\Gamma(F)$ and $m\in M$, where $\phi_m:E_m\to F_{\phi(m)}$ is the induced linear map between the fibres over the points $m\in M$ and $\phi_0(m)\in N$.
	
	Similarly as before, underlined symbols denote
	graded objects also in the geometric setting. For instance, a \textbf{graded vector bundle over a smooth manifold $M$}\index{graded vector bundle over a smooth manifold} is a
	vector bundle $q\colon \underline{E}\to M$ together with a
	direct sum decomposition
	\[
	\underline{E}=\bigoplus_{n\in\mathbb{Z}}E_n[n]
	\]
	of ordinary (non-graded) vector bundles $E_n$ over $M$. The finiteness assumption
	for the rank of $\E$ that we will generally assume implies that $\E$ is both upper and lower
	bounded, i.e.~there exists an $n_0\in\mathbb{N}$ such that
	$E_n=0$ for all $|n|>n_0$. Here, an element $e\in E_n$ is
	\textit{(degree-)homogeneous}\index{degree-homogeneous element} of degree $|e|=-n$. That is, for
	$k\in\Z$, the degree $k$ component of $\E$ (denoted with upper
	index $\E^k$) equals $E_{-k}$.
	
	All the usual algebraic constructions from the theory of ordinary
	vector bundles extend to the graded setting by  applying pointwise the constructions of linear algebra. More precisely,
	for graded vector bundles $\E,\underline{F}$, the dual $\E^*$,
	the direct sum $\E\oplus\underline{F}$, the space of graded
	homomorphisms $\underline{\Hom}(\E,\underline{F})$, the tensor
	product $\E\otimes\underline{F}$, and the symmetric and
	antisymmetric powers $\underline{S}(\E)$ and
	$\underline{A}(\E)$ are all well-defined in the realm of graded vector bundles.

	A \textbf{(cochain) complex of vector bundles}\index{cochain complex of vector bundles} is a graded vector
	bundle $\underline{E}$ over $M$ equipped with a degree one endomorphism
	(called the \textbf{differential}\index{cochain complex of vector bundles!differential})
	\[
	\ldots\overset{\partial}{\to}E_{i+1}\overset{\partial}{\to}E_{i}\overset{\partial}{\to}E_{i-1}\overset{\partial}{\to}\ldots
	\]
	which squares to zero; $\partial^2=0$.
	
	Given two complexes $(\underline{E}, \partial)$ and
	$(\underline{F},\partial)$, one may construct new complexes by
	considering all the constructions discussed
	above. Namely, the bundles $\underline{S}(\underline{E})$,
	$\underline{A}(\underline{E})$, $\underline{E}^*$,
	$\underline{\Hom}(\underline{E},\underline{F})$ and
	$\underline{E}\otimes\underline{F}$ inherit a degree one
	operator that squares to $0$. The basic principle for all the
	constructions is formal the graded derivation rule. For example, for
	$\phi\in\underline{\Hom}(\underline{E},\underline{F})$ and
	$e\in \underline{E},f\in\F$:
	\[
	\partial_{\F}(\phi(e)) = \partial_{\underline{\Hom}}(\phi)(e) + (-1)^{|\phi|}\phi(\partial_{\E}(e))	
	\]
	\[
	\partial_{\E\otimes\F}(e\otimes f) = \partial_{\E}(e)\otimes f + (-1)^{|e|}e\otimes\partial_{\F}(f).
	\]
	Using the language of (graded) commutators, the differential of $\underline{\Hom}(\underline{E},\underline{F})$ can be expressed formally as
	\[
	\partial(\phi) := \partial\circ\phi -
	(-1)^{|\phi|}\phi\circ\partial=\partial\circ\phi -
	(-1)^{|\phi|\cdot|\partial|}\phi\circ\partial =: [\partial,\phi].
	\]
	
	The \textbf{shift functor $[k]$}\index{shift functor $[k]$ for complexes of vector bundles}, for $k\in\mathbb{Z}$, yields a new
	complex $(\E[k],\partial[k])$ whose $i$-th component is
	$\E[k]^i=\E^{i+k}=E_{-i-k}$ with differential
	$\partial[k]=\partial$. Formally, $\E[k]$ is obtained by
	tensoring with $(\mathbb{R}[k],0)$ from the right\footnote{If
		one chose to tensor from the left, the resulting complex
		would still have $i$-th component $\E[k]^i=\E^{i+k}$, but
		the Leibniz rule would give the differential
		$\partial[k] = (-1)^k \partial$.}:
	\[
	\partial[k](e\otimes 1) = \partial(e)\otimes1 + (-1)^{|e|}e\otimes 0 = \partial(e)\otimes1
	\]
	for all homogeneous $e\in\E$. A \textbf{degree $k$ morphism}\index{cochain complex of vector bundles!degree $k$ morphism}
	between two complexes $(\underline{E},\partial)$ and
	$(\underline{F},\partial)$ over $M$, or simply \textbf{$k$-morphism}\index{cochain complex of vector bundles!$k$-morphism},
	is, by definition, a degree preserving morphism
	$\phi\colon \underline{E}\to \underline{F}[k]$ over the
	identity on $M$; that is, a family of vector bundle maps
	$\phi_i\colon \E^i\to F[k]^i$ over the identity on $M$ that
	commutes with the differentials:\footnote{This becomes
		$\phi\circ\partial=(-1)^k\partial\circ\phi$ for the other
		convention.} $\phi\circ\partial=\partial\circ\phi$.
	
	\section{Algebroids and basic connections}
	
	In this section, we give the basic definitions and constructions regarding the different notions of algebroids which will be used in the rest of this work.
	
	A \textbf{dull algebroid}\index{dull algebroid} \cite{Jotz18a} is a vector bundle $Q\to M$
	endowed with an \textbf{anchor map}\index{anchor map}  $\rho_Q\colon Q\to TM$ and a \textbf{bracket}\index{bracket}
	(i.e.~an $\mathbb R$-bilinear map)
	$[\cdot\,,\cdot]\colon\Gamma(Q)\times\Gamma(Q)\to\Gamma(Q)$ on
	its space of sections $\Gamma(Q)$, such that
	\begin{equation}\label{comp_anchor_bracket}
	\rho_Q[q_1,q_2] = [\rho_Q(q_1),\rho_Q(q_2)]
	\end{equation}
	and the Leibniz identity is satisfied in both entries:
	\[
	[f_1q_1,f_2q_2] = f_1f_2[q_1,q_2] +
	f_1\rho_Q(q_1)f_2\cdot q_2 - f_2\rho_Q(q_2)f_1\cdot
	q_1,
	\]
	for all $q_1,q_2\in\Gamma(Q)$ and all $f_1,f_2\in
	C^\infty(M)$. If the bracket is in addition skew-symmetric, dull algebroids are also found in the literature as \textbf{almost Lie algebroids}\index{almost Lie algebroid} \cite{PoPo19}.

	A dull algebroid gives the usual notion of a \textbf{Lie algebroid}\index{Lie algebroid} if its
	bracket is also skew-symmetric and satisfies the Jacobi
	identity
	\[
	\Jac_{[\cdot\,,\cdot]}(q_1,q_2,q_3) := [q_1,[q_2,q_3]] - [[q_1,q_2],q_3] - [q_2, [q_1,q_3]] = 0,
	\]
	for all $q_1,q_2,q_3\in\Gamma(Q)$. 
	
	Given a skew-symmetric dull algebroid
	$Q$, there is an associated operator $\diff_Q$ of degree 1 on
	the \textbf{space of $Q$-forms}\index{dull algebroid!forms}
	$\Omega^\bullet(Q) = \Gamma(\wedge^\bullet Q^*)$, defined by
	the formula
	\begin{align*}
	\diff_Q\tau(q_1,\ldots,q_{k+1}) = & \sum_{i<j}(-1)^{i+j}\tau([q_i,q_j],q_1,\ldots,\hat{q_i},\ldots,\hat{q_j},\ldots,q_{k+1}) \\
	& + \sum_i(-1)^{i+1}\rho_Q(q_i)(\tau(q_1,\ldots,\hat{q_i},\ldots,q_{k+1})),
	\end{align*}
	for $\tau\in\Omega^k(Q)$ and $q_1,\ldots,q_{k+1}\in\Gamma(Q)$; the notation $\hat{q}$ means that $q$ is omitted. The operator $\diff_Q$ satisfies as usual
	\[
	\diff_Q(\tau_1\wedge\tau_2) = (\diff_Q\tau_1)\wedge\tau_2 + (-1)^{|\tau_1|}\tau_1\wedge \diff_Q\tau_2,
	\]
	for $\tau_1,\tau_2\in\Omega^\bullet(Q)$. In general, the
	operator $\diff_Q$ squares to zero only on 0-forms
	$f\in \Omega^0(M)=C^\infty(M)$, since $\diff_Q^2f=0$ for all
	$f\in C^\infty(M)$ is equivalent to the compatibility of the
	anchor with the bracket \eqref{comp_anchor_bracket}. The
	identity $\diff_Q^2 = 0$ on $\Omega^1(Q) = \Gamma(Q^*)$ and consequently to all forms in $\Omega^{\bullet}(Q)$ due to $\Omega^1(Q)$-linearity of $\diff_{Q}^2$ is equivalent to
	$(Q,\rho_Q,[\cdot\,,\cdot])$ being a Lie algebroid. 
	
	Let now $Q\to M$ be a skew-symmetric dull algebroid and
	suppose that $E\to M$ is another vector bundle. A
	\textbf{$Q$-connection on $E$}\index{dull algebroid!connection on a vector bundle} is defined similarly as the usual case of Lie algebroids, as a map
	$\nabla\colon\Gamma(Q)\times\Gamma(E)\to\Gamma(E),(q,e)\mapsto
	\nabla_q e$ that is $C^\infty(M)$-linear in the first
	argument and satisfies
	\[
	\nabla_q(fe) = \ldr{\rho_Q(q)}f\cdot e + f\nabla_q e,
	\]
	for all $q\in\Gamma(Q),e\in\Gamma(E)$ and $f\in
	C^\infty(M)$. The \textbf{dual connection}\index{dull algebroid!dual connection}
	$\nabla^*$ is the $Q$-connection on $E^*$ defined by the
	formula
	\[
	\langle \nabla_q^* \varepsilon,e \rangle = \ldr{\rho_Q(q)}\langle \varepsilon,e \rangle - \langle \varepsilon,\nabla_q e \rangle,
	\]
	for all $\varepsilon\in\Gamma(E^*),e\in\Gamma(E)$ and
	$q\in\Gamma(Q)$, where $\langle\cdot\,,\cdot\rangle$ is the natural
	pairing between $E$ and its dual $E^*$.

	A \textbf{$Q$-connection on a graded vector bundle
	$(\underline{E}=\bigoplus_{n\in\mathbb Z} E_n[n], \partial)$}\index{dull algebroid!connection on a graded vector bundle}
	is a family of $Q$-connections $\nabla^n$, $n\in\mathbb Z$, on
	each of the bundles $E_n$. If $\underline{E}$ is a complex
	with differential $\partial$, then the $Q$-connection is a
	\textbf{connection on the complex $(\underline{E},\partial)$}\index{dull algebroid!connection on a complex of vector bundles}
	if it commutes with $\partial$, i.e.~$\partial(\nabla^{n}_qe)=\nabla_q^{n-1}(\partial e)$ for
	$q\in\Gamma(Q)$ and $e\in \Gamma(E_n)$.
	
	The \textbf{curvature of a $Q$-connection}\index{curvature of a connection} on a vector bundle $E$ is
	defined by
	\[
	R_\nabla(q_1,q_2)e = \nabla_{q_1}\nabla_{q_2} e - \nabla_{q_2}\nabla_{q_1} e - \nabla_{[q_1,q_2]} e,
	\] 
	for all $q_1,q_2\in\Gamma(Q)$ and $e\in\Gamma(E)$, and
	generally, it is an element of
	$\Gamma(Q^*\otimes Q^*\otimes E^*\otimes E)$. If the dull
	bracket of $Q$ is skew-symmetric, then the curvature is a
	2-form with values in the endomorphism bundle
	$\End(E)=E^*\otimes E$: $R_\nabla\in\Omega^2(Q,\End(E))$. A
	connection is called as usual \textbf{flat}\index{dull algebroid!flat connection} if its curvature
	$R_\nabla$ vanishes identically.
	
	Given a $Q$-connection $\nabla$ on $E$, and assuming that
	$[\cdot\,,\cdot]$ is skew-symmetric, there is an induced
	operator $\diff_\nabla$ on the space of $E$-valued $Q$-forms
	$\Omega^\bullet(Q,E) = \Omega^\bullet(Q)\otimes_{C^\infty(M)} \Gamma(E)$
	given by the usual Koszul formula
	\begin{align*}
	\diff_\nabla\tau(q_1,\ldots,q_{k+1}) = & \sum_{i<j}(-1)^{i+j}\tau([q_i,q_j],q_1,\ldots,\hat{q_i},\ldots,\hat{q_j},\ldots,q_{k+1}) \\
	& + \sum_i(-1)^{i+1}\nabla_{q_i}(\tau(q_1,\ldots,\hat{q_i},\ldots,q_{k+1})),
	\end{align*}
	for all $\tau\in\Omega^k(Q,E)$ and $q_1,\ldots,q_{k+1}\in\Gamma(Q)$. It satisfies
	\[
	\diff_\nabla(\tau_1\wedge\tau_2) = \diff_Q\tau_1\wedge\tau_2 + (-1)^{k}\tau_1\wedge \diff_\nabla\tau_2,
	\]
	for all $\tau_1\in\Omega^k(Q)$ and $\tau_2\in\Omega^\bullet(Q,E)$,
	and squares to zero if and only if $Q$ is a Lie algebroid and $\nabla$ is flat.
	
	Suppose that
	$\nabla\colon \mathfrak{X}(M)\times\Gamma(Q)\to\Gamma(Q)$ is a
	$TM$-connection on the vector bundle $Q$. The induced
	\textbf{basic connections}\index{basic connections} on $Q$ and $TM$ are defined
	similarly as the ones associated to Lie algebroids \cite{GrMe10,ArCr12}:
	\[
	\nabla^{\text{bas}}=\nabla^{\text{bas},Q}\colon \Gamma(Q)\times\Gamma(Q)\to\Gamma(Q),\
	\nabla^{\text{bas}}_{q_1} q_2 = [q_1,q_2] +
	\nabla_{\rho_Q(q_2)} q_1
	\]
	and
	\[
	\nabla^{\text{bas}}=\nabla^{\text{bas},TM}\colon\Gamma(Q)\times\mathfrak{X}(M)\to\mathfrak{X}(M),\
	\nabla^{\text{bas}}_{q} X = [\rho_Q(q),X] + \rho_Q(\nabla_X
	q).
	\]
	The \textbf{basic curvature}\index{basic curvature} is the 2-form
	$R_\nabla^{\text{bas}}\in \Omega^2(Q, \Hom(TM,Q))$
	defined by
	\[
	R_\nabla^{\text{bas}}(q_1,q_2)X = -\nabla_X[q_1,q_2] +
	[q_1,\nabla_Xq_2] +[\nabla_Xq_1,q_2] +
	\nabla_{\nabla_{q_2}^\text{bas}X}q_1 -
	\nabla_{\nabla_{q_1}^\text{bas}X}q_2.
	\]
	The
	basic connections and the basic curvature satisfy
	\begin{equation}\label{eq_bas_Q_1}
	\nabla^{\text{bas},TM}\circ\rho_Q = \rho_Q\circ\nabla^{\text{bas},Q},
	\end{equation}
	\begin{equation}\label{eq_bas_Q_2}
	\rho_Q\circ R_\nabla^{\text{bas}} = R_{\nabla^{\text{bas},TM}},
	\end{equation}
	\begin{equation}\label{eq_bas_Q_3}
	R_\nabla^{\text{bas}}\circ\rho_Q + \Jac_{[\cdot\,,\cdot]} = R_{\nabla^{\text{bas},Q}}.
	\end{equation}
	
	\section{Double vector bundles, linear splittings and duals}\label{Section: Double vector bundles, linear splittings and duals}
	
	Recall that a \textbf{double vector bundle $(D,V,F,M)$}\index{double vector bundle} is a commutative diagram
	\[
	\begin{tikzcd}
	D \arrow[dd,"\pi_V"'] \arrow[rr,"\pi_F"] & & F \arrow[dd,"q_F"] \\
	& & \\
	V \arrow[rr,"q_V"']             &          & M           
	\end{tikzcd}
	\]
	such that all the arrows are vector bundle projections and the
	structure maps of the bundle $D\to V$ are bundle morphisms over the
	corresponding structure maps of $F\to M$. More precisely, a commutative square $(D,V,F,M)$ as above is a double vector bundle if the following three conditions hold: 
	\begin{enumerate}
		\item all four arrows are vector bundles;
		
		\item the bundle projection $\pi_V$ is a vector bundle morphism over the bundle projection $q_F$;
		
		\item the addition map $+_V\colon D\times_V D\to D$ of the vector bundle $D\to V$ is a vector bundle morphism over the addition map $+\colon F\times_M F\to F$ of the vector bundle $F\to M$.
	\end{enumerate}
	We refer the reader to \cite{Mackenzie05} for a more detailed treatment of double vector bundles. The above conditions for the structure maps of $D\to V$
	are equivalent to the same conditions holding for the structure maps of $D\to F$
	over $V\to M$. The bundles $V$ and $F$ are called \textbf{the side bundles of
	$D$}\index{double vector bundle!sides}. The intersection of the kernels
	$C:=\pi_V^{-1}(0^V)\cap\pi_F^{-1}(0^F)$ is the \textbf{core of $D$}\index{double vector bundle!core}
	and is naturally a vector bundle over $M$, with projection denoted by
	$q_C\colon C\to M$. The inclusion $C\hookrightarrow D$ is denoted by
	$C_m\ni c_m\mapsto\overline{c_m}\in
	\pi_V^{-1}(0^V_m)\cap\pi_F^{-1}(0^F_m)$.
	
	A \textbf{morphism $(G_D,G_V,G_F,g)$ of two double vector bundles}\index{double vector bundle!morphism} $(D,V,F,M)$ and $(D',V',F',M')$ is a commutative cube
	\[
	\begin{tikzcd}
	& D \arrow[dl, "G_D"] \arrow[rr] \arrow[dd] & & F \arrow[dl, "G_F"] \arrow[dd] \\
	D' \arrow[rr, crossing over] \arrow[dd] & & F' \\
	& V \arrow[dl, "G_V"] \arrow[rr] & & M \arrow[dl, "g"] \\
	V' \arrow[rr] & & M' \arrow[from=uu, crossing over]
	\end{tikzcd}
	\]
	such that all the faces are vector bundle maps. 
	
	Given a double vector bundle $(D,V,F,M)$, the space of sections of $D$
	over $V$, denoted by $\Gamma_V(D)$, is generated as a
	$C^\infty(V)$-module by two special types of sections, called
	\textbf{core}\index{double vector bundle!core section} and \textbf{linear}\index{double vector bundle!linear section} sections and denoted by
	$\Gamma_V^c(D)$ and $\Gamma^l_V(D)$, respectively (see \cite{Mackenzie05}). The core sections are induced from sections of $C$ and more precisely the
	core section $c^\dagger\in\Gamma_V^c(D)$ corresponding to
	$c\in\Gamma(C)$ is defined as
	\[
	c^\dagger(v_m) = 0_{v_m}^D +_F \overline{c(m)},\, \text{ for }\, m\in M \, \text{ and }\, v_m\in V_m. 
	\]
	A section $\delta\in\Gamma_V(D)$ is linear over $f\in\Gamma(F)$, if $\delta\colon V\to D$ is a
	vector bundle morphism  $V\to D$ over $f\colon M\to F$.
	
	Finally, a section $\psi\in\Gamma(V^*\otimes C)$ defines a linear
	section $\psi^\wedge\colon V\to D$ over the zero section $0^F\colon M\to F$ by
	\[
	\psi^\wedge(v_m) = 0_{v_m}^D +_F \overline{\psi(v_m)} 
	\]
	for all $m\in M$ and $v_m\in V_m$. This type of linear section
	is called a \textbf{core-linear}\index{double vector bundle!core-linear section} section. In terms of the generators
	$\theta\otimes c\in\Gamma(V^*\otimes C)$, the correspondence above reads
	$(\theta\otimes c)^\wedge=\ell_\theta\cdot c^\dagger$, where
	$\ell_\theta$ is the linear function on $V$ associated to
	$\theta\in\Gamma(V^*)$.
	
	\begin{example}[Decomposed double vector bundle]\label{Example decomposed DVB}
		Let $V,F,C$ be vector bundles over the same manifold $M$. Set
		$D:=V\times_M F\times_M C$ with vector bundle structures
		$D=q_V^!(F\oplus C)\to V$ and $D=q_F^!(V\oplus C)\to F$. Then
		$(D,V,F,M)$ is a double vector bundle, called the \textbf{decomposed double
		vector bundle}\index{decomposed double
		vector bundle} with sides $V$ and $F$ and with core $C$. Its core
		sections over $F$ have the form $c^\dagger\colon f_m\mapsto(0^V_m,f_m,c(m))$, for
		$m\in M,f_m\in F_m$ and $c\in\Gamma(C)$, and the space of linear
		sections $\Gamma_V^l(D)$ is naturally identified with
		$\Gamma(F)\oplus\Gamma(V^*\otimes C)$ via 
		$(f,\psi)\colon v_m\mapsto(f(m),v_m,\psi(v_m))$ where
		$\psi\in\Gamma(V^*\otimes C)$ and $f\in\Gamma(F)$. This yields the
		canonical \textbf{linear horizontal lift}\index{decomposed double vector bundle!linear horizontal lift}
		$h\colon \Gamma(F)\hookrightarrow\Gamma_V^l(D)$. Given two decomposed double vector bundles $D:=V\times_M F\times_M C$ and $D':=V'\times_M F'\times_M C'$ over the same base manifold $M$, any double vector bundle morphism\index{decomposed double vector bundle!morphism} from $D$ to $D'$ is of the form
		\[
		(v,f,c) \mapsto (G_V(v),G_F(f),G_C(c) + \Phi_v(f))
		\]
		for all sections $v\in\Gamma(V),f\in\Gamma(F),\Gamma(C)$ \cite{DrJoOr15}, where $G_C\colon C\to C'$ is a vector bundle morphism and $\Phi\in\Gamma(V^*\otimes F^* \otimes C')$.
	\end{example}
	
	\begin{example}[Tangent bundle of a vector bundle]
		Given a vector bundle $q\colon E\to M$, its tangent bundle $TE$ is
		naturally a vector bundle over the manifold $E$. In addition,
		applying the tangent functor to all the structure maps of $E\to M$
		yields a vector bundle structure on $Tq\colon TE\to TM$ which is
		called the \textbf{tangent prolongation}\index{tangent prolongation of a vector bundle} of $E$. Hence,
		$(TE,TM,E,M)$ has a natural double vector bundle structure with
		sides $TM$ and $E$:
		\[
		\begin{tikzcd}
		TE \arrow[dd,"Tq"'] \arrow[rr,"p_E"] & & E \arrow[dd,"q"] \\
		& & \\
		TM \arrow[rr,"p_M"']         &              & M           
		\end{tikzcd}
		\]
		Its core is naturally identified with $E\to M$
		and the inclusion $E\hookrightarrow TE$ is given by
		$E_m\ni e_m\mapsto\left.\frac{d}{dt}\right|_{t=0}te_m\in
		T^q_{0_m^E}E$. For $e\in\Gamma(E)$, the section
		$Te\in\Gamma_{TM}^l(TE)$ is linear over $e$. The  core
		vector field $e^\dagger \in\Gamma_{TM}(TE)$ is defined by
		$e^\dagger(v_m)=T_m0^E(v_M)+_{E}\left.\frac{d}{dt}\right.\arrowvert_{t=0}te(m)$
		for $m\in M$ and $v_m\in T_ MM$ and the \textbf{vertical
			lift}\index{vertical lift} $e^\uparrow\in \Gamma_E(TE)=\mathfrak{X}(E)$ is the (core)
		vector field defined by the flow
		$\mathbb{R}\times E\to E,(t,e'_m)\mapsto e'_m + te(m)$. Elements of
		$\Gamma_E^l(TE)=:\mathfrak{X}^l(E)$ are called \textbf{linear
			vector fields}\index{linear vector field} and are equivalent to derivations
		$\delta\colon \Gamma(E)\to\Gamma(E)$ over some element in
		$\mathfrak{X}(M)$ \cite{Mackenzie05}. We denote by $X_\delta$ the linear vector field which corresponds to the derivation $\delta$. The construction goes as follows: First note that derivations of the vector bundle $q:E\to M$ are in correspondence with derivations of the dual vector bundle $q':E^*\to M$ via the map that sends $\delta$ over $X\in\mathfrak{X}(M)$ to the derivation $\delta^*$ over $X$ defined by $X\langle\varepsilon,e\rangle = \langle \delta^*\varepsilon,e \rangle + \langle \varepsilon,\delta e \rangle$ for $e\in\Gamma(E)$ and $\varepsilon\in\Gamma(E^*)$. Then the correspondence $\delta \longleftrightarrow X_\delta$ is obtained via the formulae
		\[
		X_\delta(\ell_\varepsilon) = \ell_{\delta^*\varepsilon}
		\qquad \text{and} \qquad
		X_\delta(q^*f) = q^*(X(f))
		\]
		for $f\in C^\infty(M)$ and $\varepsilon\in\Gamma(E^*)$.
	\end{example}
	
	\begin{example}[Double tangent bundle]\label{Double tangent bundle}
		Applying the construction of the previous example to the vector bundle $q_M:TM\to M$, we obtain the double vector bundle
		\[
		\begin{tikzcd}
		TTM \arrow[dd,"Tq_M"'] \arrow[rr,"q_{TM}"] & & TM \arrow[dd,"q_M"] \\
		& & \\
		TM \arrow[rr,"q_M"']            &           & M           
		\end{tikzcd}
		\]
		called the \textbf{double tangent bundle}\index{double tangent bundle} of the manifold $M$. There is a canonical isomorphism of double vector bundles $J_M:TTM\to TTM$  over the identities on the sides $TM$ and $M$ which interchanges the sides, called the \textbf{canonical involution}\index{canonical involution} of $TTM$ \cite{Mackenzie05}. To describe it, we write elements $(d,X^h,X^v,\gamma)\in TTM$ as second derivatives
		\[
		d=\frac{\partial^{2}\gamma}{\partial t\,\partial s}(0,0)
		\]
		where $\gamma:(-\varepsilon,\varepsilon)\times(-\varepsilon,\varepsilon)\rightarrow M$ is smooth, and the notation means that
		$\gamma$ is first differentiated with respect to the variable $s$, yielding a smooth curve of tangent vectors
		$X^v_{t}=\frac{\partial \gamma}{\partial s}(t,0)$ in $TM$ whose velocity at $0$ gives $d$: $\frac{d}{dt}|_{t=0}X^v_{t}=d$.
		That is, we have that
		\[
		\frac{\partial \gamma}{\partial s}(0,0)=q_{TM}(d)
		\qquad \text{and} \qquad
		\frac{\partial \gamma}{\partial t}(0,0)=Tq_{M}(d).
		\]
		The map $J_M:TTM\to TTM$ is then given as
		\[
		d=\frac{\partial^{2}\gamma}{\partial t\,\partial s}(0,0)\mapsto
		J_M(d):=\frac{\partial^{2}\gamma}{\partial s\, \partial t}(0,0).
		\]
	\end{example}
	
	A \textbf{linear splitting}\index{double vector bundle!linear splitting} of a double vector bundle $(D,V,F,M)$ with
	core $C$ is a double vector bundle embedding $\Sigma$ of the
	decomposed double vector bundle $V\times_M F$ into $D$ over the
	identities on $V$ and $F$. It is well-known that every double vector
	bundle admits a linear splitting, see
	\cite{GrRo09,delCarpio-Marek15,Pradines77} or \cite{HeJo18} for the
	general case. Moreover, a linear splitting is equivalent to a
	\textbf{decomposition}\index{double vector bundle!decomposition} of $D$, i.e.~to an isomorphism of double vector
	bundles $S:V\times_M F\times_M C\to D$ over the identity on $V, F$ and
	$C$. Given $\Sigma$, the decomposition is obtained by setting
	$S(v_m,f_m,c_m)=\Sigma(v_m,f_m) +_F (0_{f_m} +_V \overline{c_m})$, and
	conversely, given $S$, the splitting is defined by
	$\Sigma(v_m,f_m)=S(v_m,f_,,0_m^C)$.
	
	A linear splitting of $D$, and consequently a decomposition of $D$, is also equivalent to a \textbf{(linear) horizontal lift}\index{double vector bundle!linear horizontal lift}, i.e.~a right splitting of
	the short exact sequence
	\begin{equation}\label{Short exact sequence of linear sections}
		0\longrightarrow\Gamma(V^*\otimes C)\longrightarrow \Gamma_V^l(D)\longrightarrow \Gamma(F)\longrightarrow 0
	\end{equation}
	of $C^\infty(M)$-modules.  The correspondence is given by
	$\sigma_F(f)(v_m)=\Sigma(v_m,f_m)$ for $f\in\Gamma(F)$, $m\in M$ and
	$v_m\in V_m$. Note that all the previous constructions can be done
	similarly if one interchanges the roles of $V$ and $F$.
	
	\begin{example}
		For the tangent bundle $TE$ of a vector bundle $E\to M$, a linear
		splitting is equivalent to a choice of a $TM$-connection on
		$E$. Specifically, given a horizontal lift
		$\sigma\colon \mathfrak{X}(M)\to\mathfrak{X}^l(E)$, the corresponding
		connection $\nabla$ is defined by formula
		\[
		\sigma(Y) = X_{\nabla_Y}\in\mathfrak{X}^l(E) \qquad \text{for all}\  Y\in\mx(M).
		\]
	\end{example}
	
	Double vector bundles can be dualised in two ways, namely, as the dual
	of $D$ either over $V$ or over $F$ \cite{Mackenzie05}. Precisely, from a double vector
	bundle $(D,V,F,M)$ with core $C$, one obtains the \textbf{dual double vector bundles}\index{double vector bundle!duals}
	\begin{center}
		\begin{tabular}{l r}
			\begin{tikzcd}
			D^*_V \arrow[rr, "\pi_{C^*}"] \arrow[dd, "\pi_V"'] & & C^* \arrow[dd, "q_{C^*}"] \\
			& & \\
			V \arrow[rr, "q_V"']       &                        & M                       
			\end{tikzcd}
			&
			\begin{tikzcd}
			D^*_F \arrow[rr, "\pi_F"] \arrow[dd, "\pi_{C^*}"'] & & F \arrow[dd, "q_F"] \\
			& & \\
			C^* \arrow[rr, "q_{C^*}"']            &                   & M                       
			\end{tikzcd}
		\end{tabular}
	\end{center}
	with cores $F^*$ and $V^*$, respectively; here, there is a little abuse of notation for the names of the bundle projections. The vertical bundles structures of $(D_V^*,V,C^*,M)$ are the usual duals of the vector bundles $D\to V$ and $C\to M$. The projection $\pi_{C^*}:D_V^*\to C^*$ is defined by
	\[
	\langle \pi_{C^*}(\Phi),c_m \rangle = \langle \Phi,0_{v_m}^D +_F \overline{c_m} \rangle,
	\]
	where $m\in M, c_m\in C_m, \Phi:D_V|_{v_m} = \pi_V^{-1}(v_m)\to \mathbb{R}$ and $v_m\in V_m$, and the addition $+_{C^*}$ of $D_V^*$ over $C^*$ is defined by
	\[
	\langle \Phi_1 +_{C^*} \Phi_2,d_1 +_F d_2 \rangle = \langle \Phi_1,d_1 \rangle + \langle \Phi_2, d_2 \rangle,
	\]
	where $\Phi_1,\Phi_2$ are elements of $(D_V^*,V,C^*,M)$ of the form $(\Phi_1,v_1,\gamma,m)$ and $(\Phi_2,v_2,\gamma,m)$, respectively (i.e.~$v_1=\pi_V(\Phi_1), \gamma=\phi_{C^*}(\Phi_1), m=q_V(\pi_V(\Phi_1))=q_{C^*}(\pi_{C^*}(\Phi_1))$ and similarly for $\Phi_2$). In this definition we use that every element $(d,v_1+v_2,\gamma,m)$ can be written as a sum $d = d_1 +_F d_2$ with $(\Phi_1,v_1,\gamma,m)$ being arbitrary and $d_2:=d -_F d_1$. It can be shown that the definition is independent of the choice of $d_1$ and $d_2$ (see \cite{Mackenzie05} for more details).
	
	Given a linear splitting $\Sigma\colon V\times_M F\to D$, the \textbf{dual (linear) splitting}\index{double vector bundle!dual linear splitting}
	$\Sigma^*\colon V\times_M C^*\to D^*_V$ \cite{GrJoMaMe18,DrJoOr15,Jotz18b} is defined by
	\[
	\left\langle \Sigma^*(v_m,\gamma_m),\Sigma(v_m,f_m) \right\rangle_{D_V} = 0
	\qquad \text{and} \qquad
	\left\langle \Sigma^*(v_m,\gamma_m), c^\dagger(v_m) \right\rangle_{D_V} = \left\langle \gamma_m,c(m) \right\rangle_C,
	\]
	for all $m\in M$ and  $v_m\in V_m$, $f_m\in F_m$, $\gamma_m\in C^*_m$, $c\in\Gamma(C)$, where $\langle\cdot\,,\cdot\rangle_{D_V}$ denotes the pairing of $D\to V$ with $D_V^*\to V$ and $\langle\cdot\,,\cdot\rangle_{C}$ denotes the pairing of $C\to M$ with $C^*\to M$.
	
	\section{Sheaves on topological spaces}
	
	Since graded manifolds, which are the main object of this thesis, are defined in terms of sheaves on ordinary smooth manifolds, it is worth  recalling some basic definitions. Generally, there are many sources on this topic but what follows in this section is based on \cite{Hartshorne77}. Intuitively, the language of sheaf theory is a formal way of expressing local algebraic data on a general topological space (not necessarily a smooth manifold).
	
	Let $X$ be a topological space. We define the category $\text{Top}(X)$ whose objects are the open subsets of $X$ and the set of morphisms between two objects is given by the inclusion of open sets, i.e.~if $U,V\subset X$ are open then
	\[
	\Hom(V,U) := \begin{cases}
	\{*\}, & \text{if}\ V\subset U \\
	\emptyset, & \text{if}\ V\not\subset U,   
	\end{cases}
	\]
	where $\{*\}$ denotes the one-element set.
	
	\begin{definition}
		A \textbf{presheaf $\mathscr{F}$ on $X$ with values in a (fixed) category $\mathcal{C}$}\index{presheaf} is a contravariant functor from $\text{Top}(X)$ to $\mathcal{C}$.   
	\end{definition}
	
	\begin{remark}
		Equivalently, the presheaf $\mathscr{F}$ can be defined as a covariant functor $\text{Top}(X)^{\text{op}}\to\mathcal{C}$; here, the notation means the opposite category, i.e.~for a category $\mathcal{D}$, $\mathcal{D}^{\text{op}}$ has the same objects with $\mathcal{D}$ but its morphisms are given by $\Hom_{\mathcal{D^{\text{op}}}}(A,B)=\Hom_{\mathcal{D}}(B,A)$ for all objects $A,B$ of $\mathcal{D}$.
	\end{remark}
	
	Unravelling the above definition yields that a presheaf $\mathscr{F}$ on $X$ with values in the category $\mathcal{C}$ consists of the following data:
	\begin{itemize}
		\item an object $\mathscr{F}(U)$ of $\mathcal{C}$ for every open subset $U\subset X$,
		
		\item a morphism $\varrho_{UV}:\mathscr{F}(U)\to\mathscr{F}(V)$ for every inclusion of open sets $V\subset U$ such that
		\begin{enumerate}
			\item $\varrho_{UU}$ is the identity map $\mathscr{F}(U)\to\mathscr{F}(U)$,
			
			\item if $W\subset V\subset U$ are three open subsets of $X$, then $\varrho_{UW} = \varrho_{VW}\circ\varrho_{UV}$, i.e.~the following diagram is commutative
			\begin{center}
				\begin{tikzcd}
				\mathscr{F}(U) \arrow[rr, "\varrho_{UW}"] \arrow[dd, "\varrho_{UV}"'] & & \mathscr{F}(W) \\
				& & \\
				\mathscr{F}(V) \arrow[rruu, "\varrho_{VW}"']                          &               
				\end{tikzcd}
			\end{center}
		\end{enumerate}  
	\end{itemize}
	
	\begin{remark}
		\begin{enumerate}
			\item Most frequently, the category $\mathcal{C}$ is taken to be the category of sets, abelian groups, rings, etc. For our purposes, the base space will be a smooth manifold $M$ and the target category $\mathcal{C}$ will be the category of $\mathbb{Z}$-graded, graded commutative, associative, unital $C^\infty(M)$-algebras with degree-preserving morphisms.
			
			\item Usually, if the category $\mathcal{C}$ has a terminal object $0$, then $\mathscr{F}(\emptyset)=0$ is also required, e.g.~$\mathscr{F}(\emptyset) = \{*\}$ for sets, $\mathscr{F}(\emptyset) = \{0\}$ for abelian groups, etc.
		\end{enumerate}
	\end{remark}
	
	The elements of $\mathscr{F}(U)$ are called \textbf{sections over $U$}\index{presheaf!sections} and the maps $\varrho_{UV}$ are called \textbf{restrictions}\index{presheaf!restriction maps}. The restriction $\varrho_{UV}(\sigma), \sigma\in\mathscr{F}(U)$, for $V\subset U$ is also written as $\sigma|_V$.
	
	\begin{definition}
		Let $\mathscr{F}_1$ and $\mathscr{F}_2$ be two sheaves on $X$ with values in $\mathcal{C}$. A \textbf{morphism of presheaves}\index{presheaf!morphism} is a natural transformation of functors. $\phi:\mathscr{F}_1\to\mathscr{F}_2$. An \textbf{isomorphism of sheaves}\index{presheaf!isomorphism} is a natural equivalence of functors.
	\end{definition}
	
	The above definition implies that a morphism of presheaves $\phi:\mathscr{F}_1\to\mathscr{F}_2$ consists of morphisms $\phi_U:\mathscr{F}_1(U)\to\mathscr{F}_2(U)$ of $\mathcal{C}$ which are compatible with the restriction maps: $\phi_V\circ\varrho^1_{UV} = \varrho^2_{UV}\circ\phi_U$, i.e the following diagram commutes for all open subset $V\subset U$ of $X$:
	\begin{center}
		\begin{tikzcd}
		\mathscr{F}_1(U) \arrow[dd, "\varrho^1_{UV}"'] \arrow[rr, "\phi_U"] & & \mathscr{F}_2(U) \arrow[dd, "\varrho^2_{UV}"] \\
		& & \\
		\mathscr{F}_1(V) \arrow[rr, "\phi_V"']                     &        & \mathscr{F}_2(V)                            
		\end{tikzcd}
	\end{center}
	The morphism $\phi:\mathscr{F}_1\to\mathscr{F}_2$ is an isomorphism if the map $\phi_U:\mathscr{F}_1(U)\to\mathscr{F}_2(U)$ is an isomorphism in $\mathcal{C}$ for all open subsets $U\subset X$.
	
	From the discussion above, it follows that presheaves on $X$ with values in the category $\mathcal{C}$ form a category which is the category of functors $\text{PreSh}_\mathcal{C}(X) := \text{Fun}(\text{Top}^{\text{op}}(X),\mathcal{C})$.
	
	\begin{example}[Continuous functions]\label{Example: sheaf of continuous functions}
		The presheaf of continuous functions $C^0$ on $X$
		\[
		C^0(U) = \{f:U\to\mathbb{R}\ |\ f\ \text{continuous}\}
		\]
		with $\varrho_{UV}(f) = f|_V$ is a presheaf of sets, abelian groups, rings and $\mathbb{R}$-algebras. More generally, given another topological space $Y$, one obtains the sheaf of sets
		\[
		C^0_U(X,Y) = \{f:U\to Y\ |\ f\ \text{continuous}\}
		\]
		with $\varrho_{UV}(f) = f|_V$. 
	\end{example}
	
	\begin{example}[Smooth functions]\label{Example: sheaf of smooth functions}
		If $M$ is a smooth manifold, then $C^{\infty}$ defines the presheaf of smooth functions on $M$ which again is a presheaf of sets, abelian groups, rings and $\mathbb{R}$-algebras. Moreover, the inclusion $C^\infty(U) \subset C^0(U)$ for every open $U\subset M$ defines a morphism of presheaves $C^\infty\to C^0$.
	\end{example}
	
	\begin{definition}
		Given two objects $A,B$ of a category $\mathcal{C}$ together with two morphisms $f,g:A\to B$, the \textbf{equaliser of $f,g:A\to B$}\index{equaliser} consists of an object $E$ in $\mathcal{C}$ and a morphism $e:E\to A$ satisfying
		\begin{enumerate}
			\item $f\circ e = g\circ e$, and such that
			
			\item given any object $C$ in $\mathcal{C}$ and morphism $h:C\to A$, then there exists a unique $u:C\to E$ such that $e\circ u = h$.
		\end{enumerate}
	\end{definition}
	
	That is, the equaliser of $f,g:A\to B$ fits in the following commutative diagram
	\begin{center}
		\begin{tikzcd}
		E \arrow[rr, "e"]              &                    & A \arrow[rr, shift left=.75ex, "f"]\arrow[rr, shift right=.75ex, "g"'] & & B \\
		&&&& \\
		C \arrow[rruu, "h"'] \arrow[uu, "\exists!u", dashed] &           &            &  
		\end{tikzcd}
	\end{center}
	
	Assume now that all products $\prod_{i\in I}$ exist in the category $\mathcal{C}$.
	
	\begin{definition}
		A \textbf{sheaf}\index{sheaf} is a presheaf $\mathscr{F}$ such that for all open covers $U=\bigcup_{i\in I} U_i$, the map
		\[
		e_I:\mathscr{F}(U)\to\prod_{i\in I}\mathscr{F}(U_i), \qquad \sigma\mapsto \left( \sigma|_{U_i} \right)_{i\in I}
		\]
		is an equaliser for
		\[
		f_I,g_I:\prod_{i\in I}\mathscr{F}(U_i) \to \prod_{i,j\in I}\mathscr{F}(U_{ij}),
		\qquad
		f_I((\sigma_i)_{i\in I}) = \left(\sigma_i|_{U_{ij}}\right)_{i,j\in I}, 
		g_I((\sigma_i)_{i\in I}) = \left(\sigma_j|_{U_{ij}}\right)_{i,j\in I}
		\]
		where $U_{ij}:=U_i\cap U_j$.
	\end{definition}
	
	Explicitly, the above definition formalises the following statements:
	\begin{enumerate}
		\item If $U$ is an open set of $X$ together with an open cover $U=\bigcup_{i\in I}U_i$ and $\sigma,\tau\in\mathscr{F}(U)$ are sections over $U$ such that $\sigma|_{U_i} = \tau|_{U_i}$ for all $i$, then $\sigma=\tau$. In particular, if the category $\mathcal{C}$ is the category of abelian groups, rings, etc. then $\sigma|_{U_i}=0$ for all $i$ implies $\sigma=0$.
		
		\item If $U$ is an open set of $X$ together with an open cover $U=\bigcup_{i\in I}U_i$ and we have elements $\sigma_i\in\mathscr{F}(U_i)$ for each $i$ with the property that for each $i,j$, $\sigma_i|_{U_i\cap U_j} = \sigma_j|_{U_i\cap U_j}$, then there is a (unique due to the first statement) element $\sigma\in\mathscr{F}(U)$ such that $\sigma|_{U_i} = \sigma_i$ for all $i$.
	\end{enumerate}
	
	\begin{remark}
		Note that not all presheaves are automatically sheaves. One such example is the sheaf of continuous and bounded functions on $\mathbb{R}$.
	\end{remark}
	
	\begin{definition}
		A \textbf{morphism of sheaves}\index{sheaf!morphism} is just a morphism of the underline presheaves.
	\end{definition}
	
	Hence, one obtains the category of sheaves $\text{Sh}_\mathcal{C}(X)$ as a full subcategory of $\text{PreSh}_{\mathcal{C}}(X)$.
	
	\begin{example}
		The presheaves of continuous and smooth functions $C^0$ and $C^\infty$ given in Example \ref{Example: sheaf of continuous functions} and Example \ref{Example: sheaf of smooth functions} are also examples of sheaves.
	\end{example}
	
	\begin{example}
		Other examples of sheaves coming from differential geometry are the sheaf of vector fields over a smooth manifold, the sheaf of differential forms, and more generally the sheaf of sections of any fibre bundle.
	\end{example}
	
	\begin{remark}
		Note that if $\mathcal{B}$ is a basis for the topology on $X$, then one can define a sheaf $\mathscr{F}$ on $X$ by defining only objects $\mathscr{F}(U)$ for each $U\in\mathcal{B}$, which are compatible with the restrictions. 
	\end{remark}
	
	Let now $(I,\leq)$ be a directed set, i.e.~$\leq$ is a reflexive and transitive relation on the set $I$, and suppose that $(A_i)_{i\in I}$ if a family of objects in $\mathcal{C}$ together with a collection of morphisms $f_{ij}:A_i\to A_j$ for all $i\leq j$. The pair $(A_i,f_{ij})$ is called a \textbf{direct system}\index{direct system} over $I$ if $f_{ii} = \id_{A_i}$ for all $i\in I$, and $f_{ik} = f_{jk}\circ f_{ij}$ for all $i\leq j\leq k$. In a more categorical language, $I$ is an index category and a direct system is a covariant functor $f:I\to \mathcal{C}, i\mapsto A_i, (i\leq j)\mapsto (f_{ij}:A_i\to A_j)$.
	
	\begin{definition}
		Let $(A_i,f_{ij})$ be a direct system of objects and morphisms in $\mathcal{C}$. A \textbf{direct limit}\index{direct limit} of $(A_i,f_{ij})$ is a pair $(L,g_i)$ with $L$ an object in $\mathcal{C}$ and $g_i:A_i\to L$ a morphism for all $i\in I$, such that
		\begin{enumerate}
			\item $g_i = g_j\circ f_{ij}$ for all $i\leq j$,
			
			\item given another pair $(B,h_i)$ for all $i\in I$ with $B$ an object in $\mathcal{C}$ and $h_i:A_i\to B$ a morphism which has the property $h_i = h_j\circ f_{ij}$ for all $i\leq j$, then there is a unique morphism $u:L\to B$ such that $u\circ g_i = h_i$ for all $i\in I$.
		\end{enumerate}
	The object of the direct limit is denoted by $L=\lim\limits_{\underset{i\in I}{\longrightarrow}}A_i=\lim\limits_{\longrightarrow}A_i$. 
	\end{definition}
	In the language of diagrams, the direct limit of the direct system $(A_i,f_{ij})$ fits in the following commutative diagram
	\[
	\begin{tikzcd}
	A_i \arrow[rr, "f_ij"] \arrow[rd, "g_i"'] \arrow[rddd, "h_i"', bend right] &                           & A_j \arrow[ld, "g_j"] \arrow[lddd, "h_j", bend left] \\
	& L \arrow[dd, "\exists!u", dashed] &                                                      \\
	&                           &                                                      \\
	& B                         &                                                     
	\end{tikzcd}
	\]
	If the category $\mathcal{C}$ is the category of the usual algebraic structures such as groups, rings, algebras, etc. then the direct limit of $(A_i,f_{ij})$ is given by
	\[
	\lim_{\longrightarrow} A_i = \bigsqcup_{i\in I}A_i\Big/\sim,
	\]
	where by definition $A_i\ni a_i\sim a_j\in A_j$ if and only if there is $k\in I$ such that $i,j\leq k$ and $f_{ik}(a_i) = f_{jk}(a_j)$. 
	
	\begin{definition}
		Suppose that $\mathscr{F}$ is a presheaf on $X$ with values in $\mathcal{C}$ and suppose also that direct limits exist in $\mathcal{C}$. Given a point $x\in X$, the \textbf{stalk}\index{presheaf!stalk} of $\mathscr{F}$ at $x$ is defined as
		\[
		\mathscr{F}_x:=\lim_{\underset{x\in U}{\longrightarrow}}\mathscr{F}(U),
		\]
		where the direct limit is taken over all open sets $U$ containing $x$ together with the restriction maps.
	\end{definition}
	\begin{remark}
		For all $x\in U$, there is a natural map $\mathscr{F}(U)\to\mathscr{F}_x$ which sends $\sigma\in\mathscr{F}(U)$ to $\sigma_x$ called the \textbf{germ of $\sigma$ at $x$}\index{germ}.
	\end{remark}
	
	The importance of the last constructions is shown in the following definitions.
	
	\begin{definition}
		Let $\mathscr{F}$ be a presheaf on $X$. The \textbf{sheaf $\mathscr{F}^+$ associated to $\mathscr{F}$}\index{sheaf!associated to a presheaf} assigns to an open $U\subset X$ the set of functions $\sigma:U\to\bigsqcup_{x\in U}\mathscr{F}_x$ such that 
		\begin{enumerate}
			\item $\sigma(x)\in \mathscr{F}_x$ for all $x\in U$,
			
			\item for each $x\in U$, there is an open $V\subset U$ containing $x$ and a section $\tau\in$ such that $\sigma(y) = \tau_y$ for all $y\in V$. 
		\end{enumerate}
	\end{definition} 
	
	The sheaf $\mathscr{F}^+$ comes together with a natural morphism (of presheaves) $f:\mathscr{F}\to\mathscr{F}^+$ with the property that for any sheaf $\mathscr{G}$ and any morphism $g:\mathscr{F}\to\mathscr{G}$, there is a unique morphism $h:\mathscr{F}\to\mathscr{G}$ such that $g=h\circ f$:
	\[
	\begin{tikzcd}
	\mathscr{F} \arrow[rr, "f"] \arrow[dd, "g"'] &  & \mathscr{F}^+ \arrow[lldd, "\exists!h", dashed] \\
	&  &                                                 \\
	\mathscr{G}                                  &  &                                                
	\end{tikzcd}
	\]
	
	\begin{remark}
		Note that in case $\mathscr{F}$ is already a sheaf, then $\mathscr{F}^+ = \mathscr{F}$.
	\end{remark}
	
	\begin{example}
		It was mentioned above that an example of a presheaf that is not a sheaf is given by the sheaf of continuous and bounded functions on the real line:
		\[
		B^0(\mathbb{R}):=\{f\in C^0(\mathbb{R})\ |\ f\ \text{is bounded}\}.
		\]
		The sheaf associated to the presheaf $B^0$ is the sheaf of continuous functions on $\mathbb{R}$:
		\[
		(B^0)^+ = C^0(\mathbb{R}).
		\]
	\end{example}
	
	Using the sheaf associated to a presheaf, one can construct many sheaves from the usual algebraic constructions such as tensor products, inverse images, etc. Here we give the examples that will be relevant to us later.
	
	\begin{definition}
		Let $f:X\to Y$ be a continuous map and $\mathscr{F}_1,\mathscr{F}_2$ be sheaves on $X$ and $Y$, respectively. We define the \textbf{direct image sheaf $f_*\mathscr{F}_1$}\index{sheaf!direct image} on $Y$ as the sheaf which sends an open $V\subset Y$ to $f_*\mathscr{F}_1(V) = \mathscr{F}_1(f^{-1}(V))$, and the \textbf{inverse image sheaf $f^{-1}\mathscr{F}_2$}\index{sheaf!inverse image} on $X$ as the sheaf associated to the presheaf
		\[
		U\mapsto\lim_{\underset{V\supset f(U)}{\longrightarrow}}\mathscr{F}(V),
		\]	
		where the limit is taken over all open sets $V\subset Y$ containing $f(U)$.
	\end{definition}
	
	\begin{definition}
		A \textbf{ringed space}\index{ringed space} is a pair $(X,\mathscr{O}_X)$ where $X$ is a topological space and $\mathscr{O}_X$ is a sheaf of rings. A \textbf{morphism of ringed spaces $f:(X,\mathscr{O}_X)\to(Y,\mathscr{O}_Y)$}\index{ringed space!morphism} consists of a continuous map $f:X\to Y$ and a map of sheaves $f^\star:\mathscr{O}_Y\to f_*\mathscr{O}_X$. A \textbf{sheaf of $\mathscr{O}_X$-modules}\index{sheaf!of modules}, or \textbf{$\mathscr{O}_X$-module}, is a sheaf (of groups) $\mathscr{F}$ on $X$, such that $\mathscr{F}(U)$ is an $\mathscr{O}_X(U)$-module for all open subsets $U\subset X$, and such that for all open subsets $V\subset U \subset X$ the restriction $\mathscr{F}(U)\to\mathscr{F}(V)$ is compatible with the ring homomorphism $\mathscr{O}(U)\to\mathscr{O}(V)$, i.e.~$\mathscr{F}(U)\to\mathscr{F}(V)$ is a map of modules. A \textbf{morphism $\mathscr{F}_1\to\mathscr{F}_2$ of $\mathscr{O}_X$-modules}\index{sheaf!morphism of modules} is a morphism of sheaves such that for each open subset $U\subset X$, the map $\mathscr{F}_1(U)\to\mathscr{F}_2(U)$ is a morphism of $\mathscr{O}_X(U)$-modules.
	\end{definition}
	
	\begin{definition}
		Let $(X,\mathscr{O}_X)$ be a ringed space. Given two sheaves $\mathscr{F}_1,\mathscr{F}_2$ of $\mathscr{O}_X$-modules, we define the \textbf{tensor product sheaf $\mathscr{F}_1\otimes_{\mathscr{O}_X}\mathscr{F}_2$}\index{sheaf!tensor product} as the sheaf associated to the presheaf defined by the following assignment:
		\[
		U\mapsto\mathscr{F}_1(U)\otimes_{\mathscr{O}_X(U)}\mathscr{F}_2(U).
		\]
	\end{definition}
	
	Let $f:(X,\mathscr{O}_X)\to (Y,\mathscr{O}_Y)$ be a morphism of ringed spaces and $\mathscr{F}$ be a sheaf of $\mathscr{O}_Y$-modules. Then $f^{-1}\mathscr{F}$ is a sheaf of $f^{-1}\mathscr{O}_Y$-modules. Moreover, by the properties of $f$ and the universal property of direct limits, we obtain the following commutative diagram for all open subsets $U\subset X$ and all open subsets $W\subset V\subset Y$ containing $f(U)$:
	
	\[
	\begin{tikzcd}
	& \mathscr{O}_Y(V) \arrow[rr, "\varrho_{\mathscr{O}_Y}"] \arrow[ld, "f_V^\star"', bend right] \arrow[rd] &                                                          & \mathscr{O}_Y(W) \arrow[ld] \arrow[rd, "f_W^\star", bend left] &                                                     \\
	\mathscr{O}_X(f^{-1}(V)) \arrow[rrdd, "\varrho_{\mathscr{O}_X}"] \arrow[rrrr, "\varrho_{\mathscr{O}_X}", bend right=67, shift right=2] &                                                                                                        & f^{-1}\mathscr{O}_Y(U) \arrow[dd, "\exists!\widehat{f}_U", dashed] &                                                                & \mathscr{O}_X(f^{-1}(W)) \arrow[lldd, "\varrho_{\mathscr{O}_X}"'] \\
	&                                                                                                        &                                                          &                                                                &                                                     \\
	&                                                                                                        & \mathscr{O}_X(U)                                         &                                                                &                                                    
	\end{tikzcd}
	\]
	That is, the sheaf $\mathscr{O}_X$ is also a sheaf of $f^{-1}\mathscr{O}_Y$-modules via a natural morphism $\widehat{f}:f^{-1}\mathscr{O}_Y\to\mathscr{O}_X$ and thus we may give the following definition.
	
	\begin{definition}
		Let $f:(X,\mathscr{O}_X)\to (Y,\mathscr{O}_Y)$ be a morphism of ringed spaces and $\mathscr{F}$ a sheaf of $\mathscr{O}_Y$-modules. We define the \textbf{pull-back sheaf $f^!\mathscr{F}$}\index{sheaf!pull-back} as the $\mathscr{O}_X$-module 
		\[
		\mathscr{O}_X\otimes_{f^{-1}\mathscr{O}_Y}f^{-1}\mathscr{F}.
		\]
	\end{definition}
	
	\chapter{$\mathbb{Z}$- and $\mathbb{N}$-graded supergeometry}\label{Chapter: Z- and N-graded supergeometry}
	
	The goal of this chapter is to cover the basics of graded geometry. In particular, it offers quick but sufficient introduction to the theory of $\mathbb{Z}$- and $\mathbb{N}$-manifolds together with various geometric constructions on them. 
	
	\section{The categories of $\mathbb{Z}$- and $\mathbb{N}$-graded manifolds}\label{Section: The categories of Z- and N-graded manifolds}
	
	A \textbf{$\mathbb{Z}$-graded manifold $\M$ of dimension $(m;(r_i)_{i\in\mathbb{Z}})$}\index{graded manifold} with $m,r_i\in\mathbb{N}$\footnote{Only finitely many of the variables $r_i$ are assumed to be non-zero.}, for short \textbf{$\mathbb{Z}$-manifold}\index{$\mathbb{Z}$-graded manifold}\index{$\mathbb{Z}$-manifold}, is a ringed space $\M=(M,\cin(\M))$, where $M$ is an ordinary $m$-dimensional smooth manifold and
	$\cin(\M)$ is a sheaf of $\mathbb{Z}$-graded,
	graded commutative, associative, unital
	$C^\infty(M)$-algebras, such that every point $p\in M$ has an open neighbourhood $U$ which satisfies
	\begin{equation}\label{local structure of graded manifold}
	\cin_U(\M)\cong C^\infty(U)\otimes \underline{S}(\underline{V})
    \end{equation}
	as sheaves of graded commutative algebras, where $\underline{V}=\bigoplus_{i\in\mathbb{Z}} V_i[i]$ is a fixed $\mathbb{Z}$-graded vector space of finite dimension and such that $\dim(V_{i}) = r_{i}$ for all $i\in\mathbb{Z}$; here $\underline{S}(\underline{V})$ denotes the graded symmetric algebra of the graded vector space $\underline{V}$. The smooth manifold $M$ is sometimes called \textbf{base manifold}\index{graded manifold!base manifold} or \textbf{body}\index{graded manifold!body} of $\M$. An \textbf{$\mathbb{N}$-graded manifold}\index{$\mathbb{N}$-graded manifold}, or simply \textbf{$\mathbb{N}$-manifold}\index{$\mathbb{N}$-manifold} is a $\mathbb{Z}$-graded manifold $\M$, such that the vector space $\underline{V}$ is only positively graded, i.e.~$V_i=\{0\}$ for all $i\geq0$\footnote{This means that only positively graded summands of $\underline{V}$ survive as the space $V_i$ is in the $-i$ position of the graded sum.}; we write \textbf{$\n$-manifold}\index{$\n$-manifold} for an $\mathbb{N}$-graded manifold $\M$ if the model space $\underline{V}$ has $V_n\neq\{0\}$ and $V_i=\{0\}$ for all $i>n$.  \textbf{A morphism of $\mathbb{Z}$-} (respectively \textbf{$\mathbb{N}$-})\textbf{manifolds}\index{graded manifold!morphism} $\mu\colon \N\to\M$
	over a smooth map $\mu_0\colon N\to M$ of the underlying
	smooth manifolds is a (degree 0) morphism of sheaves of graded
	algebras $\mu^\star\colon \cin(\M)\to \cin(\N)$
	over $\mu_0^\ast\colon C^\infty(M)\to C^\infty(N)$. As usual, the degree of a (degree-)homogeneous element
	$\xi\in \cin(\M)$ will be denoted $|\xi|$.
	
	Given a coordinate chart $(U,x^1,\ldots,x^m)$ of $M$ as in (\ref{local structure of graded manifold}) and a choice of a basis $\xi_i^1,\ldots,\xi_i^{r_i}$ for each vector space $V_i$, one obtains the $m+\sum_i r_i$ (graded) ``local coordinates"\index{graded manifold!local coordinates} $(x^1,\ldots,x^m,\xi_i^1,\ldots,\xi_i^{r_i})$ of the graded manifold $\M$ equipped with the following degrees:
	
	\begin{itemize}
		\item the coordinates $(x^1,\ldots,x^m)$ have degree 0,
		\item the coordinates $(\xi_i^1,\ldots,\xi_i^{r_i})$ have degree $-i$, for $i\in\mathbb{Z}$ and
		$j\in \{1,\ldots,r_i\}$.
	\end{itemize} 
	
	\begin{remark}
		Note that the degree $0$ coordinates of an $\mathbb{N}$-manifold consist entirely of the coordinates of the base manifold $M$. This is not the case for the degree $0$ coordinates of a general $\mathbb{Z}$-manifold $\M$, unless the vector space $V_0$ is trivial. Even in this case, the degree $0$ functions of the $\mathbb{Z}$-manifold may have a ``polynomial" part that comes from the graded coordinates of non-zero degree, since, for instance, the local functions $\xi_i^{j}\xi_{-i}^k$ have degree $0$ for all $i\in\mathbb{Z}\setminus\{0\}$. On the contrary, by definition of the structure sheaf of a graded manifold, the degree $0$ functions of an $\mathbb{N}$-manifold $\M$ are given by $\cin(\M)^0=C^\infty(M)$.
	\end{remark}
	
	Many examples of $\mathbb{Z}$- (respectively $\mathbb{N}$-)graded manifolds come from finite graded vector bundles $\E=\bigoplus_{i\in\mathbb Z}E_i[i]\to M$ as follows: The construction of a graded manifold from $\E$ is done by assigning the degree $i$ to the fibre coordinates of the vector bundle $E_i^*$, for each $i\in\mathbb{Z}$. The resulting graded manifold has dimension $(\dim(M);(\rank(E_{-i}))_{i\in\mathbb{Z}})$ and is denoted by the same notation $\M:=\E$. The structure sheaf of $\M=\E$ is given by the sections of the (graded) symmetric algebra of $\E^*$:
	\begin{equation}\label{Functions of split [n]-manifold}
	\cin(\M) = \Gamma(\underline{S}(\underline{E}^*)) =
	\Gamma\left( \wedge E^*_{\text{odd}} \otimes S E^*_{\text{even}} \right),
	\end{equation}
	where
	\[
	E^*_{\text{odd}}= \bigoplus_{i\in2\mathbb Z + 1} E^*_i[-i]
	\qquad \text{and} \qquad
	E^*_{\text{even}}= \bigoplus_{i\in2\mathbb Z} E^*_i[-i].
	\]
	
	It is clear from the construction that $\M=\E$ is an $\mathbb{N}$-manifold if $E_i$ is zero for all $i\geq0$. In particular, as the next proposition shows, every $\mathbb{N}$-manifold is (non-canonically) isomorphic to an $\mathbb{N}$-manifold obtained in that way.
	
	\begin{proposition}[\cite{Roytenberg02,BoPo13}]
		Any $[n]$-manifold over  $M$ is isomorphic to an $[n]$-manifold corresponding to a graded vector bundle $\E=\bigoplus_{i=1}^n E_i[i]\to M$.
	\end{proposition}
	
	\begin{proof}
		The sketch of the proof can be found in \cite{Roytenberg02,BoPo13}. Here we follow \cite{Cueca19a}, where a complete proof using induction over $n$ is given.
		
		First observe that for each $\n$-manifold $\M$, there is a tower of $[i]$-manifolds
		\[
		\M_n=\M\longrightarrow\M_{n-1}\longrightarrow\ldots\longrightarrow\M_1\longrightarrow\M_0=M
		\]
		for $i=0,1,\ldots,n$, whose sheaf of functions is given by:
		\[
		\cin(\M_i) := \text{span}\{\xi\in\cin(\M)\ |\ |\xi|\leq i\}.
		\] 
		We will now prove the result by induction on the degree $n\geq1$ of the $\mathbb{N}$-manifold $\M$.
		
		For $n=1$, the local description of the $[1]$-manifold in (\ref{local structure of graded manifold}) implies that $\cin(\M)$ is generated as an algebra by $\cin(\M)^0 = C^\infty(M)$ and $\cin(\M)^1$. In other words, we have that  $\cin(\M)^k = \underline{S}^k(\cin(\M)^1)$. Moreover, the sheaf $\cin(\M)^1$ is a locally free and finitely generated sheaf of $C^\infty(M)$-modules, and thus there exists a vector bundle $E\to M$ whose sheaf of sections is given by $\Gamma(E) = \cin(\M)^1$. That is, $\M \simeq E$.
		
		Suppose the result holds for some degree $n-1\geq1$ and let $\M$ be an $\n$-manifold. This implies that there is a graded vector bundle $\E_0=\bigoplus_{i=1}^{n-1} E_i[i]\to M$ such that $\M_{n-1} \simeq \E_0$ and consequently $\cin(\M_{n-1})^k \simeq \Gamma(\underline{S}^k(\E_0^*))$. The graded algebra $\cin(\M)$ is generated by $\cin(\M_{n-1})$ and $\cin(\M)^n$, and the latter fits in the short exact sequence
		\begin{equation}\label{SES for functions of split N-manifolds}
			0\longrightarrow \cin(\M_{n-1})^n \longrightarrow \cin(\M)^n \longrightarrow \cin(\M)^n/\cin(\M_{n-1})^n\longrightarrow 0.
		\end{equation}
		It follows from $C^\infty(M)\,\cdot\,\cin(\M)^n\subset \cin(\M)^n$ that $\cin(\M)^n$ is a $C^\infty(M)$-module and from equation (\ref{local structure of graded manifold}) that it is locally freely generated. Therefore, there exists a vector bundle $E\to M$ with the property that $\Gamma(E)=\cin(\M)^n$. Setting $E_n^*:=E/\underline{S}^n(\E_0^*)\to M$ and choosing a splitting of the short exact sequence of vector bundles
		\[
		0\longrightarrow \underline{S}^n(\E_0^*) \longrightarrow E \longrightarrow E/\underline{S}^n(\E_0^*)\longrightarrow 0,
		\]
		we obtain that $E\cong E_n^* \oplus \underline{S}^n(\E_0^*)$. Hence, we conclude that $\M\simeq \E$, where we define the graded vector bundle $\E\to M$ to be the direct sum $\E:=\E_0 \oplus E[n]$.
	\end{proof}
	
	\begin{definition}
		A graded manifold associated to a graded vector bundle $\E\to M$ as above is called \textbf{split $\mathbb{Z}$-} (respectively \textbf{$\mathbb{N}$-})\textbf{graded manifold}\index{$\mathbb{N}$-graded manifold!split}\index{$\mathbb{Z}$-graded manifold!split}.
	\end{definition}
	
	\begin{remark}
		\begin{enumerate}
			\item Note that different choices of splittings of the $\mathbb{N}$-graded manifold $\M$, i.e.~isomorphisms $\M\cong\E=\bigoplus_i E_i[i]$, lead to isomorphic vector bundles $E_i$. What changes is the way to view the sections of $\Gamma(E_i^*)$ as functions of the graded manifold via the splittings of the short exact sequence (\ref{SES for functions of split N-manifolds}) (see also Example \ref{Change of splitting for [2]-manifolds}).
			
			\item Suppose that $U\subset M$ is a trivialising neighbourhood for all the vector bundles $E_i\to M$, i.e.~$E_i|_U\cong U \times \mathbb{R}^{k_i}$ for all $i\in\mathbb{Z}$, where $k_i=\rank(E_i)$. Then the graded coordinates of the manifold $\M$ are given by the following:
			\begin{itemize}
				\item a set of coordinates $(x^1,\ldots,x^m)$ of $M$ over $U$ have degree $0$,
				\item the sections of a local frame $(\varepsilon_i^1,\ldots,\varepsilon_i^{k_i})$ of $E_i^*$ over $U$ have degree $i$.
			\end{itemize}
		\end{enumerate}
	\end{remark}
	
	\begin{example}[Linear $\mathbb{Z}$-manifolds]
		Consider a finite collection of natural numbers $k_i\in\mathbb{N}$, for $i\in\mathbb{Z}$ (necessarily including $0$). Then one can define the trivial vector bundles
		\[
		E_i:=\mathbb{R}^{k_0}\times\mathbb{R}^{k_i}\to\mathbb{R}^{k_0}
		\]
		for all $i\in\mathbb{Z}$, and sum over $\mathbb{R}^{k_0}$ to obtain the graded vector bundle
		\[
		\E:=\bigoplus_{i\in\mathbb Z} E_i[i]\to\mathbb{R}^{k_0}.
		\]
		The split $\mathbb{Z}$-manifold obtained by $\E$ with body $\mathbb{R}^{k_0}$, denoted $\mathbb{R}^{\{k_i\}}$, is a called \textbf{linear $\mathbb{Z}$-manifold}\index{graded manifold!linear manifold}. Its structure sheaf is given locally by
		\[
		C^\infty(U)\otimes\underline{S}\left( \bigoplus_{i\neq0} \mathbb{R}^{k_i}[-i] \right),
		\]
		where $U\subset\mathbb{R}^{k_0}$ is open.
	\end{example}
	
	\begin{example}[Product $\mathbb{Z}$-manifold]
		Given two $\mathbb{Z}$-manifolds $\M$ and $\N$ with bodies $M$ and $N$, respectively, one can construct the \textbf{product $\mathbb{Z}$-manifold}\index{graded manifold!product} $\M\times\N$ with body $M\times N$ by defining its sheaf of functions as
		\[
		\cin(\M\times\N) = \pr_M^!\cin(\M)\, \widehat{\otimes} \pr_N^!\cin(\N),
		\]
		where $\widehat{\otimes}$ denotes the completion of the projective tensor topology; that is, the unique topology on $C^\infty(U)\otimes C^\infty(U')$ such that its completion is isomorphic to $C^\infty(U\times U')$ for all open subsets $U\subset \mathbb{R}^m$ and $U'\subset \mathbb{R}^n$ (see \cite{CaCaFi11} and references therein for the case of $\mathbb{Z}_2$-graded supermanifolds). Explicitly, given two open sets $U\subset M, U'\subset N$ with
		\[
		\cin_U(\M)\cong C^\infty(U)\otimes\underline{S}(\underline{V})
		\qquad \text{and} \qquad
		\cin_{U'}(\N)\cong C^\infty(U')\otimes\underline{S}(\underline{W}),
		\]
		the open set $U\times U'\subset M\times N$ is sent to 
		\[
		\cin_U(\M)\, \widehat{\otimes}\, \cin_{U'}(\N) \cong C^\infty(U\times U')\otimes \underline{S}(\underline{V}\oplus\underline{W}).
		\]
		In particular, the sets of coordinates $\{\xi^j\}$ and $\{\zeta^i\}$ on $\M$ and $\N$, respectively, induce the set of coordinates $\{\xi^j,\zeta^i\}$ on $\M\times \N$ and hence the homogeneous functions $\xi\in\cin(\M\times\N)^k$ can be written as a finite sum
		\[
		\xi = \sum_j \eta_j\theta_j, 
		\]
		where $\eta_j\in\cin(\M)^{p_j},\theta_j\in\cin(\M)^{q_j}$ and $p_j+q_j=k$ for all $j$.
	\end{example}
	
	\begin{example}
		From the discussion above, it is clear that every graded vector space
		\[
		\underline{V}=\bigoplus_{i\in\mathbb Z} V_i[i]
		\]
		gives a $\mathbb{Z}$-manifold over the trivial base manifold $M=\{*\}$. In particular, applying this to the $[1]$-shift of a non-graded vector space $\mathfrak{g}$, one obtains the $[1]$-manifold $\mathfrak{g}[1]$ with structure sheaf 
		$\cin(\mathfrak{g}[1]) = \wedge \mathfrak{g}^*$. The interesting scenario comes from the case where $\mathfrak{g}$ is a Lie algebra, as it is the motivation for the term ``Lie $n$-algebra" (or more generally ``Lie $n$-algebroid") that we will see later. 
	\end{example}
	
	\begin{example}
		Again from the discussion above, it follows that for every (non-graded) vector bundle $E\to M$, one obtains the $[1]$-manifold $E[1]$ whose sheaf of functions is given by $\cin(E[1]) = \Gamma(\wedge E^*)$. Applying this to the tangent and cotangent bundles of the manifold $M$ yields the $[1]$-manifolds $T[1]M$ and $T^*[1]M$ whose function sheaves are given by
		\[
		\cin(T[1]M) = \Omega(M)
		\qquad \text{and} \qquad
		\cin(T^*[1]M) = \wedge\mathfrak{X}(M).
		\]
		That is, the spaces of differential forms and multivector fields on a smooth manifold $M$ can be realised as graded functions of $[1]$-manifolds with base $M$.
	\end{example}
	
	In the split setting, a morphism $\mu\colon \F\to\E$ between two split $\n$-manifolds $\F=F_1[1]\oplus\ldots\oplus F_n[n]$ and $\E=E_1[1]\oplus\ldots\oplus E_n[n]$ over $N$ and $M$, respectively, is a map $\mu^\star\colon\Gamma(\underline{S}(\underline{E}^*))\to\Gamma(\underline{S}(\underline{F}^*))$ which is a morphism of sheaves of graded algebras over $\mu_0^*\colon C^\infty(M)\to C^\infty(N)$, where $\mu_0\colon N\to M$ is a smooth map on the base manifolds. The map $\mu^\star$ decomposes into a family of maps $\mu_i^\star\colon \Gamma(E_i^*)\to \Gamma(\underline{S}^i(\F^*)),i=1,\ldots,n$ and since we have that
	\[
	\Gamma(\underline{S}^i(\F^*))=\bigoplus_{j_1+2j_2+\ldots+nj_n=i} \Gamma\left( (F_1^*)^{\otimes^{j_1}}\otimes(F_2^*)^{\otimes^{j_2}}\otimes\ldots\otimes(F_n^*)^{\otimes^{j_n}} \right),
	\]
	the components $\mu_i^\star$ admit a decomposition
	\[
	\mu_i^\star=\sum_{j_1+2j_2+\ldots+nj_n=i} (\mu_i^{j_1,j_2\ldots,j_n})^\star,
	\]
	where all $(\mu_i^{j_1,j_2\ldots,j_n})^\star\colon \Gamma(E_i^*)\to\Gamma\left( (F_1^*)^{\otimes^{j_1}}\otimes(F_2^*)^{\otimes^{j_2}}\otimes\ldots\otimes(F_n^*)^{\otimes^{j_n}} \right)$ are morphisms over the pull-back map $\mu_0^*\colon C^\infty(M)\to C^\infty(N)$. Note that by abuse of notation, $\otimes$ in the above equation denotes the multiplication of the graded symmetric algebra. Hence, we obtain the vector bundle maps
	\[
	\mu_i^{j_1,j_2\ldots,j_n}\colon F_1^{\otimes^{j_1}}\otimes F_2^{\otimes^{j_2}}\otimes\ldots\otimes F_n^{\otimes^{j_n}}\to E_i
	\]
	covering $\mu_0\colon N\to M$. In particular, we have the component maps $\mu_i^{0,\ldots,1,\ldots,0}\colon F_i\to E_i$ ($1$ in the $i$-th slot) covering the map $\mu_0$ which we call \textbf{the principal part of $\mu$}\index{principal part of morphism of split manifolds}.
	
	\begin{example}[Change of splitting for $\text{[}2\text{]}$-manifolds]\label{Change of splitting for [2]-manifolds}
		By construction, splittings of a $[2]$-manifold $\M$ are given by splittings of the short exact sequence
		\[
		0\longrightarrow \Gamma(\wedge^2 E_1) \longrightarrow \cin(\M)^{2} \longrightarrow \Gamma(E_2)\longrightarrow 0,
		\]
		where $E_1\to M$ and $E_2\to M$ are the vector bundles whose sheaves of sections are defined by
		\[
		\Gamma(E_1) \simeq \cin(\M)^1
		\qquad \text{and} \qquad 
		\Gamma(E_2) \simeq \frac{\cin(\M)^2}{\cin(\M)^1\,\cdot\,\cin(\M)^1}.
		\]
		A change of splitting for the exact sequence above is equivalent to a section $\sigma\in\Hom(\wedge^2 E_1,E_2)$. Then the induced isomorphism of $[2]$-manifolds $\mu_\sigma:E_1[1]\oplus E_2[2]\to E_1[1]\oplus E_2[2]$ over the identity on $M$ is given by the components $\mu_\sigma^\star(\varepsilon_1) = \varepsilon_1$ and $\mu_\sigma^\star(\varepsilon_2) = \varepsilon_2 + \sigma^\star\varepsilon_2$ for all $\varepsilon_1\in\Gamma(E_1^*),\varepsilon_2\in\Gamma(E_2^*)$. The principal part of $\mu_\sigma$ consists of $\id_{E_1}$ and $\id_{E_2}$.
	\end{example}

	\section{Vector fields on graded manifolds}
	
	Using the language of graded derivations, the usual notion of  vector
	field can be generalised to a notion of vector field on a $\mathbb{Z}$-manifold
	$\M$.
	\begin{definition}
		A \textbf{vector field of degree $j$}\index{vector field} on $\M$ is a degree $j$ (graded)
		derivation of $\cin(\M)$, i.e.~a map
		$\mathcal{X}\colon\cin(\M)\to \cin(\M)$ such that
		$|\mathcal{X}(\xi)|=j+|\xi|$ and
		$\mathcal{X}(\xi\zeta)=\mathcal{X}(\xi)\zeta+(-1)^{j|\xi|}\xi
		\mathcal{X}(\zeta)$, for homogeneous elements
		$\xi,\zeta\in C^\infty(\M)$.
	\end{definition}
	As usual, $|\mathcal{X}|$ denotes the degree of a homogeneous vector field $\mathcal{X}$. The \textbf{Lie bracket of two vector fields}\index{Lie bracket of vector fields}
	$\mathcal{X},\mathcal{Y}$ on $\M$ is the graded commutator
	\begin{equation}\label{graded Lie bracket}
	[\mathcal{X},\mathcal{Y}]=\mathcal{X}\mathcal{Y}-(-1)^{|\mathcal{X}||\mathcal{Y}|}\mathcal{Y}\mathcal{X}
	\end{equation}
	which is again a vector field of degree $|\X| + |\Y|$. A straightforward computation shows that the following relations hold:
	\begin{enumerate}[(i)]
		\item $[\mathcal{X},\mathcal{Y}]=-(-1)^{|\mathcal{X}||\mathcal{Y}|}[\mathcal{Y},\mathcal{X}]$,
		\item
		$[\mathcal{X},\xi
		\mathcal{Y}]=\mathcal{X}(\xi)\mathcal{Y}+(-1)^{|\mathcal{X}||\xi|}\xi[\mathcal{X},\mathcal{Y}]$,
		\item $(-1)^{|\mathcal{X}||\mathcal{Z}|}[\mathcal{X},[\mathcal{Y},\mathcal{Z}]]
		+(-1)^{|\mathcal{Y}||\mathcal{X}|}[\mathcal{Y},[\mathcal{Z},\mathcal{X}]]
		+(-1)^{|\mathcal{Z}||\mathcal{Y}|}[\mathcal{Z},[\mathcal{X},\mathcal{Y}]]=0$,
	\end{enumerate}
	for $\mathcal{X},\mathcal{Y},\mathcal{Z}$ homogeneous vector fields on $\M$, and $\xi$
	a homogeneous element of $\cin(\M)$.
	
	\begin{example}[Coordinate vector fields]
		The local generators $\xi_i^j$ for $i\in\mathbb{Z},j=1,\ldots,r_i$ of $\cin(\M)$ over an open coordinate chart
		$U\subseteq M$ given by the definition of the $\mathbb{Z}$-manifold $\M$ define the (local)
		vector fields $\partial_{\xi_i^j}$ of degree $-|\xi^j_i|$ (alternatively denoted $\frac{\partial}{\partial \xi_i^j}$), which send
		$\xi_i^j$ to $1$ and the other local generators to $0$. The sheaf
		$\underline{\Der}_U(\cin(\M))$ of graded derivations of $\cin_U(\M)$ is locally freely generated as a $C^\infty_U(\M)$-module by all $\partial_{x^k}$
		and all $\partial_{\xi_i^j}$, where $x^1,\ldots,x^m$ are coordinates of
		$M$ defined on $U$.
	\end{example}
	
	Note that in the case of a split $\mathbb{Z}$-manifold
	$\E = \bigoplus_i E_i[i]$, each section $e\in\Gamma(E_j)$
	defines a derivation $\hat{e}$ of degree $-j$ on $\M$ by the
	relations: $\hat{e}(f)=0$ for $f\in\cin(M)$,
	$\hat{e}(\varepsilon)=\langle\varepsilon,e\rangle$ for
	$\varepsilon\in\Gamma(E_j^*)$ and $\hat{e}(\varepsilon)=0$ for
	$\varepsilon\in\Gamma(\E^*)$ with $|\varepsilon|\neq j$. In
	particular, $\widehat{e_j^i}=\partial_{\varepsilon_j^i}$ for
	$\{e_j^i\}$ a local basis of $E_j$ and $\{\varepsilon_j^i\}$ the dual
	basis of $E_j^*$.
	
	Given now an $\mathbb{N}$-manifold $\M$ of degree $n$ together with a choice of splitting $\M\cong\E$ for some graded vector bundle
	\[
	\E=\bigoplus_{i=1}^{n} E_i[i],
	\]
	the space of vector fields over $\M$ can be explicitly described in terms of ordinary vector fields over the body $M$ and connections on the bundles $E_i$. More precisely, given $TM$-connections $\nabla^i\colon\mathfrak{X}(M)\to\Der(E_i)$ on the vector bundles $E_i$ for each
	$i$, the space of vector fields on the split $\mathbb{N}$-manifold $\M\cong\E$ is generated as a
	$\cin(\M)$-module by
	\begin{equation}
	\left\{ \bigoplus_{i=1}^n\nabla^i_X \ |\
	X\in\mathfrak{X}(M) \right\}\cup\left\{ \hat{e}\ |\ e\in\Gamma(E_i)\ \text{for
		some}\ i \right\}.
	\end{equation}
	The vector fields of the form
	$\bigoplus_i\nabla^i_X$ have degree 0 and are
	understood to send $f\in C^\infty(M)$ to $X(f)\in C^\infty(M)$, and
	$\varepsilon\in\Gamma(E_i^*)$ to
	$\nabla_X^{i,*}\varepsilon\in\Gamma(E_i^*)$. The non-zero degree vector
	fields are generated by those of the form $\hat{e}$.
	
	\section{$\Q$-structures and Lie $n$-algebroids}\label{Section: Q-structures and Lie n-algebroids}
	
	Due to the sign convention for the Lie bracket in (\ref{graded Lie bracket}), it follows that if a vector field $\X$ on $\M$ has odd degree, then the equation
	\begin{equation}
	[\X,\X] = 2\X^2 = 0 
	\end{equation}
	is not necessarily satisfied and thus the search for such a vector field is not trivial. If in particular the degree of the vector field is $1$, then the structure sheaf $\cin(\M)$ becomes a complex. This leads to the following definition. 
	
	\begin{definition}
		A \textbf{homological vector field}\index{homological vector field} $\Q$ on a $\mathbb{Z}$-manifold $\M$ is a degree 1 derivation of $\cin(\M)$ such that
		$\Q^2=\frac{1}{2}[\Q,\Q]=0$.
	\end{definition}
	
	\begin{definition}
		A $\mathbb{Z}$-manifold $\M$ equipped with a homological vector field $\Q$ is called a \textbf{$\Q$-manifold}\index{$\Q$-manifold}. A \textbf{morphism between two $\Q$-manifolds}\index{$\Q$-manifold!morphism} $(\N,\Q_\N)$ and $(\M,\Q_\M)$ is a morphism $\mu:\N\to\M$ of the underlying $\mathbb{Z}$-manifolds that commutes with the homological vector fields:
		\[
		\mu^\star\circ\Q_\M = \Q_\N\circ\mu^\star.
		\]
	\end{definition}
	
	\begin{remark}
		A homological vector field on an $[1]$-manifold $\M=A[1]$ is the
		differential $\diff_A$ associated to a Lie algebroid structure on the
		vector bundle $A$ over $M$ \cite{Vaintrob97}: The functions of $A[1]$ are given by the sheaf of $A$-forms $\Omega^\bullet(A)$, and the bracket and the anchor for the Lie algebroid structure on the vector bundle $A$ are connected to the homological vector field $\diff_{A}$ via
		\[
		\langle \diff_{A}f,a \rangle = \rho(a)f
		\qquad \text{and} \qquad
		\diff_{A}\alpha(a_1,a_2) = - \alpha(a_1,a_2) + \rho(a_1)\left( \alpha(a_2) \right) - \rho(a_2)\left( \alpha(a_1) \right)
		\]
		for all $a,a_1,a_2\in\Gamma(A),\alpha\in\Gamma(A^*)$ and $f\in C^\infty(M)$.
		The following definition
		generalises this to arbitrary degrees.
	\end{remark}
	
	\begin{definition}\label{abstract Lie algebroids}
		A \textbf{Lie $n$-algebroid}\index{Lie $n$-algebroid} is an $[n]$-manifold $\M$ endowed with a
		homological vector field $\Q$ -- the pair $(\M, \Q)$ is also called
		\textbf{$\mathbb{N}\Q$-manifold of degree $n$}\index{$\mathbb{N}\Q$-manifold}. A \textbf{split Lie
			$n$-algebroid}\index{Lie $n$-algebroid!split} is a split $[n]$-manifold $\M$ endowed with a
		homological vector field $\Q$. A \textbf{morphism of (split) Lie
			$n$-algebroids}\index{Lie $n$-algebroid!morphism} is a morphism $\mu$ of the underlying [$n$]-manifolds
		such that $\mu^\star$ commutes with the homological vector fields.
	\end{definition}
	
	\begin{example}[Product $\Q$-manifold]
		Given two $\Q$-manifolds $(\M,\Q_\M)$ and $(\N,\Q_\N)$ with bodies $M$ and $N$, respectively, the product $\M\times\N$ inherits a $\Q$-manifold structure called \textbf{product $\Q$-manifold}\index{$\Q$-manifold!product}. Its homological vector field $\Q_{\M\times\N}$ is given on the generating functions by
		\[
		\Q_{\M\times\N}(\xi\zeta) = \Q_{\M}(\xi)\zeta + (-1)^{|\xi|}\xi\Q_{\N}(\zeta)
		\]
		for all homogeneous $\xi\in\cin(\M)$ and $\zeta\in\cin(\N)$.
	\end{example}
	
	The homological vector field associated to a split Lie $n$-algebroid
	$\A=A_1[1]\oplus\ldots\oplus A_n[n]\to M$ can be equivalently
	described by a family of brackets which satisfy some Leibniz and
	higher Jacobi identities \cite{ShZh17}. More precisely, a homological
	vector field on $\A$ is equivalent to an $L_\infty$-algebra
	structure\footnote{We note that the sign convention agrees with,
		e.g.~\cite{KaSt06,Voronov05}. In \cite{Schatz09}, the term
		``$L_\infty[1]$-algebra'' was used for brackets with this sign
		convention.}  on $\Gamma(\A)$ that is anchored by a vector bundle
	morphism $\rho\colon A_1\to TM$.  Such a structure is given by
	multibrackets
	$\llbracket\cdot\,,\ldots,\cdot\rrbracket_i\colon\Gamma(\A)^i\to
	\Gamma(\A)$ of degree $1$ for $1\leq i \leq n+1$ such that
	\begin{enumerate}
		\item $\llbracket\cdot\,,\cdot\rrbracket_2$ satisfies the Leibniz identity with
		respect to $\rho$,
		\item $\llbracket\cdot\,,\ldots,\cdot\rrbracket_i$ is $C^\infty(M)$-linear in each entry
		for all $i\neq2$,
		\item \textbf{(graded skew symmetry)}\index{graded skew symmetry} each
		$\llbracket\cdot,\ldots,\cdot\rrbracket_i$ is graded
		alternating: for a permutation $\sigma\in S_i$ and for all
		$a_1,\ldots,a_i\in\Gamma(\A)$ degree-homogeneous sections
		\[\llbracket a_{\sigma(1)}, a_{\sigma(2)},\ldots,a_{\sigma(i)}\rrbracket_i
		=\text{Ksgn}(\sigma,a_1,\ldots,a_i)\cdot \llbracket a_1,a_2,\ldots,a_i\rrbracket_i, \] and
		\item \textbf{(strong homotopy Jacobi identity)}\index{strong homotopy Jacobi identity} for $k\in \mathbb N$
		and $a_1,\ldots, a_k\in\Gamma(\A)$ sections of homogeneous
		degree:
		\[
		\sum_{i+j=k+1}(-1)^{i(j-1)}\sum_{\sigma\in\text{Sh}_{i,k-i}}
		\text{Ksgn}(\sigma,a_1,\ldots,a_k)\llbracket \llbracket
		a_{\sigma(1)},\ldots,a_{\sigma(i)} \rrbracket_i,
		a_{\sigma(i+1)},\ldots,a_{\sigma(k)} \rrbracket_j = 0.
		\]
	\end{enumerate}
	Here, $\text{Sh}_{i,k-i}$ is the set of all
	$(i,k-i)$-shuffles\footnote{A $(i,k-i)$-shuffle is an element
		$\sigma\in S_k$ such that $\sigma(1)<\ldots<\sigma(i)$ and
		$\sigma(i+1)<\ldots<\sigma(k)$.  } and
	$\text{Ksgn}(\sigma,a_1,\ldots,a_k)$ is the $(a_1,\dots,a_k)$-graded
	signature of the permutation $\sigma\in S_k$, i.e.
	\[
	a_1\wedge\ldots\wedge a_k = \text{Ksgn}(\sigma, a_1,
	\ldots,a_k)a_{\sigma(1)}\wedge\ldots\wedge a_{\sigma(k)}.
	\]
	
	The explicit correspondence between the homological vector field $\Q$ and the brackets $\llbracket\cdot,\ldots,\cdot\rrbracket_i$ is given by the following equations:
	\begin{enumerate}
		\item $\Q(f) = \rho^*(\diff\! f)$;
		\item $\Q(\alpha_k)(a_1^{i_1},\ldots,a_j^{i_j}) = -\langle \alpha_k,\llbracket a_1^{i_1},\ldots,a_j^{i_j}\rrbracket_j \rangle,\ j\neq2$;
		\item $\Q(\alpha_k)(a_1^1,a_2^k) = \rho(a_1^1)\langle \alpha_k,a_2^k \rangle - \langle \alpha_k,\llbracket a_1^1,a_2^k \rrbracket_2 \rangle,\ k\neq1$; 
		\item $\Q(\alpha_1)(a_1^1,a_2^1) = \rho(a_1^1)\langle \alpha_1,a_2^1 \rangle - \rho(a_2^1)\langle \alpha_1,a_1^1 \rangle - \langle \alpha_1,\llbracket a_1^1,a_2^1 \rrbracket_2 \rangle$
	\end{enumerate}
	for all $f\in C^\infty(M),\alpha_k\in\Gamma(A^*_k),a_p^{i_p}\in\Gamma(A_{i_p}),i_1+\ldots+i_j=k+2-j$.
	
	In particular, the vector bundle $A_1\to M$ of a split Lie $n$-algebroid is always endowed with skew-symmetric dull bracket on its space of sections $\Gamma(A_1)$. The following alternative geometric description in the special case of
	a split Lie $2$-algebroid
	$(\mathcal M=A_1[1]\oplus A_2[2],\mathcal Q)$ will be used extensively in the rest of the work (see \cite{Jotz19b}). For consistency with the notation in \cite{Jotz19b}, set $A_1:=Q$ and
	$A_2^*=:B$.

	\begin{definition}\label{geom split 2-alg}
		A \textbf{split Lie 2-algebroid $Q[1]\oplus B^*[2]$}\index{split Lie $2$-algebroid} is given by a pair of
		an anchored vector bundle $(Q\to M,\rho_Q)$ and a vector bundle
		$B\to M$, together with a vector bundle map $\ell\colon B^*\to Q$, a
		skew-symmetric dull bracket
		$[\cdot\,,\cdot]\colon \Gamma(Q)\times\Gamma(Q)\to \Gamma(Q)$, a
		linear $Q$-connection $\nabla$ on $B$, and a vector valued 3-form
		$\omega\in\Omega^3(Q,B^*)$ such that
		\begin{enumerate}[(i)]
			\item $\nabla^*_{\ell(\beta_1)}\beta_2 + \nabla^*_{\ell(\beta_2)}\beta_1 = 0$, for all $\beta_1,\beta_2\in\Gamma(B^*)$,
			\item $[q,\ell(\beta)]=\ell(\nabla_q^*\beta)$ for all
			$q\in\Gamma(Q)$ and $\beta\in\Gamma(B^*)$,
			\item $\Jac_{[\cdot\,,\cdot]} = \ell\circ\omega\in\Omega^3(Q,Q)$,
			\item $R_{\nabla^*}(q_1,q_2)\beta = -\omega(q_1,q_2,\ell(\beta))$ for $q_1,q_2\in\Gamma(Q)$ and $\beta\in\Gamma(B^*)$,
			\item $\diff_{\nabla^*}\omega = 0$.
		\end{enumerate}  
	\end{definition}
	To pass
	from the definition above to the homological vector field $\Q$,
	set $\Q(f)=\rho^*\diff f \in\Gamma(Q^*)$,
	$\Q(\tau)=\diff_{Q}\tau+\partial_B\tau \in \Omega^2(Q)\oplus
	\Gamma(B)$, and
	$\Q(b)=\diff_{\nabla}b - \langle\omega, b\rangle \in \Omega^1(Q,
	B)\oplus \Omega^3(Q)$ for $f\in C^\infty(M),\tau\in\Omega(Q)$
	and $b\in\Gamma(B)$, where $\partial_B:=\ell^*$. For a detailed proof of this equivalence see Appendix \ref{Appendix: Split Lie $2$-algebroids in the geometric setting} or the original source \cite{Jotz19b}.

	On the other hand we may obtain as follows the data of Definition
	\ref{geom split 2-alg} from a given homological vector field
	$\Q$. Define the vector bundle map $\ell$ to be the 1-bracket and
	$\rho$ to be the anchor. The 2-bracket induces the skew-symmetric dull bracket
	on $Q$ and the $Q$-connection on $B^*$ via the formula
	\[
	\llbracket q_1\oplus\beta_1,q_2\oplus\beta_2\rrbracket_2 =
	[q_1,q_2]_Q\oplus(\nabla_{q_1}^*\beta_2 - \nabla_{q_2}^*\beta_1).
	\]
	Finally, the 3-bracket induces the 3-form $\omega$ via the formula
	\[
	\llbracket q_1\oplus0,q_2\oplus0,q_3\oplus0\rrbracket_3 = 0\oplus\omega(q_1,q_2,q_3). 
	\]
	
	\begin{example}[Lie 2-algebras]
		If we consider a Lie 2-algebroid over a point, then we recover the
		notion of \textbf{Lie 2-algebra}\index{Lie 2-algebra} \cite{BaCr04}. Specifically, a Lie
		2-algebroid over a point consists of a pair of vector spaces
		$\mathfrak{g}_0,\mathfrak{g}_1$, a linear map
		$\ell\colon\mathfrak{g}_0\to\mathfrak{g}_1$, a skew-symmetric
		bilinear bracket
		$[\cdot\,,\cdot]\colon
		\mathfrak{g}_1\times\mathfrak{g}_1\to\mathfrak{g}_1$, a bilinear
		action bracket
		$[\cdot\,,\cdot]\colon\mathfrak{g}_1\times\mathfrak{g}_0\to\mathfrak{g}_0$,
		and an alternating trilinear bracket
		$[\cdot\,,\cdot\,,\cdot]\colon\mathfrak{g}_1\times\mathfrak{g}_1\times\mathfrak{g}_1\to\mathfrak{g}_0$
		such that
		\begin{enumerate}
			\item $[\ell(x),y] + [\ell(y),x] = 0$ for $x,y\in\mathfrak{g}_0$,
			\item $[x,\ell(y)] = \ell([x,y])$ for $x\in\mathfrak{g}_1$ and $y\in\mathfrak{g}_0$,
			\item $\Jac_{[\cdot,\cdot]}(x,y,z) = \ell([x,y,z])$ for
			$x,y,z\in\mathfrak{g}_1$,
			\item $[[x,y],z] +[y,[x,z]] -[x,[y,z]] = [x,y,\ell(z)]$ for
			$x,y\in\mathfrak{g}_1$ and $z\in\mathfrak{g}_0$,
			\item and the higher Jacobi identity
			\begin{align*}
			0 = & [x,[y,z,w]] - [y,[x,z,w]]+[z,[x,y,w]]- [w,[x,y,z]]\\ 
			&  - [[x,y],z,w] + [[x,z],y,w] - [[x,w],y,z]- [[y,z],x,w] + [[y,w],x,z] - [[z,w],x,y].
			\end{align*}
			holds for $x,y,z,w\in\mathfrak{g}_1$.
		\end{enumerate}
	\end{example}
	
	\begin{example}[Derivation Lie 2-algebr(oid)]
		For any Lie algebra $(\mathfrak{g},[\cdot\,,\cdot]_\mathfrak{g})$, the
		\textbf{derivation Lie 2-algebra}\index{derivation Lie 2-algebra} is defined as the complex
		\[
		\ad\colon \mathfrak{g}\to\Der(\mathfrak{g})
		\]
		with brackets given by
		$[\delta_1,\delta_2] = \delta_1\delta_2 - \delta_2\delta_1$,
		$[\delta,x] = \delta x$, $[\delta_1,\delta_2,\delta_3] = 0$
		for all $\delta,\delta_i\in\Der(\mathfrak{g}),i=1,2,3$, and
		$x\in\mathfrak{g}$.
		
		A global analogue of this construction can be achieved only
		under strong assumptions on the Lie algebroid $A\to
		M$. Precisely, let $A\to M$ be a Lie algebra bundle. Then the
		space of all derivations $D$ of the vector bundle $A$ which
		preserve the bracket
		\[
		D[a_1,a_2] = [Da_1,a_2] + [a_1,Da_2] \qquad \text{for all}\ a_1,a_2\in\Gamma(A)
		\]
		is the module of sections of a vector bundle over $M$, denoted
		$\Der_{[\cdot,\cdot]}(A)\to M$. Together with the usual
		commutator bracket and the anchor $\rho'(D)=X$, where $D$ is a
		derivation of $\Gamma(A)$ covering $X\in\mathfrak{X}(M)$, the
		vector bundle $\Der_{[\cdot,\cdot]}(A)$ is a Lie algebroid
		over $M$ \cite{Mackenzie05}. Since the anchor of $A$
		is trivial, the complex
		\[
		A\overset{\ad}{\to}\Der_{[\cdot\,,\cdot]}(A)\overset{\rho'}{\to}TM
		\]
		becomes a Lie 2-algebroid with the
		$\Der_{[\cdot,\cdot]}(A)$-connection on $A$ given by
		$\nabla_Da = Da$ and $\omega=0$.
	\end{example}

	\begin{example}[Courant algebroids]\label{Split_symplectic_Lie_2-algebroid_example}
		Let    $E\to   M$    be   a    Courant   algebroid    with   pairing
		$\langle\cdot\,,\cdot\rangle\colon E\times_M E\to  E$, anchor $\rho$
		and bracket  $\llbracket\cdot\,,\cdot\rrbracket$, and choose
		a               metric               linear               connection
		$\nabla\colon   \mathfrak{X}(M)\times\Gamma(E)\to\Gamma(E)$.    Then
		$E[1]\oplus T^*M[2]$  becomes as follows a  split Lie $2$-algebroid.
		The     skew-symmetric     dull      bracket     is     given     by
		$[e,e'] =  \llbracket e,e'  \rrbracket -  \rho^*\langle \nabla_.e,e'
		\rangle$ for all $e,e'\in\Gamma(E)$.  The basic connection is
		$\nabla^\text{bas}\colon\Gamma(E)\times\mathfrak{X}(M)\to\mathfrak{X}(M)$
		and the basic curvature is given by
		$\omega_\nabla\in\Omega^2(E,\Hom(TM,E))$
		\[
		\omega_\nabla(e,e')X = -\nabla_X\llbracket e,e' \rrbracket +
		\llbracket \nabla_Xe,e' \rrbracket + \llbracket e,\nabla_Xe'
		\rrbracket + \nabla_{\nabla_{e'}^{\text{bas}} X} e -
		\nabla_{\nabla_{e}^{\text{bas}} X} e' - P^{-1}\langle
		\nabla_{\nabla_{.}^{\text{bas}}X}e,e' \rangle,
		\]
		where $P\colon E\to E^*$ is the isomorphism defined by the
		pairing, for all $e,e'\in\Gamma(E)$ and
		$X\in\mathfrak{X}(M)$. The map $\ell$ is
		$\rho^*\colon T^*M\to E$, the $E$-connection on $T^*M$ is
		$\nabla^{\text{bas},*}$ and the form
		$\omega\in\Omega^3(E,T^*M)$ is given by
		$\omega(e_1,e_2,e_3)=\langle \omega_\nabla(e_1,e_2)(\cdot),e_3
		\rangle$. The kind of split Lie 2-algebroids that arise in
		this way are the \textbf{split symplectic Lie 2-algebroids}\index{split symplectic Lie 2-algebroid}
		\cite{Roytenberg02}. They are splittings of the symplectic Lie
		$2$-algebroid which is equivalent to the tangent prolongation
		of $E$. The tangent prolongation of $E$ is an LA-Courant algebroid
		\cite{Jotz19b,Jotz18d}.
	\end{example}
	
	\section{Generalised functions of a Lie $n$-algebroid}
	
	We now turn to the space of generalised functions\index{Lie $n$-algebroid!generalised functions} (or alternatively the vector valued functions\index{Lie $n$-algebroid!vector valued functions}) of a Lie $n$-algebroid. In the following, $(\M,\Q)$ is a Lie $n$-algebroid with underlying
	manifold $M$. Consider the space
	$\cin(\M)\otimes_{C^\infty(M)}\Gamma(\E)$ for a graded vector bundle
	$\E\to M$ of finite rank. For simplicity, the tensor product
	$\cin(\M)\otimes_{C^\infty(M)}\Gamma(\E)$ is sometimes written as
	$\cin(\M)\otimes\Gamma(\E)$. That is, in the
	rest of work this particular kind of tensor products is always considered as a tensor product of $C^\infty(M)$-modules.
	
	First suppose that $(\M,\Q)=(A[1],\diff_A)$ is a Lie algebroid. The
	space of \textbf{$\E$-valued differential forms}\index{vector valued differential form}
	$\Omega(A;\E):=\Omega(A)\otimes_{C^\infty(M)}\Gamma(\E)=\cin(A[1])\otimes_{C^\infty(M)}\Gamma(\E)$
	has a natural grading given by
	\[
	\Omega(A;\E)_p=\bigoplus_{i-j=p}\Omega^i(A;E_j).
	\]
	It is well-known (see \cite{ArCr12}) that any degree preserving vector
	bundle map $h\colon \E\otimes \underline{F}\to \underline{G}$ induces a wedge product operation\index{generalised wedge product}
	\[
	(\cdot\wedge_h\cdot)\colon\Omega(A;\E)\times\Omega(A;\underline{F})\to
	\Omega(A;\underline{G})
	\]
	which is defined on  $\omega\in\Omega^p(A;E_i)$ and $\eta\in\Omega^q(A;F_j)$ by
	\[
	(\omega\wedge_h\eta)(a_1,\ldots,a_{p+q})=\sum_{\sigma\in
		\text{Sh}_{p,q}}(-1)^{qi}\sgn(\sigma)h\left(\omega(a_{\sigma(1)},\ldots,a_{\sigma(p)}),\eta(a_{\sigma(p+1)},\ldots,a_{\sigma(p+q)})\right),
	\]
	for all $a_1,\ldots,a_{p+q}\in\Gamma(A)$. 
	
	In particular, the above rule reads
	\[
	\theta\wedge_h\zeta=(-1)^{qi}\left(\omega\wedge\eta\right)\otimes h(e,f),
	\]
	for all $\theta=\omega\otimes e$ and $\zeta=\eta\otimes f$ where
	$\omega$ is a $p$-form, $\eta$ is a $q$-form, and $e$ and $f$
	are homogeneous sections of $\E$ and $\underline{F}$ of degree $i$ and $j$,
	respectively.
	
	Some notable cases for special choices of the map $h$ are
	given by the identity, the composition of endomorphisms, the
	evaluation and the `twisted' evaluation maps, the graded
	commutator of endomorphisms and the natural pairing of a
	graded vector bundle with its dual. In particular, the evaluation $(\Phi,e)\mapsto \Phi(e)$ and the twisted
	evaluation $(e,\Phi)\mapsto(-1)^{|\Phi||e|}\Phi(e)$ make $\Omega(A;\E)$ a graded
	$\Omega(A;\underline{\End}(\E))$-bimodule.
	
	In the general case of a Lie $n$-algebroid $(\M,\Q)$, the space $\Omega(A)$ is
	replaced by the generalised smooth functions $\cin(\M)$ of
	$\M$. The space $\cin(\M)\otimes\Gamma(\E)$ has a natural grading,
	where the homogeneous elements of degree $p$ are given by
	\[
	\bigoplus_{i-j=p}\cin(\M)^i\otimes\Gamma(E_j).
	\]
	
	Similarly as in the case of a Lie algebroid, given a degree preserving map
	\[
	h\colon \E\otimes \underline{F}\to \underline{G},
	\]
	one obtains the multiplication
	\begin{align*}
	\left(\cin(\M)\otimes\Gamma(\E)\right)\times \left(\cin(\M)\otimes\Gamma(\underline{F})\right)\to\, & \cin(\M)\otimes\Gamma(\underline{G})\\
	(\omega,\eta)\mapsto\, & \omega\wedge_h\eta.
	\end{align*}
	In particular, for elements of the form
	$\xi\otimes e\in \cin(\M)^i\otimes\Gamma(E_j),\zeta\otimes f\in
	\cin(\M)^k\otimes\Gamma(F_\ell)$ the above rule reads
	\[
	\left(\xi\otimes e\right)\wedge_h\left(\zeta\otimes f\right)=(-1)^{(-j)k}\xi\zeta\otimes h(e,f),
	\]
	where on the right hand side the multiplication $\xi\zeta$ is the one
	in $\cin(\M)$. The special cases above are defined similarly for the
	$n$-algebroid case. Moreover,
	$\cin(\M)\otimes\Gamma(\E)$ is endowed with the
	structure of a graded
	$\cin(\M)\otimes\Gamma(\underline{\End}(\E))$-bimodule.
	
	Finally, the following lemma will be useful later as it is a
	generalisation of \cite[Lemma A.1]{ArCr12}, and gives the connection
	between the space
	$\cin(\M)\otimes\Gamma(\underline{\Hom}(\E,\underline{F}))$ and the
	homomorphisms from $\cin(\M)\otimes\Gamma(\E)$ to
	$\cin(\M)\otimes\Gamma(\underline{F})$.
	\begin{lemma}\label{wedge_product-operators_Correspondence_Lemma}
		There is a one-to-one correspondence between the degree $n$ elements of
		$\cin(\M)\otimes\Gamma(\underline{\Hom}(\E,\underline{F}))$ and the operators
		$\Psi\colon \cin(\M)\otimes\Gamma(\E)\to \cin(\M)\otimes\Gamma(\underline{F})$
		of degree $n$ which are $\cin(\M)$-linear in the graded sense:
		\[
		\Psi(\xi\wedge\eta)=(-1)^{nk}\xi\wedge \Psi(\eta),
		\]
		for all $\xi\in \cin(\M)^k$, and all $\eta\in \cin(\M)\otimes\Gamma(\E)$.
	\end{lemma}
	\begin{proof}
		The element $\Phi\in \cin(\M)\otimes\Gamma(\underline{\End}(\E))$
		induces the operator $\widehat{\Phi}$ given by left multiplication
		by $\Phi$:
		\[
		\widehat{\Phi}(\eta)=\Phi\wedge\eta.
		\]
		This clearly satisfies
		$\widehat{\Phi}(\xi\wedge\eta)=(-1)^{nk}\xi\wedge\widehat{\Phi}(\eta)$,
		for all
		$\xi\in \cin(\M)^k,\ \eta\in
		\cin(\M)\otimes\Gamma(\E)$.  Conversely, an operator $\Psi$ of degree $n$
		must send a section $e\in\Gamma(E_k)$ into the sum
		\[
		\Gamma(F_{k-n}) \oplus \left(
		\cin(\M)^1\otimes\Gamma(F_{k-n+1}) \right) \oplus
		\left( \cin(\M)^2\otimes\Gamma(F_{k-n+2}) \right)
		\oplus\dots,
		\]
		defining the elements
		\[
		\Psi_i\in C^\infty(\M)^i\otimes\Gamma(\underline{\Hom}^{n-i}(\E,\underline{F})).
		\]
		Thus, this yields the (finite) sum
		$\widetilde{\Psi} = \sum_i \Psi_i\in
		\Big(\cin(\M)\otimes\Gamma(\underline{\Hom}(\E,\underline{F}))\Big)^n$.
		Clearly,
		\[
		\widetilde{\widehat{\Phi}} = \Phi\ \text{and}\ \widehat{\widetilde{\Psi}} = \Psi.\qedhere
		\]
	\end{proof}
	Schematically, for a Lie $n$-algebroid $\M$, the above lemma gives the following diagram:
	\[
	\Big(\cin(\M)\otimes\Gamma(\underline{\Hom}(\E,\underline{F}))\Big)^n\xleftrightarrow{\text{1-1}}
	\left\{\begin{array}{c}
	\text{Degree}\ n\ \text{operators}\ \Psi \\
	\cin(\M)\otimes\Gamma(\E)\to \cin(\M)\otimes\Gamma(\underline{F}) \\
	\text{which are}\ \cin(\M)\text{-linear in the graded sense}
	\end{array}\right\}.
	\]
	In particular, if $\E = \underline{F}$, then
	\[
	\Big(\cin(\M)\otimes\Gamma(\underline{\End}(\E))\Big)^n\xleftrightarrow{\text{1-1}}
	\left\{\begin{array}{c}
	\text{Degree}\ n\ \text{operators}\ \Psi\ \text{on}\ \cin(\M)\otimes\Gamma(\E)\ \text{which}\\
	\text{are}\ \cin(\M)\text{-linear in the graded sense}
	\end{array}\right\}.
	\]
	
	\section{Graded Poisson structures}
	
	Poisson structures can be defined on $\mathbb{Z}$-manifolds for any $k\in\mathbb{Z}$. More precisely, a \textbf{degree $k$ Poisson bracket}\index{graded Poisson bracket} on a $\mathbb{Z}$-manifold $\M$ is an $\mathbb{R}$-bilinear map
	$\{\cdot\,,\cdot\}_k\colon \cin(\M)\times \cin(\M)\to \cin(\M)$ of degree $k$, i.e.~$|\{\xi_1,\xi_2\}_k| = |\xi_1| + |\xi_2| + k$, such that it satisfies the following conditions:
	\begin{enumerate}
		\item $\{ \xi_1,\xi_2 \}_k = - (-1)^{(|\xi_1|+k)(|\xi_2|+k)}\{ \xi_2,\xi_1 \}_k$,
		
		\item $\{\xi_1,\xi_2\xi_3\}_k = \{\xi_1,\xi_2\}_k\xi_3 + (-1)^{(|\xi_1|+k)|\xi_2|}\xi_2\{\xi_1,\xi_3\}_k$,
		
		\item $\{\xi_1,\{\xi_2,\xi_3\}\}_k = \{\{\xi_1,\xi_2\}_k,\xi_3\}_k + (-1)^{(|\xi_1|+k)(|\xi_2|+k)}_k\{\xi_2,\{\xi_1,\xi_3\}\}_k$,
	\end{enumerate} 
	for homogeneous elements $\xi_1,\xi_2,\xi_3\in \cin(\M)$. $\mathbb{Z}$-graded manifolds endowed with a Poisson bracket of degree $k$ are called \textbf{$\mathcal{P}_k$-manifolds}\index{$\mathcal{P}_k$-manifold}, or simply \textbf{$\mathcal{P}$-manifolds}\index{$\mathcal{P}$-manifold}. An $\mathbb{N}$-manifold together with a Poisson bracket of degree $k$ is called \textbf{$\mathbb{N}\mathcal{P}_k$-manifold}\index{$\mathbb{N}\mathcal{P}_k$-manifold}, or simply \textbf{$\mathbb{N}\mathcal{P}$-manifold}\index{$\mathbb{N}\mathcal{P}$-manifold}. A \textbf{morphism of two
		$\mathcal{P}$-manifolds}\index{$\mathcal{P}$-manifold!morphism} (respectively $\mathbb{N}\mathcal{P}$-manifolds\index{$\mathbb{N}\mathcal{P}$-manifold!morphism}) $(\N,\{\cdot\,,\cdot\}_\N)$ and $(\M,\{\cdot\,,\cdot\}_\M)$ with both brackets of degree $k$ is a morphism
	of graded manifolds $\mathcal{F}\colon\N\to\M$ which respects the Poisson
	brackets:
	$\mathcal{F}^\star\{\xi_1,\xi_2\}_\M =
	\{\mathcal{F}^\star\xi_1,\mathcal{F}^\star\xi_2\}_\N$ for all
	$\xi_1,\xi_2\in \cin(\M)$.
	
	As in the ordinary case, every degree $k$ Poisson bracket on $\M$ induces a degree $k$ map
	\[
	\sharp\colon \cin(\M)\to \mathfrak{X}(\M)
	\]
	which sends $\xi$ to its \textbf{Hamiltonian vector field}\index{Hamiltonian vector field} $\mathcal{X}_\xi=\{\xi\,,\cdot\}_k$ (of degree $|\xi|+k$). A graded manifold is called \textbf{symplectic} if it is equipped with a degree $k$ Poisson bracket whose Hamiltonian vector fields generate all of $\mathfrak{X}(\M)$. Note that the Hamiltonian map characterises the morphisms between graded Poisson manifolds as morphisms $\mathcal{F}:\N\to\M$ such that the Hamiltonian vector field $\X_{\xi}$ on $\M$ is $\mathcal{F}$-related to the Hamiltonian vector field $\X_{\mathcal{F}^\star\xi}$ on $\N$ for all $\xi\in\cin(\M)$.
	
	A homogeneous vector field $\X$ of a graded Poisson manifold $\M$ is called \textbf{Poisson vector field}\index{Poisson vector field} if it is a derivation of the Poisson bracket:
	\[
	\X\{\xi_1,\xi_2\}_k = \{\X(\xi_1),\xi_2\}_k + (-1)^{|\X|(|\xi_1|+k)}\{\xi_1,\X(\xi_2)\},
	\]
	for all homogeneous $\xi_1,\xi_2\in\cin(\M)$. Equivalently, using the Hamiltonian vector field notation the defining formula of Poisson vector fields can be expressed as $\X_{\X(\xi)}=[\X,\X_\xi]$ for all $\xi\in\cin(\M)$. Yet another equivalent form of the compatibility condition of $\X$ with $\{\cdot\,,\cdot\}_k$ will be given later using the generalisation of the Schouten bracket on the space of (pseudo)multivector fields on $\M$. Note that similarly to the case of ordinary Poisson manifolds, the Jacobi identity of a Poisson bracket can be restated as the commutativity of the Hamiltonian map $\sharp$ with the brackets $\{\cdot\,,\cdot\}_k$ and $[\cdot\,,\cdot]$:
	\[
	\X_{\{\xi_1,\xi_2\}_k} = [\X_{\xi_1},\X_{\xi_2}],
	\]
	for all $\xi_1,\xi_2\in\cin(\M)$, or equivalently, every Hamiltonian vector field is Poisson:
	\[
	\X_{\xi_1}\{\xi_2,\xi_3\}_k = \{\X_{\xi_1}(\xi_2),\xi_3\}_k + (-1)^{(|\xi_1|+k)(|\xi_2|+k)}\{\xi_2,\X_{\xi_1}(\xi_3)\}_k,
	\]
	for all $\xi_1,\xi_2,\xi_3\in\cin(\M)$.
	
	A \textbf{$\mathcal{PQ}$-manifold}\index{$\mathcal{PQ}$-manifold} $(\M,\Q,\{\cdot\,,\cdot\}_k)$ is a graded manifold
	$\M$ endowed with a compatible pair of a homological vector field $\Q$ and a Poisson
	bracket $\{\cdot\,,\cdot\}_k$, i.e~a \textbf{Poisson homological vector field}\index{Poisson homological vector field}:
	\[
	\Q\{\xi_1,\xi_2\}_k = \{\Q(\xi_1),\xi_2\}_k + (-1)^{|\xi_1|+k}\{\xi_1,\Q(\xi_2)\}_k
	\]
	for all $\xi_1,\xi_2\in \cin(\M)$. If the underlying manifold $\M$ is $\mathbb{N}$-graded of degree $n\in\mathbb{N}$, then the $(\M,\Q,\mathcal{P})$ is called \textbf{Poisson Lie $n$-algebroid}\index{Poisson Lie $n$-algebroid}, or \textbf{symplectic Lie $n$-algebroid}\index{symplectic Lie $n$-algebroid} if the Poisson bracket is symplectic. A \textbf{morphism of $\mathcal{PQ}$-manifolds}\index{$\mathcal{PQ}$-manifold!morphism} (similarly of Poisson and symplectic Lie $n$-algebroids) is a morphism of the underlying
	graded manifolds which is a morphism of both $\Q$- and $\mathcal{P}$-manifolds.
	
	Clearly, a Poisson (symplectic) Lie 0-algebroid is a usual Poisson (symplectic) manifold $M$. A Poisson Lie 1-algebroid is a Lie bialgebroid $(A,A^*)$ and a symplectic Lie 1-algebroid is again in correspondence with a usual Poisson manifold as it becomes the Lie bialgebroid $(T^*M,TM)$ of a Poisson manifold $M$. In \cite{Roytenberg02}, it was shown by Roytenberg that split symplectic Lie 2-algebroids correspond to Courant algebroids. More details about these correspondences are given in Chapter \ref{Chapter: Graded tangent and cotangent bundles} and in particular in Section \ref{Section: Homotopy Lie bialgebroids} and Section \ref{Section: Poisson manifolds and Courant algebroids in the graded setting}.
	
	\chapter{Graded tangent and cotangent bundles}\label{Chapter: Graded tangent and cotangent bundles}
	
	This chapter defines (graded) vector bundles in the category of graded manifolds and analyses, in particular, two fundamental examples of vector bundles accompanying any graded manifold. Namely, the tangent and cotangent bundles of a graded manifold. As the tangent and cotangent bundles are our main objects for the theory of modules and representations in the next chapters, we provide here many details about their structure. Moreover, we explain the equivalence of Poisson manifolds and Courant algebroids with symplectic Lie 1- and 2-algebroids, respectively. The following is based on \cite{Mehta06}.
	
	\section{Vector bundles over graded manifolds}\label{Section: Vector bundles over graded manifolds}
	
	Suppose that $\{k_i\}$ is a finite collection of natural numbers. As usual, a \textbf{vector bundle of rank $\{k_i\}$}\index{vector bundle over a graded manifold} \cite{Mehta06,Mehta09} in the category of $\mathbb{Z}$-graded manifolds is defined as a surjection of graded manifolds $q\colon \mathcal{E}\to \M$ endowed with an atlas of local trivialisations $\mathcal{E}|_{q^{-1}(U)} \cong \M|_U\times\mathbb{R}^{\{k_i\}}$ over the open set $U\subset M$, such that the transition functions between two local trivialisations are linear in the fibre coordinates.
	
	\begin{remark}
		The standard algebraic constructions such as tensor products, direct sums, duals, etc., can be applied in this category yielding new vector bundles. In particular, one may apply the degree shifting functor $[j]$ to $\mathcal{E}$ to produce another vector bundle $q[j]\colon \mathcal{E}[j]\to\M$ obtained by $\mathcal{E}$ after shifting the fibre coordinates by $j$, i.e.~the fibre coordinates $\{e^i\}$ of $\Em$ induce the fibre coordinates $\{e^i\}$ of degree $|e^i|-j$ for $\Em[j]$.
	\end{remark}
	
	The \textbf{space of sections of $\mathcal{E}$}\index{vector bundle over a graded manifold!sections} is defined as $\Gamma(\mathcal{E}):=\bigoplus_{j\in\mathbb{Z}}\Gamma^j(\mathcal{E})$, where $\Gamma^j(\mathcal{E})$ consists of the (homogeneous) \textbf{degree $j$ sections} of $\mathcal{E}$, i.e.~morphisms of graded manifolds $\sigma\colon\M\to\mathcal{E}[j]$ such that $q\circ\sigma=\id_\M$. There is another equivalent definition for a vector bundle given by its sheaf of sections $\Gamma(\mathcal{E})$ as an analogue of the Serre-Swan theorem of ordinary vector bundles over smooth manifolds \cite{Mehta06}. Explicitly, a vector bundle over $\M$ can be defined as a sheaf $\mathcal{E}$ of locally freely generated (left) graded $\cin(\M)$-modules over the smooth manifold $M$.
	
	Note that there are two special classes of functions associated to every vector bundle $\mathcal{E}$ called \textbf{basic}\index{vector bundle over a graded manifold!basic functions} and \textbf{linear}\index{vector bundle over a graded manifold!linear functions} functions and denoted $\cin_{\text{bas}}(\mathcal{E})$ and $\cin_{\text{lin}}(\mathcal{E})$, respectively. The former consists of the pull-back of functions of $\cin(\M)$ via the projection $q$ while the latter is given by the functions in $\cin(\mathcal{E})$ which are linear in the fibre coordinates. In other words, any linear function $\ell\in\cin(\Em)$ can be written locally as $\ell = e^j\xi_j$, where ${e^j}$ is a set of fibre coordinates of $\mathcal{E}$ and $\xi_j$ are local smooth functions on $\M$. In fact, the linear functions of a vector bundle $\Em$ are identified with $\Gamma(\Em^*)$, the $\cin(\M)$-module of sections of the dual bundle $\Em^*\to\M$. We define the bigrading $(p,q)$ on the space of functions $\cin(\Em)$ of the total space of $q:\Em\to\M$ so that it satisfies
	\[
	\cin(\Em)^i = \bigoplus_{p+k=i} \cin(\Em)^{p,k},
	\]
	for all $i\in\mathbb{Z}$; that is, the horizontal degree $p$ is the ``polynomial'' degree obtained from the linear coordinates. This implies in particular that a basic function $q^\star(\xi)\in\cin_{\text{bas}}(\mathcal{E})^i$ has bidegree $(0,i)$ for all $\xi\in\cin(\M)^i$ and a linear function $\ell\in\cin_{\text{lin}}(\mathcal{E})^i$ has bidegree $(1,i-1)$.
	
	\begin{example}[Split vector bundles]\label{Splitting of vector bundles example}
		Suppose $\M$ is a $\mathbb{Z}$-manifold over $M$ and $\E=\bigoplus_{i\in\mathbb{Z}} E[i]\to M$ is a graded vector bundle over $M$. Then the tensor product $\cin(\M)\otimes\Gamma(\E)$
		defines the sheaf of sections of a vector bundle $\Em\to\M$. In the case of a split $\mathbb{Z}$-manifold $\M=\A=\bigoplus_iA_i[i]$, this implies in particular that the coordinates of the $\mathbb{Z}$-manifold $\Em$ are given by
		\begin{enumerate}
			\item \textbf{Basic:}\index{vector bundle over a graded manifold!basic coordinates} coordinates $x^i$ of $M$ of degree $0$ (bidegree $(0,0)$) and coordinates $\alpha^i_k$ of degree $i$ (bidegree $(0,i)$), where $\{\alpha^i_k\}$ is a local frame of $A_i^*$. 
			\item \textbf{Linear:}\index{vector bundle over a graded manifold!linear coordinates} coordinates $\varepsilon_k^i$ of degree $i$ (bidegree $(1,i-1)$), where $\{\varepsilon_k^i\}$ is a local frame of $E_i^*$.
		\end{enumerate}
	\end{example}
	
	\begin{remark}\label{Remark: Splitting of vector bundles}
		Mehta proved in \cite{Mehta14} that all the vector bundles over $\mathbb{N}$-graded manifolds come non-canonically from this construction. Briefly, the proof goes as
		follows: Let $\Em\to\M$ be a vector bundle over the $\mathbb{N}$-manifold $\M$. The pull-back\footnote{Obtained by the induced $C^\infty(M)$-module structure of $\Gamma(\mathcal{E})$} $0_{\M}^!\mathcal{E}$ with respect to the zero embedding
		$0_{\M}\colon M\to\M$ is an ordinary graded
		vector bundle $\E$ over $M$ and hence splits as a direct sum
		$\E=\bigoplus_i E_i[i]$. The double pull-back $\pi_{\M}^!0_{\M}^!\mathcal{E}$, where
		$\pi_{\M}\colon \M\to M$ is the projection map, is non-canonically isomorphic
		to $\mathcal{E}$ as vector bundles over $\M$. Then, as a sheaf over $M$,
		$\Gamma(\mathcal{E})$ is identified with
		$\Gamma(\pi_{\M}^!0_{\M}^!\mathcal{E})=\Gamma(\pi_{\M}^!\E)$, which
		in turn is canonically isomorphic to
		$\cin(\M)\otimes\Gamma(\E)$.
	\end{remark}
	
	\begin{example}[Tangent and cotangent bundles]
		For every graded manifold $\M$, the \textbf{tangent bundle $T\M$}\index{tangent bundle of a graded manifold} is a vector bundle over $\M$ whose sheaf of sections is given by the vector fields $\mathfrak{X}(\M)$. For a description of $T\M$ in terms of local trivialisations and transition maps, we proceed as follows \cite{Mehta06,Mehta09}: Given a local chart $U\subset M$ of $\M$ with coordinates $\xi_i^j,i\in\mathbb{Z},j=1,\ldots,r_i,$ of degree $|\xi_i^j|=-i$ as in (\ref{local structure of graded manifold}), we form the product manifold $\T\M|_U:=\M|_U\times \mathbb{R}^{\{r_i\}}$ with coordinates $\xi_i^j,\dot{\xi}_i^j$ of degree $|\xi_i^j| = -i = |\dot{\xi}_i^j|$. Given two charts $U,V\subset M$ of $\M$ with coordinates $\xi_i^j$ and $\zeta_i^j$, respectively, and the coordinate transformation $\mu_{UV}:\M|_U\to \M|_V$ from $\xi_i^j$ to $\zeta_i^j$, we obtain the coordinate transformation $T\mu_{UV}:T\M|_U\to T\M|_V$ defined by
		\[
		\left(T\mu_{UV}\right)^\star\left(\zeta_i^j\right) = \mu_{UV}^\star\left(\zeta_i^j\right)
		\qquad \text{and} \qquad
		\left(T\mu_{UV}\right)^\star \left(\dot{\zeta}_i^j\right) = \sum_{s,t} \dot{\xi}_s^t\frac{\partial}{\partial \xi_s^t}\mu_{UV}^\star\left( \zeta_i^j \right).
		\]
		The graded manifold $T\M$ is obtained by gluing the local charts $T\M|_U$ with the transformations defined above. The dual of $T\M$ is the \textbf{cotangent bundle $T^*\M$} with sheaf of sections the $1$-forms $\Omega^1(\M)$. If $\M$ is $\mathbb{N}$-graded, then in terms of the isomorphism of Example \ref{Splitting of vector bundles example}, the tangent and cotangent bundles are represented by the sheaves
		\[
		\Gamma(T\M)\cong\cin(\M)\otimes\Gamma(\A\oplus TM[0])
		\qquad \text{and} \qquad
		\Gamma(T^*\M)\cong\cin(\M)\otimes\Gamma(T^*M[0]\oplus\A^*),
		\]
		where the vector bundle $\A=\bigoplus_{i} A_i[i]\to M$ is a splitting of the base manifold $\M$, i.e.~$\cin(\M)\cong\Gamma(\underline{S}(\A^*))$. We will see more geometric structures on $T\M$ and $T^*\M$ in the coming sections.
	\end{example}
	
	\section[(Pseudo)differential forms and the Weil algebra]{(Pseudo)differential forms and the Weil algebra of a $\Q$-manifold}\label{Section: (Pseudo)differential forms and the Weil algebra of a Q-manifold}
	
	Let $\M$ be an $\mathbb{Z}$-manifold over a smooth manifold $M$ and
	$\xi_i^1,\ldots,\xi_i^{r_i},i\in\mathbb{Z}$
	be its local generators over some open $U\subset M$, as in (\ref{local structure of graded manifold}). By definition, the \textbf{(shifted) tangent bundle}\index{shifted tangent bundle}\footnote{Note that here there is a
		sign difference in the notation with \cite{Mehta06} and
		\cite{Mehta09}. $T[1]\M$ here is the same as $T[-1]\M$ in these papers.} $T[1-k]\M$
	of $\mathcal M$, for $k\in\mathbb{Z}$, is the $\mathbb{Z}$-manifold \cite{Mehta06,Mehta09} whose structure sheaf assigns to every coordinate domain
	$(U,x^1,\ldots,x^m)$ of $M$ that trivialises $\mathcal M$, the space
	\[
	\cin_U(T[1-k]\M) =
	\bigoplus_{i\in\mathbb{Z}}\underset{\text{$(0,i)$}}{\underbrace{\cin_U(\M)^i}}\left<
	\underset{(1,-k)}{\underbrace{\left(\diff x^s\right)_{s=1}^m}};
	\underset{(1,-k+|\xi_j^s|)}{\underbrace{\left(\diff\xi_j^s\right)_{s=1}^{r_j}}}, j\in\mathbb{Z} 
	\right>.
	\]
	It carries the bigrading $(p,q)$, where $p$ is the grading of ``(pseudo)differential forms'' (or ``polynomial degree"\footnote{This is obtained by counting the ``$\diff$'s" and thus it is non-negative; e.g., the element $\xi\in\cin(\M)$ viewed as an element in $\cin(T[1-k]\M)$ has polynomial degree $0$ and the element $\diff\xi_1\ldots\diff\xi_n$ has polynomial degree $n$.}) and $q$ comes from the grading of
	$\M$. Note that as the structure sheaf of a $\mathbb{Z}$-manifold the above algebra of functions is considered with the \textit{total degree}. That is, an element of bidegree $(p,q)$ is a function of degree $p+q$; e.g.~the degrees of the functions $\xi$ and $\diff\xi$ in $\cin(T[1-k]\M)$ are $|\xi|$ and $|\xi|+1-k$, respectively, for all homogeneous $\xi\in\cin(\M)$.
	
	We are interested in the case of $k=0$. Suppose that $(\M,\Q)$ is a $\Q$-manifold. Then $T[1]\M$ is a $\mathbb{Z}$-manifold,
	which inherits the two commuting differentials $\ldr{\Q}$ and $\dr$ defined as follows:
	\begin{itemize}
		\item $\dr\colon\cin(T[1]\M)^\bullet\to\cin(T[1]\M)^{\bullet+1}$ is the \textbf{de Rham differential}\index{de Rham differential} defined on generators by $C^\infty(M)\ni f  \mapsto \diff f$,
		$\xi_i^j  \mapsto \diff \xi_i^j$, $\diff f\mapsto 0$ and 
		$\diff \xi_i^j  \mapsto 0$,
		and is extended to the whole algebra as a derivation of bidegree $(1,0)$:
		\[
		\diff(\xi_1\diff\xi_2) = \diff\xi_1\xi_2 + (-1)^{|\xi_1|}\xi_1\diff\xi_2,
		\qquad \text{for homogeneous}\ \xi_1\in\cin(\M)\ \text{and all}\ \xi_2\in\cin(\M). 
		\]
		Note that this does not require the existence of the $\Q$-manifold structure on $\M$ and is hence well-defined for all $\mathbb{Z}$-manifolds.

		\item
		$\ldr{\Q}\colon\cin(T[1]\M)^\bullet\to\cin(T[1]\M)^{\bullet+1}$
		is the \textbf{Lie derivative}\index{Lie derivative} with respect to the vector
		field $\Q$, i.e.~the graded commutator
		$\ldr{\Q} = [i_{\Q},\dr] = i_\Q\circ\dr - \dr\circ i_Q$, and
		it is a derivation of bidegree $(0,1)$. Here, $i_{\Q}$ is the bidegree
		$(-1,1)$-derivation on $T[1]\M$, which sends
		$\xi\in \cin(\M)$ to $0$, $\diff \xi$ to $\Q(\xi)$ for
		$\xi\in \cin(\M)$, and is extended to the whole algebra as
		a derivation of bidegree $(-1,1)$:
		\[
		i_\Q(\xi_1\diff\xi_2) = i_\Q(\xi_1)\diff\xi_2 + (-1)^{|\xi_1|(1-1)}\xi_1i_\Q(\diff(\xi_2)) = \xi_1\Q(\xi_2).
		\]
	\end{itemize}
	
	Note that the above operators $\dr,i_\Q$ and $\ldr{\Q}$ are vector fields on the manifold $T[1]\M$ of degree $1,0$ and $1$, and bidegree $(1,0),(-1,1)$ and $(0,1)$, respectively. In particular, $i_\Q$ and $\ldr{\Q}$ are special cases of the two vector fields $i_\X$ and $\ldr{\X}$ on $T[1]\M$ obtained for any $\X\in\mathfrak{X}(\M)$; the former is of bidegree $(-1,|\X|)$ and is defined on the generators by $i_\X(\xi) = 0$ and $i_\X(\diff\xi) = \X(\xi)$, while the latter is of bidegree $(0,|\X|)$ and is defined as the Lie bracket $\ldr{\X}=[i_\X,\diff]$. Locally, they have the following coordinate representations: Given local coordinates $\{\xi^j\}$ on $\M$, the family $\{\xi^j,\dot{\xi}^j\}$ forms a local system of coordinates for the graded manifold $T[1]\M$ with $\xi^j$ and $\dot{\xi}^j$ of degree $|\xi^j|$ and $|\xi^j|+1$, respectively; here we do not use the letter $x^i$ for the coordinates of the body and we write $\dot{\xi}^j$ for the linear coordinate $\diff\xi^j$ for simplicity. Then for a homogeneous $\X=\zeta^j\frac{\partial}{\partial\xi^j}\in\mathfrak{X}(\M)$, the vector fields above become
	\[
	\diff = \dot{\xi}^j\frac{\partial}{\partial\xi^j},\qquad 
	i_\X = \zeta^j\frac{\partial}{\partial\dot{\xi}^j},\qquad 
	\ldr{\X} = \zeta^j\frac{\partial}{\partial \xi^j} + (-1)^{|\X|}\diff\zeta^j\frac{\partial}{\partial\dot{\xi}^j}.
	\]
	
	By checking their values on local generators, it is easy to see that
	$\ldr{\Q}^2 = 0, \dr^2 = 0$ and
	$[\ldr{\Q},\dr] = \ldr{\Q}\circ\dr + \dr\circ\ldr{\Q} = 0$. Hence, the pair
	\[
	W(\M):=\left( \Omega(\M):=\bigoplus_{i\in\mathbb Z}\Omega^{\bullet,\bullet}(\M),\ldr{{\Q}} + \dr \right),
	\]
	where $\Omega^{\bullet,\bullet}(\M):=\cin(T[1]\M)^{\bullet,\bullet}$, forms a double complex. This double complex is called the \textbf{Weil algebra}\index{Weil algebra} of $(\M,\Q)$.
	
	\begin{remark}
		The terminology ``Weil algebra" becomes clear from the case of a Lie $n$-algebroid $(\M,\Q)$, which is analysed in detail in the next chapter.
	\end{remark}
	
	\begin{remark}\label{Cohomologies from differential forms}
		In total, one has three differential complexes, and consequently three cohomologies, associated to any $\Q$-manifold $(\M,\Q)$:
		
		\begin{enumerate}
			\item The \textbf{de Rham complex}\index{de Rham complex} consisting of the columns of $\Omega^{\bullet,\bullet}(\M)$ with differential $\dr$ and the induced \textbf{de Rham cohomology}\index{de Rham cohomology} $H_{dR}^\bullet(\M)$ of $\M$.
			\item The \textbf{Lie complex}\index{Lie complex} of differential forms consisting of the rows of $\Omega^{\bullet,\bullet}(\M)$ with differential $\ldr{{\Q}}$ and the induced \textbf{Lie cohomology}\index{Lie cohomology} of (pseudo)differential forms $H_{\Omega,L}^\bullet(\M)$ of $\M$.
			\item The \textbf{Weil complex}\index{Weil complex} $W(\M)$ and \textbf{Weil cohomology}\index{Weil cohomology} $H^\bullet_W(\M)$ of $(\M,\Q)$ consisting of the the differential graded algebra given by the Weil algebra defined above, i.e.~the total complex of $\Omega^{\bullet,\bullet}(\M)$ and its cohomology.
		\end{enumerate}
	\end{remark}
	
	\section[(Pseudo)multivector fields and the Poisson-Weil algebra]{(Pseudo)multivector fields and the Poisson-Weil algebra of a $\mathcal{PQ}$-manifold}\label{Section:(Pseudo)multivector fields and the Poisson-Weil algebra}
	
	From the theory of ordinary Poisson manifolds, it is well-known that a Poisson bracket $\{\cdot\,,\cdot\}$ on a smooth manifold $M$ is equivalent to a bivector field $\pi\in\mathfrak{X}^2(M):=\Gamma(\wedge^2TM)$ that satisfies the differential equation $[\pi,\pi]=0$, where $[\cdot\,,\cdot]$ is the Schouten bracket on the space $\mathfrak{X}^\bullet(M)$ of multivector fields on $M$. The precise correspondence is given by the derived bracket formula\index{derived bracket formula}
	\[
	\{f,g\} = [[f,\pi],g]
	\]
	for all $f,g\in C^\infty(M)$. Furthermore, the last equality defines a cohomology theory for Poisson manifolds on $\mathfrak{X}^\bullet(M)$ with differential given by $\diff_\pi:=[\pi,\cdot]$. As we will see now, this is also true in the setting of graded manifolds carrying the graded version of a Poisson bracket as defined above.
	
	Suppose that $\M$ is a $\mathbb{Z}$-graded manifold over $M$ and $k\in\mathbb{Z}$. The algebra of functions of the graded manifold $T^*[1-k]\M$ is, by definition, $\underline{S}(\mathfrak{X}(\M)[k-1])$ considered with the total grading, and locally it is given by
	\[
	\cin_U(T^*[1-k]\M) = \bigoplus_{i\in\mathbb{Z}}\underset{(0,i)}{\underbrace{\cin_U(\M)^i}}\left< \underset{(1,-k)}{\underbrace{\left(\frac{\partial}{\partial x^s}\right)_{s=1}^m}};
	\underset{(1,-k-|\xi_j^s|)}{\underbrace{\left(\frac{\partial}{\partial \xi_j^s}\right)_{s=1}^{r_j}}},\ j\in\mathbb{Z} \right>.
	\]
	where $U\subset M$ is a local chart with coordinates $\{x^s\}$, and $\{\xi^s_j\}$ are local coordinates of $\M$ as in (\ref{local structure of graded manifold}). The resulting $\mathbb{Z}$-manifold $T^*[1-k]\M$ is called \textbf{(shifted) cotangent bundle}\index{shifted cotangent bundle} and its structure sheaf will be denoted by $\mathfrak{A}_k^{\bullet,\bullet}(\M)$. The bigrading $(p,q)$ is such that the horizontal degree $p$ counts the ``multivector field" (or ``polynomial") degree\footnote{This is obtained by counting the partial derivatives ``$\frac{\partial}{\partial \xi_j^s}$" and is always non-negative.} and the vertical degree $q$ counts the ``derivation" degree. These are denoted $|\cdot|_p$ and $|\cdot|_q$, respectively. The total degree is denoted $|\cdot|:=|\cdot|_p+|\cdot|_q$. If $k=0$, the resulting function algebra of the graded manifold $T^*[1]\M$ is given by the space of \textbf{(pseudo)multivector fields}\index{pseudomultivector fields} $\mathfrak{X}^{\bullet,\bullet}(\M):=\mathfrak{A}_{0}^{\bullet,\bullet}(\M)$, and if in particular $\M$ is an ordinary (non-graded) smooth manifold $M$ then $\mathfrak{A}_0(M)=\Gamma(\wedge TM)$. See \cite{Cueca19} for more on graded cotangent bundles over $\mathbb{N}$-manifolds. 
	
	The \textbf{bidegree $(-1,k)$ Schouten bracket}\index{Schouten bracket of bidegree $(-1,k)$} on $\mathfrak{A}_{k}^{\bullet,\bullet}(\M)$ is defined inductively by the following conditions:
	\begin{enumerate}
		\item For all $\xi\in\cin(\M)$ and all vector fields $\X,\Y\in\mathfrak{X}^\bullet(\M)$:
		\[
		[\X,\xi]_k:=\X(\xi)\qquad \text{and}\qquad [\X,\Y]_k:=[\X,\Y].
		\]
		\item For homogeneous elements of higher degree $\X,\Y,\mathcal{Z}\in\mathfrak{A}_k^{\bullet,\bullet}(\M)$, it is extended by the rules:
		\begin{enumerate}
			\item $[\X,\Y]_k = -(-1)^{(|\X|+k-1)(|\Y|+k-1)}[\Y,\X]_k$,
			\item $[\X,\Y\wedge\mathcal{Z}]_k = [\X,\Y]_k\wedge\mathcal{Z} + (-1)^{(|\X|+k-1)|\Y|}\Y\wedge[\X,\mathcal{Z}]_k$.
		\end{enumerate}
	\end{enumerate}
	Using an induction argument, it can be shown that the Schouten bracket on $\mathfrak{A}_{k}^{\bullet,\bullet}(\M)$ satisfies the following graded version of the Jacobi identity:
	\[
	[\X,[\Y,\mathcal{Z}]_k]_k = [[\X,\Y]_k,\mathcal{Z}]_k + (-1)^{(|\X|+k-1)(|\Y|+k-1)}[\Y,[\X,\mathcal{Z}]_k]_k,
	\]
	for all homogeneous elements $\X,\Y,\mathcal{Z}\in\mathfrak{A}_k^{\bullet,\bullet}(\M)$.
	
	As in the ordinary case, the antisymmetry and the Leibniz identity of a degree $k$ Poisson bracket on $\M$ imply the existence of a bivector field $\pi$ of bidegree $(2,-k)$, i.e.~$\pi\in\mathfrak{A}_k^{2,-k}(\M)$. The correspondence is given again by Voronov's derived bracket formula\index{derived bracket formula}
	\begin{equation*}\label{Derived bracket formula}
	\{ \xi_1,\xi_2 \}_k = [[\xi_1,\pi]_k,\xi_2]_k
	\end{equation*}
	for all $\xi_1,\xi_2\in\cin(\M)$. Note that this implies the equality 
	\[
	[\xi,\pi]_k = \X_{\xi}\qquad \text{for all } \xi\in\cin(\M).
	\]
	
	A similar formula can be used to determine the component functions of any multivector field over a local coordinate chart of $\M$. Specifically, given a (local) element $\X\in\mathfrak{A}_k^{\ell,\bullet}(\M)$ whose coordinate representation is a finite sum of the form
	\[
	\X = \sum \xi_{i_1\ldots i_\ell} \frac{\partial}{\partial \xi_{i_1}}\wedge\ldots\wedge\frac{\partial}{\partial \xi_{i_\ell}}
	\]  
	with $\xi_{i_1},\ldots,\xi_{i_\ell}$ coordinate functions and $\xi_{i_1\ldots i_\ell}$ a local smooth function, one has
	\[
	[\xi_{i_\ell},[\xi_{i_{\ell-1}},[\ldots[\xi_{i_1},\X]_k]_k]_k]_k = \pm \xi_{i_1\ldots i_\ell}.
	\]
	In particular, we obtain the following Lemma.
	\begin{lemma}\label{Triviality of a multivector field}
		An $\ell$-vector field $\X\in\mathfrak{A}_k^{\ell,\bullet}(\M)$ is zero if and only if 
		\[
		[\xi_{\ell},[\xi_{\ell-1},[\ldots[\xi_{1},\X]_k]_k]_k]_k = 0
		\]
		for all $\xi_{1},\ldots,\xi_{\ell}\in\cin(\M)$.
	\end{lemma}
	
	\begin{proof}
		If $\X = 0$, then clearly $[\xi_{\ell},[\xi_{\ell-1},[\ldots[\xi_{1},\X]_k]_k]_k]_k = 0$ for all $\xi_{1},\ldots,\xi_{\ell}\in\cin(\M)$. Conversely, suppose $[\xi_{\ell},[\xi_{\ell-1},[\ldots[\xi_{1},\X]_k]_k]_k]_k = 0$ for all $\xi_{1},\ldots,\xi_{\ell}\in\cin(\M)$ and write the vector field $\X$ in coordinates as
		\[
		\X = \sum \xi_{i_1\ldots i_\ell} \frac{\partial}{\partial \xi_{i_1}}\wedge\ldots\wedge\frac{\partial}{\partial \xi_{i_\ell}}.
		\]
		For the coordinate functions we have
		\[
		0 = [\xi_{i_\ell},[\xi_{i_{\ell-1}},[\ldots[\xi_{i_1},\X]_k]_k]_k]_k = \pm \xi_{i_1\ldots i_\ell}
		\]
		and hence $\X = 0$.
	\end{proof}
	
	\begin{proposition}\label{Differential equation for Poisson bracket}
		Let $\{\cdot\,,\cdot\}_k$ be Poisson bracket of even degree $k$ on a graded manifold $\M$ and denote the corresponding bivector field by $\pi\in\mathfrak{A}_k^{2,-k}(\M)$. Then the Jacobi identity for $\{\cdot\,,\cdot\}_k$ is equivalent to the differential equation
		\[
		[\pi,\pi]_k = 0.
		\]
	\end{proposition}
	\begin{proof}
		For $\xi_1,\xi_2\in\cin(\M)$ we compute
		\begin{align*}
		[\xi_1,[\xi_2,[\pi,\pi]_k]_k]_k = &\ 2(-1)^{|\xi_2|+k-1}[\xi_1,[\pi,[\xi_2,\pi]_k]_k]_k \\
		= &\ 2 (-1)^{|\xi_2|+k-1} \Big( [\X_{\xi_1},\X_{\xi_2}]_k - \X_{\{\xi_1,\xi_2\}_k} \Big).
		\end{align*}
		Applying Lemma \ref{Triviality of a multivector field} above to the last equation, one sees that $[\pi,\pi]_k=0$ if and only if the Hamiltonian map commutes with the brackets $\{\cdot\,,\cdot\}_k$ and $[\cdot\,,\cdot]$, or equivalently, if and only if $\{\cdot\,,\cdot\}_k$ satisfies the Jacobi identity. 
	\end{proof}
	
	In addition to the double complex and the Weil algebra explained above, we will see now that the shifted space of multivector fields $\mathfrak{A}_k^{\bullet,\bullet}(\M)$ of a $\mathbb{Z}$-graded manifold $\M$ with $\Q$ compatible with $\{\cdot\,,\cdot\}_k$ carries two commuting differentials, yielding another double complex and a differential graded algebra associated to $\M$.
	
	Suppose first that $(\M,\Q)$ is a $\Q$-manifold. Then one may use the Schouten bracket with $\Q$ to obtain a canonical operator of bidegree $(0,1)$ on $\mathfrak{A}_k^{\bullet,\bullet}(\M)$ which extends the canonical differential of the adjoint action. Explicitly, one has the \textbf{Lie derivative}\index{Lie derivative}
	\[
	\ldr{\Q}\colon\mathfrak{A}_k^{\bullet,\bullet}(\M)\to\mathfrak{A}_k^{\bullet,\bullet+1}(\M),\qquad \X \mapsto [\Q,\X]_k.
	\]
	The Jacobi identity of the Schouten bracket on $\mathfrak{A}^{\bullet,\bullet}(\M)$ together with $[\Q,\Q]_k=0$ yield that $\ldr{\Q}^2 = 0$ on the entire $\mathfrak{A}_k^{\bullet,\bullet}(\M)$ and hence $\ldr{\Q}$ is a differential operator:
	\[
	[\Q,[\Q,\X]_k]_k = [[\Q,\Q]_k,\X]_k + (-1)^{(1+1-k+k-1)^2}[\Q,[\Q,\X]_k]_k = - [\Q,[\Q,\X]_k]_k,
	\]
	for every $\X\in\mathfrak{A}_k^{\bullet,\bullet}(\M)$.
	
	Suppose now that the $\mathbb{Z}$-manifold $\M$ carries a Poisson bracket $\{\cdot\,,\cdot\}_k$ of degree $k$ with corresponding Poisson bivector field $\pi\in\mathfrak{A}_k^{2,-k}(\M)$. Then $\mathfrak{A}_k^{\bullet,\bullet}(\M)$ is equipped with another canonical differential operator of bidegree $(1,0)$ given by the Schouten bracket with respect to $\pi$:
	\[
	\diff_{\pi}\colon\mathfrak{A}_k^{\bullet,\bullet}(\M)\to\mathfrak{A}_k^{\bullet+1,\bullet}(\M),\qquad \X \mapsto [\pi,\X]_k.
	\]
	Due to Jacobi identity and Proposition \ref{Differential equation for Poisson bracket}, one computes that $\diff_\pi$ indeed squares to zero:
	\[
	[\pi,[\pi,\X]_k]_k = [[\pi,\pi]_k,\X]_k + (-1)^{(2-k+k-1)^2}[\pi,[\pi,\X]_k]_k = -[\pi,[\pi,\X]_k]_k,
	\]
	for all $\X\in\mathfrak{A}_k^{\bullet,\bullet}(\M)$.
	
	\begin{proposition}
		Let $\M$ be a $\mathbb{Z}$-manifold with a Poisson bivector field $\pi\in\mathfrak{A}_k^{2,-k}(\M)$. Then the following are equivalent for any 1-vector field $\Y\in\mathfrak{A}_k^{1,j}(\M)=\mathfrak{X}^{j+k}(\M)$:
		\begin{enumerate}
			\item $\Y$ is a Poisson vector field with respect to the bracket induced by $\pi$:
			\[
			[\Y,\X_\xi] = \X_{\Y(\xi)}, \qquad \text{for all}\ \xi\in\cin(\M).
			\]
			
			\item The vector field $\Y$ commutes with $\pi$, i.e.~$[\pi,\Y]_k=0$.
		\end{enumerate}
	\end{proposition}
	
	\begin{proof}
		For any homogeneous function $\xi\in\cin(\M)$ we compute
		\begin{align*}
		[\xi,[\Y,\pi]_k]_k = &\ [[\xi,\Y]_k,\pi]_k + (-1)^{(|\xi|+k-1)(j+k)}[\Y,[\xi,\pi]_k]_k \\
		= &\ -(-1)^{(|\xi|+k-1)(j+k)} [\Y(\xi),\pi]_k + (-1)^{(|\xi|+k-1)(j+k)}[\Y,\X_\xi] \\
		= &\ -(-1)^{(|\xi|+k-1)(j+k)} \X_{\Y(\xi)} + (-1)^{(|\xi|+k-1)(j+k)}[\Y,\X_\xi] \\
		= &\ (-1)^{(|\xi|+k-1)(j+k)}\left( [\Y,\X_\xi] - \X_{\Y(\xi)} \right).
		\end{align*}
	Since $[\pi,\Y]_k=0$ if and only if $[\Y,\pi]_k=0$ the result follows.
	\end{proof}
	
	Suppose now that $(\M,\Q)$ carries a Poisson bivector field $\pi\in\mathfrak{A}_k^{2,-k}(\M)$. The analysis above implies that $\mathfrak{A}^{\bullet,\bullet}(\M)$ is endowed with two canonical differential operators $\diff_\pi$ and $\ldr{\Q}$ of bidegree $(1,0)$ and $(0,1)$, respectively. Assuming that $\Q$ is a Poisson vector field with respect to the given Poisson bracket, one computes
	\[
	\diff_{\pi}\circ\ldr{\Q}(\X) = [\pi,[\Q,\X]_k]_k = [[\pi,\Q]_k,\X]_k + (-1)^{1^2}[\Q,[\pi,\X]_k]_k = - \ldr{\Q}\circ\diff_{\pi}(\X),
	\]
	for all $\X\in\mathfrak{A}^{\bullet,\bullet}(\M)$. It follows that $\mathfrak{A}_{k}^{\bullet,\bullet}(\M)$ together with the ``twisted'' differential $(-1)^{k-1}\diff_{\pi}$ and $\ldr{\Q}$ forms a double complex with anticommuting differentials which we call the \textbf{Poisson-Weil double complex of $\M$}\index{Poisson-Weil double complex}. Hence, given a graded manifold $\M$ with a homological vector field $\Q$ and a compatible Poisson bracket $\{\cdot\,,\cdot\}_k$, one obtains the following cohomologies:
	
	\begin{enumerate}
		\item The \textbf{Poisson} or \textbf{Lichnerowicz complex}\index{Poisson complex}\index{Lichnerowicz complex} consisting of the columns of $\mathfrak{A}_{k}^{\bullet,\bullet}(\M)$ with differential $\diff_{\pi}$ and the induced Poisson cohomology $H^{\bullet}_P(\M)$ of $\M$.
		\item The \textbf{Lie complex of (pseudo)multivector fields}\index{Lie complex of (pseudo)multivector fields} consisting of the rows of $\mathfrak{A}_{k}^{\bullet,\bullet}(\M)$ with differential $\ldr{\Q}$ and the induced cohomology $H^\bullet_{\mathfrak{A}_{k},L}(\M)$ of $\M$.
		\item The \textbf{Poisson-Weil algebra}\index{Poisson-Weil algebra} consisting of the differential graded algebra given by the total complex
		\[
		\left(\mathfrak{A}_k(\M):=\bigoplus_{i\in\mathbb{Z}}\bigoplus_{p+q=i}\mathfrak{A}_{k}^{p,q}(\M),\delta:=(-1)^{k-1}\diff_{\pi} + \ldr{\Q}\right)
		\]
		and the induced Poisson-Weil cohomology $H^\bullet_{PW}(\M)$ of $\M$. Moreover, $\delta$ is a degree $k-1$ derivation of the Schouten bracket:
		\[
		\delta[\X,\Y]_k = [\delta(\X),\Y]_k + (-1)^{i+j+k-1}[\X,\delta(\Y)]_k
		\]
		for all $\X\in\mathfrak{A}_k^{i,j}(\M),\Y\in\mathfrak{A}_k(\M)$, and therefore we obtain the DGLA 
		\[
		\left( \widehat{\mathfrak{A}}_k(\M):=\mathfrak{A}_k(\M)[1-k],\delta,[\cdot\,,\cdot]_k \right).
		\]
	\end{enumerate}
	
	\begin{remark}
		Some special cocycles of the Poisson-Weil cohomology are the following:
		\begin{enumerate}
			\item The cocycles of bidegree $(0,i)$ consist of Casimir functions of degree $i$ on $\M$ which are closed but not exact in the Lie $\Q$-manifold cohomology.
			\item The cocycles of bidegree $(1,i)$ consist of degree $i+k$ vector fields on $\M$ which are Poisson but not symplectic and are invariant under $\Q$.
			
		\end{enumerate}
	\end{remark}
	
	\begin{remark}
		As we will see in Section \ref{Section: PQ-manifolds: Weil vs Poisson-Weil algebras}, the three cohomologies defined above are naturally related to the three cohomologies from Remark \ref{Cohomologies from differential forms} via a map of bigraded algebras $\pi^\sharp:\Omega^{\bullet,\bullet}(\M)\to\mathfrak{A}_k^{\bullet,\bullet}(\M)$ induced from the $\mathcal{PQ}$-manifold structure on $\M$. This map is also the reason for the choice of the sign $(-1)^{k-1}$ in the horizontal differential of $\mathfrak{A}_k(\M)$.
	\end{remark}
	
	\section{Homotopy Poisson structures and their deformations}
	
	As DGLA's have connection to deformation theory, it is worth mentioning the deformations that are related to the DGLA obtained by the Poisson-Weil algebra defined in the previous section. For this we need the notion of homotopy Poisson structures, which extend simultaneously the notions of $\Q$-manifolds, $\mathcal{P}$-manifolds and $\mathcal{PQ}$-manifolds. These can be found in \cite{Mehta11}, in \cite{CaFe07,Cattaneo08,Schaetz09} under the name \textit{$\mathcal{P}_\infty$-manifolds}, and in \cite{Bruce10,Voronov05,Voronov05a} under the name \textit{higher Poisson manifolds}.
	
	A \textbf{homotopy Poisson structure of degree $k$}\index{homotopy Poisson structure} on a graded manifold $\M$ is, by definition, an element $\Theta\in\mathfrak{A}_k^{2-k}(\M)$ such that $[\Theta,\Theta]_k = 0$. The name ``homotopy" Poisson is explained by the following remarks: Considering bidegrees, one can write
	\[
	\mathfrak{A}_k^{2-k}(\M) = \cin(\M)^{2-k} \oplus
	\mathfrak{A}_k^{1,1-k}(\M)\oplus
	\mathfrak{A}_k^{2,-k}(\M)\oplus
	\mathfrak{A}_k^{3,-1-k}(\M)\oplus\ldots
	\]
	and in particular
	\begin{enumerate}
		\item if $\Theta=\Q\in\mathfrak{A}_k^{1,1-k}(\M)=\mathfrak{X}^1(\M)$, then $\Q$ is homotopy Poisson if and only if $(\M,\Q)$ is a $\mathcal{Q}$-manifold;
		
		\item if $\Theta=\pi\in\mathfrak{A}_k^{2,-k}(\M)$, then $\pi$ is homotopy Poisson if and only if $(\M,\pi)$ is a $\mathcal{P}_k$-manifold;
		
		\item if $\Theta=\Q+\pi\in\mathfrak{A}_k^{1,1-k}(\M)\oplus\mathfrak{A}_k^{2,-k}(\M)$, then $\Q+\pi$ is homotopy Poisson if and only if
		\[
		[\Q,\Q]_k = 0,\qquad [\Q,\pi,]_k = 0\qquad \text{and}\qquad [\pi,\pi]_k = 0,
		\]
		i.e.~$\pi$ is a Poisson bivector field and $\Q$ is a Poisson homological vector field on $\M$. That is, $\M$ is a $\mathcal{P}_k\Q$-manifold;
		
		\item if $\Theta =
		\sum_{i\geq0}\Theta_i\in\bigoplus_{i\geq0}\mathfrak{A}_k^{i,2-k-i}(\M)$, then expanding $[\Theta,\Theta]_{k}=0$ and comparing degrees, we obtain the following system of equations:
		\[
		\sum_{p+q=j}[\Theta_p,\Theta_q] = 0,\qquad j=1,2,\ldots
		\]
		In particular, for $j=2,3,4$ the equation reads
		\begin{enumerate}
			\item $[\Theta_1,\Theta_1]_k + 2[\Theta_0,\Theta_2]_k = 0$,
			\item $[\Theta_1,\Theta_2]_k + [\Theta_0,\Theta_3]_k = 0$,
			\item $[\Theta_2,\Theta_2]_k + 2[\Theta_1,\Theta_3]_k + 2[\Theta_0,\Theta_4] = 0$,
		\end{enumerate}
		i.e.~the failure of $(\M,\Theta_1,\Theta_2)$ to be a $\mathcal{P}\Q$-manifold is measured by the above higher homotopy terms.
	\end{enumerate}
	
	\begin{definition}
		Let $\Theta\in\mathfrak{A}_k^{2-k}(\M)$ be a homotopy Poisson structure on $\M$.
		\begin{enumerate}
			\item An \textbf{infinitesimal deformation of $\Theta$}\index{homotopy Poisson structure!infinitesimal deformation}, or \textbf{formal deformation of $\Theta$}\index{homotopy Poisson structure!formal deformation}, is an element $\Theta'\in\mathfrak{A}_k^{2-k}(\M)$ for which the infinitesimal change of $\Theta$ up to order 2 in the direction of $\Theta'$ is a homotopy Poisson structure of degree $k$, i.e.~such that 
			\[
			[\Theta+\varepsilon\Theta',\Theta+\varepsilon\Theta']_k = \mathcal{O}(\varepsilon^2) \equiv 0 \mod\varepsilon^2,
			\]
			where $\varepsilon$ is a formal infinitesimal parameter.
			\item A \textbf{deformation of $\Theta$}\index{homotopy Poisson structure!deformation} is an element $\Theta'\in\mathfrak{A}_k^{2-k}(\M)$ such that $\Theta+\Theta'$ is a homotopy Poisson structure, i.e.
			\[
			[\Theta+\Theta',\Theta+\Theta']_k = 0 .
			\]
		\end{enumerate}
	\end{definition}
	
	Given $\Theta:=\pi+\Q$ with $\Q$ and $\pi$ a $\mathcal{P}\Q$-structure on $\M$, one computes that for every element $\Theta'\in\mathfrak{A}_k^{2-k}(\M)$
	\[
	[\Q+\pi+\varepsilon\Theta',\Q+\pi+\varepsilon\Theta'] = 2\varepsilon\Big( [\Q,\Theta']_k +[\pi,\Theta']_k \Big) + \varepsilon^2[\Theta',\Theta'] \equiv 2\varepsilon\delta(\Theta') \mod\varepsilon^2.
	\]
	In other words, the degree $2-k$ cocycles of $\mathfrak{A}_k^{\bullet}(\M)$ (or the degree $1$ cocycles of $\widehat{\mathfrak{A}}_k^{\bullet}(\M)$) are in one-to-one correspondence with infinitesimal deformations of $\Theta=\pi+\Q$ viewed as a homotopy Poisson structure on $\M$. In particular, the cocycles of the form $\Lambda+\X\in\mathfrak{A}_{k}^{2,-k}(\M)\oplus\mathfrak{A}_{k}^{1,1-k}(\M)$ are in one-to-one correspondence with infinitesimal deformations of $\pi+\Q$ as a $\mathcal{PQ}$-manifold structure on $\M$, i.e.~$(\pi+\varepsilon\Lambda,\Q+\varepsilon\X)$ forms a compatible pair of a degree $k$ Poisson bivector and a homological vector field if and only if $\Lambda+\X$ is a cocycle in $H^{2-k}_{PW}(\M)$. Counting degrees, one sees that for $\Theta' = \Lambda + \X\in\mathfrak{A}_{k}^{2,-k}(\M)\oplus\mathfrak{A}_{k}^{1,1-k}(\M)$:
	\begin{equation}\label{Infinitesimal deformations of homotopy Poisson structures}
		\delta(\Theta') = 0
		\qquad \Longleftrightarrow \qquad
		\begin{cases}
			\ldr{\Q}(\X) = 0 \\[2pt]
			\diff_{\pi}(\Lambda) = 0 \\[2pt]
			\diff_{\pi}(\X) + \ldr{\Q}(\Lambda) = 0
		\end{cases}
	\end{equation}
	
	Similarly, the \textbf{Maurer-Cartan elements}\index{Maurer-Cartan element} of $\widehat{\mathfrak{A}}_k^{\bullet}(\M)$, i.e.~$\Theta'\in\widehat{\mathfrak{A}}_k^{1}(\M)=\mathfrak{A}_k^{2-k}(\M)$ such that
	\[
	\delta(\Theta') + \frac{1}{2}[\Theta',\Theta']_k = 0,
	\]
	are in one-to-one correspondence with deformations of $\Theta=\pi+\Q$ as a homotopy Poisson structure, and the Maurer-Cartan elements of the subspace $\mathfrak{A}_{k}^{2,-k}(\M)\oplus\mathfrak{A}_{k}^{1,1-k}(\M)$ are in correspondence with deformations of $\Q+\pi$ viewed as a $\mathcal{PQ}$-manifold structure. By considering terms of the same degrees, one obtains the following equivalence
	\begin{equation}\label{Deformations of homotopy Poisson structures}
		\delta(\Theta') + \frac{1}{2}[\Theta',\Theta'] = 0
		\qquad \Longleftrightarrow \qquad
		\begin{cases}
			\ldr{\Q}(\X) + \frac{1}{2}[\X,\X]_k = 0 \\[2pt]
			\diff_{\pi}(\Lambda) + \frac{1}{2}[\Lambda,\Lambda]_k = 0 \\[2pt]
			\diff_{\pi}(\X) + \ldr{\Q}(\Lambda) + [\X,\Lambda]_k = 0
		\end{cases}
	\end{equation}
	where $\Theta' = \Lambda + \X\in\mathfrak{A}_{k}^{2,-k}(\M)\oplus\mathfrak{A}_{k}^{1,1-k}(\M)$.
	
	\section{Homotopy Lie bialgebroids}\label{Section: Homotopy Lie bialgebroids}
	
	In this section, we analyse some of the constructions of this chapter in the case of $[1]$-manifolds of the form $A[1]$, where $A\to M$ is a vector bundle.
	
	Recall that a \textbf{Lie bialgebroid}\index{Lie bialgebroid} is given by a pair of two Lie algebroids in ``duality" over $M$ $(A,A^*)$, such that for all $a,b\in\Gamma(A)$ 
	\[
	\diff_{A^*}[a,b] = [\diff_{A^*}a,b] + [a,\diff_{A^*}b],
	\]
	where $[\cdot\,,\cdot]\colon\Gamma(A)\times\Gamma(A)\to\Gamma(A)$ is the Lie bracket on the space of sections of $A$ and $\diff_{A^*}\colon\Omega^\bullet(A^*)=\Gamma(\wedge^\bullet A)\to\Omega^{\bullet+1}(A^*)$ is the Lie algebroid differential of $A^*$. The definition of Lie bialgebroids is symmetric, in the sense that $(A,A^*)$ is a Lie bialgebroid if and only if $(A^*,A)$ is a Lie bialgebroid \cite{MaXu94,Kosmann95,Roytenberg99}. Moreover, it is well-know that Lie bialgebroids structures on $(A,A^*)$ are in one-to-one correspondence with degree $-1$ Poisson Lie $1$-algebroid structures on the $[1]$-manifold $A[1]$ \cite{MaXu00}, i.e.~with a pair of a homological vector field compatible and a degree $-1$ Poisson bracket on $A[1]$. We now briefly recall this correspondence.
	
	Suppose that $\diff_{A}$ is a homological vector field on the $[1]$-manifold $A[1]$ giving rise to a Lie algebroid structure on $A\to M$. A degree $-1$ Poisson bracket $\{\cdot\,,\cdot\}$ on the function space $\cin(A[1]) = \Omega(A)$ is characterised by two maps $\rho_*:\Gamma(A^*)\to\mathfrak{X}(M)$ and $[\cdot\,,\cdot]_*\colon\Gamma(A^*)\times\Gamma(A^*)\to\Gamma(A^*)$ defined via
	\[
	\{\alpha,f\} = \rho_*(\alpha)f \qquad \text{and} \qquad \{\alpha,\beta\} = [\alpha,\beta]_*
	\]
	for all $f\in C^\infty(M)$ and all $\alpha,\beta\in\Gamma(A^*)$. The antisymmetry, Leibniz and Jacobi identities of $\{\cdot\,,\cdot\}$ are equivalent to $(A^*,\rho_*,[\cdot\,,\cdot]_*)$ being a Lie algebroid.
	
	For the Lie bialgebroid condition one proceeds as follows: Let $\pi\in\mathfrak{A}_{-1}^{2,1}(A[1])$ be the Poisson bivector field corresponding to $\{\cdot\,,\cdot\}$. Then $(A[1],\diff_A,\{\cdot\,,\cdot\})$ is a Poisson Lie $1$-algebroid if and only if $\diff_A$ is a Poisson vector field, or equivalently, if and only if $[\pi,\diff_A]_{-1} = 0$, where $[\cdot\,,\cdot]_{-1}$ is the bidegree $(-1,-1)$ Schouten bracket on $\mathfrak{A}_{-1}(A[1])$. Using the Jacobi identity of $[\cdot\,,\cdot]_{-1}$, we compute
	\[
	[[\alpha,[\pi,\diff_A]_{-1}]_{-1},\beta]_{-1} = [\diff_A\alpha,\beta]_* - \diff_A[\alpha,\beta]_* + [\alpha,\diff_A\beta]_*.
	\]
	Hence, $\diff_A$ is a Poisson vector field if and only if $(A,A^*)$ is a Lie bialgebroid.
	
	From the above, it follows that $(\widehat{\mathfrak{A}}_{-1}(A[1]),\delta)$ is the DGLA that governs the (infinitesimal) deformations of a Lie bialgebroid $(A,A^*)$, which in particular correspond to (cocycles) Maurer-Cartan elements in $\mathfrak{A}_{-1}^{2,1}(A[1]) \oplus \mathfrak{A}_{-1}^{1,2}(A[1])$. We will express all the operators locally in classical differential geometric language in terms of the splitting, but before doing that, we will use the homotopy Poisson structures on graded manifolds that were explained before to relax the definition of a Lie bialgebroid structure on $(A,A^*)$. A detailed study of the following structures can be found e.g.~in \cite{Roytenberg99,Ko-Sc05} under the names of ``proto-bialgebroid'', ``quasi-Lie bialgebroid'' and ``Lie quasi-bialgebroid''.
	
	Suppose that $A\to M$ is a vector bundle with a homotopy Poisson structure $\Theta$ of degree $-1$ on the graded manifold $A[1]$. Then we distinguish the following three cases:
	 
	\begin{itemize}
		\item \textbf{Lie bialgebroid:} This case is for $\Theta = \pi + \diff_{A}\in\mathfrak{A}_{-1}^{2,1}(A[1])\oplus\mathfrak{A}_{-1}^{1,2}(A[1])$ and was discussed above.
		
		\item \textbf{Quasi-Lie bialgebroid:}\index{Quasi-Lie bialgebroid} If $\Theta = \pi + \diff_{A} + \omega
		\in\mathfrak{A}_{-1}^{2,1}(A[1])\oplus\mathfrak{A}_{-1}^{1,2}(A[1])
		\oplus\mathfrak{A}_{-1}^{0,3}(A[1])$, it follows that $\pi$ is a degree $-1$ Poisson structure on $A[1]$, and the maps $[[\alpha,\pi]_{-1},f]_{-1}=\rho_*(\alpha)f$ and $[[\alpha,\pi]_{-1},\beta]=[\alpha,\beta]_*$ turn the triple $(A^*,\rho_*,[\cdot\,,\cdot]_*)$ into a Lie algebroid over $M$. The degree 1 operator $\diff_{A}:\Gamma(\wedge^\bullet A)\to\Gamma(\wedge^{\bullet+1}A)$ is a  quasi-algebroid structure\footnote{A quasi-algebroid is, by definition, a vector bundle $Q\to M$ together with an anchor $\rho:Q\to TM$ and a skew-symmetric bracket $[\cdot\,,\cdot]$ on $\Gamma(Q)$ that satisfies $[a,fb]=f[a,b] + (\rho(a)f)b$, for all $f\in C^\infty(M),a,b\in\Gamma(Q)$.} on $A$ whose failure to become a Lie algebroid differential is controlled by the Hamiltonian vector field $\X_{\omega}=[\omega,\pi]_{-1}$ of the $\diff_{A}$-closed function $\omega$ on $A[1]$ of degree $3$: $[\diff_{A},\diff_{A}]_{-1} + 2 \X_\omega = 0$. Moreover, the operator $\diff_{A}$ is a derivation of the Lie bracket on sections of $A^*$. In other words, the failure of $(A,A^*)$ to be a Lie bialgebroid comes from the failure of $\diff_{A}$ to be a Lie algebroid differential on $A$. It is a Lie bialgebroid if and only if $\omega$ is a Casimir function for the Poisson structure on $A[1]$, or equivalently, if and only if $\ad_{[\cdot\,,\cdot]_*}(\omega):=[\omega,\cdot]_*=0$.
		
		\item \textbf{Lie quasi-bialgebroid:}\index{Lie quasi-bialgebroid} If $\Theta = \Lambda + \pi + \diff_{A}
		\in\mathfrak{A}_{-1}^{3,0}(A[1])\oplus\mathfrak{A}_{-1}^{2,1}(A[1])\oplus\mathfrak{A}_{-1}^{1,2}(A[1])$, it follows that $\diff_{A}$ is a Lie algebroid differential on $A$ which is also a derivation of the quasi-algebroid structure on $A^*$ induced by $\pi$. Moreover, $\Lambda$ commutes with $\pi$ and itself, and the failure of $\pi$ to be a Lie algebroid structure on $A^*$ is measured by the homotopy term $[\diff_{A},\Lambda]_{-1}$: $[\pi,\pi]_{-1} + 2[\diff_{A},\Lambda]_{} = 0$.
	\end{itemize}
	
	Recall that given a $TM$-connection $\nabla$ on the vector bundle $A$, one can describe the module of vector fields of the graded manifold $A[1]$ via the $\Omega(A)$-module isomorphism
	\[
	\ad_\nabla(A):=\Omega(A)\otimes\Gamma(TM[0]\oplus A[1]) \to \mathfrak{X}(A[1])
	\]
	which maps the generators $X\in\mathfrak{X}(M)$ and $a\in\Gamma(A)$ to $\nabla_X^*$ and $\widehat{a}$, respectively. Locally, given the two dual frames $\{a_j\}$ and $\{\alpha^j\}$ of $A$ and $A^*$, we have the equality $\widehat{a}_j=\frac{\partial}{\partial \alpha^j}$. That is, we have the isomorphism
	\[
	\mathfrak{X}(A[1])=\bigoplus_{i\in\mathbb{Z}}\mathfrak{X}^i(A[1]) \cong \bigoplus_{i\in\mathbb{Z}}\left( \Omega^i(A,TM) \oplus \Omega^{i+1}(A,A)  \right) = \ad_\nabla(A)
	\]
	and thus 
	\[
	\mathfrak{X}(A[1])[-1]=\bigoplus_{i\in\mathbb{Z}}\mathfrak{X}^{i-1}(A[1]) \cong \bigoplus_{i\in\mathbb{Z}}\left( \Omega^{i-1}(A,TM) \oplus \Omega^{i}(A,A) \right) = \ad_\nabla(A)[-1].
	\]
	The DGLA of deformations of the Lie bialgebroid is given by the (graded) symmetric algebra $(\underline{S}(\ad_\nabla(A)[-1]),\delta)$ considered with the total grading.
	
	Next, we express the differential $\delta=\diff_{\pi} + \ldr{\diff_{A}}$ in this setting. As we will see in Chapter \ref{Chapter: Representations up to homotopy}, the operator $\ldr{\diff_{A}}$ is the differential of the adjoint module and thus becomes the differential of the adjoint representation up to homotopy from \cite{ArCr12} 
	\[
	\D_{\ad_\nabla(A)} = \rho \oplus \nabla^{\text{bas}} \oplus R_\nabla^{\text{bas}}.
	\]
	The horizontal differential $\diff_{\pi}$ acts on the generators as follows: Let $\xi_1,\xi_2\in\Omega(A)$ be homogeneous functions of $A[1]$. Then we compute
	\[
	[\pi,\xi_1]_{-1}(\xi_2) = [[\pi,\xi_1]_{-1},\xi_2]_{-1}
	= -(-1)^{|\xi_1|}[\X_{\xi_1},\xi_2]_{-1} = -(-1)^{|\xi_1|}\{\xi_1,\xi_2\}.
	\]
	Applying this to functions of $A[1]$ of the form $f,g\in C^\infty(M)$ and $\alpha,\beta\in\Omega^1(A)$, we obtain
	\[
	\diff_{\pi}(f)(g) = 0,\qquad \diff_{\pi}(f)(\beta) = \rho_*(\beta)f,\qquad \diff_{\pi}(\alpha)(g) = \rho_*(\alpha)g, \qquad \diff_{\pi}(\alpha)(\beta) = [\alpha,\beta]_*.
	\]
	Locally, suppose that we have the equations
	\[
	\rho_*(\alpha^j) = (\rho_*)^m_j\frac{\partial }{\partial x^m}
	\qquad \text{and} \qquad
	[\alpha^j,\alpha^s]_* = c^{js}_t \alpha^t
	\]
	where $\{x^i\}$ are coordinates in a chart of $M$ over which we have the  dual frames $\{a_j\}$ and $\{\alpha^j\}$ of $A$ and $A^*$, respectively. From the formulas above, we obtain the vector fields
	
	\[
	\diff_{\pi}(f) = (\rho_*)^m_j\frac{\partial f}{\partial x^m}\frac{\partial}{\partial \alpha^j}
	\qquad \text{and}\qquad
	\diff_{\pi}(\alpha) = \rho_*(\alpha)x^i\frac{\partial }{\partial x^i} + [\alpha,\alpha^s]_*\frac{\partial }{\partial \alpha^s}.
	\]
	On the coordinate functions, the vector fields become
	\[
	\diff_{\pi}(x^n) = (\rho_*)_j^n\frac{\partial}{\partial \alpha^j}
	\qquad \text{and} \qquad
	\diff_{\pi}(\alpha^j) = (\rho_*)_j^i\frac{\partial}{\partial x^i} + c_t^{js}\alpha^t\frac{\partial}{\partial \alpha^s}.
	\]
	
	An easy computation shows that a general vector field of the form
	\[
	\Lambda=\zeta\frac{\partial}{\partial \xi^i}\wedge\frac{\partial}{\partial \xi^j}\in\mathfrak{A}_{-1}^{2,s}(A[1])
	\]
	satisfies the equation $[[\xi_j,\Lambda]_{-1},\xi_i]_{-1} = - (-1)^{|\xi_j|s}\zeta$; that is, the coordinate functions of a $2$-vector field $\Lambda$ can be recovered by computing expressions of the form $[[\xi_i,\Lambda],\xi_j]$, where $\xi_i$ are the coordinates of the graded manifold $A[1]$. Now we use this to obtain the coordinate representations for the $2$-vector fields on $A[1]$ of the form $\diff_\pi(\nabla_X^*)$ and $\diff_\pi(\widehat{a})$, where $X\in\mathfrak{X}(M)$ and $a\in\Gamma(A)$: Using the graded anti-symmetry and the Jacobi identity of $[\cdot\,,\cdot]_{-1}$, we compute for $f,g\in C^\infty(M)$ and $\alpha,\beta\in\Omega^1(A)$
	\[
	\{0\} = \mathfrak{A}_{-1}^{0,-1}(A[1])\ni [[f,[\pi,\nabla_X^*]_{-1}]_{-1},g]_{-1} = 0 
	\]
	\[
	[[\alpha,[\pi,\nabla_X^*]_{-1}]_{-1},\beta]_{-1} = - \nabla_X^*[\alpha,\beta]_* + [\nabla_X^*\alpha,\beta]_* + [\alpha,\nabla_X^*\beta]_*
	\]
	\[
	[[f,[\pi,\nabla_X^*]_{-1}]_{-1},\alpha]_{-1} = [X,\rho_*(\alpha)]f - \rho(\nabla_X^*\alpha)f = \left((\nabla^*)_\alpha^{\text{bas}}X \right)f
	\]
	which imply that
	\begin{align*}
	\diff_{\pi}(\nabla_X^*) = & \left((\nabla^*)_{\alpha^j}^{\text{bas}}X \right)x^i \frac{\partial}{\partial \alpha^j}\wedge\frac{\partial}{\partial x^i} \\ 
	&\ + \Big( [\nabla_X^*\alpha^n,\alpha^m]_* + [\alpha^n,\nabla_X^*\alpha^m]_*
	- \nabla_X^*[\alpha^n,\alpha^m]_* \Big)\frac{\partial}{\partial \alpha^m}\wedge\frac{\partial}{\partial \alpha^n}. 
	\end{align*}
	Moreover, we have that 
	\[
	\{0\} = \mathfrak{A}_{-1}^{0,-2}(A[1])\ni[[f,[\pi,\widehat{a}]_{-1}]_{-1},g]_{-1} = 0
	\]
	\[
	\{0\} = \mathfrak{A}_{-1}^{0,-1}(A[1])\ni[[f,[\pi,\widehat{a}]_{-1}]_{-1},\alpha]_{-1} = 0
	\]
	\[
	[[\alpha,[\pi,\widehat{a}]_{-1}]_{-1},\beta]_{-1} = \rho_*(\alpha)(\beta(a)) - \rho_*(\beta)(\alpha(a)) - [\alpha,\beta]_*(a) = \diff_{A^*}a(\alpha,\beta),
	\]
	which yield
	\[
	\diff_{\pi}(\widehat{a}) = \diff_{A^*}a(\alpha^i,\alpha^j)\, \frac{\partial}{\partial \alpha^i}\wedge\frac{\partial}{\partial \alpha^j}.
	\]
	
	\begin{remark}
		Note that with the same reasoning we can compute the local expression of $\pi\in\mathfrak{A}_{-1}^{2,1}(A[1])$. In particular, from the equations $\{f,g\} = 0, \{\alpha,f\} = \rho_*(\alpha)f$ and $\{\alpha,\beta\} = - [\alpha,\beta]_*$, it follows that
		\[
		\pi = \rho_*(\alpha^j)\wedge\frac{\partial}{\partial \alpha^j} + [\alpha^s,\alpha^t]_*\frac{\partial}{\partial \alpha^s}\wedge\frac{\partial}{\partial \alpha^t} = (\rho_*)_j^i \frac{\partial}{\partial x^i}\wedge\frac{\partial}{\partial \alpha^j} + c_m^{st}\alpha^m\frac{\partial}{\partial \alpha^s}\wedge\frac{\partial}{\partial \alpha^t}.
		\]
	\end{remark}
	
	\section{Graded symplectic forms}
	
	In ordinary manifolds, it is well-known that a manifold with a symplectic form $\omega\in\Omega^2(M)$ is also Poisson with respect to the induced bracket on functions given by
	\[
	\{f,g\} := \omega(X_f,X_g) = -\langle \diff f,X_g \rangle = -X_g(f) = X_f(g),
	\] 
	for $f,g\in C^\infty(M)$. However, the converse is not true in general. In this section, we will see that this bridge between non-degenerate Poisson brackets and symplectic forms can be carried over to the world of supergeometry.
	
	Suppose that $\M$ is a $\mathbb{Z}$-manifold and $k\in\mathbb{Z}$. A $2$-form $\omega\in\Omega^{2,-k}(\M)$ is called \textbf{symplectic}\index{symplectic form} (of degree $-k$) if it is closed, i.e.~$\dr\omega = 0$, and non-degenerate, i.e.~the induced map
	\[
	\omega^\flat\colon T\M\to T^*[-k]\M, \qquad
	\X\mapsto i_\X\omega,
	\] 
	is a vector bundle isomorphism. Given a homogeneous function $\xi\in\cin(\M)$, the (unique) vector field $\X_\xi\in\mathfrak{X}^{|\xi|+k}(\M)$ which satisfies the equation
	\[
	i_{\X_\xi}\omega = -(-1)^{|\xi|} \diff\xi
	\]
	is called \textbf{Hamiltonian vector field of $\xi$}\index{Hamiltonian vector field} (with respect to $\omega$). A \textbf{symplectic vector field}\index{symplectic vector field} is a vector field $\X\in\mathfrak{X}(\M)$ such that $\ldr{\X}\omega = 0$. Due to $\dr^2 = 0$ and $\diff\omega = 0$, it follows that every Hamiltonian vector field is symplectic: For all $\xi\in\cin(\M)$
	\[
	\ldr{\X_\xi}\omega = [i_{\X_\xi},\diff]\omega = i_{\X_\xi}(\diff \omega) - (-1)^{|\xi|+k-1}\diff(i_{\X_\xi}\omega) = (-1)^{k-1}\diff^2\xi = 0.
	\]
	
	As in the case of ordinary (non-graded) symplectic manifolds, one can define a bracket $\{\cdot\,,\cdot\}_k$ of degree $k$ on $\M$ as
	\[
	\{\xi_1,\xi_2\}_k := \X_{\xi_1}(\xi_2) = i_{\X_{\xi_1}}(\diff\xi_2) = (-1)^{|\xi_2|} i_{\X_{\xi_1}}i_{\X_{\xi_2}}\omega.
	\]
	\begin{proposition}
		The bracket $\{\cdot\,,\cdot\}_k$ defined above is a degree $k$ Poisson bracket on $\M$.
	\end{proposition}
	\begin{proof}
		Given two homogeneous functions $\xi_1,\xi_2\in\cin(\M)$, we have that $|\{\xi_1,\xi_2\}_k| = |\X_{\xi_1}(\xi_2)| = \xi_1 + \xi_2 + k$. It remains to show that $\{\cdot\,,\cdot\}_k$ satisfies the graded anti-symmetry, Leibniz, and Jacobi identities. In order to justify these, we need the identities from Appendix \ref{Cartan calculus on Z-graded manifolds}. Then we compute
		\[
		\{\xi_1\,,\xi_2\}_k = -(-1)^{|\xi_2|}i_{\X_{\xi_1}}i_{\X_{\xi_2}}\omega
		= (-1)^{(|\xi_1|+k)(|\xi_2|+k)+|\xi_1|}i_{\X_{\xi_2}}i_{\X_{\xi_1}}\omega 
		= - (-1)^{(|\xi_1|+k)(|\xi_2|+k)}\{\xi_2,\xi_1\}_k,
		\]
		\begin{align*}
		\{\xi_1,\xi_2\xi_3\}_k = &\ i_{\X_{\xi_1}}(\diff(\xi_2\xi_3)) \\ 
		= &\ i_{\X_{\xi_1}}((\diff\xi_2)\xi_3) + (-1)^{|\xi_2|}i_{\X_{\xi_1}}(\xi_2\diff\xi_3) \\
		= &\ \{\xi_1,\xi_2\}_k\xi_3 + (-1)^{(|\xi_1|+k)|\xi_2|}\xi_2\{\xi_1,\xi_3\}_k.
		\end{align*}
		Therefore, anti-symmetry and Leibniz rule hold. For the Jacobi identity, we prove the equivalent equation $[\X_{\xi_1},\X_{\xi_2}] = \X_{\{\xi_1,\xi_2\}_k}$:
		\begin{align*}
		i_{[\X_{\xi_1},\X_{\xi_2}]}\omega = &\ [\ldr{\X_{\xi_1}},i_{\X_{\xi_2}}]\omega \\
		= &\ \ldr{\X_{\xi_1}}(i_{\X_{\xi_2}}\omega) - (-1)^{(|\xi_1|+k)(|\xi_2|+k-1)}i_{\X_{\xi_2}}(\ldr{\X_{\xi_1}}\omega) \\
		= &\ -(-1)^{|\xi_2|}\ldr{\X_{\xi_1}}(\diff \xi_2) \\
		= &\ (-1)^{|\xi_2|+|\xi_1|+k-1}\diff i_{\X_{\xi_1}}(\diff \xi_2) \\
		= &\ -(-1)^{|\xi_1|+|\xi_2|+k}\diff\{\xi_1,\xi_2\}_k \\
		= &\ i_{\X_{\{\xi_1,\xi_2\}_k}}\omega.
		\end{align*}
		The result follows due to non-degeneracy of $\omega$.
	\end{proof}
	
	Clearly, the Hamiltonian vector fields with respect to $\{\cdot\,,\cdot\}_k$ and $\omega$ coincide. As the proposition below shows, the same is true also for symplectic and Poisson vector fields. For its proof, we will need the following lemma.
	
	\begin{lemma}\label{Lemma about the vanishing of a form}
		Given a $p$-form $\eta\in\Omega^{p,\bullet}(\M)$, the following are equivalent:
		\begin{enumerate}
			\item $\eta = 0$,
			
			\item $i_{\X_1}\ldots i_{\X_p}\eta = 0$ for all $\X_1,\ldots,\X_p\in\mathfrak{X}(\M)$.
		\end{enumerate}
	\end{lemma}
	\begin{proof}
		One direction is obvious. For the second, consider local coordinates $\{\xi^j\}$ of $\M$ and write
		\[
		\eta = \sum_{j_1\ldots j_p}\zeta_{j_1\ldots j_p} \diff\xi^{j_1}\wedge\ldots\wedge \diff\xi^{j_p}
		\]
		for some local functions $\zeta_{j_1\ldots j_p}$ of $\M$. Using repeatedly the Leibniz rule for the contraction with respect to the coordinate vector fields $i_{\frac{\partial}{\partial\xi^{j_1}}},\ldots,i_{\frac{\partial}{\partial\xi^{j_p}}}$ yields
		\[
		0=i_{\frac{\partial}{\partial\xi^{j_1}}}\ldots i_{\frac{\partial}{\partial\xi^{j_p}}}\eta = i_{\frac{\partial}{\partial\xi^{j_1}}}\ldots i_{\frac{\partial}{\partial\xi^{j_p}}} \left(\sum_{j_1\ldots j_p}\zeta_{j_1\ldots j_p} \diff\xi^{j_1}\wedge\ldots\wedge \diff\xi^{j_p}\right) = \pm \zeta_{j_1\ldots j_p}.\qedhere
		\]
	\end{proof}
	
	\begin{proposition}
		A homogeneous vector field $\X\in\mathfrak{X}(\M)$ is symplectic for $\omega$ if and only if it is Poisson for $\{\cdot\,,\cdot\}_k$.
	\end{proposition} 
	\begin{proof}
		Consider any homogeneous $\X\in\mathfrak{X}(\M)$ and any homogeneous function $\xi\in\cin(\M)$. Then we compute
		\begin{align*}
		i_{[\X,\X_\xi]}\omega = &\ [\ldr{\X},i_{\X_\xi}]\omega \\
		= &\ \ldr{\X}(i_{\X_\xi}\omega) - (-1)^{|\X|(|\xi|+k-1)}i_{\X_\xi}(\ldr{\X}\omega) \\
		= &\ -(-1)^{|\xi|}\ldr{\X}(\diff\xi) - (-1)^{|\X|(|\xi|+k-1)}i_{\X_\xi}(\ldr{\X}\omega) \\
		= &\ (-1)^{|\xi|+|\X|-1}\diff\X(\xi) - (-1)^{|\X|(|\xi|+k-1)}i_{\X_\xi}(\ldr{\X}\omega) \\
		= &\ i_{\X_{\X(\xi)}}\omega - (-1)^{|\X|(|\xi|+k-1)}i_{\X_\xi}(\ldr{\X}\omega)
		\end{align*}
		
		Hence, if $\X$ is symplectic, then $[\X,\X_\xi] = \X_{\X(\xi)}$ for all $\xi\in\cin(\M)$, which is equivalent to $\X$ being Poisson. Conversely, suppose $\X$ is Poisson. Then it follows from the computation above that $i_{X_\xi}(\ldr{\X}\omega) = 0$ for all $\xi\in\cin(\M)$. Due to non-degeneracy of $\omega$, the Hamiltonian vector fields span all of $\mathfrak{X}(\M)$ and thus $i_\Y(\ldr{\X}\omega)$ for all $\Y\in\mathfrak{X}(\M)$. An application of Lemma \ref{Lemma about the vanishing of a form} finishes the proof.
	\end{proof}
	
	\begin{example}[Cotangent bundle]\label{Symplectic structure of cotangent bundles}
		Similarly to the non-graded case, all shifted cotangent bundles $T^*[1-k]\M$ carry a canonical symplectic form $\omega_{\text{can}}\in\Omega^{k-1}(T^*[1-k]\M)$. Its corresponding Poisson bivector is given locally by
		\[
		\pi=\sum_i \frac{\partial}{\partial e^i}\wedge\frac{\partial}{\partial \xi^i},
		\] 
		where $\{\xi^j\}$ are local coordinates on $\M$ and $\{e^j\}$ are the linear coordinates of $T^*[1-k]\M$ corresponding to the local frame $\{\frac{\partial}{\partial \xi^j}\}$ of $T[k-1]\M$. That is, $\{\xi^j,e^j\}$ are the canonical \textbf{Darboux coordinates} for $\omega_{\text{can}}$. 
	\end{example}
	
	\section[Poisson manifolds and Courant algebroids]{Poisson manifolds and Courant algebroids in the graded setting}\label{Section: Poisson manifolds and Courant algebroids in the graded setting}
	
	In this section, we recall the correspondence of ordinary Poisson manifolds and Courant algebroids with symplectic Lie $1$- and symplectic Lie $2$-algebroids, respectively. The results were first described in \cite{Roytenberg02,Severa05} and we refer the reader to these works for more details. For the Courant algebroid case see also \cite{CuMe21}. The following is based on \cite{Roytenberg02} and \cite{CuMe21}.
	
	Note that on a symplectic $\n$-manifold $(\M,\omega)$ with $\omega\in\Omega^{2,n}(\M)$, every symplectic vector field $\X\in\mathfrak{X}^{i}(\M)$ with $i+n\neq0$ is Hamiltonian and
	\[
	i_{\X}\omega = \diff\left( \frac{1}{i+n}i_{\X_E}i_{\X}\omega \right),
	\]
	where $\X_{E}$ is the Euler vector field on $\M$ (Roytenberg \cite{Roytenberg02}). In particular, every vector field of degree $1$ is Hamiltonian, and thus any vector field $\Q=\X_{\Theta}=\{\Theta,\cdot\}\in\mathfrak{X}^1(\M)$ is homological if and only if the function $\Theta\in\cin(\M)^{n+1}$ satisfies the \textbf{(classical) Master equation}\index{Master equation}
	\[
	\{\Theta,\Theta\} = 0.
	\]
	
	Suppose $A[1]$ is a symplectic $[1]$-manifold over $M$ with corresponding Poisson bracket $\{\cdot\,,\cdot\}$ of degree $-1$. Due to the Leibniz rule of the Poisson bracket, we obtain a map $A^*\to TM,\alpha\to\{\alpha,\cdot\}$, which is an isomorphism due to the non-degeneracy of $\{\cdot\,,\cdot\}$. Finally, due to Jacobi identity, the Poisson bracket on $\cin(A[1])^1=\Gamma(A^*)$ corresponds under this isomorphism to the
	commutator of vector fields on the base manifold $M$. This implies that symplectic $[1]$-manifolds over $M$ are always of the form $T^*[1]M$. A homological vector field $\Q$ on $T^*[1]M$ is of the form $\Q = \{\Theta,\cdot\}$, where $\Theta$ is a degree $2$ function on $T^*[1]M$. Since $\cin(T^*[1]M) = \Gamma(\wedge TM)$, it follows that $\cin(T^*[1]M)^2 = \Gamma(\wedge^2 TM)$ and thus the function $\Theta$ is a Poisson tensor $\pi$ on $M$. Therefore, we have proved the following result.
	
	\begin{theorem}[\cite{Roytenberg02}]
		Symplectic Lie $1$-algebroids with induced Poisson bracket of degree $-1$ are in one-to-one correspondence with ordinary Poisson manifolds.
	\end{theorem}
	
	Passing to the degree $n=2$ case, we obtain the graded geometric description of Courant algebroids. The correspondence can be seen as follows: Let $(\M,\Q,\omega)$ be a symplectic Lie $2$-algebroid with corresponding degree $-2$ Poisson bracket $\{\cdot\,,\cdot\}$ on $\cin(\M)$, and choose a splitting for the underlying $[2]$-manifold $\M$ so that $\cin(M)^0=C^\infty(M)$ and $\cin(\M)^1\cong\Gamma(E)$ for a vector bundle $E\to M$. The restriction of the Poisson bracket on $\Gamma(E)\times\Gamma(E)$ defines the pairing $\langle\cdot\,,\cdot\rangle$ on $E$. Suppose now that $\Theta\in\cin(M)^3$ is such that $\Q=\X_{\Theta}$ and it satisfies the classical Master equation: $\{\Theta,\Theta\} = 0$. Then we define
	\[
	\rho(e)f = \{\{e,\Theta\},f\} = \{\{\Theta,f\},e\}
	\qquad \text{and} \qquad
	\llbracket e_1,e_2 \rrbracket = \{\{e_1,\Theta\},e_2\},
	\]
	for $f\in C^\infty(M),e,e_1,e_2\in\Gamma(E)$, and hence we have additionally a bundle map $\rho:E\to TM$ and a bracket $\llbracket\cdot\,,\cdot\rrbracket$ on $\Gamma(E)$. The quadruple $(E,\rho,\llbracket\cdot\,,\cdot\rrbracket,\langle\cdot\,,\cdot\rangle)$ has the structure of a Courant algebroid over the base manifold $M$.
	
	The converse can be found in \cite{Roytenberg02} using local coordinates. Here, we adopt the method from \cite{CuMe21} and describe it in an invariant way. Suppose $(E,\rho,\llbracket\cdot\,,\cdot\rrbracket,\langle\cdot\,,\cdot\rangle)$ is a Courant algebroid over the smooth manifold $M$. We define the $[2]$-manifold $\M$ whose sheaf of functions is given by
	\[
	\cin(\M)^0 = C^\infty(M),
	\qquad
	\cin(\M)^1 = \Gamma(E),
	\qquad
	\cin(\M)^2 = \Gamma(\Der_{\langle\cdot\,,\cdot\rangle}(E))  ,
	\]
	where $\Der_{\langle\cdot\,,\cdot\rangle}(E)$ denotes the derivations of the vector bundle $E\to M$ preserving the pairing $\langle\cdot\,,\cdot\rangle$ on $E$: $D\in\Der_{\langle\cdot\,,\cdot\rangle}(E)$ if and only if $D$ is a derivation of $E$ over the vector field $X_D\in\mathfrak{X}(M)$ such that
	\[
	X_D\langle e_1,e_2 \rangle = \langle De_1,e_2 \rangle + \langle e_1,De_2 \rangle
	\]
	for all $e_1,e_2\in\Gamma(E)$. We now recall the \textbf{Keller-Waldmann algebra}\index{Courant algebroid!Keller-Waldmann algebra} \cite{KeWa15} $(C^\bullet(E),\bullet)$ of the Courant algebroid $E\to M$. A \textbf{$k$-cochain}\index{Courant algebroid!$k$-cochain} on $E$, $k\geq1$, is a map
	\[
	\omega\colon\underset{k\text{-times}}{\underbrace{\Gamma(E)\times\ldots\Gamma(E)}}\to C^\infty(M)
	\] 
	that is $C^\infty(M)$-linear in the last entry and such that for $k\geq2$ there exists a map
	\[
	\sigma_\omega\colon\underset{(k-2)\text{-times}}{\underbrace{\Gamma(E)\times\ldots\Gamma(E)}}\to \mathfrak{X}(M),
	\]
	called the \textbf{symbol}\index{Courant algebroid!symbol of a $k$-cochain}, with
	\[
	\sigma_\omega(e_1,\ldots,\widehat{e}_i,\widehat{e}_{i+1},\ldots,e_k)\langle e_i,e_{i+1} \rangle =
	\omega(e_1,\ldots,e_i,e_{i+1},\ldots,e_k) + \omega(e_1,\ldots,e_{i+1},e_i,\ldots,e_k)
	\]
	for all $e_j\in\Gamma(E)$ and $1\leq i\leq k-1$. The space of $k$-cochains is denoted $C^\bullet(E)$ and by definition $C^0(E):=C^\infty(M)$. The space $C^\bullet(E)$ has the structure of a graded commutative algebra with product $\bullet$ defined by
	\[
	(\omega\bullet\eta)(e_1,\ldots,e_{k+m}) := 
	\sum_{{\sigma}\in\text{Sh}_{k,m}} \text{sgn}(\sigma)\,
	\omega(e_{\sigma(1)},\ldots,e_{\sigma(k)}) \,
	\eta(e_{\sigma(k+1)},\ldots,e_{\sigma(m)})
	\]
	for $\omega\in C^k(E),\eta\in C^m(E)$ and $e_j\in\Gamma(E)$. Note that in degrees $0,1$ and $2$, $C^\bullet(E)$ is described as follows:
	\begin{itemize}
		\item $C^0(E)=C^\infty(M)=\cin(\M)^0$;
		
		\item $C^1(E) = \Gamma(E^*) \simeq \Gamma(E) = \cin(\M)^1$ via the non-degenerate pairing $\langle\cdot\,,\cdot\rangle$ on $E$;
		
		\item Given $\omega\in C^2(E)$, let $\widehat{\omega}\colon\Gamma(E)\to \Gamma(E)$ be defined by $\langle\widehat{\omega}(e_1),e_2\rangle = \omega(e_1,e_2)$, for $e_1,e_2\in\Gamma(E)$. Then $\widehat{\omega}\in\Der_{\langle\cdot\,,\cdot\rangle}(E)$ over the vector field $X_{\widehat{\omega}}=\sigma_\omega\in\mathfrak{X}(M)$. That is, $C^2(E)\cong \cin(\M)^2$.
	\end{itemize}
	
	Using now the pairing $\langle\cdot\,,\cdot\rangle$ of $E$, one can define a symplectic structure on the graded manifold $\M$ whose corresponding degree $-2$ Poisson bracket $\{\cdot\,,\cdot\}$ is given in low degrees by
	\[
	\{e_1,e_2\} = \langle e_1,e_2\rangle,
	\quad
	\{D,f\} = X_D(f),
	\quad
	\{D,e\} = De,
	\quad
	\{D_1,D_2\} = D_1D_2 - D_2D_1,
	\]
	for $f\in C^\infty(M),e,e_1,e_2\in\Gamma(E),D,D_1,D_2\in\Der_{\langle\cdot\,,\cdot\rangle}(E)$. In particular, since only the pairing of $E$ was used, we have proved the following result.
	
	\begin{proposition}[\cite{Roytenberg02}]
		The above construction establishes a one-to-one correspondence between symplectic $[2]$-manifolds over $M$ with induced Poisson bracket of degree $-2$ and vector bundles $E\to M$ equipped with a non-degenerate pairing $\langle\cdot\,,\cdot\rangle$.
	\end{proposition}
	
	Note that the description of $C^0(E),C^1(E)$ and $C^2(E)$ implies that  $(C^\bullet(E),\bullet)$ and $(\cin(\M),\cdot)$ coincide in low degrees. In fact, they are isomorphic as graded commutative algebras via the map $\Upsilon\colon\cin(\M)^k\to C^k(E)$ \cite[Thm.~2.5]{CuMe21} defined by
	\[
	\Upsilon(\xi)(e_1,\ldots,e_k) := \{ e_k,\{ e_{k-1},\{\ldots\{ e_1,\xi \}\ldots \} \} \}
	\]
	for $e_1,\ldots,e_k\in \Gamma(E)$ and $\xi\in\cin(\M)^k$. Now we can use this isomorphism to define a function $\Theta\in\cin(\M)$ of degree $3$ satisfying $\{\Theta,\Theta\}=0$; namely, $\Theta:=\Upsilon^{-1}(T)$, where $T\in C^3(E)$ is given by $T(e_1,e_2,e_3) := \langle\llbracket e_1,e_2 \rrbracket, e_3  \rangle$, for $e_1,e_2,e_3\in\Gamma(E)$. The symplectic homological vector field on the symplectic $[2]$-manifold $(\M,\{\cdot\,,\cdot\})$ is given by $\Q:=\X_{\Theta}$. Therefore, we obtain the following well-known result.
	
	\begin{theorem}[\cite{Roytenberg02}]
		The above construction gives a one-to-one correspondence between symplectic Lie $2$-algebroids with induced Poisson bracket of degree $-2$ and Courant algebroids. 
	\end{theorem}
	
	\chapter{Differential graded modules}\label{Chapter: Differential graded modules}
	
	This chapter defines the notion of a differential graded module
	over a Lie $n$-algebroid (or generally a $\Q$-manifold) $(\M,\Q)$ and gives the two fundamental
	examples of modules which come canonically with $(\M,\Q)$,
	namely the adjoint and the coadjoint modules. Moreover, we investigate the influence of a Poisson structure on $(\M,Q)$ on its adjoint and coadjoint modules. Note that the case of
	differential graded modules over a Lie 1-algebroid $A\to M$ is studied
	in detail in \cite{Mehta14}.
	
	\section{The category of differential graded modules}\label{Section: The category of differential graded modules}
	
	Let $A\to M$ be a Lie 1-algebroid. A \textbf{Lie algebroid module}\index{Lie algebroid!module}
	\cite{Vaintrob97} over $A$ is defined as a 
	sheaf $\mathcal{B}$ of locally freely generated graded
	$\Omega(A)$-modules over $M$ together with a degree $1$ map $\D\colon \mathcal{B}\to\mathcal{B}$ which
	squares to zero and satisfies the Leibniz rule
	\[
	\D(\alpha\eta) = (\diff_A\alpha)\eta + (-1)^{|\alpha|}\alpha\D(\eta),
	\]
	for $\alpha\in\Omega(A)$ and $\eta\in\mathcal{B}$. By definition, the sheaf $\mathcal{B}$ is the sheaf of sections of a vector bundle $\Em\to A[1]$ in the category of graded manifolds: $\mathcal{B} = \Gamma(\Em)$. For a Lie $n$-algebroid
	$(\M,\Q)$ over $M$ (or even a general $\mathbb{Z}$-graded $\Q$-manifold), this is generalised to the following definitions.
	
	\begin{definition}
		\begin{enumerate}
			\item A \textbf{left differential graded module of $(\M,\Q)$}\index{Lie $n$-algebroid!left differential graded module} is a
			sheaf $\Gamma(\Em)$ of locally freely generated left graded $\cin(\M)$-modules
			over $M$ (corresponding to a vector bundle $\Em\to\M$), together with a map $\D\colon\Gamma(\Em)\to\Gamma(\Em)$ of degree $1$, such
			that $\D^2=0$ and
			\[
			\D(\xi\eta) = \Q(\xi)\eta + (-1)^{|\xi|}\xi\D(\eta)
			\]
			for all $\xi\in\cin(\M)$ and $\eta\in\Gamma(\Em)$.
			\item A \textbf{right differential graded module of $(\M,\Q)$}\index{Lie $n$-algebroid!right differential graded module} is
			a sheaf $\Gamma(\Em)$ of right graded modules as above together with a map $\D\colon\Gamma(\Em)\to\Gamma(\Em)$ of
			degree $1$, such that $\D^2=0$ and
			\[
			\D(\eta\xi) = \D(\eta)\xi + (-1)^{|\eta|}\eta\Q(\xi)
			\]
			for all $\xi\in\cin(\M)$ and $\eta\in\Gamma(\Em)$.
			\item A \textbf{differential graded bimodule of $(\M,\Q)$}\index{Lie $n$-algebroid!differential graded bimodule} is a
			sheaf $\Gamma(\Em)$ as above together with left and right differential
			graded module structures such that the gradings and the
			differentials coincide, and the two module structures commute:
			$(\xi_1\eta)\xi_2 = \xi_1(\eta\xi_2)$ for all
			$\xi_1,\xi_2\in\cin(\M)$ and $\eta\in\Gamma(\Em)$.
		\end{enumerate}
	\end{definition}
	
	For short we write \textbf{left DG $(\M,\Q)$-module}\index{Lie $n$-algebroid!left DG-module} and \textbf{right DG
		$(\M,\Q)$-module}\index{Lie $n$-algebroid!right DG-module}, or simply \textbf{left DG $\M$-module}\index{Lie $n$-algebroid!left DG-module} and
	\textbf{right DG $\M$-module}\index{Lie $n$-algebroid!right DG-module}. The cohomology\index{Lie $n$-algebroid!cohomology of a DG-module} of the induced complexes
	is denoted by $H_L^\bullet(\M,\Q;\Em)$ and $H_R^\bullet(\M,\Q;\Em)$,
	respectively, or simply by $H_L^\bullet(\M,\Em)$ and
	$H_R^\bullet(\M,\Em)$. If there is no danger of confusion, the
	prefixes ``left" and ``right", as well as the subscripts ``$L$" and
	``$R$", will be omitted.
	
	\begin{definition}
		Let $(\Em_1,\D_1)$ and $(\Em_2,\D_2)$ be two left (right) differential graded
		modules over the Lie $n$-algebroids $(\M,\Q_{\M})$ and
		$(\N,\Q_{\N})$, respectively. A \textbf{degree $0$
		morphism}\index{Lie $n$-algebroid!degree $0$-morphism of DG-modules}, or simply a \textbf{morphism}\index{Lie $n$-algebroid!morphism of DG-modules}, between $\Em_1$ and $\Em_2$ consists of a morphism of
		Lie $n$-algebroids $\phi\colon \N\to\M$ and a degree preserving map
		$\mu\colon \Gamma(\Em_1)\to\Gamma(\Em_2)$ which is left (right) linear:
		$\mu(\xi\eta) = \phi^\star(\xi) \mu(\eta)$, for all $\xi\in\cin(\M)$
		and $\eta\in\Gamma(\Em_1)$, and commutes with the differentials $\D_1$ and
		$\D_2$.
	\end{definition}
	
	As in the case of Lie algebroids, new examples of DG $\M$-modules are obtained by considering the usual algebraic
	constructions. In the following, we describe these constructions only
	for left DG modules but the case of right DG modules is treated
	similarly.
	
	\begin{definition}[\textbf{Dual module}\index{Lie $n$-algebroid!dual module}]
		Given a DG $\M$-module $\Em$ with differential $\D_{\Em}$,
		one defines a DG $\M$-module structure on the dual sheaf
		$\Em^*:=\underline{\Hom}(\Em,\cin)$ with differential $\D_{\Em^*}$
		defined via the property
		\[
		\Q(\langle\psi,\eta\rangle) =
		\langle\D_{\Em^*}(\psi),\eta\rangle +
		(-1)^{|\psi|}\langle\psi,\D_{\Em}(\eta)\rangle,
		\] 
		for all $\psi\in\Gamma(\Em^*)$ and $\eta\in\Gamma(\Em)$, where
		$\langle\cdot\,,\cdot\rangle$ is the pairing of $\Em^*$ and
		$\Em$ \cite{Mehta06}.
	\end{definition}
	
	\begin{definition}[\textbf{Tensor product module}\index{Lie $n$-algebroid!tensor product module}]
		For DG $\M$-modules $\Em$ and $\Fm$ with operators $\D_{\Em}$ and
		$\D_{\Fm}$, the corresponding operator $\D_{\Em\otimes \Fm}$ on
		$\Em\otimes \Fm$ is uniquely characterised by the formula
		\[
		\D_{\Em\otimes \Fm}(\eta\otimes\eta') =
		\D_{\Em}(\eta)\otimes\eta' +
		(-1)^{|\eta|}\eta\otimes\D_{\Fm}(\eta'),
		\]
		for all $\eta\in\Gamma(\Em)$ and $\eta'\in\Gamma(\Fm)$.
	\end{definition}
	
	\begin{definition}[\textbf{$\underline{\Hom}$-module}\index{Lie $n$-algebroid!module of homomorphisms}]
		For DG $\M$-modules $\Em,\Fm$ with operators $\D_{\Em}$
		and $\D_{\Fm}$, the differential $\D_{\underline{\Hom}(\Em,\Fm)}$ on
		$\underline{\Hom}(\Em,\Fm)$ is defined via
		\[
		\D_{\Fm}(\psi(\eta)) = \D_{\underline{\Hom}(\Em,\Fm)}(\psi)(\eta) + (-1)^{|\psi|}\psi(\D_{\Em}(\eta)),
		\]
		for all $\psi\in
		\underline{\Hom}(\Gamma(\Em),\Gamma(\Fm))$ and
		$\eta\in \Gamma(\Em)$.
	\end{definition}
	
	\begin{definition}[\textbf{(Anti)symmetric powers module}\index{Lie $n$-algebroid!symmetric powers module}\index{Lie $n$-algebroid!anti-symmetric powers module}]
		For a DG $\M$-module $\Em$ with operator $\D_{\Em}$, the corresponding operator
		$\D_{\underline{S} (\Em)}$ on ${\underline{S}^k(\Em)}$ is uniquely
		characterised by the formula
		\begin{align*}
		\D_{\underline{S} (\Em)}(\eta_1\eta_2\ldots\eta_k) = &\ \D_{\Em}(\eta_1)\eta_2\ldots\eta_k \\
		& + \eta_1\sum_{i=2}^k(-1)^{|\eta_1|+\ldots+|\eta_{i-1}|}\eta_2\ldots\D(\eta_i)\ldots\eta_k,
		\end{align*}
		for all $\eta_1,\ldots,\eta_k\in \Gamma(\Em)$. A similar formula
		gives also the characterisation for the
		operator $\D_{\underline{A} (\Em)}$ of the antisymmetric
		powers $\underline{A}^q(\Em)$.
	\end{definition}
	
	\begin{definition}[\textbf{Direct sum module}\index{Lie $n$-algebroid!direct sum module}]
		For DG $\M$-modules $\Em,\Fm$ with operators $\D_{\Em}$
		and $\D_{\Fm}$, the differential operator $\D_{\Em\oplus \Fm}$ on
		$\Em\oplus \Fm$ is defined as
		\[
		\D_{\Em\oplus \Fm} = \D_{\Em} \oplus \D_{\Fm}.
		\]
	\end{definition}

	\begin{definition}[\textbf{Shifted module}\index{Lie $n$-algebroid!shifted module}]
		For $k\in\mathbb{Z}$, the DG $\M$-module $\mathbb{R}[k]$ is defined
		as $\cin(\M)\otimes\Gamma(M\times\mathbb{R}[k])$ with differential
		given by $\Q$; here, $M\times \mathbb{R}[k]$ is the $[k]$-shift of
		the trivial line bundle over $M$, i.e.~$M\times \mathbb{R}$ in
		degree $-k$ and zero otherwise.  Given now a DG-module $\Em$ with
		differential $\D_{\Em}$, we define the shifted module
		$\Em[k]:=\Em\otimes\mathbb{R}[k]$. Due to the definition of the
		tensor module, its differential $\D[k]$ acts via
		\[
		\D_{\Em}[k](\eta\otimes1) = \D_{\Em}(\eta)\otimes1
		\]
		for all $\eta\in\Gamma(\Em)$. Abbreviating the element
		$\eta\otimes1$ simply as $\eta$, the shifted differential
		$\D_{\Em}[k]$ coincides\footnote{Again one could choose to
			tensor with $\mathbb{R}[k]$ from the left. Then on elements
			of the form $1\otimes\eta$, the resulting differential would
			act as $\D_{\Em}[k]=(-1)^k\D_{\Em}$.} with $\D_{\Em}$.
	\end{definition}
	
	\begin{definition}
		Let $(\M,\Q_{\M})$ and $(\N,\Q_{\N})$ be $\Q$-manifolds, and
		suppose that $\Em_1$ and $\Em_2$ are left (right) DG-modules over $\M$ and $\N$,
		respectively. A \textbf{degree $k$-morphism}\index{Lie $n$-algebroid!degree $k$-morphism of DG-modules} or simply $k$-morphism\index{Lie $n$-algebroid!$k$-morphism of DG-modules}, for $k\in\mathbb{Z}$, from
		$\Em_1$ to $\Em_2$ is defined as a left (right) degree $0$ morphism
		$\mu:\Em_1\to\Em_2[k]$; that is, a map sending elements of degree
		$i$ in $\Gamma(\Em_1)$ to elements of degree $i+k$ in $\Gamma(\Em_2)$, such that it
		is linear over a morphism of $\Q$-manifolds $\phi:\N\to\M$ and
		commutes with the differentials. A \textbf{$k$-isomorphism}\index{Lie $n$-algebroid!$k$-ismorphism of DG-modules} is a $k$-morphism
		with an inverse.
	\end{definition}
	
	\begin{remark}
		\begin{enumerate}
			\item The inverse of a $k$-isomorphism is necessarily a $-k$-isomorphism.
			\item For all $k\in\mathbb{Z}$ and all DG $\M$-modules
			$\Em$, there is an obvious $k$-isomorphism
			$\Em\to \Em[k]$ over the identity on $\M$.
		\end{enumerate}
	\end{remark}
	
	Considering the special case of $\M = \N$ in the definition above
	yields $k$-morphisms between DG $\M$-modules over the same $\Q$-manifold. The resulting graded categories\index{Lie $n$-algebroid!category of DG-modules} of left and right DG
	$\M$-modules are denoted by $\underline{\mathbb{M}\text{od}}_L(\M,\Q)$
	and $\underline{\mathbb{M}\text{od}}_R(\M,\Q)$, or simply by
	$\underline{\mathbb{M}\text{od}}_L(\M)$ and
	$\underline{\mathbb{M}\text{od}}_R(\M)$. The isomorphism classes of
	these categories are denoted by $\underline{\text{Mod}}_L(\M,\Q)$ and
	$\underline{\text{Mod}}_R(\M,\Q)$, or simply by
	$\underline{\text{Mod}}_L(\M)$ and
	$\underline{\text{Mod}}_R(\M)$. Again if there is no danger of
	confusion, the subscripts ``$L$" and ``$R$" will be omitted.
	
	\section{Adjoint and coadjoint modules}\label{Section: Adjoint and coadjoint modules}
	
	Recall that every $\Q$-manifold $\M$ comes together  with its tangent bundle $T\M$ whose sheaf of sections is the space of vector fields $\mathfrak{X}(\M)$ over $\M$. Its left $\cin(\M)$-module structure is defined by
	the property $(\xi_1\mathcal{X})(\xi_2) = \xi_1\mathcal{X}(\xi_2)$ for
	all $\xi_1,\xi_2\in\cin(\M)$ and
	$\mathcal{X}\in\mathfrak{X}(\M)$. In addition to the left
	module structure, the space of vector fields is also endowed with a
	right $\cin(\M)$-module structure. The right multiplication with
	functions in $\cin(\M)$ is obtained by viewing the elements of
	$\mathfrak{X}(\M)$ of degree $i$ as functions of bidegree
	$(1,i)$ of the graded manifold $T^*[1]\M$. The two module structures on
	$\mathfrak{X}(\M)$ are related by
	$\mathcal{X}\xi = (-1)^{|\xi|(|\mathcal{X}|+1)}\xi\mathcal{X}$, for
	all homogeneous $\xi\in\cin(\M)$ and
	$\mathcal{X}\in\mathfrak{X}(\M)$.
	
	Suppose now that $\M$ is endowed with a homological vector field $\Q$,
	i.e.~$(\M,\Q)$ is a $\Q$-manifold. Then the Lie derivative on the
	space of vector fields $\ldr{\Q}:=[\Q,\cdot]$ is a degree 1 operator
	which squares to zero and has both the left and right Leibniz
	identities with respect to the left and right module structures
	explained above:
	\[
	\ldr{\Q}(\xi\X) = \Q(\xi)\X + (-1)^{|\xi|}\xi\ldr{\Q}(\X)
	\quad \text{and} \quad
	\ldr{\Q}(\X\xi) = \ldr{\Q}(\X)\xi + (-1)^{|\X|}\X\Q(\xi)
	\]
	for homogeneous $\xi\in\cin(\M)$ and $\X\in\mathfrak{X}(\M)$. That is, the sheaf of vector fields over $(\M,\Q)$
	has a canonical DG $\M$-bimodule structure which we call the
	\textbf{adjoint module}\index{Lie $n$-algebroid!adjoint module} of $\M$:
	\[
	(\mathfrak{X}(\M),\ldr{\Q}).
	\]
	
	The dual module $\bigoplus_p\cin(T[1]\M)_{(p,1)}$ of 1-forms over $\M$
	carries the grading obtained from the horizontal grading of the Weil
	algebra -- that is, the elements of $\cin(T[1]\M)_{(p,1)}$ have degree
	$p$. Its structure operator as a left DG module is given by the Lie
	derivative $\ldr{\Q} = [i_\Q,\dr]$, and as a right DG module is given
	by $-\ldr{\Q}$.  These left and right DG $\M$-modules are called the
	\textbf{coadjoint module}\index{Lie $n$-algebroid!coadjoint module} of $(\M,\Q)$ and denoted by
	\[
	(\Omega^1(\M),\ldr{\Q})\qquad \text{and} \qquad (\Omega^1(\M),-\ldr{\Q})
	\]
	
	\section{$\mathcal{PQ}$-manifolds: coadjoint vs adjoint modules}\label{Section PQ-manifolds: coadjoint vs adjoint modules}
	
	In \cite{MaXu94}, it was shown that a Lie algebroid $A\to M$ with a linear
	Poisson structure satisfies the Lie bialgebroid compatibility
	condition if and only if the map $T^*A \to TA$ induced by the Poisson
	bivector is a Lie algebroid morphism from $T^*A = T^*A^* \to A^*$ to
	$TA \to TM$ over $\rho*:A^*\to TM$. This is now generalised to give a
	characterisation of $\mathcal{PQ}$-manifolds in the general setting.
	
	Let $\M$ be a $\mathbb{Z}$-manifold equipped with a homological vector field
	$\Q$ and a Poisson bivector field $\pi\in\mathfrak{A}_k^{2,-k}(\M)$ with corresponding degree $k$ Poisson bracket $\{\cdot\,,\cdot\}_k$. The Poisson
	bracket on $\M$ induces a map
	$\pi^\sharp\colon\Omega^1(\M)\to\mathfrak{X}(\M)[k]$ defined on the generators via the
	property
	\begin{equation}\label{eqn:sharp}
	\pi^\sharp(\diff\xi_1) (\xi_2) = \{ \xi_1,\xi_2 \}_k,
	\end{equation}
	for all $\xi_1,\xi_2\in\cin(\M)$, and extended odd linearly by the rules\footnote{The reason for this particular choice of signs is given in Appendix \ref{Appendix: Signs for the linearity of pi^sharp}.}
	\[
	\pi^\sharp(\xi_1\diff \xi_2) = (-1)^{|\xi_1|}\xi_1\pi^\sharp(\diff\xi_2)
	\qquad \text{and} \qquad
	\pi^\sharp((\diff\xi_1) \xi_2) = (-1)^{|\xi_2|}\pi^\sharp(\diff\xi_1)\,\xi_2.
	\]
	
	\begin{theorem}\label{thm_poisson}
		Let $\M$ be a graded manifold equipped with a homological vector
		field $\Q$ and a degree $k$ Poisson bracket $\{\cdot\,,\cdot\}_k$. Then
		$(\M,\Q,\{\cdot\,,\cdot\}_k)$ is a $\mathcal{PQ}$-manifold if and only
		if $\pi^\sharp\colon \Omega^1(\M)\to\mathfrak{X}(\M)$ is a degree $k$
		anti-morphism\footnote{Using convention of tensoring from the left to obtain the shift, the anti-morphism condition reads $\pi^\sharp\circ\ldr{\Q}=-\ldr{\Q}\circ\pi^\sharp=-(-1)^k\ldr{\Q}[k]\circ\pi^\sharp=-(-1)^{k^2}\ldr{\Q}[k]\circ\pi^\sharp=-(-1)^{k|\pi^\sharp|}\ldr{\Q}[k]\circ\pi^\sharp$.} of left DG $\M$-modules, or if and only if it is a morphism of right DG $\M$-modules,
		i.e.~$\pi^\sharp\circ\ldr{\Q}=-\ldr{\Q}\circ\pi^\sharp$.
	\end{theorem}
	\begin{proof}
		From \eqref{eqn:sharp}, 
		\[
		\Big( \ldr{\Q}(\pi^\sharp(\diff\xi_1)) +
		\pi^\sharp(\ldr{\Q}(\diff\xi_1)) \Big)\xi_2 = \Q\{\xi_1,\xi_2\}
		- (-1)^{|\xi_1|+k}\{\xi_1,\Q(\xi_2)\} - \{\Q(\xi_1),\xi_2\}.
		\]
		In other words, the compatibility of $\Q$ with
		$\{\cdot\,,\cdot\}_k$ is equivalent to
		$\ldr{\Q}\circ\pi^\sharp = - \pi^\sharp\circ\ldr{\Q}$.
	\end{proof} 
	A detailed analysis of this map in the cases of Poisson Lie $n$-algebroids equipped with a degree $-n$ bracket for $n\leq2$ is given in Section \ref{Section: Poisson Lie algebroids of low degree}. The two following corollaries can be
	realised as obstructions for a $\Q$-manifold endowed with a Poisson
	bracket to be symplectic. In particular, if the manifold is $\mathbb{N}$-graded of degree $n=2$ one obtains the
	corresponding results for Courant algebroids.
	
	\begin{corollary}\label{cor_poisson}
		Let $\M$ be an $\mathbb{Z}$-manifold equipped with a homological vector
		field $\Q$ and a degree $k$ Poisson bracket $\{\cdot\,,\cdot\}_k$. Then
		$(\M,\Q,\{\cdot\,,\cdot\}_k)$ is symplectic if and only if $\pi^\sharp$ is
		an anti-isomorphism of left DG $\M$-modules, or if and only if it is an isomorphism of right DG $\M$-modules.
	\end{corollary}
	
	\begin{corollary}
		For any $\mathcal{PQ}$-manifold $(\M,\Q,\{\cdot\,,\cdot\}_k)$ there are
		natural degree $k$ maps in cohomologies
		$\pi^\sharp\colon H^\bullet_L(\M,\Omega^1)\to
		H^{\bullet+k}_L(\M,\mathfrak{X})$ and $\pi^\sharp\colon H^\bullet_R(\M,\Omega^1)\to
		H^{\bullet+k}_R(\M,\mathfrak{X})$ which are isomorphisms if the
		bracket is symplectic.
	\end{corollary}
	
	\section{$\mathcal{PQ}$-manifolds: Weil vs Poisson-Weil algebras}\label{Section: PQ-manifolds: Weil vs Poisson-Weil algebras}
	
	 In this section, we will see that map $\pi^\sharp$ constructed above can be extended to a natural map between the bigraded algebras $\Omega^{\bullet,\bullet}(\M)$ and $\mathfrak{A}_{k}^{\bullet,\bullet}(\M)$ which anti-commutes with the horizontal and vertical differentials and consequently with the total differentials. This extension is obtained as follows:
	\begin{itemize}
		\item On the space $\cin(\M)$ it is given, up to a sign, by the identity map
		\[
		\cin(\M)^i\to\cin(\M)^i,\qquad \xi\mapsto (-1)^{i}\xi.
		\]
		\item On $q$-forms of the form $\xi\diff\xi_1\wedge\ldots\wedge\diff\xi_q$, where $\xi,\xi_1,\ldots,\xi_q\in\cin(\M)$ with $\xi$ homogeneous, it is given by
		\[
		\pi^\sharp(\xi\diff\xi_1\wedge\ldots\wedge\diff\xi_q) = \pi^\sharp(\xi)\pi^\sharp(\diff\xi_1)\wedge\ldots\wedge\pi^\sharp(\diff\xi_q) = (-1)^{|\xi|}\xi\X_{\xi_1}\wedge\ldots\wedge\X_{\xi_q}.
		\]
	\end{itemize}
	
	Consider now two homogeneous functions $\xi_1,\xi_2\in\cin(\M)$. For the horizontal differential $\ldr{\Q}$, the equality $\ldr{\Q}\circ \pi^\sharp=-\pi^\sharp\circ\ldr{\Q}$ is satisfied on functions because 
	\[
	\ldr{\Q}(\pi^\sharp(\xi)) = \ldr{\Q}((-1)^{|\xi|}\xi) = -(-1)^{|\xi|+1}\ldr{\Q}(\xi) = -\pi^\sharp(\ldr{\Q}(\xi))
	\]
	and on $1$-forms it is the  (anti-)morphism between the coadjoint and the adjoint modules which was explained before.
	
	For the vertical differentials $\diff$ and $(-1)^{k-1}\diff_{\pi}$, we compute on functions
	\[
	(-1)^{k-1}\diff_{\pi}(\pi^\sharp(\xi)) = (-1)^{k-1+|\xi|}[\pi,\xi] = -[\xi,\pi] = - \X_{\xi} = - \pi^\sharp(\diff\xi),
	\]
	and since all Hamiltonian vector fields are Poisson
	\[
	(-1)^{k-1}\diff_{\pi}(\pi^\sharp(\xi\diff\xi_1\wedge\ldots\wedge\diff\xi_q)) = - \X_{\xi}\wedge\X_{\xi_1}\wedge\ldots\wedge\X_{\xi_q} = - \sharp(\diff(\xi\diff\xi_1\wedge\ldots\wedge\diff\xi_q)).
	\]
	Therefore, we have proved the following result about the map $\pi^\sharp$.
	\begin{theorem}\label{thm_Weil_Poisson-Weil}
		Let $(\M,\Q,\{\cdot\,,\cdot\}_k)$ be a $\mathcal{PQ}$-manifold and consider the bidegree $(0,0)$ map  $\pi^\sharp\colon\Omega^{\bullet,\bullet}(\M)\to\mathfrak{A}_{k}^{\bullet,\bullet}(\M)$ of bigraded algebras constructed above.
		\begin{enumerate}
			\item The map $\pi^\sharp$ is an anti-morphism of complexes and DGA's for the columns of $\Omega^{\bullet,\bullet}(\M)$ and $\mathfrak{A}_k^{\bullet,\bullet}(\M)$.
			
			\item The map $\pi^\sharp$ is an anti-morphism of complexes and DGA's for the rows of $\Omega^{\bullet,\bullet}(\M)$ and $\mathfrak{A}_k^{\bullet,\bullet}(\M)$.
			
			\item The map $\pi^\sharp$ is an anti-morphism of complexes and DGA's between (the total complexes) $\Omega(\M)$ and $\mathfrak{A}_k(\M)$.
		\end{enumerate}
	Consequently, the map $\pi^\sharp$ induces a natural map on the level of the vertical and horizontal cohomologies
	\[
	\pi^\sharp\colon H^\bullet_{dR}(\M)\to H^\bullet_P(\M)
	\qquad \text{and}\qquad
	\pi^\sharp\colon H^\bullet_{\Omega,L}(\M)\to H^\bullet_{\mathfrak{A}_k,L}(\M).
	\]
	and also on the level of total cohomologies
	\[
	\pi^\sharp\colon H_W^{\bullet}(\M)\to H^{\bullet}_{PW}(\M),
	\]
	Moreover, it is an anti-isomorphism of differential bigraded algebras if and only if the Poisson structure is symplectic, and in this case all the induced maps in the cohomologies are also isomorphisms.
	\end{theorem}
	
	\chapter{Representations up to homotopy}\label{Chapter: Representations up to homotopy}
	
	This chapter is devoted to the study of representations up to homotopy for general Lie $n$-algebroids with an extra focus on the case $n=2$ where we explain thoroughly the adjoint representation. In addition, we analyse the Weil algebra of a split Lie $n$-algebroid and give the precise map between the coadjoint and adjoint representations of a split Poisson Lie $n$-algebroid for $n=0,1,2$.
	
	\section{The category of representations up to homotopy}\label{Section: The category of representations up to homotopy}
	
	Recall that a \textbf{representation up to homotopy of a Lie algebroid}\index{Lie algebroid!representation up to homotopy} $A$ \cite{ArCr12,Mehta14} is
	given by an $A$-module of the form
	$\Omega(A,\E)=\Omega(A)\otimes\Gamma(\E)$ for a graded vector bundle
	$\E$ over $M$. In the same manner, a \textbf{left representation up to homotopy of
		a Lie $n$-algebroid $(\M,\Q)$}\index{Lie $n$-algebroid!left representation up to homotopy} is defined as a left DG $\M$-module of the
	form $\cin(\M)\otimes\Gamma(\E)$ for a graded vector bundle $\E\to M$. \textbf{Right representations up to homotopy}\index{Lie $n$-algebroid!right representation up to homotopy} are defined similarly as right DG $\M$-modules of the form $\Gamma(\E)\otimes\cin(\M)$. In what follows, we will focus only on left representations and the prefix ``left" will be omitted for simplicity.
	
	Following the notation from \cite{ArCr12}, we denote the category of
	representations up to homotopy of $(\M,\Q)$ by $\underline{\mathbb{R}\text{ep}}^\infty(\M,Q)$,
	or simply by $\underline{\mathbb{R}\text{ep}}^\infty(\M)$. The isomorphism classes of representations up to
	homotopy of this category are denoted by $\underline{\text{Rep}}^\infty(\M,\Q)$,
	or by $\underline{\text{Rep}}^\infty(\M)$. A representation of the form
	$\E=E_0\oplus\ldots\oplus E_{k-1}$ is called \textbf{$k$-term
		representation}\index{Lie $n$-algebroid! $k$-term representation}, or simply \textbf{$k$-representation}\index{Lie $n$-algebroid!$k$-representation}.
	
	\begin{remark}
		Since by Remark \ref{Remark: Splitting of vector bundles} every vector bundle over an $\mathbb{N}$-manifold comes from an ordinary graded vector bundle over the base, it follows that the notions of DG-modules and representations up to homotopy are equivalent. However, the equivalence is not canonical.
	\end{remark}
	
	\begin{example}[$\Q$-closed functions]
		Let $(\M,\Q)$ be a Lie $n$-algebroid and suppose $\xi\in \cin(\M)^k$
		such that
		$\Q(\xi) = 0$. Then one can construct a representation up to
		homotopy $\cin(\M)\otimes\Gamma(\E_\xi)$ of $\M$ on the graded vector bundle
		$\E_\xi=M\times(\mathbb{R}[0]\oplus\mathbb{R}[1-k])\to M$
		(i.e.~$\mathbb{R}$ in degrees 0 and $k-1$, and zero otherwise). Its
		differential $\D_\xi$ is given in components by the map
		\[
		\D_\xi = \sum_i \D_\xi^i,
		\]
		where
		\[
		\D_\xi^i\colon \cin(\M)^i\oplus \cin(\M)^{i-k+1}\to \cin(\M)^{i+1}\oplus \cin(\M)^{i-k+2}
		\]
		is defined by the formula
		\[
		\D_\xi^i(\zeta_1,\zeta_2)=(\Q(\zeta_1) + (-1)^{i-k+1}\zeta_2\xi,\Q(\zeta_2)).
		\]
		If there is an element $\xi'\in \cin(\M)^k$ which is
		$\Q$-cohomologous to $\xi$, i.e.~$\xi-\xi'=\Q(\xi'')$ for some
		$\xi''\in \cin(\M)^{k-1}$, then the representations
		$\E_\xi$ and $\E_{\xi'}$ are isomorphic via the isomorphism
		$\mu\colon \E_\xi\to \E_{\xi'}$ defined in components by
		\[
		\mu^i\colon \cin(\M)^i\oplus \cin(\M)^{i-k+1}\to \cin(\M)^{i}\oplus \cin(\M)^{i-k+1}
		\]
		given by the formula
		\[
		\mu^i(\zeta_1,\zeta_2)=(\zeta_1+\zeta_2\xi'',\zeta_2).
		\]
		Hence, one obtains a well-defined map
		$H^\bullet(\M)\to\underline{\text{Rep}}^\infty(\M)$. In particular, if $\M$ is a Lie
		algebroid, the above construction recovers Example 3.5 in
		\cite{ArCr12}.
	\end{example}
	
	\section{The case of (split) Lie $2$-algebroids}\label{Section: The case of (split) Lie 2-algebroids}
	
	Fix now a split Lie $2$-algebroid $\M$, and recall that from the
	analysis of Section \ref{Section: Q-structures and Lie n-algebroids}, $\M$ is given by the sum $Q[1]\oplus B^*[2]$ together with the anchor $\rho_Q:Q\to TM$, the skew-symmetric dull bracket $[\cdot\,,\cdot]:\Gamma(Q)\times\Gamma(Q)\to\Gamma(Q)$ on the vector bundle $Q\to M$, the vector bundle map $\ell:B^*\to Q$, the $Q$-connection $\nabla$ on the vector bundle $B\to M$, and the vector valued $3$-form $\omega\in\Omega^3(Q,B^*)$. In particular, we have obtain the complex 
	which forms the complex
	\[
	B^*\overset{\ell}{\longrightarrow} Q\overset{\rho_Q}{\longrightarrow} TM.
	\]
	In order to make the notation lighter, we will write $Q[1]\oplus B^*[2]$ for the split Lie $2$-algebroid $(Q[1]\oplus B^*[2],\rho_Q,\ell,[\cdot\,,\cdot],\nabla,\omega)$. 
	
	Unravelling the data of the definition of representations up to
	homotopy for the special case where $E$ is concentrated only in degree
	0 yields the following characterisation.
	
	\begin{proposition}\label{Representations_of_Lie_2-algebroids}
		A representation\index{split Lie $2$-algebroid!representation} of the Lie $2$-algebroid $Q[1]\oplus B^*[2]$
		consists of a (non-graded) vector bundle $E$ over $M$, together with
		a $Q$-connection $\nabla$ on $E$ such that\footnote{Note that all
			the objects that appear in the following equations act via the
			generalised wedge products that were discussed before. For
			example, $\partial(\diff_\nabla e)$ and $\omega_2(\omega_2(e))$
			mean $\partial\wedge\diff_\nabla e$ and
			$\omega_2\wedge\omega_2(e)$, respectively. This is explained in
			detail in the Appendix of \cite{ArCr12}.}  :
		\begin{enumerate}[(i)]
			\item $\nabla$ is flat, i.e.~$R_\nabla = 0$ on $\Gamma(E)$,
			\item $\nabla_{\ell(\beta)}e = 0$ for all $\beta\in\Gamma(B^*)$ and $e\in\Gamma(E)$, where $\partial_B:=\ell^*$. That is, $\partial_B\circ \diff_\nabla = 0$ on $\Gamma(E)$.
		\end{enumerate}
	\end{proposition}
	\begin{proof}
		Let $(E,\D)$ be a representation of the Lie $2$-algebroid. Due to
		the Leibniz rule, $\D$ is completely characterised by what it does
		on $\Gamma(E)$. By definition, it sends $\Gamma(E)$ into
		$\Omega^1(Q,E)$. Using the Leibniz rule once more together with the
		definition of the homological vector field $\Q$ on $\Omega^1(Q)$,
		for all $f\in C^\infty(M)$ and all $e\in\Gamma(E)$ yields
		\[
		\D(fe) = (\rho_Q^*\diff f)\otimes e + f\D(e),
		\]
		which implies that $\D = \diff_\nabla$ for a
		$Q$-connection $\nabla$ on $\Gamma(E)$. Moreover, by definition of $\D$
		one must have $\D^2(e) = 0$ for all $e\in\Gamma(E)$. On the
		other hand, a straightforward computation yields
		\[
		\D^2(e) = \D(\diff_\nabla e) = \diff_\nabla^2 e +
		\partial_B(\diff_\nabla
		e)\in\Omega^2(Q,E)\oplus\Gamma(B\otimes E).\qedhere
		\]
	\end{proof}
	
	\begin{example}[Trivial line bundle]\label{Trivial line bundle representation example}
		The trivial line bundle $\R[0]$ over $M$ with $Q$-connection defined by
		\[
		\diff_\nabla f = \diff_Q f =\rho_Q^* \diff f
		\]
		is a representation of the Lie $2$-algebroid $Q[1]\oplus
		B^*[2]$. The operator $\D$ is given by the homological vector
		field $\Q$ and thus the cohomology induced by the
		representation is the Lie $2$-algebroid cohomology:
		$H^\bullet(\M,\R) = H^\bullet(\M)$. The shifted version of this example was used before to define general shifts of DG $\M$-modules.
	\end{example}
	
	\begin{example}\label{Trivial representation of rank k example}
		More generally, for all $k>0$, the trivial vector bundle $\R^k$ of
		rank $k$ over $M$ with $Q$-connection defined component-wise as in
		the example above becomes a representation with cohomology
		$H^\bullet(\M,\R^k)=H^\bullet(\M)\oplus\ldots\oplus H^\bullet(\M)$
		($k$-times).
	\end{example}
	
	\begin{remark}
		Given a split Lie $n$-algebroid $A_1[1]\oplus\ldots\oplus A_n[n]$
		over a smooth manifold $M$,  with $n\geq 2$, the vector bundle
		$A_1\to M$ carries a skew-symmetric dull algebroid structure induced by the
		2-bracket and the anchor $\rho\colon A_1\to TM$ given by
		$\Q(f)=\rho^*\diff f$, for $f\in C^\infty(M)$. Hence, Proposition
		\ref{Representations_of_Lie_2-algebroids}, Example \ref{Trivial line
			bundle representation example} and Example \ref{Trivial
			representation of rank k example} can be carried over verbatim to
		the general case.
	\end{remark}
	
	A more interesting case is for representations of the form $\E=E_0\oplus E_1\oplus E_2$
	concentrated in 3 degrees, say in degrees $-2,-1$ and $0$. An explicit description
	of those representations is given below. The reader should note the
	similarity of the following proposition with the description of 2-term
	representations of Lie algebroids from \cite{ArCr12}.
	
	\begin{proposition}\label{3-term_representations}
		A 3-term representation up to homotopy\index{split Lie $2$-algebroid!3-term representation}
		$(\E = E_0[2]\oplus E_1[1]\oplus E_2[0],\D)$ of $Q[1]\oplus B^*[2]$ is
		equivalent to the following data:
		\begin{enumerate}[(i)]
			\item A degree 1 map $\partial\colon \E\to \E$ such that $\partial^2 = 0$,
			\item a $Q$-connection $\nabla$ on the complex $\partial\colon E_\bullet\to E_{\bullet + 1}$,
			\item an element $\omega_i\in\Omega^i(Q,\underline{\End}^{1-i}(\E))$ for $i=2,3$,
			\item an element
			$\phi_j\in\Gamma(B)\otimes\Omega^j(Q,\underline{\End}^{-j-1}(\E))$
			for $j=0,1$
		\end{enumerate}
		such that\footnote{In the following
			equations, the map
			$\partial_B\colon\Omega^1(Q)\to\Gamma(B)$ extends to
			$\partial_B\colon \Omega^k(Q)\to\Omega^{k-1}(Q,B)$ by the
			rule
			$\partial_B(\tau_1\wedge\ldots\wedge\tau_k) =
			\sum_{i=1}^k
			(-1)^{i+1}\tau_1\wedge\ldots\wedge\hat{\tau}_i\wedge\ldots\wedge\tau_k\otimes\partial_B\tau_i$,
			for $\tau_i\in\Omega^1(Q)$.}
		\begin{enumerate}
			\item $\partial\circ\omega_2 + \diff_\nabla^2 + \omega_2\circ\partial = 0$,
			\item
			$\partial\circ\phi_0 + \partial_B\circ \diff_\nabla
			+ \phi_0\circ\partial = 0$,
			\item
			$\partial\circ\omega_3 + \diff_\nabla\circ\omega_2 +
			\omega_2\circ \diff_\nabla + \omega_3\circ\partial =
			\langle \omega,\phi_0 \rangle$,
			\item
			$\diff_{\overline{\nabla}}\phi_0 +
			\partial\circ\phi_1 + \partial_B\circ\omega_2 +
			\phi_1\circ\partial = 0$,
			\item
			$\diff_\nabla\circ\omega_3 + \omega_2\circ\omega_2 +
			\omega_3\circ \diff_\nabla = \langle \omega,\phi_1
			\rangle$,
			\item
			$\diff_{\overline{\nabla}}\phi_1 +
			\omega_2\circ\phi_0 + \partial_B\circ\omega_3 +
			\phi_0\circ\omega_2 = 0$,
			\item $\phi_0\circ\phi_0 + \partial_B\circ\phi_1 = 0$,
		\end{enumerate}
		where $\overline{\nabla}$ is the $Q$-connection on
		$B\otimes\underline{\End}^{-j-1}(\E)$ induced by
		$\nabla$ on $B$ and $\nabla^{\underline{\End}}$ on
		$\underline{\End}(\E)$.
	\end{proposition}
	\begin{remark}
		\begin{enumerate}
			\item  If both of the bundles $E_1$ and $E_2$ are zero, the
			equations agree with those of a 1-term representation.
			\item  The equations in
			the statement can be summarised as follows:
			\[
			[\partial,\phi_0] + \partial_B\circ \diff_\nabla = 0,\qquad
			\phi_0\circ\phi_0 + \partial_B\circ\phi_1 = 0,
			\] and for all $i$:
			\[
			[\partial,\omega_i] + [\diff_\nabla,\omega_{i-1}]
			+\omega_2\circ\omega_{i-2} +
			\omega_3\circ\omega_{i-3} + \ldots
			+\omega_{i-2}\circ\omega_2 = \langle
			\omega,\phi_{i-3} \rangle,
			\]
			\[
			\partial_B\circ\omega_{i+2} + [\partial,\phi_{i+1}]
			+ \diff_{\overline{\nabla}}\phi_i +
			\sum_{j\geq2}[\omega_j,\phi_{i-j+1}] = 0.
			\]
		\end{enumerate}
	\end{remark}
	\begin{proof}
		It is enough to check how $\D$ acts on $\Gamma(\E)$. Since $\D$ is of degree
		1, it maps each $\Gamma(E_i)$ into the direct sum
		\[
		\Gamma(E_{i+1}) \oplus
		\left(\cin(\M)^1\otimes\Gamma(E_i)\right) \oplus
		\left(\cin(\M)^2\otimes\Gamma(E_{i-1})\right) \oplus
		\left(\cin(\M)^3\otimes\Gamma(E_{i-2})\right).
		\]
		Considering the components of $\D$, this translates to the
		following three equations: 
		\[
		\D(e) = \partial(e) + d(e)\in\Gamma(E_1)\oplus\Omega^1(Q,E_0)
		\]
		for $e\in\Gamma(E_0)$,
		\[
		\D(e) = \partial(e) + d(e) + \omega_2(e) +
		\phi_0(e)\in\Gamma(E_2)\oplus\Omega^1(Q,E_1)\oplus\Omega^2(Q,E_0)\oplus\left(\Gamma(B)\otimes\Gamma(E_0)\right)
		\]
		for $e\in\Gamma(E_1)$, and
		
		\begin{align*}
		\D(e) = &\ d(e) + \omega_2(e) + \phi_0(e) + \omega_3(e) +\phi_1(e)\\
		& \in\Omega^1(Q,E_2)\oplus\Omega^2(Q,E_1)\oplus\left(\Gamma(B)\otimes\Gamma(E_1)\right)\\
		& \oplus\Omega^3(Q,E_0)\oplus\left(\Gamma(B)\otimes\Omega^1(Q,E_0)\right)
		\end{align*}
		for $e\in\Gamma(E_2)$. Due to Lemma
		\ref{wedge_product-operators_Correspondence_Lemma} and the
		Leibniz rule for $\D$, 
		$\partial\in\underline{\End}^1(\E)$, $d=\diff_\nabla$ where
		$\nabla$ are $Q$-connections on the vector bundles $E_i$ for
		$i = 0,1,2$,
		$\omega_i\in\Omega^i(Q,\underline{\End}^{1-i}(\E))$ for
		$i = 2,3$, and
		$\phi_i\in\Gamma(B)\otimes\Omega^i(Q,\underline{\End}^{-i-1}(\E))$
		for $i = 0,1$.
		
		A straightforward computation and a degree count in the expansion of
		the equation $\D^2=0$ shows that $(\E,\partial)$ is a complex,
		$\nabla$ commutes with $\partial$, and the equations in the statement hold. Indeed, suppose $e\in\Gamma(E_0)$. Then we have
		\[
		\D^2(e) = \D(\partial(e) + \diff_{\nabla}e) = \partial^2(e) + \diff_{\nabla}(\partial(e)) + \omega_2(\partial(e)) + \phi_0(\partial(e)) + \D(\diff_{\nabla}e).
		\]
		In order to compute $\D(\diff_{\nabla}e)$, we write $\diff_{\nabla}e\in\Omega^1(Q,E_0)$ as a (finite) sum
		\[
		\diff_{\nabla}e = \sum_i \tau_i\otimes e_i',
		\]
		with $\tau_i\in\Gamma(Q^*)$ and $e_i'\in\Gamma(E_0)$. Due to the linearity and the Leibniz rule of $\D$, we have
		\begin{align*}
			\D(\diff_{\nabla}(e)) = & \sum_i \D(\tau_i\otimes e_i') \\
			= & \sum_i \Q(\tau_i)\otimes e_i' - \tau_i \otimes \D(e_i') \\
			= & \sum_i (\diff_Q\tau_i + \partial_B\tau_i)\otimes e_i' - \tau_i\otimes(\partial(e_i') + \diff_{\nabla}e_i') \\
			= & \sum_i \diff_Q\tau_i \otimes e_i' - \tau_i\otimes \diff_{\nabla}e_i' + \partial_B\tau_i \otimes e_i' - \tau_i\otimes \partial(e_i') \\
			= & \sum_i \diff_{\nabla}(\tau_i\otimes e_i') + \partial_B(\tau_i \otimes e_i') + \partial(\tau_i \otimes e_i') \\
			= &\ \diff_{\nabla}\left(\sum_i \tau_i\otimes e_i'\right) + \partial_B\left(\sum_i \tau_i\otimes e_i'\right) + \partial\left(\sum_i \tau_i\otimes e_i'\right) \\
			= &\ \diff_{\nabla}^2 e + \partial_B(\diff_{\nabla} e) + \partial(\diff_{\nabla} e).
		\end{align*}
	Coupling bidegrees for $\D^2(e)$, we obtain
	\begin{align*}
		\D^2(e) = &\ \partial^2(e) + \Big(\diff_{\nabla}(\partial(e))+\partial(\diff_{\nabla}e)\Big) +
		\Big(\omega_2(\partial(e)) + \diff_{\nabla}^2 e\Big) +
		\Big(\phi_0(\partial(e)) + \partial_B(\diff_{\nabla} e)\Big) \\
		 & \in \Gamma(E_2) \oplus \Omega^1(Q,E_1) \oplus \Omega^2(Q,E_0) \oplus \left(\Gamma(B)\otimes\Gamma(E_0)\right).
	\end{align*}
	One proceeds similarly for the cases $e\in\Gamma(E_1)$ and $e\in\Gamma(E_2)$.
    \end{proof}
	
	\section{Adjoint representation of a split Lie $2$-algebroid}\label{Section: Adjoint representation of a split Lie $2$-algebroid}
	
	This section shows that any split Lie $2$-algebroid
	$Q[1]\oplus B^*[2]$ admits an adjoint representation which is a 3-term representation up to homotopy. It is a
	generalisation of the adjoint representation of a (split) Lie
	$1$-algebroid studied in \cite{ArCr12}.
	
	\begin{proposition}\label{Adjoint_representation_of_Lie_2-algebroid}
		Any split Lie $2$-algebroid $Q[1]\oplus B^*[2]$ admits a 3-term
		representation up to homotopy as follows: Choose arbitrary
		$TM$-connections on $Q$ and $B^*$ and denote them both by $\nabla$. Then
		the structure objects are\footnote{Some signs are chosen
			so that the map given in
			\ref{adjoint_module_adjoint_representation_isomorphism} is an
			isomorphism for the differential of the adjoint module defined
			earlier.}
		\begin{enumerate}[(i)]
			\item the \textbf{adjoint complex}\index{split Lie $2$-algebroid!adjoint complex} $B^*[2]\to Q[1]\to TM[0]$
			with maps $-\ell$ and $\rho_Q$,
			\item the two $Q$-connections $\nabla^{\text{bas}}$ on $Q$ and
			$TM$, and the $Q$-connection $\nabla^*$ on $B^*$ given by
			the split Lie 2-algebroid,
			\item the element
			$\omega_2\in\Omega^2(Q,\Hom(Q,B^*)\oplus\Hom(TM,Q))$ defined
			by
			\[
			\omega_2(q_1,q_2)q_3 =
			-\omega(q_1,q_2,q_3)\in\Gamma(B^*)\ \text{and}\
			\omega_2(q_1,q_2)X =
			-R_\nabla^\text{bas}(q_1,q_2)X\in\Gamma(Q)
			\]
			for $q_1,q_2,q_3\in\Gamma(Q)$ and $X\in\mathfrak{X}(M)$,
			\item the element $\omega_3\in\Omega^3(Q,\Hom(TM,B^*))$
			defined by
			\[
			\omega_3(q_1,q_2,q_3)X = - (\nabla_X\omega)(q_1,q_2,q_3)\in\Gamma(B^*)
			\]
			for $q_1,q_2,q_3\in\Gamma(Q)$ and $X\in\mathfrak{X}(M)$,
			\item the element $\phi_0\in\Gamma(B)\otimes(\Hom(Q,B^*)\oplus\Hom(TM,Q))$ defined by
			\[
			\phi_0(\beta)X = \ell(\nabla_X\beta) - \nabla_X(\ell(\beta))\in\Gamma(Q)\ \text{and}\ 
			\phi_0(\beta)q = \nabla_{\rho(q)}\beta - \nabla^*_q\beta\in\Gamma(B^*)
			\]
			for $\beta\in\Gamma(B^*),q\in\Gamma(Q),X\in\mathfrak{X}(M)$,
			\item the element $\phi_1\in\Gamma(B)\otimes\Omega^1(Q,\Hom(TM,B^*))$ defined by
			\[
			\phi_1(\beta,q)X =  \nabla_X\nabla^*_q \beta - \nabla^*_q\nabla_X
			\beta-\nabla^*_{\nabla_X q} \beta 
			+ \nabla_{\nabla^{\rm bas}_  qX} \beta\in\Gamma(B^*)
			\]
			for $\beta\in\Gamma(B^*),q\in\Gamma(Q),X\in\mathfrak{X}(M)$.
		\end{enumerate}
	\end{proposition}
	
	The proof can be done in two ways. First, one could check explicitly
	that all the conditions of a 3-representation of $Q[1]\oplus B^*[2]$
	are satisfied. This is an easy but long computation and it can be found in Appendix \ref{Appendix: Adjoint representation of a split Lie $2$-algebroid}. Instead, the following
	section shows that given a splitting and $TM$-connections on the
	vector bundles $Q$ and $B^*$, there exists an isomorphism of sheaves
	of $\cin(\M)$-modules between the adjoint module $\mathfrak{X}(\M)$
	and $\cin(\M)\otimes\Gamma(TM[0]\oplus Q[1]\oplus B^*[2])$, such that the
	objects defined above correspond to the differential
	$\ldr{\Q}$. Another advantage of this approach is that it gives a
	precise recipe for the definition and the explicit formulas for the
	components of the adjoint representation of a Lie $n$-algebroid for
	general $n$.
	
	\begin{remark}[Adjoint representation of a Courant algebroid]\index{Courant algebroid!adjoint representation}\label{Adjoint representation of Courant algebroid}
		The adjoint representation of a Courant algebroid $E\to M$ can be
		deduced from the formulas above and from Example
		\ref{Split_symplectic_Lie_2-algebroid_example}. Choose a linear
		connection $\nabla\colon\mx(M)\times\Gamma(E)\to\Gamma(E)$ that
		preserves the metric underlying the Courant algebroid structure on
		$E$. As in Example \ref{Split_symplectic_Lie_2-algebroid_example},
		define the basic connection
		$\nabla^{\rm bas}\colon \Gamma(E)\times\mx(M)\to\mx(M)$,
		$\nabla^{\rm bas}_eX=[\rho(e),X]+\rho(\nabla_Xe)$. Recall that
		$\nabla$ defines as well the dull bracket $[\cdot\,,\cdot]$ on
		sections of $E$:
		\[ [e_1,e_2]=\lb e_1, e_2\rb -\rho^*\langle \nabla_\cdot e_1,e_2\rangle.
		\]
		The dull bracket and the $TM$-connection on $E$ defines the basic
		$E$-connection
		$\nabla^{\rm bas}\colon \Gamma(E)\times\Gamma(E)\to\Gamma(E)$,
		$\nabla^{\rm bas}_{e_1}e_2= [e_1,e_2]+\nabla_{\rho(e_2)}e_1=\lb e_1, e_2\rb -\rho^*\langle
		\nabla_\cdot e_1,e_2\rangle+\nabla_{\rho(e_2)}e_1$. Choose in
		addition a $TM$-connection $\nabla$ on $TM$. The complex is $T^*M[2]\to E[1]\to TM[0]$ with maps $-\rho^* $ and
		$\rho$. The $E$-connections on $T^*M$, $E$ and $TM$ are
		${\nabla^{\rm bas}}^*$, $\nabla^{\rm bas}$ and $\nabla^{\rm bas}$
		defined as above, respectively. The form
		$\omega_2\in \Omega^2(E, \Hom(E,T^*M))$ is  given by
		\[\langle\omega_2(e,e'),X\rangle=\nabla_X\llbracket e,e' \rrbracket -
		\llbracket \nabla_Xe,e' \rrbracket -\llbracket e,\nabla_Xe'
		\rrbracket -\nabla_{\nabla_{e'}^{\text{bas}} X} e +
		\nabla_{\nabla_{e}^{\text{bas}} X} e' + P^{-1}\langle
		\nabla_{\nabla_{.}^{\text{bas}}X}e,e' \rangle\]
		for $e,e'\in\Gamma(E)$ and $X\in\mx(M)$, while the second summand
		$\omega_2\in \Omega^2(E, \Hom(TM,E))$ is given by
		\[\omega_2(e, e')X=\nabla_X[e,e'] -
		[e,\nabla_Xe'] -[\nabla_Xe,e'] -
		\nabla_{\nabla_{e'}^\text{bas}X}e +
		\nabla_{\nabla_{e}^\text{bas}X}e'.
		\]
		This terms can be compared with the components of the
		representation up to homotopy of a Lie algebroid $A\to M$,
		after the choice of a $TM$-connection on $A$, see
		\cite{GrMe10}. The remaining terms are given by (iv), (v)
		and (vi) in Proposition
		\ref{Adjoint_representation_of_Lie_2-algebroid}, and as they
		do not seem more instructive than the general form
		in the proposition, they are not computed in more detail
		here.
	\end{remark}
	
	\section{Adjoint module vs adjoint representation}\label{adjoint_module_adjoint_representation_isomorphism}
	
	Recall that for a split $[n]$-manifold $\M=\bigoplus E_i[i]$, the
	space of vector fields over $\M$ is generated as a $\cin(\M)$-module
	by two special kinds of vector fields. Namely, the degree $-i$ vector
	fields $\hat{e}$ for $e\in\Gamma(E_i)$, and the family of vector
	fields $\nabla^1_X \oplus \ldots \oplus \nabla^n_X$ for
	$X\in\mathfrak{X}(M)$ and a choice of $TM$-connections $\nabla^i$ on
	the vector bundles $E_i$.
	
	Consider now a Lie 2-algebroid $(\M,\Q)$ together with a splitting
	$\M\simeq Q[1]\oplus B^*[2]$ and a choice of $TM$-connections
	$\nabla^{B^*}$ and $\nabla^Q$ on $B^*$ and $Q$, respectively. These
	choices give as follows the adjoint representation $\ad_\nabla$, whose
	complex is given by $TM[0]\oplus Q[1]\oplus B^*[2]$. Define a map
	$\mu_\nabla\colon \cin(\M)\otimes\Gamma(TM[0]\oplus Q[1]\oplus B^*[2])\to \mathfrak{X}(\M)$ on the generators by
	\[
	\Gamma(B^*)\ni \beta  \mapsto \hat{\beta},\qquad
	\Gamma(Q)\ni q \mapsto \hat{q},\qquad
	\mathfrak{X}(M)\ni X \mapsto \nabla^{B^*}_X \oplus \nabla^Q_X
	\]
	and extend $\cin(\M)$-linearly to the whole space to obtain a degree-preserving
	isomorphism of sheaves of $\cin(\M)$-modules. A straightforward
	computation shows that
	\[
	\ldr{\Q}(\hat{\beta}) = \mu_\nabla\left(-\ell(\beta) + \diff_{\nabla^*}\beta\right),
	\]
	\[
	\ldr{\Q}(\hat{q}) = \mu_\nabla\left(\rho_Q(q) +
	\diff_{\nabla^{\text{bas}}}q +\omega_2(\cdot\,,\cdot)q +
	\phi_0(\cdot)q\right),
	\]
	\[
	\ldr{\Q}(\nabla_X^{B^*} \oplus \nabla_X^Q) = \mu_\nabla\left(
	\diff_{\nabla^{\text{bas}}} X +\phi_0(\cdot)X
	+\omega_2(\cdot\,,\cdot)X + \omega_3(\cdot\,,\cdot\,,\cdot)X +
	\phi_1(\cdot\,,\cdot)X \right)
	\]
	and therefore, the objects in the statement of Proposition
	\ref{Adjoint_representation_of_Lie_2-algebroid} define the differential $\D_{\ad_\nabla}:=\mu_\nabla^{-1}\circ\ldr{\Q}\circ\mu_\nabla$ of a
	3-representation of
	$Q[1]\oplus B^*[2]$, called the \textbf{adjoint representation of the Lie $2$-algebroid}\index{split Lie $2$-algebroid!adjoint representation}\index{adjoint representation} and
	denoted as $(\ad_\nabla,\D_{\ad_\nabla})$.  The adjoint representation
	is hence, up to isomorphism, independent of the choice of splitting
	and connections (see the following section for the precise transformations), and gives so a
	well-defined class $\ad\in\underline{\text{Rep}}^\infty_L(\M)$.
	
	Due to the result above, one can also define the \textbf{coadjoint
		representation of a Lie 2-algebroid}\index{split Lie $2$-algebroid!coadjoint representation}\index{coadjoint representation} $(\M,\Q)$ as the isomorphism
	class $\ad^*\in\underline{\text{Rep}}^\infty_L(\M)$. To find an explicit representative of $\ad^*$,
	suppose that $Q[1]\oplus B^*[2]$ is a splitting of $\M$, and consider
	its adjoint representation $\ad_\nabla$ as above for some choice of
	$TM$-connections $\nabla$ on $B^*$ and $Q$. Recall that given a
	representation up to homotopy $(\E,\D)$ of a general Lie $n$-algebroid $(\M,\Q)$, its dual $\E^*$
	becomes a representation up to homotopy with operator $\D^*$
	characterised by the formula
	\[
	\Q(\xi\wedge\xi') = \D^*(\xi)\wedge\xi' + (-1)^{|\xi|}\xi\wedge\D(\xi'),
	\]
	for all $\xi\in \cin(\M)\otimes\Gamma(\E^*)$ and
	$\xi'\in \cin(\M)\otimes\Gamma(\E)$.  Here,
	$\wedge=\wedge_{\langle\cdot\,,\cdot\rangle}$, where
	$\langle\cdot\,,\cdot\rangle$ is the pairing of $\E$ with $\E^*$. Unravelling the definition of the dual for the left
	representation $\ad_\nabla$, one finds that the structure differential of $\ad_\nabla^*=\cin(\M)\otimes\Gamma(B[-2]\oplus Q^*[-1] \oplus TM[0])$ is
	given by the following objects:
	\begin{enumerate}
		\item the \textbf{coadjoint complex}\index{split Lie $2$-algebroid!coadjoint complex} $T^*M\to Q^* \to B$ obtained by
		$-\rho_Q^*$ and $-\ell^*$,
		\item the $Q$-connections $\nabla$ on $B$ and $\nabla^{\text{bas},*}$ on $Q^*$ and $T^*M$,
		\item the elements
		\begin{center}
			\begin{tabular}{l l}
				$\omega_2^*(q_1,q_2)\tau=\tau\circ\omega_2(q_1,q_2)$, & $\omega_2^*(q_1,q_2)b=-b\circ\omega_2(q_1,q_2)$, \\
				$\phi_0^*(\beta)\tau=\tau\circ\phi_0(\beta)$, & $\phi_0^*(\beta)b=-b\circ\phi_0(\beta)$, \\
				$\omega_3^*(q_1,q_2,q_3)b=-b\circ\omega_3(q_1,q_2,q_3)$, & $\phi_1^*(\beta,q)b=-b\circ\phi_1(\beta,q)$,
			\end{tabular}
		\end{center}
		for all $q,q_1,q_2,q_3\in\Gamma(Q),\tau\in\Gamma(Q^*),b\in\Gamma(B)$ and $\beta\in\Gamma(B^*)$.
	\end{enumerate}
	
	\begin{remark}\label{Iso_coad_mod_coad_rep}
		The coadjoint representation can also be obtained from the coadjoint
		module $\Omega^1(\M)$ by the right $\cin(\M)$-module isomorphism
		$\mu^\star_\nabla\colon\Omega^1(\M)\to\Gamma(B[-2] \oplus
		Q^*[-1] \oplus T^*M[0])\otimes\cin(\M)$ which is dual to
		$\mu_\nabla\colon \cin(\M)\otimes\Gamma(TM[0]\oplus Q[1]\oplus B^*[2])\to
		\mathfrak{X}(\M)$ above. Explicitly, it is defined as the pull-back map
		$\mu^\star_\nabla(\omega)=\omega\circ\mu_\nabla$ for all $\omega\in\Omega^1(\M)$,
		whose inverse is given on the generators by 
		$\Gamma(B[-2] \oplus Q^*[-1] \oplus T^*M[0])\otimes\cin(\M)\to
		\Omega^1(\M)$,
		\[ 
		\Gamma(B)\ni b\mapsto \dr b-\diff_{\nabla^*} b, \quad \Gamma(Q^*)\ni\tau\mapsto  \dr\tau - \diff_{\nabla^*}\tau,
		\quad \text{and}\quad \Omega^1(M)\ni\theta\mapsto\theta.
		\]
	\end{remark}
	
	\section{Coordinate transformation of the adjoint representation}\label{Section: Coordinate transformation of the adjoint representation}
	
	The adjoint representation up to homotopy of a Lie 2-algebroid was constructed after a choice of splitting and $TM$-connections. This
	section explains how the adjoint representation transforms under
	different choices.
	
	First, a morphism of 3-representations of a split Lie 2-algebroid can
	be described as follows.
	
	\begin{proposition}\label{morphism_of_3-term_representations}
		Let $(\E,\D_{\E})$ and $(\underline{F},\D_{\underline{F}})$ be
		3-term representations up to homotopy of the split Lie 2-algebroid
		$Q[1]\oplus B^*[2]$. A morphism $\mu\colon \E\to \underline{F}$ is
		equivalent to the following data:
		\begin{enumerate}[(i)]
			\item For each $i=0,1,2$, an element
			$\mu_i\in\Omega^{i}(Q,\underline{\Hom}^{-i}(\E,\underline{F}))$.
			\item An element
			$\mu_0^b\in\Gamma(B\otimes\underline{\Hom}^{-2}(\E,\underline{F}))$.
		\end{enumerate}
		The above objects are subject to the relations
		\begin{enumerate}
			\item
			$[\partial,\mu_i] + [\diff_\nabla,\mu_{i-1}] +
			{\displaystyle\sum_{j+k=i,i\geq2}[\omega_j,\mu_k]} = \langle
			\omega,\mu^b_{i-3} \rangle$,
			\item $[\partial,\mu_0^b] + [\phi_0,\mu_0] + \partial_B\circ\mu_1 = 0$,
			\item
			$\diff_{\overline{\nabla}}\mu_0^b + [\phi_0,\mu_1] +
			[\phi_1,\mu_0] + \partial_B\circ\mu_2 = 0$.
		\end{enumerate}
	\end{proposition}
	\begin{proof}
		As before it  suffices to check how $\mu$ acts on $\Gamma(\E)$, by the
		same arguments. Then it must be of the type
		\[
		\mu = \mu_0 + \mu_1 + \mu_2 + \mu_0^b,
		\]
		where
		$\mu_i\in\Omega^{i}(Q,\underline{\Hom}^{-i}(\E,\underline{F}))$
		and
		$\mu^b\in\Gamma(B)\otimes\Gamma(\underline{\Hom}^{-2}(\E,\underline{F}))$. The three equations in the statement come from
		the expansion of
		$\mu\circ\D_{\E} = \D_{\underline{F}}\circ\mu$ when $\mu$ is
		written in terms of the components defined above. We check here the case $\mu\circ \D_{\E}(e) = \D_{\underline{F}}\circ\mu(e)$ for $e\in\Gamma(E_0)$ and omit the rest. One the one hand, we have that
		\[
		\mu(\D_{\E}(e)) = \mu(\partial(e)) + \mu(\diff_{\nabla}e) = \mu_0(\partial(e)) + \mu_1(\partial(e)) + \mu(\diff_{\nabla}e),
		\]
		while on the other hand
		\[
		\D_{\F}(\mu(e)) = \D_{\F}(\mu_0(e)) = \partial(\mu_0(e)) + \diff_{\nabla}(\mu_0(e)).
		\]
		Write $\diff_{\nabla}e\in\Omega^1(Q,E_0)$ as a (finite) sum of the form
		\[
		\diff_{\nabla}e = \sum_i \tau_i\otimes e'_i,
		\]
		with $\tau_i\in\Gamma(Q^*)$ and $e'_i\in\Gamma(E_0)$. Then we have
		\[
		\mu(\diff_{\nabla}e) = \mu\left( \sum_i \tau_i\otimes e'_i \right) 
		= \sum_i \tau_i\otimes \mu(e'_i) 
		= \sum_i \tau_i\otimes \mu_0(e'_i)
		= \sum_i \mu_0\wedge(\tau_i\otimes e'_i)
		= \mu_0(\diff_{\nabla}e).
		\]
		Equalising $\mu(\D_{\E}(e))$ with $\D_{\F}(\mu(e))$ and comparing degrees, we obtain
		\[
		\mu_0(\partial(e)) = \partial(\mu_0(e))
		\qquad \text{and} \qquad
		\diff_{\nabla}(\mu_0(e)) -\mu_0(\diff_{\nabla}e) = \mu_1(\partial(e)). \qedhere
		\]
	\end{proof}
	
	The transformation of $\ad \in \underline{\text{Rep}}^\infty(\M)$ for a fixed splitting
	$Q[1]\oplus B^*[2]$ of $\M$ and different choices of $TM$-connections
	is given by their difference. More precisely, let $\nabla$ and
	$\nabla'$ be the two $TM$-connections. Then
	the map
	$\mu=\mu_{\nabla'}^{-1}\circ \mu_\nabla\colon\ad_\nabla\to \ad_{\nabla'}$ is defined by
	$\mu = \mu_0 + \mu_1 + \mu_0^b$, where
	\begin{align*}
	\mu_0 = &\ \id \\
	\mu_1(q)X = &\ \nabla'_X q - \nabla_X q \\
	\mu_0^b(\beta)X = &\ \nabla'_X \beta - \nabla_X \beta,
	\end{align*}
	for $X\in\mathfrak{X}(M)$, $q\in\Gamma(Q)$ and $\beta\in\Gamma(B^*)$.
	The equations in Proposition \ref{morphism_of_3-term_representations}
	are automatically satisfied since by construction
	\[
	\D_{\ad_{\nabla'}}\circ\mu=\D_{\ad_{\nabla'}}\circ\mu_{\nabla'}^{-1}\circ\mu_\nabla=\mu_{\nabla'}^{-1}\circ\ldr{\Q}\circ\mu_\nabla=\mu_{\nabla'}^{-1}\circ\mu_\nabla\circ
	\D_{\ad_{\nabla}}=\mu\circ \D_{\ad_{\nabla}}.
	\]
	This yields the following result.
	
	\begin{proposition}\label{Isomorphism with change of connections}
		Given two pairs of $TM$-connections on the bundles $B^*$ and $Q$,
		the difference of the two $TM$-connections induces the isomorphism $\mu\colon\ad_\nabla\to \ad_{\nabla'}$ between the
		corresponding adjoint representations. Explicitly, it is given by
		$\mu=\id \oplus \Big( \nabla'-\nabla \Big)$.
	\end{proposition}

	The next step is to show how the adjoint representation transforms
	after a change of splitting of the Lie 2-algebroid. Fix a Lie
	2-algebroid $(\M,Q)$ over the smooth manifold $M$ and choose a
	splitting $Q[1]\oplus B^*[2]$, with structure objects
	$(\ell,\rho,[\cdot\,,\cdot]_1,\nabla^1,\omega^1)$ as before. Recall
	that a change of splitting does not change the vector bundles $B^*$
	and $Q$, and it is equivalent to a section
	$\sigma\in\Omega^2(Q,B^*)$. The induced isomorphism of [2]-manifolds
	over the identity on $M$ is given by:
	$\mathcal{F}_\sigma^\star(\tau) = \tau$ for all $\tau\in\Gamma(Q^*)$
	and
	$\mathcal{F}^\star_\sigma(b) = b + \sigma^\star
	b\in\Gamma(B)\oplus\Omega^2(Q)$ for all $b\in\Gamma(B)$. If
	$(\ell,\rho,[\cdot\,,\cdot]_2,\nabla^2,\omega^2)$ is the structure
	objects of the second splitting, then the compatibility of $\sigma$
	with the homological vector fields reads the following:
	\begin{itemize}
		\item The skew-symmetric dull brackets are related by $[q_1,q_2]_2 = [q_1,q_2]_1 - \ell(\sigma(q_1,q_2))$.
		\item The connections are related by
		$\nabla^2_q b = \nabla^1_q b + \partial_B\langle
		\sigma(q,\cdot),b \rangle$, or equivalently on the dual by
		$\nabla^{2*}_q \beta = \nabla^{1*}_q \beta -
		\sigma(q,\ell(\beta))$.
		\item The curvature terms are related by
		$\omega^2 = \omega^1 + \diff_{2,\nabla^1}\sigma$, where the
		operator
		\[
		\diff_{2,\nabla^1}\sigma\colon \Omega^\bullet(Q,B^*)\to\Omega^{\bullet+1}(Q,B^*)\]
		is defined by the usual Koszul formula using the dull bracket
		$[\cdot\,,\cdot]_2$ and the connection $\nabla^{1*}$.
	\end{itemize}
	The above equations give the following identities between the
	structure data for the adjoint representations\footnote{Note that the
		two pairs of $TM$-connections are identical} $\ad_\nabla^1$ and
	$\ad_\nabla^2$.

	\begin{lemma}\label{Identities_for_different_splitting_of_Lie_2-algebroid}
		Let $q,q_1,q_2\in\Gamma(Q),\beta\in\Gamma(B^*)$ and
		$X\in\mathfrak{X}(M)$. Then
		\begin{enumerate}[(i)]
			\item $\ell_2 = \ell_1$ and $\rho_2 = \rho_1$.
			\item $\nabla^{2,\text{bas}}_{q_1} q_2 = \nabla^{1,\text{bas}}_{q_1} q_2 - \ell(\sigma(q_1,q_2))$
			
			$\nabla^{2,\text{bas}}_{q} X = \nabla^{1,\text{bas}}_{q} X$
			
			$\nabla^{2,*}_{q} \beta = \nabla^{1,*}_{q} \beta - \sigma(q,\ell(\beta))$.
			\item $\omega_2^2(q_1,q_2)q_3 = \omega_2^1(q_1,q_2)q_3 + \diff_{2,\nabla^1}\sigma(q_1,q_2,q_3)$\\
			$\omega_2^2(q_1,q_2)X = \omega_2^1(q_1,q_2)X +
			\nabla_X(\ell(\sigma(q_1,q_2))) -
			\ell(\sigma(q_1,\nabla_X q_2)) +
			\ell(\sigma(q_2,\nabla_X q_1))$.
			\item $\omega_3^2(q_1,q_2,q_3)X = \omega_3^1(q_1,q_2,q_3)X +  (\nabla_X(\diff_{2,\nabla^1}\sigma))(q_1,q_2,q_3)$.
			\item $\phi_0^2(\beta)q = \phi_0^1(\beta)q + \sigma(q,\ell(\beta))$\\
			$\phi_0^2(\beta)X = \phi_0^1(\beta)X$.
			\item $\phi_1^2(\beta,q)X = \phi_1^1(\beta,q)X -
			\sigma(\nabla_X q_1,\ell(\beta)) -
			\sigma(q,\ell(\nabla_X \beta)) +
			\nabla_X(\sigma(q,\ell(\beta)))$.
		\end{enumerate}
	\end{lemma}
	
	Consider now two Lie $n$-algebroids $\M_1$ and $\M_2$ over $M$, and
	an isomorphism
	\[
	\mathcal{F}\colon(\M_1,\Q_1)\to(\M_2,\Q_2)
	\]
	given by the maps
	$\mathcal{F}_Q\colon Q_1\to Q_2$,
	$\mathcal{F}_B\colon B_1^*\to B^*_2$, and
	$\mathcal{F}_0\colon\wedge^2Q_1\to B_2^*$.
	Recall that a 0-morphism between two representations up to homotopy
	$(\E_1,\D_1)$ and $(\E_2,\D_2)$ of $\M_1$ and $\M_2$, respectively, is
	given by a degree 0 map
	\[
	\mu\colon \cin(\M_2)\otimes\Gamma(\E_2)\to \cin(\M_1)\otimes\Gamma(\E_1),
	\]
	which is $\cin(\M_2)$-linear:
	$\mu(\xi\otimes e) = \mathcal{F}^\star\xi\otimes\mu(e)$ for all
	$\xi\in \cin(\M_2)$ and $e\in\Gamma(\E_2)$, and makes the following
	diagram commute
	\begin{center}
	\begin{tikzcd}
		\cin(\M_2)\otimes\Gamma(\E_2)\arrow[rr,"\mu"]\arrow[dd,"\D_2"'] & & \cin(\M_1)\otimes\Gamma(\E_1)\arrow[dd,"\D_1"] \\
		& & \\
		\cin(\M_2)\otimes\Gamma(\E_2)\arrow[rr,"\mu"'] & &
		\cin(\M_1)\otimes\Gamma(\E_1).
	\end{tikzcd}
    \end{center}
	The usual analysis as before implies that $\mu$ must be given by a
	morphism of complexes $\mu_0\colon (\E_2,\partial_2)\to (\E_1,\partial_1)$ and
	elements
	\[
	\mu_1\in\Omega^1(Q_1,\underline{\Hom}^{-1}(\E_2,\E_1)),
	\]
	\[
	\mu_2\in\Omega^2(Q_1,\underline{\Hom}^{-2}(\E_2,\E_1)),
	\]
	\[
	\mu_0^b\in \Gamma(B)\otimes\Gamma(\underline{\Hom}^{-2}(\E_2,\E_1)),
	\]
	which satisfy equations similar to the set of equations in Proposition \ref{morphism_of_3-term_representations}.
	
	A change of splitting of the Lie 2-algebroid transforms as follows the
	adjoint representation.  Since changes of choices of connections are
	now fully understood, choose the same connection for both splittings
	$\M_1\simeq Q[1]\oplus B^*[2]\simeq\M_2$. Suppose that $\sigma\in\Omega^2(Q,B^*)$ is the
	change of splitting and denote by $\mathcal{F}_\sigma$ the induced
	isomorphism of the split Lie 2-algebroids whose components are given
	by
	$\mathcal{F}^\star_{\sigma,Q}=\id_{Q^*}, \mathcal{F}^\star_{\sigma,B}=\id_B,
	\mathcal{F}^\star_{\sigma,0}=\sigma^\star$. The composition map $\mu^\sigma:\ad_\nabla^1\to\mathfrak{X}(\M)\to\ad_{\nabla}^2$ is given in components by
	\begin{align*}
	\mu_0^\sigma = &\ \id \\
	\mu_1^\sigma(q_1)q_2 = &\ \sigma(q_1,q_2) \\
	\mu_2^\sigma(q_1,q_2)X = &\ (\nabla_X \sigma)(q_1,q_2).
	\end{align*}
	A similar argument as before implies that $\mu^\sigma$ is a morphism between the two adjoint
	representations and therefore the following result follows.
	
	\begin{proposition}\label{Isomorphism with change of splitting}
		Given two splittings of a Lie 2-algebroid with induced change of
		splitting $\sigma\in\Omega^2(Q,B^*)$ and a pair of $TM$-connections
		on the vector bundles $B^*$ and $Q$, the corresponding adjoint representations are isomorphic via
		$\mu=\id\oplus\ \sigma\oplus\nabla_\cdot\sigma$.
	\end{proposition}
	
	\section{Adjoint representation of a Lie $n$-algebroid}\label{Adjoint representation of a Lie n-algebroid}
	
	The construction of the adjoint representation up to homotopy of a Lie
	$n$-algebroid $(\M,\Q)$ for general $n\in\mathbb{N}$ is similar to the case $n=2$. Specifically, choose a splitting $\M\simeq \bigoplus_{i=1}^n A_i[i]$
	and $TM$-connections $\nabla^{A_i}$ on the vector bundles $A_i\to M$ for all $i=1,\ldots,n$. Then there is an
	induced isomorphism of $\cin(\M)$-modules
	\begin{align*}
	\mu_\nabla\colon \cin(\M)\otimes\Gamma(TM[0]\oplus A_1[1]\oplus\ldots\oplus A_n[n]) & \to \mathfrak{X}(\M),
	\end{align*}
	which at the level of generators is given by
	\begin{align*}
	\Gamma(A_i)\ni a & \mapsto \hat{a} \quad \text{and} \quad \mathfrak{X}(M)\ni X \mapsto \nabla^{A_n}_X \oplus \ldots \oplus \nabla^{A_1}_X.
	\end{align*}
	The isomorphism  $\mu$ is used to transfer the differential $\ldr{\Q}$ of the adjoint module $\mathfrak{X}(\M)$
	to the differential
	$\D_{\ad_\nabla} := \mu^{-1}_\nabla\circ\ldr{\Q}\circ\mu_\nabla$ on
	$\cin(\M)\otimes\Gamma(TM[0]\oplus A_1[1]\oplus\ldots\oplus A_n[n])$. We obtain thus the \textbf{adjoint representation of a Lie $n$-algebroid}\index{split Lie $n$-algebroid!adjoint representation}\index{adjoint representation} as the equivalence class of $(\ad_\nabla,\D_{\ad_{\nabla}})$ and the \textbf{coadjoint representation}\index{split Lie $n$-algebroid!adjoint representation}\index{coadjoint representation} as the equivalence class of its dual $(\ad_{\nabla}^*,\D_{\ad_{\nabla}}^*)$. Moreover, a formula similar to the transformation of the adjoint representation of a Lie $2$-algebroid from Proposition \ref{Isomorphism with change of connections} holds also for general Lie $n$-algebroids, $n\in\mathbb{N}$, equipped with two choices of $TM$-connections on the vector bundles $A_i\to M$, $i=1,\ldots,n$. This is proved in Proposition \ref{Isomorphism of A-representations for decompositions of VB-Lie n-algebroids} and explained in Remark \ref{Transformation of adjoint representation of split Lie n-algebroid}.
	
	\section{The Weil algebra of a Lie $n$-algebroid}\label{Section: The Weil algebra of a Lie n-algebroid}
	
	In this section, we analyse in detail the Weil algebra of a Lie $n$-algebroid\index{split Lie $n$-algebroid!Weil algebra} $(\M,\Q)$ using the adjoint representation from the previous section. Suppose $U\subset M$ is a coordinate chart of $\M$ with graded coordinates $\xi_i^1,\ldots,\xi_i^{r_i}$ of degree $i$, for all $i=1,\ldots,n$. Recall from Section \ref{Section: (Pseudo)differential forms and the Weil algebra of a Q-manifold}, that the tangent bundle $T\M$
	of $\mathcal M$ is an $[n]$-manifold over $TM$ \cite{Mehta06,Mehta09},
	whose local generators over $TU\subset TM$ are given in degree zero by $\cin_{TU}(T\M)^0 = C^\infty(TU)$ and in degree $i$ by $\xi_i^1,\ldots,\xi_i^{r_i},\diff\xi_i^1,\ldots,\diff\xi_i^{r_i}$. The shifted tangent bundle $T[1]\M$
	is an $[n+1]$-manifold over $M$, with the following local generators over $U$:
	\begin{itemize}
		\item \textbf{degree 0:} $C^\infty(U)$
		
		\item \textbf{degree 1:} $\xi_1^1,\ldots,\xi_1^{r_1}, \Omega^1(U)$
		
		\item \textbf{degree 2:} $\xi_2^1,\ldots,\xi_2^{r_2}, \diff\xi_1^1,\ldots,\diff\xi_1^{r_1}$
		
		\item[]  $\qquad \vdots$
		
		\item \textbf{degree $n$:} $\xi_n^1,\ldots,\xi_n^{r_n}, \diff\xi_{n-1}^1,\ldots,\diff\xi_{n-1}^{r_{{n-1}}}$
		
		\item \textbf{degree $n+1$:} $\diff\xi_n^1,\ldots,\diff\xi_n^{r_n}$.
	\end{itemize}
	
	\noindent
	In other words,
	the structure sheaf of $T[1]\M$ assigns to every coordinate domain
	$(U,x^1,\ldots,x^m)$ of $M$ that trivialises $\mathcal M$, the space
	\[
	\cin_U(T[1]\M) =
	\bigoplus_i\underset{(0,i)}{\underbrace{\cin_U(\M)^i}}\left<
	\underset{(1,0)}{\underbrace{(\diff x^k)_{k=1}^m}},
	\underset{(1,1)}{\underbrace{(\diff\xi_1^k)_{k=1}^{r_1}}},\ldots,
	\underset{(1,n)}{\underbrace{(\diff\xi_n^k)_{k=1}^{r_n}}}
	\right>,
	\]
	as in Section \ref{Section: (Pseudo)differential forms and the Weil algebra of a Q-manifold}.
	
	In the case of a Lie 1-algebroid $A\to M$, this is the Weil algebra
	from \cite{Mehta06,Mehta09}. For the 1-algebroid case, see also
	\cite{ArCr12} for an approach with without the language of supergeometry. Here we extend the approach of \cite{ArCr12} in the case of split Lie $n$-algebroids for general $n$.
	
	Suppose first that $\M = Q[1]\oplus B^*[2]$ is a split Lie 2-algebroid
	and consider two $TM$-connections on the vector bundles $Q$ and $B^*$,
	both denoted by $\nabla$. Recall from Section
	\ref{adjoint_module_adjoint_representation_isomorphism} the
	(non-canonical) isomorphism of DG $\M$-modules
	\[
	\mathfrak{X}(\M)\cong\cin(\M)\otimes\Gamma(TM[0]\oplus Q[1]\oplus B^*[2]).
	\]
	This implies that
	\[
	\Omega^1(\M)\cong\cin(\M)\otimes\Gamma(B[-2]\oplus Q^*[-1]\oplus T^*M[0])
	\]
	as (left) DG $\M$-modules, and thus the generators of the Weil algebra can be
	identified with
	\[
	\underset{(0,t)}{\underbrace{\cin(\M)^t}}, 
	\underset{(u,0)}{\underbrace{\Gamma(\wedge^uT^*M)}},
	\underset{(\upsilon,\upsilon)}{\underbrace{\Gamma(S^\upsilon Q^*)}}, 
	\underset{(w,2w)}{\underbrace{\Gamma(\wedge^w B)}}.
	\]
	Using also that
	$\cin(\M)^t=\bigoplus_{t=r+2s} \Gamma(\wedge^rQ^*)\otimes\Gamma(S^s
	B)$, the space of $(p,q)$-forms is decomposed as
	\begin{align*}
	W^{p,q}(\M,\nabla) = & \bigoplus_{\substack{q=t+v+2w \\ p=u+w+v}} \cin(\M)^t\otimes
	\Gamma\left( \wedge^uT^*M\otimes S^vQ^*\otimes \wedge^wB \right) \\
	= & \bigoplus_{\substack{q=r+2s+v+2w \\ p=u+w+v}} \Gamma\left( \wedge^uT^*M\otimes
	\wedge^rQ^*\otimes S^vQ^*\otimes \wedge^wB\otimes S^sB \right).
	\end{align*}
	Therefore, after a choice of splitting and $TM$-connections $\nabla$
	on $Q$ and $B^*$, the total space of the Weil algebra of $\M$ can be
	written as
	
	\[
	W(\M,\nabla) = \bigoplus_{r,s,u,v,w} \Gamma\left(
	\wedge^uT^*M\otimes \wedge^rQ^*\otimes S^vQ^*\otimes
	\wedge^wB\otimes S^sB \right).
	\]
	
	The next step is to express the differentials $\ldr{\Q}$ and $\dr$ on
	$W(\M,\nabla)$ in terms of the two $TM$-connections $\nabla$. For the
	vertical differential, recall that by definition the $p$\,-th column of
	the double complex $W(\M,\nabla)$ equals the space of $p$-forms
	$\Omega^p(\M)$ on $\M$ with differential given by the Lie derivative
	$\ldr{\Q}$. Due to the identification of DG $\M$-modules
	\[
	\Omega^p(\M)=\Omega^1(\M)\wedge\ldots\wedge\Omega^1(\M)
	=\cin(\M)\otimes\Gamma(\ad_\nabla^*\wedge\ldots\wedge\ad_\nabla^*)
	\]
	($p$-times) and the Leibniz identity for $\ldr{\Q}$, it follows that
	the $p$-th column of $W(\M,\nabla)$ becomes the $p$-symmetric power of
	the coadjoint representation $\underline{S}^p(\ad_\nabla^*)$ and
	$\ldr{\Q}=\D_{\underline{S}^p(\ad_\nabla^*)}$.
	
	The horizontal differential $\dr$ is built from two 2-representations of
	the tangent Lie algebroid $TM$, namely the dualisation of the $TM$-representations on the graded vector bundles
	$\E_{Q}=Q[0]\oplus Q[-1]$ and $\E_{B^*}=B^*[0]\oplus B^*[-1]$
	whose
	differentials are given by the chosen $TM$-connections
	$(\id_Q,\nabla,R_\nabla)$ and $(\id_{B^*},\nabla,R_\nabla)$,
	respectively. Indeed, suppose first that $\tau\in\Gamma(Q^*)$ and
	$b\in\Gamma(B)$ are functions on $\M$, i.e.~$0$-forms. Then from Remark \ref{Iso_coad_mod_coad_rep}, it follows
	that $\dr$ acts via
	\begin{equation}\label{d acting on functions}
		\dr\tau = \tau + \diff_{\nabla^*}\tau\qquad \text{and}\qquad \dr b
		= b + \diff_{\nabla^*}b.
	\end{equation}
	If now $\tau\in\Gamma(Q^*),b\in\Gamma(B)$ are 1-forms on $\M$, then
	\[
	\dr\tau=\dr(\tau+\diff_{\nabla^*}\tau-\diff_{\nabla^*}\tau)
	=\dr^2\tau-\dr(\diff_{\nabla^*}\tau)=\diff_{\nabla^*}\tau-\diff_{\nabla^*}^2\tau,
	\]
	\[
	\dr b=\dr(b+\diff_{\nabla^*}b-\diff_{\nabla^*}b)=\dr^2b-\dr(\diff_{\nabla^*}b)=\diff_{\nabla^*}b-\diff_{\nabla^*}^2b,
	\]
	where in both lines the second equation uses the formulae in (\ref{d acting on functions}) and the last equation can be seen by a direct computation, e.g., in local coordinates.

	\begin{remark}
		Note that if $B^*=0$, i.e.~$\M$ is an ordinary Lie algebroid
		$A\to M$, the above construction recovers (up to isomorphism) the
		connection version of the Weil algebra $W(A,\nabla)$ from
		\cite{ArCr11,ArCr12,Mehta09}.
	\end{remark}
	
	In the general case of a split Lie $n$-algebroid
	$\M=A_1[1]\oplus\ldots\oplus A_n[n]$ with a choice of $TM$-connections
	on all the bundles $A_i$, one may apply the same procedure as above to
	obtain the (non-canonical) DG $\M$-module isomorphisms
	\[
	\mathfrak{X}(\M)\cong\cin(\M)\otimes\Gamma(TM[0]\oplus A_1[1]\oplus\ldots\oplus A_n[n])
	\]
	\[
	\Omega^1(\M)\cong\cin(\M)\otimes\Gamma(A_n^*[-n]\oplus\ldots\oplus A_1^*[-1]\oplus T^*M[0]),
	\]
	and hence the identification of the generators of the Weil algebra
	with
	
	\[
	\underset{(0,t)}{\underbrace{\cin(\M)^t}}, 
	\underset{(u,0)}{\underbrace{\Gamma(\wedge^{u}T^*M)}},
	\underset{(\upsilon_1,\upsilon_1)}{\underbrace{\Gamma(S^{\upsilon_1} A_1^*)}},
	\underset{(\upsilon_2,2\upsilon_2)}{\underbrace{\Gamma(\wedge^{\upsilon_2} A_2^*)}},\ldots, 
	\underset{(\upsilon_n,n\upsilon_n)}{\underbrace{\Gamma(\wedge^{\upsilon_n} A_n^*)}}.
	\]
	
	\noindent
	This then yields
	\begin{align*}
	W^{p,q}(\M,\nabla) = & \bigoplus_{\substack{q=t+v_1+2v_2+\ldots \\ p=u+v_1+v_2+\ldots}} \cin(\M)^t\otimes\Gamma\left( \wedge^uT^*M\otimes S^{v_1}A_1^*\otimes \wedge^{v_2}A_2^*\otimes\ldots \right) \\
	= & \bigoplus_{\substack{q=r_1+v_1+2r_2+2v_2+\ldots \\ p=u+v_1+v_2+\ldots}} \Gamma\left( \wedge^uT^*M\otimes \wedge^{r_1}A_1^*\otimes S^{v_1}A_1^*\otimes S^{r_2}A_2^*\otimes \wedge^{v_2}A_2^*\otimes\ldots \right).
	\end{align*}
	Similar considerations as before imply that the $p$\,-th column of
	$W(\M,\nabla)$ is given by $\underline{S}^p(\ad_\nabla^*)$ with
	$\ldr{\Q}=\D_{\underline{S}^p(\ad_\nabla^*)}$, and that $\dr$ is built
	again by the dualisation of the 2-representations of $TM$ on the
	graded vector bundles $\underline{E}_{A_i}=A_i[0]\oplus A_i[-1]$, for
	$i=1,\ldots,n$, whose differentials are given by
	$(\id_{A_i},\nabla,R_{\nabla})$.
	
	\section{Poisson Lie algebroids of low degree}\label{Section: Poisson Lie algebroids of low degree}
	
	This section describes in detail the degree $-n$ (anti-)morphism
	$\sharp\colon\ad_\nabla^*\to\ad_\nabla$ of (left) right $n$-representations in the case of Poisson Lie
	$n$-algebroids with a degree $-n$ bracket for $n=0,1,2$. Recall that the map $\sharp$ sends an
	exact $1$-form $\diff\xi$ of the graded manifold $\M$ to the vector
	field $\{\xi,\cdot\}$.
	
	First, consider a Poisson Lie 0-algebroid, i.e.~a usual Poisson
	manifold $(M,\{\cdot\,,\cdot\})$. Then the Lie $0$-algebroid is just
	$M$, with a trivial homological vector field -- it can be thought of as a
	trivial Lie algebroid $A=0\times M\to M$ with trivial differential $\diff_A=0$, and consequently trivial homological vector field.  The
	coadjoint and adjoint representations are just the vector bundles
	$T^*M[0]$ and $TM[0]$, respectively, with zero module differentials,
	and the map $\sharp$ simply becomes the usual vector bundle map
	induced by the Poisson bivector field that corresponds to the Poisson
	bracket
	\[
	\sharp\colon T^*M[0] \to TM[0].
	\]
	
	Consider a Lie algebroid $A\to M$ with anchor $\rho\colon A\to TM$ and
	a linear Poisson structure $\{\cdot\,,\cdot\}$, i.e.~a Lie algebroid
	structure on the dual $A^*\to M$. As we explained before, the latter is equivalent to a Poisson bracket of degree $-1$ on the $[1]$-manifold $A[1]$ and $(A[1],\{\cdot\,,\cdot\})$ is a Poisson Lie algebroid if and only if $(A,A^*)$ is a Lie bialgebroid \cite{MaXu00}. Recall also that the correspondence is given by $\rho'\colon A^*\to TM$,
	$\alpha\mapsto \{\alpha,\cdot\}|_{C^\infty(M)}$ and
	$[\cdot\,,\cdot]_*:=\{\cdot\,,\cdot\}|_{\Omega^1(A)\times\Omega^1(A)}$. After a choice of a
	$TM$-connection $\nabla$ on the vector bundle $A$, the map
	$\sharp\colon \ad_\nabla^*\to\ad_\nabla$ acts via
	\[
	\sharp(\diff f)=\sharp_0(\diff f)= \{f,\cdot\}=-\rho'^*(\diff f)\in\Gamma(A)
	\]
	\[
	\sharp(\beta)=\sharp_0(\beta)+\sharp_1(\cdot)\beta\in\mathfrak{X}(M)\oplus(\Gamma(A^*)\otimes\Gamma(A))
	\]
	for all $f\in C^\infty(M),\beta\in\Gamma(A)$, where we consider
	$\beta$ as an \textit{1-form} on $A[1]$ and identify it with $\dr\beta-\diff_{\nabla^*}\beta\in\Omega^1(A[1])$ as in Remark \ref{Iso_coad_mod_coad_rep}, i.e.~$\sharp(\beta) = \sharp(\dr\beta) - \sharp(\diff_{\nabla^*}\beta) =
	\{\beta,\cdot\}-\sharp(\diff_{\nabla^*}\beta)$. Computing how these
	act on $\alpha\in\Omega^1(M)$ and $g\in C^\infty(M)$, viewed as
	functions of the graded manifold $A[1]$, gives the components of $\sharp(\beta)$: From the right-hand-side of the equation we obtain
	\[
	\left(\sharp_0(\beta)+\sharp_1(\cdot)\beta\right)g = \sharp_0(\beta)g\in C^\infty(M)
	\]
	while from the left-hand-side we obtain
	\[
	\sharp(\beta)g = \sharp(\dr \beta - \diff_{\nabla^*}\beta)g  = 
	\{\beta,g\} - \sharp(\diff_{\nabla^*}\beta)g =  \rho'(\beta)g.
	\]
	From this, it follows that $\sharp_0(\beta) = \rho'(\beta)$. Using now this, the right-hand-side gives
	\[
	\left(\sharp_0(\beta)+\sharp_1(\cdot)\beta\right)\alpha =  \nabla^*_{\rho'(\beta)}\alpha + \left(\sharp_1(\cdot)\beta\right)\alpha\in\Gamma(A^*)\oplus\Gamma(A^*)
	\]
	while the left-hand-side gives
	\[
	\sharp(\beta)\alpha = \sharp(\dr \beta - \diff_{\nabla^*}\beta)\alpha  = 
	\{\beta,\alpha\} - \sharp(\diff_{\nabla^*}\beta)\alpha =
	[\beta,\alpha]_* + \nabla^*_{\rho'(\alpha)}\beta = (\nabla^*)^{\text{bas}}_\beta\alpha.
	\]
	This implies that $(\sharp_1(\cdot)\beta)\alpha = (\nabla^*)^{\text{bas}}_\beta\alpha - \nabla^*_{\rho'(\alpha)}\beta$ and thus $\sharp$ consists
	of the ($-1$)-chain map $\sharp_0$ given by the anti-commutative diagram
	\[
	\begin{tikzcd}
	T^*M[0] \arrow[rr, "-\rho^*"] \arrow[dd, "-\rho'^*"'] & & A^*[-1] \arrow[dd, "\rho'"] \\
	& & \\
	A[1] \arrow[rr, "\rho"']           &                 & TM  [0]                  
	\end{tikzcd}
	\]
	together with
	$\sharp_1(a)\beta = \langle (\nabla^*)^{\text{bas}}_\beta(\cdot) -
	\nabla^*_{\rho'(\beta)}(\cdot),a
	\rangle\in\Gamma(A^{**})\simeq\Gamma(A)$, for all
	$\beta\in\Gamma(A^*)$ and $a\in\Gamma(A)$.
	
	By Theorem
	\ref{thm_poisson}, $\sharp$ is an (anti-)morphism of $2$-representations if
	and only $(A[1],\diff_A,\{\cdot\,,\cdot\})$ is a Poisson Lie
	$1$-algebroid. Hence, $\sharp$ is an (anti-)morphism of $2$-representations if
	and only if $(A,A^*)$ is a Lie bialgebroid.  Similarly,
	\cite{GrJoMaMe18} shows that $\ad_\nabla^*$ and $\ad_\nabla$ form a
	\emph{matched pair} if and only if $(A,A^*)$ is a Lie bialgebroid.
	
	Note that $(A,\{\cdot\,,\cdot\})$ is a Poisson Lie algebroid if the
	induced vector bundle morphism $\sharp\colon T^* A\to TA$ over $A$ is
	a VB-algebroid morphism over $\rho'\colon A^*\to TM$
	\cite{MaXu00}. Then the fact that
	$\sharp\colon \ad_\nabla^*\to\ad_\nabla$ is an (anti-)morphism of
	$2$-representations follows immediately \cite{DrJoOr15}, since
	$\ad_\nabla^*$ and $\ad_\nabla$ are equivalent to decompositions of
	the VB-algebroids $(T^*A\to A^*, A\to M)$ and $(TA\to TM, A\to M)$,
	respectively.
	
	Now we consider the case of 2-algebroids. First recall that a
	symplectic Lie 2-algebroid over a point, that is, a Courant algebroid
	over a point, is a usual Lie algebra $(\mathfrak{g},[\cdot\,,\cdot])$
	together with a non-degenerate pairing
	$\langle \cdot\,,\cdot \rangle\colon
	\mathfrak{g}\times\mathfrak{g}\to\mathfrak{g}$, such that
	\[
	\langle [x,y],z \rangle + \langle y,[x,z] \rangle = 0
	\]
	for all $x,y,z\in\mathfrak{g}$. Using the adjoint and coadjoint representations
	$\ad\colon\mathfrak g\to \End(\mathfrak g)$, $x\mapsto [x,\cdot]$, and
	$\ad^*\colon \mathfrak g\to\End(\mathfrak g^*)$, $x\mapsto -\ad(x)^*$, and
	denoting the canonical linear isomorphism induced by the pairing by
	$P\colon \mathfrak{g}\to\mathfrak{g^*}$, the equation above reads
	\[
	P(\ad(x)y) = \ad^*(x)P(y)
	\]
	for all $x,y\in\mathfrak{g}$. In other words, this condition is precisely what is needed to turn the
	vector space isomorphism $P$ into an isomorphism of Lie algebra
	representations between $\ad$ and $\ad^*$. In fact, the map of representations up to homotopy
	$\sharp\colon \ad^*\to\ad$
	for Poisson Lie 2-algebroids is
	a direct generalisation of this construction.
	
	Let $B\to M$ be a usual Lie algebroid with a 2-term representation
	$(\nabla^Q,\nabla^{Q^*},R)$ on a complex $\partial_Q\colon Q^*\to
	Q$. The representation is called \textbf{self dual}\index{Lie algebroid!self-dual representation} \cite{Jotz18b} if
	it equals its dual, i.e.~$\partial_Q=\partial_Q^*$, the connections
	$\nabla^Q$ and $\nabla^{Q^*}$ are dual to each other, and
	$R^*=-R\in\Omega^2(B,\Hom(Q,Q^*))$,
	i.e.~$R\in\Omega^2(B,\wedge^2Q^*)$.  \cite{Jotz18b} further shows that
	Poisson brackets $\{\cdot\,,\cdot\}$ on a split Lie 2-algebroid
	$Q[1]\oplus B^*[2]$ correspond to self dual 2-representations of $B$
	on $Q^*[1]\oplus Q[0]$ as follows: the bundle map $\partial_Q\colon
	Q^*\to Q$ is
	$\tau\mapsto\{ \tau,\cdot \}|_{\Omega^1(Q)}$, the anchor $\rho_B\colon B\to TM$ is
	$b\mapsto\{ b,\cdot \}|_{C^\infty(M)}$, the $B$-connection on $Q^*$ is given by
	$\nabla^{Q^*}_b\tau=\{b,\tau\}$, and the 2-form $R$ and the Lie bracket of
	$B$ are defined by
	$\{b_1,b_2\} = [b_1,b_2] - R(b_1,b_2)\in\Gamma(B)\oplus\Omega^2(Q)$.
	
	Fix now a Poisson Lie 2-algebroid $(\M,\Q,\{\cdot\,,\cdot\})$ together
	with a choice of a splitting $Q[1]\oplus B^*[2]$ for $\M$, a pair of
	$TM$-connections on $B^*$ and $Q$, and consider the representations
	$\ad_\nabla$ and $\ad_\nabla^*$. Similarly as before, we have that
	\[
	\sharp(\diff f) = \sharp_0(\diff f) = \{f,\cdot\} = -\rho_B^*(\diff f)\in\Gamma(B^*)
	\]
	\[
	\sharp(\tau) = \sharp_0(\tau) +
	\sharp_1(\cdot)\tau\in\Gamma(Q)\oplus(\Omega^1(Q)\otimes\Gamma(B^*))
	\]
	\[
	\sharp(b) = \sharp_0(b) + \sharp_1(\cdot)b + \sharp_2(\cdot\,,\cdot)b +
	\sharp^b(\cdot)b\in\mathfrak{X}(M)\oplus\Omega^1(Q,Q)\oplus\Omega^2(Q,B^*)\oplus(\Gamma(B)\otimes\Gamma(B^*))
	\]
	for $f\in C^\infty(M),\tau\in\Gamma(Q^*),b\in\Gamma(B)$, where we
	identify again $\tau$ with $\dr\tau-\diff_{\nabla^*}\tau$ and $b$ with
	$\dr b-\diff_{\nabla^*}b$ as in Remark \ref{Iso_coad_mod_coad_rep}. Then the map
	$\sharp\colon \ad^*_\nabla\to\ad_\nabla$ consists of the ($-2$)-chain
	map given by the anti-commutative diagram
	
	\[
	\begin{tikzcd}
	T^*M[0] \arrow[rr, "-\rho_Q^*"] \arrow[dd, "-\rho_B^*"'] & & Q^*[-1]
	\arrow[rr, "-\partial_B"] \arrow[dd, "\partial_Q"]
	& & B[-2] \arrow[dd, "\rho_B"] \\
	& & & & \\
	B^*[2] \arrow[rr, "-\partial_B^*"'] & & Q[1] \arrow[rr, "\rho_Q"'] & & TM[0]
	\end{tikzcd}
	\]
	and the elements
	\[
	\sharp_1(q)\tau = \langle \tau, \nabla^Q_\cdot q - \nabla_{\rho_B(\cdot)}q \rangle\in\Gamma(B^*)
	\]
	\[
	\sharp_1(q)b = \nabla_b^Q q - \nabla_{\rho_B(b)}q\in\Gamma(Q)
	\]
	for $q\in\Gamma(Q),\tau\in\Gamma(Q^*),b\in\Gamma(B)$,
	\[
	\sharp_2(q_1,q_2)b = - \langle R(b,\cdot)q_1,q_2 \rangle\in\Gamma(B^*)
	\]
	for $q_1,q_1\in\Gamma(Q),b\in\Gamma(B)$, where $R$ is the
	component that comes from the self-dual 2-representation of
	$B$ from the Poisson structure,
	\[
	\sharp^b(\beta)b = \langle \beta,\nabla^{\text{bas}}_b(\cdot) -
	\nabla^*_{\rho_B(b)}(\cdot) \rangle \in\Gamma(B^*)
	\]
	for $\beta\in\Gamma(B^*),b\in\Gamma(B)$. 
	
	Suppose now that the split Lie 2-algebroid is symplectic, i.e.~that it
	is of the form $E[1]\oplus T^*M[2]$ for a Courant algebroid $E\to
	M$. The only thing that is left from the construction in the Example
	\ref{Split_symplectic_Lie_2-algebroid_example} is a choice of a
	$TM$-connection on $TM$, and hence on the dual $T^*M$. The (anti-)isomorphism
	$\sharp\colon \ad_\nabla^*\to\ad_\nabla$ consists of the (-2)-chain
	map of the anti-commutative diagram
	\[
	\begin{tikzcd}
	T^*M[0] \arrow[rr, "-\rho^*"] \arrow[dd, "-\id"'] & & E^*[-1] \arrow[rr, "-\rho"] \arrow[dd, "P^{-1}"] & & TM[-2] \arrow[dd, "\id"] \\
	& & & & \\
	T^*M[2] \arrow[rr, "-\rho^*"'] & & E[1] \arrow[rr, "\rho"'] & & TM [0]
	\end{tikzcd}
	\]
	where $P\colon E\overset{\sim}{\to} E^*$ is the pairing, and the elements
	$\langle \sharp_2(e_1,e_2)X,Y \rangle = - \langle
	R_\nabla(X,Y)e_1,e_2 \rangle$ and
	$\langle \sharp^b(\alpha)X, Y \rangle = - \langle
	\alpha,T_\nabla(X,Y) \rangle$. Its inverse consists of the
	2-chain map given by the anti-commutative diagram
	\[
	\begin{tikzcd}
	T^*M[2] \arrow[rr, "-\rho^*"] \arrow[dd, "-\id"'] & & E[1] \arrow[rr, "\rho"] \arrow[dd, "P"] & & TM[0] \arrow[dd, "\id"] \\
	&&&& \\
	T^*M[0] \arrow[rr, "-\rho^*"']         &          & E^*[-1] \arrow[rr, "-\rho"']          &            & TM  [-2]              
	\end{tikzcd}
	\]
	and the elements
	$\langle \sharp^{-1}_2(e_1,e_2)X,Y \rangle = - \langle
	R_\nabla(X,Y)e_1,e_2 \rangle$ and
	$\langle (\sharp^{-1})^b(\alpha)X, Y \rangle = - \langle
	\alpha,T_\nabla(X,Y) \rangle$. In other words, $\sharp^2=\id$. If the
	connection on $TM$ is torsion-free, then the terms $\sharp^b$ and
	$(\sharp^{-1})^b$ vanish, as well. In particular, if the base manifold
	$M$ is just a point, then the bundles $TM$ and $T^*M$, and the
	elements $\sharp_2$ and $\sharp^{-1}_2$ are zero. Therefore, the map
	$\ad^*_\nabla\to\ad_\nabla$ reduces to the linear isomorphism of the pairing and agrees with the one explained above.
	
	\chapter{Linear structures on vector bundles}\label{Chapter: Linear structures on vector bundles}
	
	In this chapter, we consider the case where the structures defined before are are considered on vector bundles over graded manifolds and are in some sense ``linear". It puts the notions studied before in the context of linear structures on graded vector bundles over graded manifolds and compares with the existing literature.
	In what follows, suppose that $q:\mathcal{E}\to\M$ is a vector bundle in the category of $\mathbb{Z}$-graded manifolds.
	
	\section{Linear multivector fields on vector bundles}
	
	Recall from Section \ref{Section: Vector bundles over graded manifolds} the basic and linear functions of the graded manifold $\Em$; the former are elements of $q^\star(\cin(\M))$, while the latter are functions on $\Em$ which are linear in the fibre coordinates (i.e.~locally of the form $e^j\xi_j$, where $e^j$ are fibre coordinates of $\Em$ and $\xi_j$ are local functions on $\M$). Let $\widetilde{\X}$ be a degree $j$ vector field on the graded manifold $\Em$. Then $\widetilde{\X}$ is called \textbf{linear} if  the spaces of basic and linear functions are stable under its action, i.e.~if
	
	\[
	\widetilde{\X}(\cin_{\text{lin}}(\Em)) \subset \cin_{\text{lin}}(\Em)
	\qquad \text{and} \qquad 
	\widetilde{\X}(\cin_{\text{bas}}(\Em)) \subset \cin_{\text{bas}}(\Em).
	\]
	The graded subspace of \textbf{linear vector fields}\index{vector bundle over a graded manifold!linear vector field} of the vector bundle $\Em$ is denoted $\mathfrak{X}_\text{lin}(\Em)\subset\mathfrak{X}(\Em)$. It follows easily from the definition that $\mathfrak{X}_\text{lin}(\Em)$ is closed under the Lie bracket of vector fields, and hence, one obtains the graded Lie subalgebra $(\mathfrak{X}_\text{lin}(\Em),[\cdot\,,\cdot])$ of $(\mathfrak{X}(\Em),[\cdot\,,\cdot])$.
	
	Similarly to the ordinary case of vector bundles over smooth manifolds, it follows that a linear vector field $\widetilde{\X}$ induces a degree $j$ vector field on the base manifold $\X\colon\cin(\M)\to\cin(\M)$ and corresponds to a \textbf{degree $j$ derivation of $\Em^*$ (over $\X$): $\D^*\in\mathfrak{Der}^j(\Em^*)$}\index{vector bundle over a graded manifold!derivation}, i.e.~a degree $j$ linear map $\D^*\colon\Gamma(\Em^*)\to\Gamma(\Em^*)$ with the following property:
	\[
	\D^*(\xi\theta) = \X(\xi)\theta + (-1)^{j|\xi|}\xi \D^*(\theta),
	\]
	for all $\xi\in\cin(\M)$ and $\theta\in\Gamma(\Em^*)$. The correspondence is obtained via the formulae
	\[
	\widetilde{\X}(\ell_\theta) = \ell_{\D^*(\theta)}
	\qquad \text{and} \qquad
	\widetilde{\X}(q^\star(\xi)) = q^\star(\X(\xi))
	\] 
	for all $\xi\in\cin(\M)$ and $\theta\in\Gamma(\Em^*)$. A straightforward computation shows that given two linear vector fields $\widetilde{\X},\widetilde{\Y}\in\mathfrak{X}_\text{lin}(\Em)$,
	\[
	[\widetilde{\X},\widetilde{\Y}](\ell_\theta) = \ell_{[\D_X^*,\D_Y^*](\theta)}
	\qquad \text{and} \qquad
	[\widetilde{\X},\widetilde{\Y}](q^\star\xi) = q^\star\left( [\X,\Y](\xi) \right)
	\]
	for all $\xi\in\cin(\M)$ and $\theta\in\Gamma(\Em^*)$, and therefore, it follows that the map
	\[\mathfrak{X}_\text{lin}(\Em) \to \mathfrak{Der}(\Em^*):=\bigoplus_{j\in\mathbb{Z}}\mathfrak{Der}^j(\Em^*),\qquad (\widetilde{\X},\X)\mapsto \D_{(\widetilde{\X},\X)}
	\]
	is an isomorphism of graded Lie algebras, where the Lie bracket of $\mathfrak{Der}(\Em^*)$ is given by the graded commutator.
	
	A degree $j$ derivation $\D^*$ is called \textbf{flat} if $\D^*\circ\D^*=0$. Flat derivations of odd degree correspond to linear vector fields which square to zero. In particular, flat\index{vector bundle over a graded manifold!flat derivation} derivations of degree $1$ on $\Em^*$, i.e.~DG-module structures on $\Em^*$, correspond to linear homological vector fields on $\Em$. This will be analysed extensively in classical differential geometric terms in Chapter \ref{Chapter: Higher split VB-algebroid structures}.
	
	Locally, if $\{\xi^i\}$ are coordinates on $\M$ and $\{e^j\}$ are fibre coordinates of $\Em$ with dual coordinates $\{\theta_j\}$ on the fibres of $\Em^*$, then
	\[
	\widetilde{\X} = \zeta_i\frac{\partial}{\partial \xi^i} + e^k\eta_k^j\frac{\partial}{\partial e^j},
	\qquad
	\X = \zeta_i\frac{\partial}{\partial \xi^i},
	\qquad \text{and} \qquad
	\D_{(\widetilde{\X},\X)}(\theta_j) = e^k\eta_k^j
	\]
	where $e^k\eta_k^j\in\cin_{\text{lin}}(\Em)$ and $\eta_k^j,\zeta_i\in\cin(\M)$.
	
	Let now $k\in\mathbb{Z}$ and consider the space $\mathfrak{A}_k(\Em)$ together with its the Schouten bracket. 
	
	\begin{definition}
		An element $\widetilde{\X}\in\mathfrak{A}_k^{s,\bullet}(\Em)$ on a vector bundle $\Em$ is called \textbf{linear multivector field}\index{vector bundle over a graded manifold!linear multivector field} if the following conditions hold:
		\begin{enumerate}
			\item $[\cin_{\text{lin}}(\Em),[\ldots[\cin_{\text{lin}}(\Em),[\cin_{\text{lin}}(\Em),\widetilde{\X}]_k]_k\ldots]_k]_k \subset \cin_{\text{lin}}(\Em)$ ($s$ linear functions),
			\item $[\cin_{\text{lin}}(\Em),[\ldots[\cin_{\text{lin}}(\Em),[\cin_{\text{bas}}(\Em),\widetilde{\X}]_k]_k\ldots]_k]_k \subset \cin_{\text{bas}}(\Em)$ ($s-1$ linear and one basic functions),
			\item $[\cin_{\text{lin}}(\Em),[\ldots[\cin_{\text{bas}}(\Em),[\cin_{\text{bas}}(\Em),\widetilde{\X}]_k]_k\ldots]_k]_k = 0$ (two or more basic functions).
		\end{enumerate}
	\end{definition}
	
	The space of linear multivector fields on the vector bundle $\Em$ will be denoted $\mathfrak{A}_{\text{lin},k}^{\bullet,\bullet}(\Em)$. From the above definition, it follows that the local description of a linear multivector field $\widetilde{\X}\in\mathfrak{A}_{\text{lin},k}^{s,\bullet}(\Em)$ is given by
	\[
	\widetilde{\X} = \zeta_{i,J} \frac{\partial}{\partial \xi^{i}}\wedge\frac{\partial}{\partial e^J} + \lambda_{K}\frac{\partial}{\partial e^K}
	\]
	where $J$ and $K$ are multi-indices of degree $s-1$ and $s$, respectively, $\zeta_{i,J}\in\cin(\M)$ and $\lambda_{K}\in\cin_{\text{lin}}(\Em)$.
	\begin{remark}
		Note that the space of linear multivector fields is \textit{not} closed under the wedge product. This is easily seen from the local description because a product of two multivector fields may contain a term of the form $\frac{\partial}{\partial \xi^{i_1}}\wedge\frac{\partial}{\partial \xi^{i_2}}$.
	\end{remark}

\begin{remark}\label{Linear multivector fields - Wedge products and Lie brackets}
	If $\widetilde{\X},\widetilde{\Y}\in\mathfrak{X}_{\text{lin}}(\Em)$, then also $[\widetilde{\X},\widetilde{\Y}]\in\mathfrak{X}_{\text{lin}}(\Em)$.
	In particular, the vector fields $[\widetilde{\X},\widetilde{\Y}]$ is linear over $[\X,\Y]$.
\end{remark}
	
	\section{Linear $\Q$-manifold structures}
	
	A \textbf{linear $\Q$-manifold}\index{$\Q$-manifold!linear} is a vector bundle $q:\Em\to\M$ equipped with a \textbf{linear homological vector field}\index{homological vector field!linear}: $\widetilde{\Q}\in\mathfrak{X}_{\text{lin}}^1(\Em)$ such that $\widetilde{\Q}^2=0$. From the discussion above, it follows that given a linear $\Q$-manifold $\Em$ with homological vector field $\widetilde{\Q}$, there is an induced homological vector field $\Q$ on the base $\M$. Since the space of linear vector fields on $\Em$ is a graded subalgebra of $\mathfrak{X}(\Em)$, we obtain the cochain complex
	\[
	\ldr{\widetilde{\Q}}\colon\mathfrak{X}_{\text{lin}}^\bullet(\Em)\to\mathfrak{X}_{\text{lin}}^{\bullet+1}(\Em),\qquad \widetilde{\X}\mapsto[\widetilde{\Q},\widetilde{\X}].
	\]
	It is a subcomplex of $(\mathfrak{X}(\Em),\ldr{\widetilde{\Q}})$ and, as in Chapter \ref{Chapter: Graded tangent and cotangent bundles}, it is the DGLA governing the deformations of the linear  $\Q$-manifold structure of $\Em$. As usual, the infinitesimal deformations and the deformations of $\widetilde{\Q}$ are in one-to-one correspondence with degree 1 cochains and Maurer-Cartan elements of $(\mathfrak{X}_{\text{lin}}(\Em),\ldr{\widetilde{\Q}})$.
	
	\begin{example}[\textbf{Tangent $\Q$-manifold}\index{tangent $\Q$-manifold}]
		It was mentioned before that linear homological vector fields on the vector bundle $\Em$ correspond to DG-module structures on $\Em^*$, and consequently on $\Em$. Hence, the tangent bundle $\pr_{\M}\colon T\M\to\M$ of a $\Q$-manifold $(\M,\Q)$ carries a canonical linear $\Q$-manifold structure corresponding to the adjoint module $(\mathfrak{X}(\M),\ldr{\Q})$. Its linear homological vector field $\Q_T$ is characterised by
			\begin{equation}\label{Tangent prolongation of Q}
				\Q_T(\pr^\star_\M(\xi)) = \pr^\star_\M(\Q(\xi))
				\qquad \text{and} \qquad
				\Q_T(\ell_\theta) = \ell_{\ldr{\Q}(\theta)}
			\end{equation}
			for all $\xi\in\cin(\M)$ and all $\theta\in\Gamma(T^*\M)=\Omega^1(\M)$; here, $\ell_\theta$ is the linear function of $T\M$ that corresponds to the $1$-form $\theta\in\Omega^{1,\bullet}(\M)$. Locally, suppose that $\{\xi^i\}$ are coordinates in a chart of $\M$ such that $\Q$ takes the form
			\[
			\Q = \zeta^i\frac{\partial}{\partial \xi^i}
			\]
			for some local functions $\zeta^i\in\cin(\M)$. Denote the linear coordinates of $T\M$ over this chart by $\{\dot{\xi}^i\}$; that is, $\dot{\xi}^i$ is the linear function that corresponds to the $1$-form $\diff\xi^i$. Then an application of (\ref{Tangent prolongation of Q}) for the coordinate functions $\{\xi^i,\dot{\xi}^i\}$ of the graded manifold $T\M$ implies that the vector field $\Q_T$ has the following local expression:
			\[
			\Q_T = \zeta^i\frac{\partial}{\partial \xi^i} - \dot{\xi}^j\frac{\partial \zeta^s}{\partial \xi^j} \frac{\partial}{\partial \dot{\xi}^s}.
			\]
	\end{example}
	
	\section{Linear homotopy Poisson structures}
	
	A \textbf{linear Poisson structure of degree $k$}\index{graded Poisson bracket!linear} on a graded vector bundle $\Em\to\M$ is a degree $k$ Poisson bracket $\{\cdot\,,\cdot\}_k$ on the graded manifold $\Em$ such that
	\begin{enumerate}
		\item $\{\cin_{\text{lin}}(\Em),\cin_{\text{lin}}(\Em)\}_k \subset \cin_{\text{lin}}(\Em)$,
		\item $\{\cin_{\text{lin}}(\Em),\cin_{\text{bas}}(\Em)\}_k \subset \cin_{\text{bas}}(\Em)$,
		\item $\{\cin_{\text{bas}}(\Em),\cin_{\text{bas}}(\Em)\}_k = 0$.
	\end{enumerate}
	
	Using the corresponding Poisson bivector field $\pi\in\mathfrak{A}_k^{2,-k}(\Em)$ discussed in Section \ref{Section:(Pseudo)multivector fields and the Poisson-Weil algebra}, the above axioms may be rewritten as
	\begin{enumerate}
		\item $[[\cin_{\text{lin}}(\Em),\pi]_k,\cin_{\text{lin}}(\Em)]_k \subset \cin_{\text{lin}}(\Em)$,
		\item $[[\cin_{\text{lin}}(\Em),\pi]_k,\cin_{\text{bas}}(\Em)]_k \subset \cin_{\text{bas}}(\Em)$
		\item $[[\cin_{\text{bas}}(\Em),\pi]_k,\cin_{\text{bas}}(\Em)]_k = 0$.
	\end{enumerate}
	In other words, $\{\cdot\,,\cdot\}_k$ is a linear Poisson bracket on $\Em$ if and only if the corresponding Poisson bivector field is linear: $\pi\in\mathfrak{A}_{\text{lin},k}^{2,-k}(\Em)$. More generally, a \textbf{linear homotopy Poisson structure}\index{homotopy Poisson structure!linear} on a $\mathbb{Z}$-manifold $\M$ is an element $\Theta\in\mathfrak{A}_{\text{lin},k}^{2-k}(\Em)$ such that $[\Theta,\Theta]_k = 0$.
	
	Similarly as in Section \ref{Section:(Pseudo)multivector fields and the Poisson-Weil algebra}, due to the invariance of the space $\mathfrak{A}_{\text{lin},k}(\Em)$ under the Schouten bracket explained in Remark \ref{Linear multivector fields - Wedge products and Lie brackets}, one obtains the induced differential operator of bidegree $(1,0)$
	\[
	\diff_{\pi}:\mathfrak{A}_{\text{lin},k}^{\bullet,\bullet}(\Em)\to\mathfrak{A}_{\text{lin},k}^{\bullet+1,\bullet}(\Em)\qquad \widetilde{\X}\mapsto[\pi,\widetilde{\X}].
	\]
	It follows from the discussion above that there are two embeddings
	\[
	(\mathfrak{A}_{\text{lin},k}^{\bullet,\bullet}(\Em),\ldr{\widetilde{\Q}})\hookrightarrow(\mathfrak{A}_{k}^{\bullet,\bullet}(\Em),\ldr{\widetilde{\Q}})
	\qquad \text{and} \qquad
	(\mathfrak{A}_{\text{lin},k}^{\bullet,\bullet}(\Em),\diff_{\pi})\hookrightarrow(\mathfrak{A}_{k}^{\bullet,\bullet}(\Em),\diff_\pi),
	\] 
	and hence an embedding on the level of total complexes
	\[
	(\mathfrak{A}_{\text{lin},k}^{\bullet,\bullet}(\Em),\ldr{\widetilde{\Q}}+(-1)^{k-1}\diff_\pi)\hookrightarrow(\mathfrak{A}_{k}^{\bullet,\bullet}(\Em),\ldr{\widetilde{\Q}}+(-1)^{k-1}\diff_\pi).
	\]
	
	This yields the natural question of whether the induced map on cohomologies is also an embedding. Although the obvious answer is ``yes'', the proof requires some non-trivial work. For instance, consider the case of the Lie complex $\ldr{\widetilde{\Q}}:\mathfrak{X}^\bullet(\Em)\to \mathfrak{X}^{\bullet+1}(\Em)$ of $1$-vector fields of $\Em$ endowed with a linear homological vector field $\widetilde{\Q}$. One has to prove the following: Given $\widetilde{\X}\in\mathfrak{X}_{\text{lin}}^{\bullet}(\Em)$ with $\ldr{\widetilde{\Q}}(\Y) = \widetilde{\X}$ for some $\Y\in\mathfrak{X}(\Em)$ \textit{not necessarily linear}, there is a \textit{linear} vector field $\widetilde{\mathcal{Z}}\in\mathfrak{X}_{\text{lin}}(\Em)$ such that $\ldr{\widetilde{\Q}}(\mathcal{Z}) = \widetilde{\X}$. For the case of VB-algebroids, i.e.~linear $\mathbb{N}\Q$-manifolds of degree 1, it is solved in \cite{LaPastina20,LaVi18}, using classical differential geometric language and following an analytic approach. 
	
	\section{Graded Lie algebroids and linear $\mathcal{P}$-manifolds}
	
	Now we will prove a result which is familiar from ordinary Lie algebroids and linear Poisson structures over smooth manifolds. Namely, there is a correspondence between Lie algebroid structures of degree $k$ on the vector bundle $\mathcal{A}\to\M$ (in the sense of Mehta \cite{Mehta06,Mehta09}) and linear Poisson structures of degree $k$ on the dual bundle $\mathcal{A^*}$ viewed as a $\mathbb{Z}$-graded manifold.
	
	\begin{definition}
		A \textbf{(graded) Lie algebroid of degree $k$}\index{graded Lie algebroid} over a graded manifold $\M$ is a graded vector bundle $q\colon\mathcal{A}\to\M$ together with a degree $k$ anchor map $\rho$ and a degree $k$ Lie bracket on $\Gamma(\mathcal{A})$, i.e.~a degree $0$ vector bundle morphism $\rho\colon\Gamma(\mathcal{A})\to\mathfrak{X}(\M)[k]$ and a degree $0$ bilinear operator $[\cdot\,,\cdot]\colon\Gamma(\mathcal{A})\times\Gamma(\mathcal{A})\to\Gamma(\mathcal{A})[k]$ satisfying the following conditions:
		\begin{enumerate}
			\item $[\sigma,\tau] = -(-1)^{(|\sigma|+k)(|\tau|+k)}[\tau,\sigma]$,
			\item $[\sigma,[\tau,\upsilon]] = [[\sigma,\tau],\upsilon] + (-1)^{(|\sigma|+k)(|\tau|+k)}[\tau,[\sigma,\upsilon]]$,
			\item $[\sigma,\xi\tau] = (\rho(\sigma)(\xi))\,\tau + (-1)^{(|\sigma|+k)|\xi|}\,\xi\,[\sigma,\tau]$,
		\end{enumerate}
		for homogeneous elements $\xi\in\cin(\M)$ and $\sigma,\tau\in\Gamma(\mathcal{A})$.
	\end{definition}
	
	\begin{example}\label{Shifted tangent bundle as a graded Lie algebroid}
		Given any graded manifold $\M$, the shifted tangent bundle $T[k]\M$ with anchor $\rho=\id_{\mathfrak{X}(\M)[k]}$ and bracket the usual Lie bracket of vector fields $[\cdot\,,\cdot]$ is a Lie algebroid of degree $k$ over $\M$.
	\end{example}
	
	Consider a graded vector bundle $q\colon\mathcal{A}\to\M$ and its dual $p\colon\mathcal{A^*}\to\M$. Suppose that $\{\cdot\,,\cdot\}_k$ is a linear Poisson structure of degree $k$ on $\mathcal{A^*}$. Define a Lie algebroid structure of degree $k$ on $\mathcal{A}$ by the properties
	\begin{equation}\label{Graded Lie algebroids and graded linear Poisson brackets}
		\{\ell_\sigma,p^\star(\xi)\}_k = p^\star(\rho(\sigma)(\xi)),
		\qquad
		\{\ell_\sigma,\ell_{\tau}\}_k = \ell_{[\sigma,\tau]}
		\qquad \text{and}\qquad
		\{p^\star(\xi_1),p^\star(\xi_2)\}_k = 0
	\end{equation}
	for all $\xi_1,\xi_2\in\cin(\M)$ and all $\sigma,\tau\in\Gamma(\mathcal{A})$, where $\ell_\sigma$ and $\ell_{\tau}$ are the linear functions of $\mathcal{A^*}$ corresponding to the sections $\sigma$ and $\tau$, respectively. We show now that the graded anti-symmetry of the Poisson bracket $\{\cdot\,,\cdot\}_k$ is equivalent to the graded anti-symmetry of $[\cdot\,,\cdot]$, the Leibniz identity of $\{\cdot\,,\cdot\}_k$ is equivalent to $\rho(\sigma)$ being a degree $|\sigma|+k$ derivation of $\cin(\M)$, $\rho$ being $\cin(\M)$-linear and $[\cdot\,,\cdot]$ having the Leibniz identity, and the Jacobi identity of $\{\cdot\,,\cdot\}_k$ is equivalent to the Jacobi identity of $[\cdot\,,\cdot]$. In order to see this, consider homogeneous elements $\xi,\xi_1,\xi_2\in\cin(\M)$ and $\sigma,\tau,\upsilon\in\Gamma(\mathcal{A})$. Then we compute
	\[
	\ell_{[\sigma,\tau]} = \{\ell_\sigma,\ell_\tau\}_k = -(-1)^{(|\sigma|+k)(|\tau|+k)}\{\ell_\tau,\ell_\sigma\}_k = \ell_{-(-1)^{(|\sigma|+k)(|\tau|+k)}[\tau,\sigma]},
	\]
	\begin{align*}
		p^\star(\rho(\sigma)(\xi_1\xi_2)) = &\ \{\ell_\sigma,p^\star(\xi_1)\,p^\star(\xi_2)\}_k \\
		= &\ \{ \ell_\sigma,p^\star(\xi_1) \}_k\,p^\star(\xi_2) + (-1)^{(|\sigma|+k)|\xi_1|} p^\star(\xi_1)\{ \ell_\sigma,p^\star(\xi_2) \}_k \\
		= &\ p^\star(\rho(\sigma)(\xi_1)\,\xi_2 + (-1)^{(|\sigma|+k)|\xi_1|} \xi_1\rho(\sigma)(\xi_2)),
	\end{align*}
	\begin{align*}
		p^\star(\rho(\xi_1\sigma)(\xi_2)) = &\ \{\ell_{\xi_1\sigma},p^\star(\xi_2)\}_k \\
		= &\ \{p^\star(\xi_1)\,\ell_\sigma,p^\star(\xi_2)\}_k \\
		= & -(-1)^{(|\xi_2|+k)(|\xi_1|+|\sigma|+k)} \{p^\star(\xi_2),p^\star(\xi_1)\,\ell_\sigma\}_k \\
		= & -(-1)^{(|\xi_2|+k)(|\sigma|+k)}\, p^\star(\xi_1)\,\{p^\star(\xi_2),\ell_\sigma\}_k \\
		= &\ p^\star(\xi_1\rho(\sigma)(\xi_2)),
	\end{align*}
	\begin{align*}
		\ell_{[\sigma,\xi\tau]} = &\ \{ \ell_\sigma,p^\star(\xi)\,\ell_\tau \}_k \\
		= &\ \{ \ell_\sigma,p^\star(\xi) \}_k\, \ell_\tau + (-1)^{(|\sigma|+k)|\xi|}\, p^\star(\xi)\,\{\ell_\sigma,\ell_\tau\}_k \\
		= &\ p^\star(\rho(\sigma)(\xi))\, \ell_\tau + (-1)^{(|\sigma|+k)|\xi|}\, p^\star(\xi)\,\ell_{[\sigma,\tau]} \\
		= &\ \ell_{(\rho(\sigma)(\xi))\,\tau + (-1)^{(|\sigma|+k)|\xi|}\,\xi\,[\sigma,\tau]},
	\end{align*}
	\begin{align*}
		\ell_{[\sigma,[\tau,\upsilon]]} = &\ \{\ell_\sigma,\ell_{[\tau,\upsilon]}\}_k \\
		= &\ \{ \ell_\sigma,\{ \ell_\tau,\ell_\upsilon \}_k \}_k \\
		= &\ \{ \{ \ell_\sigma,\ell_\tau \}_k,\ell_\upsilon \}_k + (-1)^{(|\sigma|+k)(|\tau|+k)} \{ \ell_\tau,\{ \ell_\sigma,\ell_\upsilon \}_k \}_k \\
		= &\ \{ \ell_{[\sigma,\tau]},\ell_\upsilon \}_k + (-1)^{(|\sigma|+k)(|\tau|+k)}\{ \ell_\tau,\ell_{[\sigma,\upsilon]} \}_k \\
		= &\ \ell_{[\sigma,[\tau,\upsilon]] = [[\sigma,\tau],\upsilon] + (-1)^{(|\sigma|+k)(|\tau|+k)}[\tau,[\sigma,\upsilon]]}.
	\end{align*}
	
	Conversely, given a Lie algebroid structure of degree $k$ on $\mathcal{A}\to\M$, the formulae in (\ref{Graded Lie algebroids and graded linear Poisson brackets}) define the action of a linear bivector field $\pi\in\mathfrak{A}_{\text{lin},k}^{2,-k}(\mathcal{A^*})$ on the coordinate functions of $\mathcal{A^*}$, which characterises uniquely the linear Poisson bracket on the whole $\cin(\mathcal{A^*})$.
	
	\begin{example}[Symplectic structure on $T^*\text{[}1-k\text{]}\M$]
		The natural degree $k-1$ Lie algebroid structure on $T[k-1]\M$ obtained from a graded manifold $\M$ (Example \ref{Shifted tangent bundle as a graded Lie algebroid}) corresponds to the linear Poisson structure on the graded manifold $T^*[1-k]\M$ given by the canonical symplectic form $\omega_{\text{can}}$ defined in Example \ref{Symplectic structure of cotangent bundles}. In particular, recall that the algebra of functions of the graded manifold $T^*[1-k]\M$ is given in degree $i\in\mathbb{Z}$ by
		\[
		\mathfrak{A}_{k}^{i}(\M)=\bigoplus_{p+q=i}\mathfrak{A}_k^{p,q}(\M),
		\]
		with basic and linear functions
		\[
		\cin_{\text{bas}}(T^*[1-k]\M)^i = \bigoplus_{p+q=i}\mathfrak{A}_{\text{bas},k}^{p,q}(\M) 
		\qquad \text{and} \qquad
		\cin_{\text{lin}}(T^*[-k]\M)^i = \bigoplus_{p+q=i}\mathfrak{A}_{\text{lin},k}^{p,q}(\M) .
		\]
		The linear Poisson bracket coincides with the degree $k-1$ Schouten bracket defined in Section \ref{Section:(Pseudo)multivector fields and the Poisson-Weil algebra} and one may recover its local description from the formulae in (\ref{Graded Lie algebroids and graded linear Poisson brackets}). For this, write the Poisson bivector field $\pi$ in local coordinates as
		
		\[
		\pi = \kappa^{st} \frac{\partial}{\partial \xi^s}\wedge\frac{\partial}{\partial \xi^t} +
		\eta^{mn} \frac{\partial}{\partial e^m}\wedge\frac{\partial}{\partial e^n} +
		\theta^{ji}\frac{\partial}{\partial e^j}\wedge\frac{\partial}{\partial \xi^i}
		\]
		where $\{\xi^j\}$ are local coordinates on $\M$, $\{ e^j := \ell_{\partial/\partial \xi^j} \}$ are the linear coordinates of $T^*[1-k]\M$ corresponding to the local frame $\{\frac{\partial}{\partial \xi^j}\}$ of $T[k-1]\M$, and $\kappa^{st},\eta^{mn},\theta^{ji}$ are functions of $T^*[1-k]\M$ over this coordinate chart. We recover the functions $\kappa^{st},\eta^{mn},\theta^{ji}$ for the coordinate representation of $\pi$ by plugging the coordinate functions $\{\xi^j,e^j\}$ into $[\cdot\,,[\cdot\,,\pi]_k]_k$, as in Section \ref{Section:(Pseudo)multivector fields and the Poisson-Weil algebra}. The last formula in (\ref{Graded Lie algebroids and graded linear Poisson brackets}) implies $\kappa^{st} = 0$. The second formula in (\ref{Graded Lie algebroids and graded linear Poisson brackets}), together with
		\[
		\left[\frac{\partial}{\partial \xi^m},\frac{\partial}{\partial \xi^n}\right] = 0
		\qquad \Longleftrightarrow \qquad
		\{ e^m,e^n \}_k = 0,
		\]
		yields $\eta^{mn} = 0$. Finally, a straightforward computation using the anti-symmetry and the Leibniz rule of the Schouten bracket $[\cdot\,,\cdot]_k$, together with the first formula in (\ref{Graded Lie algebroids and graded linear Poisson brackets}), shows that
		\[
		\delta_{ij} =  \frac{\partial \xi^i}{\partial \xi^j} = \id\left(\frac{\partial}{\partial \xi^j}\right)(\xi^i) = \{ e^j,\xi^i \}_k = - [\xi^i,[e^j,\pi]_k]_k = (-1)^{|\xi^i|(k-1)}\theta^{ji}.
		\] 
		The coordinate transformation $\zeta^i:= (-1)^{|\xi^i|(k-1)}\xi^i$ brings $\pi$ into the form of Example \ref{Symplectic structure of cotangent bundles}:
		\[
		\pi = \sum_i \frac{\partial}{\partial e^i}\wedge\frac{\partial}{\partial \zeta^i}.
		\]
	\end{example}
	
	\begin{proposition}\label{Lie algebroids and linear Poisson brackets}
		Suppose $(\mathcal{A},\rho,[\cdot\,,\cdot])$ is a Lie algebroid of degree $k$ over $\M$ with corresponding linear Poisson bivector  $\pi$ on $\mathcal{A^*}$. Then a linear vector field $\widetilde{\X}\in\mathfrak{X}_{\text{lin}}^i(\mathcal{A^*})$ over $\X\in\mathfrak{X}(\M)$ is Poisson if and only if for all (homogeneous) $\sigma,\tau\in\Gamma(\mathcal{A})$
		\begin{enumerate}
			\item the corresponding operator $\D$ on $\Gamma(\mathcal{A})$ is a derivation of the Lie bracket:
			\[
			\D[\sigma,\tau] = [\D(\sigma),\tau] + (-1)^{i(\sigma+k)}[\sigma,\D(\tau)],
			\]
			\item $\D$ and $\ldr{\X}$ commute with $\rho$:
			\[
			\ldr{\X}(\rho(\sigma))=[\X,\rho(\sigma)]=\rho(\D(\sigma)).
			\]
		\end{enumerate}
	\end{proposition}
	\begin{proof}
		Both conditions follow easily from the actions of the linear vector field and the linear Poisson bracket on basic and linear functions of $\mathcal{A^*}$.
	\end{proof}
	
	\begin{remark}
		Linear Poisson vector fields on $\mathcal{A^*}$ are related to the ``morphic vector fields'' on $\mathcal{A}$ defined in \cite{Mehta06,Mehta09}. By definition, a vector field $\X$ on $\mathcal{A}$ is morphic if $[\diff_{\mathcal{A}},\X]=0$, where $\diff_{\mathcal{A}}$ is the Lie algebroid differential that encodes the bracket and the anchor of $\mathcal{A}$. As it is shown by Mehta, morphic vector fields can also be described as operators $\D$ on $\Gamma(\mathcal{A})$ with the properties of Proposition \ref{Lie algebroids and linear Poisson brackets}.
	\end{remark}
	
	\chapter{Higher split VB-algebroid structures}\label{Chapter: Higher split VB-algebroid structures}
	
	In this chapter, we introduce the notion of higher split VB-Lie algebroids as a generalisation of VB-algebroids. Intuitively, VB-algebroids may be viewed as vector bundles in the category of Lie algebroids and so in the same spirit higher VB-Lie algebroids may be thought of as vector bundles in the category of higher split Lie algebroids. They are essentially a classical differential geometric description of the linear structures that appeared in the previous chapter and we prove that they are in correspondence with representations up to homotopy.
	
	\section{Classical interpretation}\label{Section: Classical interpretation}
	
	Suppose that $(\underline{D},V,\A,M)$ is a double vector bundle (see Section \ref{Section: Double vector bundles, linear splittings and duals})
	together with graded vector bundle decompositions
	$\underline{D}=D_1[1]\oplus\ldots\oplus D_n[n]$ and
	$\A=A_1[1]\oplus\ldots\oplus A_n[n]$, over $V$ and $M$, respectively, which are compatible with the
	projection $\underline{D}\to\A$. This means that each of the
	individual squares $(D_i,V,A_i,M)$ also forms a double vector bundle\footnote{In supergeometric terminology, the double vector bundle $(\underline{D},V\A,M)$ is the splitting of a vector bundle whose total space and base are both $\n$-manifolds. See Appendix \ref{Appendix: The geometrisation of N-graded vector bundles} for more details about this correspondence.}.
	Schematically, this yields the following sequence of diagrams
	\[
	\begin{tikzcd}
	& D_1[1] \arrow[ldd] \arrow[rr] \arrow[d,symbol=\oplus] &   & A_1[1] \arrow[ldd, crossing over] \arrow[d,symbol=\oplus] \\
	& D_2[2] \arrow[ld] \arrow[rr, crossing over] \arrow[d,symbol=\oplus]  &   & A_2[2] \arrow[ld] \arrow[d,symbol=\oplus]  \\
	V \arrow[rr] & \vdots                     & M & \vdots \\
	& D_n[n] \arrow[lu] \arrow[rr] \arrow[u,symbol=\oplus]  &   & A_n[n] \arrow[lu] \arrow[u,symbol=\oplus] 
	\end{tikzcd}
	\]
	where all the ``planes" are double vector bundles. This yields that
	the core of $(\underline{D},V,\A,M)$ is the graded vector
	bundle $\underline{C}=C_1[1]\,\oplus\ldots\oplus\, C_n[n]$, where $C_i$ is the
	core of $(D_i,V,A_i,M)$, for  $i=1,\ldots,n$.
	
	\begin{definition}\label{VB_lien}
		The quadruple $(\underline{D},V,\A,M)$ is a \textbf{(split) VB-Lie
			$n$-algebroid}\index{VB-Lie $n$-algebroid} if 
		\begin{enumerate}
			\item the graded vector bundle $\underline{D}\to V$ 
			is endowed with a homological vector field
			$\Q_{\underline{D}}$, 
			\item the Lie $n$-algebroid structure of $\underline{D}\to V$
			is \textbf{linear}\index{Lie $n$-algebroid!linear} over $A\to M$, in the sense that
			\begin{enumerate}
				\item the anchor $\rho_D\colon D_1\to TV$ is a double vector bundle morphism,
				\item the map $\partial_{D_i}$ fits into a morphism
				of double vector bundles
				$(\partial_{D_i},\id_V,\partial_{A_i},\id_M)$ between $(D_i,V,A_i,M)$ and
				$(D_{i-1},V,A_{i-1},M)$ for all $i$,
				\item the multi-brackets of
				$\underline{D}$ satisfy the following relations:
				\begin{enumerate}
					\item the $i$-bracket of $i$ linear sections is a linear section;
					\item the $i$-bracket of $i-1$ linear sections with a core section is a core section;
					\item the $i$-bracket of $i-k$ linear sections with $k$ core sections, $i\geq k \geq 2$, is zero.
				\end{enumerate} 
			\end{enumerate} 
		\end{enumerate}
	\end{definition}
	
	\begin{remark}
		\begin{enumerate}
			\item A VB-Lie 1-algebroid as in the definition above is just a VB-algebroid.
			
			\item A VB-Lie $n$-algebroid structure on $(\underline{D},V,\A,M)$ defines uniquely a Lie $n$-algebroid structure on $\A\to M$ as follows: The anchor $\rho_{D}\colon D_1\to TV$ is linear over the anchor $\rho\colon A_1\to TM$, and if all $d_j\in\Gamma_V^l(\underline{D})$ are linear sections over $a_j\in\Gamma(\A)$ for $j=1,2,\ldots,i$, then the section $\llbracket d_1,d_2,\ldots,d_i \rrbracket_{\underline{D}}\in\Gamma_V^l(\underline{D})$ is linear over the section $\llbracket a_1,a_2,\ldots,a_i \rrbracket_{\A}\in\Gamma(\A)$. Therefore, the
			graded vector bundles $\underline{D}\to V$ and $\A\to M$ are endowed with homological vector fields $\Q_{\underline{D}}$
			and $\Q_{\A}$ for which the bundle projection $\underline{D}\to \A$ is a morphism of Lie $n$-algebroids over
			the projection $V\to M$. In particular, the homological vector field $\Q_{\A}$ on the $[n]$-manifold $\A=A_1[1]\oplus\ldots\oplus A_n[n]$ is determined by the equations
			\[
			\Q_{\underline{D}}(q_V^*f) = \pi_{\A}^\star(\Q_{\A}(f))
			\qquad \text{and} \qquad
			\Q_{\underline{D}}(\pi_{\A}^\star(\alpha)) = \pi_{\A}^\star(\Q_{\A}(\alpha))
			\]
			for $f\in C^\infty(M)$ and $\alpha\in\Gamma(\underline{S}(\A^*))$.
		\end{enumerate}
	\end{remark}
	
	We now need to define the right notion of morphism of VB-Lie $n$-algebroids. Intuitively, this is given by conditions similar to the bracket relations of a VB-Lie $n$-algebroid, i.e.~a set of linear sections will remain intact (due to the fixed side Lie $n$-algebroid $\A$), a set of linear sections and a core section is mapped to a core section, a set of linear sections and two or more core sections is mapped to zero. In order to make it precise, we consider the following bigrading on the functions of the Lie $n$-algebroid $\underline{D}\to V$:
	\begin{itemize}
		\item basic functions in $ q^*(C^\infty(M))\subset C^\infty(V)$ have bidegree $(0,0)$;
		
		\item basic functions in $\pi_{\A}^\star(\Gamma(\underline{S}^i(\A^*)))$ have bidegree $(0,i)$;
		
		\item linear functions $\ell_\psi\in C_{\text{lin}}^\infty(V)$, where $\psi\in\Gamma(V^*)$, have bidegree $(1,-1)$;
		
		\item linear functions corresponding to elements in $\Gamma(C_i^*)$ have bidegree $(1,i-1)$.
	\end{itemize}
	In fact, this is the bigrading of the function space $\cin(\underline{D}_V)$ of the $\n$-manifold $\underline{D}_V$ viewed as a vector bundle over the $\n$-manifold $\A$ in the category of graded manifolds.
	
	\begin{definition}\label{Definition of morphism of VB-Lie n-algebroids}
		\textbf{A morphism of VB-Lie $n$-algebroids}\index{VB-Lie $n$-algebroid!morphism} from $(\underline{D},V,\A,M)$ to $(\underline{D}',V',\A,M)$ (with fixed side $\A$) is a bidegree $(0,0)$ morphism of Lie $n$-algebroids $G_D\colon \underline{D} \to \underline{D}'$ over $G_V\colon V\to V'$ whose action on basic functions is trivial\footnote{By abuse of notation, we use the same symbol $\pi_{\A}$ to denote both vector bundle projections $\underline{D}\to \A$ and $\underline{D}'\to \A$.}: $G^\star_D(f\circ q_{V'}) = f\circ q_V$ and $G^\star_D(\pi_{\A}^\star(\alpha)) = \pi_{\A}^\star(\alpha)$ for all $f\in C^\infty(M)$ and all $\alpha\in\Gamma(\underline{S}(\A^*))$.
	\end{definition}
	
	\begin{remark}
		The condition of the trivial action of a VB-Lie $n$-algebroid morphism on basic functions is imposed because here we restrict our study to morphisms of representations up to homotopy over a fixed Lie $n$-algebroid. That is, our aim is to prove a one-to-one correspondence between morphisms of VB-Lie $n$-algebroids with the same side Lie $n$-algebroid $\A\to M$ and morphisms of representations up to homotopy of $\A$ fixing $\A$. Dropping this condition yields a degree-preserving morphism $G_{\A}^\star\colon\Gamma(\underline{S}(\A^*))\to \Gamma(\underline{S}(\A^*))$ satisfying $G^\star_D(\pi_{\A}^\star(\alpha)) = \pi_{\A}^\star(G_{\A}^\star(\alpha))$ and hence a morphism or representations up to homotopy over $G_{\A}^\star$.
	\end{remark}
	
	The resulting category is denoted $\mathbb{VB}\text{-Lie}_n(\A)$ and the set of isomorphy classes of this category is denoted $\text{VB-Lie}_n(\A)$. The category whose objects are the \textbf{decomposed VB-Lie $n$-algebroids}\index{VB-Lie $n$-algebroid!decomposed}, i.e.~VB-Lie $n$-algebroids with decomposed $(D_i,V,A_i,M)$, with morphisms defined as above is denoted $\mathbb{DVB}\text{-Lie}_n(\A)$ and the isomorphy classes $\text{DVB-Lie}_n(\A)$. In fact, since every double vector bundle admits a decomposition, given a VB-Lie $n$-algebroid $(\underline{D},V,\A,M)$, a choice of a decomposition for each double vector bundle $(D_i,V,A_i,M)$ gives a decomposition of $(\underline{D},V,\A,M)$. Therefore, we immediately obtain the following proposition.
	
	\begin{proposition}
		There is an equivalence of categories between $\mathbb{VB}\text{-Lie}_n(\A)$ and $\mathbb{DVB}\text{-Lie}_n(\A)$, and therefore an equivalence between the isomorphy classes $\text{VB-Lie}_n(\A)$ and $\text{DVB-Lie}_n(\A)$.
	\end{proposition}
	
	\begin{example}[\textbf{Tangent prolongation of a (split) Lie $n$-algebroid}\index{Lie $n$-algebroid!tangent prolongation}]\label{Tangent prolongation of a (split) Lie n-algebroid}
		The basic example of a split VB-Lie $n$-algebroid is obtained by
		applying the tangent functor to a split Lie $n$-algebroid
		$\A=A_1[1]\oplus\ldots\oplus A_n[n]\to M$. The double vector
		bundle is given by the diagram
		\[
		\begin{tikzcd}
		\underline{TA} \arrow[dd] \arrow[rr] & & \A \arrow[dd] \\
		& & \\
		TM \arrow[rr]          &             & M           
		\end{tikzcd}
		\]
		where the Lie $n$-algebroid structure of
		$\underline{TA}=T\A=TA_1[1]\oplus\ldots\oplus TA_n[n]$ over
		the manifold $TM$ is defined by the relations
		\begin{enumerate}
			\item $\rho_{TA}=J_M\circ T\rho_A\colon TA_1\to TTM$, where
			$J_M\colon TTM\to TTM$ is the canonical involution from Example \ref{Double tangent bundle}.
			
			\item $\llbracket Ta_{1},\ldots,Ta_{i}\rrbracket = T\llbracket a_{1},\ldots,a_{i}\rrbracket$,
			
			\item
			$\llbracket
			Ta_{1},\ldots,Ta_{i-1},a_{i}^\dagger\rrbracket =
			\llbracket
			a_{k_1},\ldots,a_{{i-1}},a_{i}\rrbracket^\dagger$,
			
			\item
			$\llbracket
			Ta_{1},\ldots,Ta_{j},a_{j+1}^\dagger,\ldots,a_{i}^\dagger\rrbracket
			= 0$ for all $1\le j\le i-2$,
			
			\item $\llbracket a_{1}^\dagger,\ldots,a_{i}^\dagger\rrbracket = 0$,
		\end{enumerate}	
		for all sections $a_{j}\in\Gamma(A_{k_j})$.
	\end{example}
	
	Applying the above construction to a split Lie 2-algebroid
	$Q[1]\oplus B^*[2]\to M$ with structure
	$(\rho_Q,\ell,\nabla^*,\omega)$ yields  as follows the objects
	$(\rho_{TQ},T\ell,T\nabla^*,T\omega)$ of the split Lie 2-algebroid
	structure of $TQ[1]\oplus TB^*[2]\to TM$: The complex
	$TB^*\to TQ\to TTM$ consists of the anchor of $TQ$ given by
	$\rho_{TQ}=J_M\circ T\rho_Q$, and the vector bundle map
	$T\ell\colon TB^*\to TQ$. The bracket of $TQ$ (over $TM$) is defined by the
	relations
	\[
	[Tq_1,Tq_2]_{TQ} = T[q_1,q_2]_Q,\qquad 
	[Tq_1,q_2^\dagger]_{TQ} = [q_1,q_2]_Q^\dagger,\qquad 
	[q_1^\dagger,q_2^\dagger]_{TQ} = 0,
	\]
	for $q_1,q_2\in\Gamma(Q)$. The $TQ$-connection
	$T\nabla^*\colon
	\Gamma_{TM}(TQ)\times\Gamma_{TM}(TB^*)\to\Gamma_{TM}(TB^*)$ is
	defined by
	\[
	(T\nabla^*)_{Tq}(T\beta) = T(\nabla^*_q\beta),\qquad
	(T\nabla^*)_{Tq}(\beta^\dagger) = (\nabla^*_q\beta)^\dagger = (T\nabla^*)_{q^\dagger}\beta,\qquad
	(T\nabla^*)_{q^\dagger}(\beta^\dagger) = 0,
	\]
	for $q\in\Gamma(Q)$ and $\beta\in\Gamma(B^*)$. Finally, the 3-form $T\omega\in\Omega^3(TQ,TB^*)$ is defined by
	\[
	(T\omega)(Tq_1,Tq_2,Tq_3) = T(\omega(q_1,q_2,q_3)),\qquad
	(T\omega)(Tq_1,Tq_2,q_3^\dagger) = \omega(q_1,q_2,q_3)^\dagger,
	\]
	\[
	(T\omega)(q_1,q_2^\dagger,q_3^\dagger) = 0 =(T\omega)(q_1^\dagger,q_2^\dagger,q_3^\dagger),
	\]
	for $q_1,q_2,q_3\in\Gamma(Q)$.
	
	\section{Supergeometric interpretation}\label{Section: Supergeometric interpretation}
	
	Now we give an equivalent description of a VB-Lie $n$-algebroid $(\underline{D},V,\A,M)$ with decomposed double vector bundles $(D_i,V,A_i,M)$ in terms of its homological vector field.
	
	\begin{proposition}\label{Homological vector field of VB-Lie n-algebroids}
		Suppose $(\underline{D},V,\A,M)$ is a decomposed double vector bundle. Then a Lie $n$-algebroid structure on the $\n$-manifold $\underline{D}_V\to V$ is a VB-Lie $n$-algebroid structure for $(\underline{D},V,\A,M)$ if and only if its corresponding homological vector field $\Q_{\underline{D}}$ satisfies the following relations:
		\begin{enumerate}
			\item $\Q_{\underline{D}}$ maps $q_V^*(C^\infty(M))$ into $\pi_{\A}^\star(\Gamma(A_1^*))$;
			\item $\Q_{\underline{D}}$ maps $C^\infty_{\text{lin}}(V)$ into $\Gamma(A_1^*\otimes V^*)\oplus\Gamma(C_1^*)$;
			\item $\Q_{\underline{D}}$ leaves $\pi_{\A}^\star(\Gamma(\underline{S}(\A)))$ invariant;
			\item $\Q_{\underline{D}}$ maps $\Gamma(C_i^*)$ into the degree $i+1$ elements of
			$\pi_{\A}^\star(\Gamma(\underline{S}(\A)))\otimes \Gamma(\underline{C}^*)$.
		\end{enumerate}
	\end{proposition}
	\begin{proof}
		The first and third assertions follow from the fact that the projection $\pi_{\A}$ is a morphism of Lie $n$-algebroids from $\underline{D}_V\to V$ to $\A\to M$. To see assertion (2), one proceeds as follows: Recall that the spaces $\Gamma_V(D_i^*)$ are generated as $C^\infty(V)$-modules by core and linear sections, where for the latter we have the equality
		\begin{equation}\label{Sections of dual DVD}
		\Gamma_V^l(D_i^*) = \Gamma(A_i^*\otimes V^*)\oplus\Gamma(C_i^*)
		\end{equation}
		given by the decomposition, and the (canonical) identification of $\alpha\in\Gamma(\A^*)$ with $\pi_{\A}^\star(\alpha)\in\Gamma(D_V^*)$. Moreover, as vector bundles over $V$, we have the equality 
		\begin{equation}\label{Decomposed DVD}
		D_i=q_V^!(A_i\oplus C_i).
		\end{equation}
		Note that for $\psi\in\Gamma(V^*)$ the 1-form $\diff\ell_\psi$ is a linear section of $T^*V\to V$ over $\psi:M\to V^*$ and $\rho_{D_1}:D_1\to TV$ is a morphism of double vector bundles. Hence, (2) follows from
		\[
		\Q_{\underline{D}}(\ell_\psi) = \rho_{D_1}^*\diff\ell_\psi\in\Gamma_V^l(D_1^*)=\Gamma(A_1^*\otimes V^*)\oplus\Gamma(C_1^*).
		\]
		
		Suppose now that $\gamma_i\in\Gamma(C_i^*)$. Then using (\ref{Decomposed DVD}) we may write the $(i+1)$-degree function $\Q_{\underline{D}}(\gamma)$ in terms of its various components. The components with $\Gamma(\underline{S}^k(\A))$ together with two or more $\Gamma(C_j^*)$ vanish due to the bracket conditions of a VB-Lie $n$-algebroid. Recall that from the definition of split Lie $n$-algebroids in terms of higher brackets (Section \ref{Section: Q-structures and Lie n-algebroids}), we have that for terms of the form $\Gamma(C_j^*\wedge C_{k}^*)$ with $j+k=i+1$
		\begin{equation}\label{Homological vector field with two C_j}
		\Q_{\underline{D}}(\gamma_i)(c_j,c_k)=\rho_{D_1}(c_j)\langle \gamma_i,c_k \rangle - \rho_{D_1}(c_k)\langle \gamma_i,c_j \rangle - \langle \gamma_i,\llbracket c_j,c_k \rrbracket_2 \rangle,
		\end{equation}
		where $\rho_{D_1}(c_j)$ and $\rho_{D_1}(c_k)$ are zero if $j\neq1$ and $k\neq1$, respectively. If $i\neq j$ and $i\neq k$, then the first two terms are again zero, due to the vanishing of the pairing. Observe now that the last term is always zero because it is a bracket with two core sections.
		
		It remains to prove that the terms of the form $\rho_{D_1}(c_1)\langle \gamma_i,c_i \rangle$ are always zero. As a morphism of double vector bundles, the anchor $\rho_{D_1}$ sends the core section $c_1$ to a vector field on $V$ obtained by a vertical lift. By definition, a vertical lift has flow which lies entirely on the fibre directions. A simple observation of the fact that the sections $\gamma_i\in\Gamma_V(D_i^*), c_i\in\Gamma_V(D_i), c_1\in\Gamma_V(D_1)$ are obtained by the pull-back under $q_V$ of the corresponding sections $\gamma_i\in\Gamma(C_i^*),c_i\in\Gamma(C_i),c_1\in\Gamma(C_1)$ implies that the function $\langle \gamma_i,c_i \rangle\in C^\infty(V)$ is invariant under the flow of $\rho_{D_1}(c_1)$. Therefore, all the terms in equation (\ref{Homological vector field with two C_j}) are zero and we obtain the result. 
	\end{proof}
	
	\section{The Weil algebra of a split VB-Lie $n$-algebroid}\label{Section: The Weil algebra of a split VB-Lie n-algebroid}
	
	We now define the Weil algebra of a double vector bundle $(\underline{D},V,\A,M)$ as above (with no Lie $n$-algebroid structure a priori) and see how this can be used to give an alternative description of a VB-Lie $n$-algebroid structure. The case of $n=1$, i.e.~the ordinary double vector bundles $(D,V,A,M)$, is studied in great detail in \cite{MePi19}.
	
	\begin{definition}\label{Weil algebra of DVB}
		The \textbf{Weil algebra of $(\underline{D},V,\A,M)$}\index{double vector bundle!Weil algebra}, denoted $W(\underline{D})$, is defined as the space of functions of the graded manifold $\underline{D}[1]$ over $M$, where the shift functor $[1]$ is applied to the horizontal arrows of $(\underline{D},V,\A,M)$.
	\end{definition}
	
	\begin{remark}
		\begin{enumerate}
			\item The graded manifold $\underline{D}[1]$ whose space of functions defines the Weil algebra of $(\underline{D},V,\A,M)$ is the total space of the $[1]$-shifted vector bundle $\underline{D}[1]\to\A$ in the category of graded manifolds.
			
			\item Note that in the case of an ordinary (non-graded) double vector bundle $(D,V,A,M)$, our notation for the Weil algebra $W(D)$ is different from the notation in \cite{MePi19}, where a double shift $[1,1]$ for both horizontal and vertical arrows of $(D,V,A,M)$ is used. According to our convention, $D$ is already concentrated in degree $-1$ and so the induced $[1]$-manifold over $V$ satisfies $\cin(D)^1=\Gamma_V(D^*_V)$. In \cite{MePi19}, the vector bundle $D$ over $V$ is concentrated in degree $0$ and thus a (vertical) shift over $V$ is required, in order to give the same $[1]$-manifold.
		\end{enumerate}
	\end{remark}
	
	Unravelling the definition of the Weil algebra of $(\underline{D},V,\A,M)$, we find that $W(\underline{D})$ carries a bigrading such that $0$-functions of the form $f\circ q_V\in C^\infty(V)$ obtained by pulling-back smooth functions $f\in C^\infty(M)$ are elements of bidegree $(0,0)$, $1$-functions $\ell_\psi\in C^\infty(V)$ for $\psi\in\Gamma(V^*)$ are elements of bidegree $(1,0)$, $i$-functions $\alpha\in\Gamma(\underline{S}^i(\A))$ are elements of bidegree $(0,i)$ and $(i+1)$-functions $\gamma_i\in\Gamma(C_i^*)$ are elements of bidegree $(1,i)$. In other words, the Weil algebra of $(\underline{D},V,\A,M)$ is generated by
	\[
	W^{0,0}(\underline{D})=C^\infty(M),
	\qquad
	W^{0,i}(\underline{D})=\Gamma(\underline{S}^i(\A^*)),
	\]
	\[
	W^{1,0}(\underline{D})=\Gamma(V^*),
	\qquad
	W^{1,i}(\underline{D})=\bigoplus_{j+k=i}\Gamma(\underline{S}^j(\A^*))\otimes\Gamma(C_k^*).
	\]
	
	\begin{remark}
		The Weil algebra of a split Lie $n$-algebroid $\A$ over $M$ discussed in Section \ref{Section: The Weil algebra of a Lie n-algebroid} can be obtained from the definition above as $W(T\A)$, where $(T\A,TM,\A,M)$ is the tangent prolongation of $\A\to M$.
	\end{remark}
	In this setting, it is clear that Proposition \ref{Homological vector field of VB-Lie n-algebroids} above may be reformulated as follows.
	\begin{theorem}\label{VB-Lie n-algebroids via the Weil algebra}
		Let $(\underline{D},V,\A,M)$ be a decomposed double vector bundle. A VB-Lie $n$-algebroid structure on $(\underline{D},V,\A,M)$ is equivalent to a homological vector field of bidegree $(0,1)$ on the graded manifold $\underline{D}[1]$.
	\end{theorem}
	
	\section{The fat Lie $n$-algebroid}\label{Section: The fat Lie n-algebroid}
	
	The space of linear functions $\cin_{\text{lin}}(\underline{D})$ of a VB-Lie $n$-algebroid inherits an obvious structure of an algebra over the ring $C^\infty(M)$ of smooth functions on $M$. With this structure, $\cin_{\text{lin}}(\underline{D})$ becomes a sheaf of locally free, $\mathbb{N}$-graded, graded commutative, associative, unital $C^\infty(M)$-algebras over $M$, with local generators given by local frames of $(V^*\otimes \underline{C})^*$ and by pulling-back the local generators of $\A$ under $\pi_{\A}$. Hence, it is the space of functions of an $\n$-manifold over $M$, denoted $\underline{\widehat{A}}$. In fact, one has that as a vector bundle over $M$,
	\[
	\underline{\widehat{A}}=\bigoplus_{i=1}^n \widehat{A_i}[i],
	\]
	where $\widehat{A_i}$ is the (non-graded) vector bundle over $M$ whose sheaf of sections corresponds to the locally free $C^\infty(M)$-module structure on $\Gamma_V^l(D_i)$. From the short exact sequence (\ref{Short exact sequence of linear sections}), it follows that the vector bundle $\widehat{A_i}$ has rank equal to $\rank(A_i)+\rank(V)\cdot\rank(C_i)$, and thus the $\n$-manifold corresponding to $\underline{\widehat{A}}$ is of dimension
	\[
	(\dim(M);\rank(A_n)+\rank(V)\cdot\rank(C_n),\ldots,\rank(A_1)+\rank(V)\cdot\rank(C_1)).
	\]
	
	Another interesting fact about the $\n$-manifold $\underline{\widehat{A}}$ is that it carries a Lie $n$-algebroid structure which is induced by the VB-Lie $n$-algebroid structure on $(\underline{D},V,\A,M)$. In particular, it is given by the multi-brackets
	\[
	\llbracket a_1,\ldots,a_i \rrbracket_{\underline{\widehat{A}}} = 
	\llbracket a_1,\ldots,a_i \rrbracket_{\underline{D}} 
	\]
	for all $a_1,\ldots,a_i\in\Gamma(\underline{\widehat{A}})=\Gamma_V^l(\underline{D})$, or equivalently, by the homological vector field on $\underline{\widehat{A}}$ defined by $\Q_{\underline{\widehat{A}}}(\alpha) = \Q_{\underline{D}}(\alpha)$, for all $\alpha\in\Gamma(\underline{S}(\underline{\widehat{A}}))$. We call $\underline{\widehat{A}}$ the \textbf{fat Lie $n$-algebroid}\index{Lie $n$-algebroid!fat}. 
	
	The projection map $\underline{\widehat{A}}\to\A$ is a Lie $n$-algebroid morphism whose kernel may be identified with $V^*\otimes\underline{C}=\Hom(V,\underline{C})$. This yields that the graded vector bundle
	\[
	\underline{\Hom(V,\underline{C})}:=\bigoplus_{i=1}^n(V^*\otimes C_i)[i]=V^*\otimes\underline{C}=\Hom(V,\underline{C})=\bigoplus_{i=1}^n(\Hom(V,C_i))[i]
	\]
	may be equipped with a Lie $n$-algebroid structure so that the sequence
	\begin{equation}\label{Short exact sequence of fat Lie n-algebroid}
		0\longrightarrow\underline{\Hom(V,\underline{C})}\longrightarrow\underline{\widehat{A}}\longrightarrow \A\longrightarrow 0
	\end{equation}
	is a short exact sequence of Lie $n$-algebroids over $M$. Therefore, we obtain the following result.
	
	\begin{theorem}
		Given a VB-Lie $n$-algebroid $(\underline{D},V,\A,M)$, the construction explained above equips the graded vector bundle $\underline{\widehat{A}}$ with a Lie $n$-algebroid structure over the manifold $M$. With this Lie $n$-algebroid structure, $\underline{\widehat{A}}$ fits in the short exact sequence (\ref{Short exact sequence of fat Lie n-algebroid}) of Lie $n$-algebroids over $M$.
	\end{theorem}
	
	\section{Split VB-Lie $n$-algebroids and $(n+1)$-representations}\label{Section: Split VB-Lie n-algebroids and (n+1)-representations}
	
	As it is shown in \cite{GrMe10}, an interesting fact about the
	tangent prolongation of a Lie algebroid is that it encodes its
	adjoint representation. The same holds for a split Lie $n$-algebroid $A_1[1]\oplus\ldots\oplus A_n[n]$, since by definition the adjoint module is exactly the space of sections of the $\Q$-vector bundle $T(A_1[1]\oplus\ldots\oplus A_n[n])\to A_1[1]\oplus\ldots\oplus A_n[n]$. The next example shows this correspondence explicitly in the case of split Lie
	2-algebroids $Q[1]\oplus B^*[2]$.
	
	\begin{example}
		Choose two $TM$-connections on $Q$ and $B^*$, both denoted by
		$\nabla$. These choices induce the horizontal lifts
		$\Gamma(Q)\to\Gamma_{TM}^l(TQ)$ and
		$\Gamma(B^*)\to\Gamma_{TM}^l(TB^*)$, both denoted by $h$. More
		precisely, given a section $q\in\Gamma(Q)$, its lift is defined as
		$h(q) = Tq - (\nabla_{.}q)^\wedge$.
		A similar formula holds for
		$h(\beta)$, as well. Then an easy computation yields the following:
		\begin{enumerate}
			\item $\rho_{TQ}(q^\dagger) = \rho(q)^\uparrow$ and $(T\ell)(\beta^\dagger) = \ell(\beta)^\uparrow$
			\item $\rho_{TQ}(h(q)) = X_{\nabla_q^{\text{bas}}}$
			\item $(T\ell)(h(\beta)) = h(\ell(\beta)) + (\nabla_.(\ell(\beta)) - \ell(\nabla_.\beta))^\wedge$
			\item $[h(q_1),h(q_2)]_{TQ} = h[q_1,q_2]_Q - R_\nabla^{\text{bas}}(q_1,q_2)^\wedge$
			\item $[h(q_1),q_2^\dagger]_{TQ} = (\nabla_{q_1}^{\text{bas}}q_2)^\dagger$
			\item $(T\nabla^*)_{h(q)}(\beta^\dagger) = (\nabla_q^*\beta)^\dagger$ 
			\item $(T\nabla^*)_{q^\dagger}(h(\beta)) = (\nabla_q^*\beta - \nabla_{\rho(q)}\beta)^\dagger$
			\item
			$(T\nabla^*)_{h(q)}(h(\beta)) =
			h(\nabla^*_q\beta) + \left(\nabla_{\nabla_\cdot q}^*\beta -
			\nabla_{\rho(\nabla_\cdot q)}\beta + \nabla^*_q\nabla_\cdot\beta -
			\nabla_\cdot\nabla_q^*\beta -
			\nabla_{[\rho(q),\cdot]}\beta\right)^\wedge$
			\item $(T\omega)(h(q_1),h(q_2),h(q_3)) = h(\omega(q_1,q_2,q_3)) + ((\nabla_\cdot\omega)(q_1,q_2,q_3))^\wedge$
			\item $(T\omega)(h(q_1),h(q_2),q_3^\dagger) = (\omega(q_1,q_2,q_3))^\dagger$.
		\end{enumerate}
	\end{example}
	
	The reader should compare the calculations above with the structure objects of the adjoint representation up to homotopy of a split Lie $2$-algebroid given in Proposition \ref{Adjoint_representation_of_Lie_2-algebroid}. In fact, this result is a special case of a
	correspondence between VB-Lie $n$-algebroid structures on a decomposed graded double vector bundle
	$(\underline{D},V,\A,M)$ and $(n+1)$-representations of the Lie $n$-algebroid
	$\A$ on the vector bundle
	$\E:=V[0] \oplus C_1[1] \oplus \ldots \oplus C_n[n]$ over $M$. In the
	general case, it is easier to give the correspondence in terms
	of the homological vector field on $\underline{D}$ and the
	dual representation of $\A$ on
	$\E^*=C_n^*[-n]\oplus\ldots\oplus C_1^*[-1] \oplus V^*[0]$.
	
	Suppose that $(\underline{D},V,\A,M)$ is a VB-Lie
	$n$-algebroid with homological vector fields
	$\Q_{\underline{D}}$ and $\Q_{\A}$, and choose a decomposition
	for each double vector bundle $(D_i,V,A_i,M)$\footnote{In the
		case of the tangent Lie $n$-algebroid, this corresponds to
		choosing the $TM$-connections on the vector bundles of the
		adjoint complex.}, and consequently for
	$(\underline{D},V,\A,M)$. Consider the dual
	$\underline{D}_V^*$ and recall that the spaces
	$\Gamma_V(D_i^*)$ are generated as $C^\infty(V)$-modules by their
	core and linear sections. For the latter, use the
	identification
	$\Gamma_V^l(D_i^*) = \Gamma(A_i^*\otimes
	V^*)\oplus\Gamma(C_i^*)$ induced by the
	decomposition. Moreover, the element
	$\alpha\in\Gamma(A_i^*)$ is identified with the core section
	$\pi_{\A}^{\star}(\alpha)\in\Gamma_V^c(\underline{D}^*)$. Define the representation
	$\D^*$ of $\A$ on the dual complex $\E^*$ by the equations
	\begin{equation}\label{Equation 1 for VB-Lie n-algebroids and (n+1)-representation}
		\D^*(\psi) = \Q_{\underline{D}}(\ell_\psi) = \rho_{D_1}^*\diff\ell_\psi
		\qquad \text{and} \qquad
		\D^*(\gamma) = \Q_{\underline{D}}(\gamma)
	\end{equation}
	for all $\psi\in\Gamma(V^*)$ and all
	$\gamma\in\Gamma(C_i^*)$. From Proposition \ref{Homological vector field of VB-Lie n-algebroids} and the fact that $\Q_{\underline{D}}^2 = 0$, it follows that $\D^*$ defined by (\ref{Equation 1 for VB-Lie n-algebroids and (n+1)-representation}) is the differential of a representation up to homotopy of $(\A,\Q_{\A})$ on the graded vector bundle $\E^*$ over $M$.
	
	Conversely, given a representation $\D^*$ of $(\A,\Q_{\A})$ on $\E^*$, the equations in (\ref{Equation 1 for VB-Lie n-algebroids and (n+1)-representation}) together with
	\[
	\Q_{\underline{D}}(q_V^*f) = \pi_{\A}^\star(\Q_{\A}(f))
	\qquad \text{and} \qquad
	\Q_{\underline{D}}(\pi_{\A}^\star(\alpha)) = \pi_{\A}^\star(\Q_{\A}(\alpha)),                       
	\]
	for all $f\in C^\infty(M)$ and $\alpha\in\Gamma(\A^*)$, define a
	VB-Lie $n$-algebroid structure on the double vector bundle
	$(\underline{D},V,\A,M)$.
	This yields the following theorem.
	\begin{theorem}\label{Theorem RUTHS and higher VB-Lie algebroids}
		Let $(\underline{D},V,\A,M)$ be a decomposed graded double vector
		bundle with core $\underline{C}$. There is a one-to-one
		correspondence between VB-Lie $n$-algebroid structures on
		$(\underline{D},V,\A,M)$ and $(n+1)$-representations up to
		homotopy of $\A$ on the complex
		$V[0] \oplus C_1[1] \oplus \ldots \oplus C_n[n]$.
	\end{theorem} 
	
	We will show now that the correspondence of decomposed VB-Lie $n$-algebroids with fixed side algebroid $\A\to M$ and $(n+1)$-representations of $\A\to M$ can be also achieved on the level of morphisms, and therefore, we obtain an equivalence of categories.
	
	\begin{proposition}\label{Correspondence of morphisms}
		Let $(\underline{D},V,\A,M)$ and $(\underline{D}',V',\A,M)$ be two decomposed VB-Lie $n$-algebroids with cores $\underline{C}$ and $\underline{C}'$, respectively. There exists a one-to-one correspondence between morphisms of VB-Lie $n$-algebroids from $(\underline{D},V,\A,M)$ to $(\underline{D}',V',\A,M)$ and morphisms of $\A$-representations from $V[0]\oplus \underline{C}$ to $V'[0]\oplus \underline{C}'$. 
	\end{proposition}
	\begin{proof}
		Suppose first that 
		\[
		\mu\colon\Gamma(\underline{S}(\A^*))\otimes\Gamma(\underline{C}\oplus V)\to \Gamma(\underline{S}(\A^*))\otimes\Gamma(\underline{C}'\oplus V')
		\] 
		is a morphism of representations up to homotopy. Denote the differentials of $V[0]\oplus \underline{C}$ and $V'[0]\oplus \underline{C}'$ by $\D$ and $\D'$, respectively, and the dual of $\mu$ by $\mu^\star$.
		
		Recall that the space of sections $\Gamma_V(\underline{D}^*)$ of a double vector bundle is generated as $C^\infty(V)$-module by core and linear sections, and that the coordinate functions of the smooth manifold $V$ can be taken of the form $\ell_\psi$ (fibre-wise linear) and $f\circ q_V$ (basic functions), for $\psi\in\Gamma(V^*)$ and $f\in C^\infty(M)$. We define the map $G^\star_\mu\colon\Gamma(\underline{D'}_{V'}^*) \to \Gamma(\underline{D}_V^*)$ on the generators by
		\begin{equation}\label{Equations for correspondence of morphisms}
		G^\star_\mu(f\circ q_{V'})=f\circ q_{V}, \quad G^\star_\mu(\ell_\psi)=\ell_{\mu_0^*(\psi)}, \quad G^\star_{\mu}(\alpha)=\alpha, \quad G^\star_{\mu}(\gamma)=\mu^\star(\gamma),
		\end{equation}
		for all $\psi\in\Gamma(V'^*),\alpha\in\Gamma(\underline{S}(\A^*)),\gamma\in\Gamma(\underline{C'}^*),f\in C^\infty(M)$. A straightforward computation, using $\D'\circ\mu=\mu\circ\D$, shows that $G_\mu^\star\circ\Q_{\underline{D'}}=\Q_{\underline{D}}\circ G_\mu^\star$, and thus $G_\mu\colon\underline{D}_V\to \underline{D}'_{V'}$ is a morphism of Lie $n$-algebroids over $\mu_0|_V\colon V\to V'$. It is clear from the equations in (\ref{Equations for correspondence of morphisms}) that $G_\mu\colon\underline{D}_V\to\underline{D}'_{V'}$ is of bidegree $(0,0)$, and hence, together with $\mu_0|_V\colon V\to V'$, they define a morphism of VB-Lie $n$-algebroids as in Definition \ref{Definition of morphism of VB-Lie n-algebroids}. Its principal part $(G_{\mu,0})_i\colon D_i\to D'_i,i=1,\ldots,n,$ with respect to equality (\ref{Decomposed DVD}) is given by $\id_{A_i}\oplus~\mu_0|_{C_i}$, i.e.
		\[
		G_{\mu,0}^\star=\id_{\A^*} \oplus\ (\mu_0|_{\underline{C}})^\star=(\id_{A_1^*} \oplus\ \mu_0|_{C_1}^*)\oplus\ldots\oplus(\id_{A_n^*}\oplus\ \mu_0|_{C_n}^*)
		\]
		and therefore the induced core morphism is $\mu_0|_C\colon C\to C'$ over the identity on $M$.
		
		Conversely, suppose that $(G_D,G_V,\id_{\A},\id_M)$ is a morphism between two (decomposed) VB-Lie $n$-algebroids $(\underline{D},V,\A,M)$ and $(\underline{D}',V',\A,M)$. Observe that, for degree reasons, a morphism of $\A$-representations from $V'^*\oplus\underline{C'}^*$ to $V^*\oplus\underline{C}^*$ must send the vector bundle $V'^*$ into $V^*$ via a bundle map $\mu_0^*$ over the identity on $M$. Since $G_D^\star$ is a morphism of VB-Lie $n$-algebroids over the vector bundle map $G_V$, it has the following properties: 
		\begin{itemize}
			\item it maps the basic function $f\circ q_{V'}$ to $f\circ q_V$;
			
			\item it maps the fibre-wise linear function $\ell_\psi$ to a fibre-wise linear function $\ell_{\mu_0^*(\psi)}$, defining thus a $C^\infty(M)$-linear map $\mu_0^*\colon\Gamma(V'^*)\to\Gamma(V^*)$;
			
			\item it maps the section $\gamma_i\in\Gamma_V^l(D_i'^*)$ into the sum
			\[
			\Big(\Gamma(\underline{S}^i(\A^*))\otimes \Gamma(V^*)\Big) \oplus \left( \bigoplus_{j=1}^{i-1} \Gamma(\underline{S}^{i-j}(\A^*))\otimes \Gamma(C_j^*) \right) \oplus \Gamma(C_i^*) =
			\left( \Gamma(\underline{S}(\A^*))\otimes \Gamma(V^*\oplus \underline{C}^*) \right)^i,
			\]
		\end{itemize}
		Therefore, we may use again the formulae in equation (\ref{Equations for correspondence of morphisms}) to define a $\Gamma(\underline{S}(\A^*))$-linear map
		\[
		\mu^\star:\Gamma(\underline{S}(\A^*))\otimes\Gamma(V'^*\oplus\underline{C'}^*)\to \Gamma(\underline{S}(\A^*))\otimes\Gamma(V^*\oplus\underline{C}^*),
		\] 
		which commutes with the differentials due to the identity $G_D^\star\circ\Q_{\underline{D'}}=\Q_{\underline{D}}\circ G_D^\star$.
	\end{proof}
	
	Combining now Theorem \ref{Theorem RUTHS and higher VB-Lie algebroids} with Proposition \ref{Correspondence of morphisms}, we obtain the following result.
	
	\begin{theorem}\label{Equivalence of categories}
		There is an equivalence of categories between $\mathbb{DVB}\text{-Lie}_n(\A)$ and $\R\text{ep}^{n+1}(\A)$, and hence, an equivalence of categories $\mathbb{VB}\text{-Lie}_n(\A)$ and $\R\text{ep}^{n+1}(\A)$.
	\end{theorem}
	
	\begin{remark}
		The theorem above could be used to understand the adjoint representation of a split Lie $n$-algebroid $\A$ in classical differential geometric terms. Recall that Section \ref{Adjoint representation of a Lie n-algebroid} gave a method for deriving the structure objects of the adjoint representation, once there is a choice of splitting $\A$ and $TM$-connections on the vector bundles $A_i$ for all $i$. This could be very hard in practice if one were to work with a split Lie $n$-algebroid and its adjoint representation. Already in the case $n=2$ the structure objects tend to be long and complicated. The advantage of Theorem \ref{Equivalence of categories} is that it gives a more compact way to define the adjoint representation of a split Lie $n$-algebroid $\A$ for any $n\in\mathbb{N}$. It is given by the isomorphy class in $\text{VB-Lie}_n(\A)$ of the tangent prolongation $\underline{TA}$ from Example \ref{Tangent prolongation of a (split) Lie n-algebroid}.
	\end{remark}
	
	\section{Change of decomposition}\label{Section: Change of decomposition}
	
	In the previous section, we gave the equivalence of VB-Lie $n$-algebroids and $(n+1)$-representations in terms of a decomposition. Now we will see that two different choices of a decomposition for the VB-Lie $n$-algebroid $(\underline{D},V,\A,M)$ lead to isomorphic representations up to homotopy of $\A$.
	
	Suppose that $(\underline{D},V,\A,M)$ is a VB-Lie $n$-algebroid with core $\underline{C}$ and consider a pair of decompositions for $\underline{D}$ obtained by the horizontal lifts $h_i,h_i':\Gamma(A_i)\to\Gamma_V^l(D_i)$ for all the double vector bundles $(D_i,V,A_i,M)$. One then has the isomorphisms of decomposed double vector bundles
	\[
	G_i:V\times_M A_i \times_M C_i\to D_i \to V\times_M A_i \times_M C_i
	\] 
	defined by mapping $(v,a_i,c_i)\in V\times_M A_i\times_M C_i$ from the decomposition obtained by $h_i'$ to the element $(v,a_i,c_i + \Delta^i_{a_i}(v))$ in the decomposition of $h_i$, where $\Delta^i\in\Gamma(A_i^*\otimes V^*\otimes C_i)$ is given by the difference $h_i-h_i'$. Summing up all $G_i$, one obtains a morphism of decomposed double vector bundles
	\[
	G:V\times_M \A \times_M \underline{C}\to \D \to V\times_M \A \times_M \underline{C}
	\] 
	inducing the identity on $\A$. Therefore, it follows from Proposition \ref{Correspondence of morphisms} that $G$ corresponds to a morphism of $(n+1)$-representations of $\A$
	\[
	\mu:\Gamma(\Sym(\A^*))\otimes\Gamma(\underline{C}\oplus V)\to \Gamma(\Sym(\A^*))\otimes\Gamma(\underline{C}\oplus V).
	\]
	
	As usual, we denote the dual morphism by $\mu^\star$. Clearly, the homogeneous part $G^\star_0$ of $G^\star$ with respect to the decomposition $\underline{D}_V=q_V^!(\A\oplus\underline{C})$ is given by $\id_{\A^*\oplus\underline{C}^*}$. Recalling then the construction of the correspondence between morphisms of VB-Lie $n$-algebroids and $(n+1)$-representations, one sees that $\mu_0|_{\underline{C}^*}=\id_{\underline{C}^*}:\underline{C}^*\to\underline{C}^*$. In addition, given a section $\gamma_i\in\Gamma(C_i^*)$, one has that
	\[
	G^\star(\gamma_i) = \gamma_i + (\Delta^i)^*(\gamma_i) = \gamma_i + \Big(h_i^*(\gamma_i) - h_i'^*(\gamma_i)\Big)
	\]
	Finally, using the second equation in (\ref{Equations for correspondence of morphisms}), we compute
	\[
	\langle \mu_0^*|_{V^*}(\psi(m)),G(v_m,a^i_m,c^i_m) \rangle = G^\star(\ell_\psi)(v_m,a^i_m,c^i_m) =
	\langle \psi(m),G(v_m,a^i_m,c^i_m) \rangle
	\]
	for a point $(v_m,a^i_m,c^i_m)\in\underline{D}$ in the decomposition of $h'$ and a section $\psi\in\Gamma(V^*)$. Therefore, we also obtain that $\mu_0^*|_{V^*}=\id_{V^*}:V^*\to V^*$ and we conclude that $\mu^\star = \id_{\underline{C}^*\oplus V^*} \oplus\ \Delta^*$, or equivalently, $\mu = \id_{V\oplus\underline{C}} \oplus\ \Delta$. 
	
	One may use now the same construction for the isomorphism of decomposed double vector bundles $\widetilde{G}$ given by the difference $\widetilde{\Delta} = h' - h$ and obtain another morphism of representations up to homotopy $\widetilde{\mu}$ from the decomposition of $h$ to the decomposition of $h'$. Observe now that, by construction, the maps $G^\star$ and $\widetilde{G}^\star$ act trivially on the space $\Gamma(\underline{A}^*\otimes V^*)$. Since $(\Delta^i)^*(\gamma_i)\in\Gamma(\A^*\otimes V^*)$ and $(\widetilde{\Delta}^i)^*(\gamma_i)\in\Gamma(\A^*\otimes V^*)$, it follows that $\widetilde{G}^\star((\Delta^i)^*(\gamma_i))=(\Delta^i)^*(\gamma_i)$ and $G^\star((\widetilde{\Delta}^i)^*(\gamma_i))=(\widetilde{\Delta}^i)^*(\gamma_i)$ for all $\gamma_i\in\Gamma(C_i^*)$, and similarly for $h_i'$. We conclude that the morphisms $\widetilde{\mu}$ and $\mu$ are inverses to each other and therefore we have proved the following result.
	
	\begin{proposition}\label{Isomorphism of A-representations for decompositions of VB-Lie n-algebroids}
		Suppose $(\underline{D},V,\A,M)$ is a VB-Lie $n$-algebroid. Let $h_i$ and $h_i'$ be two decompositions of the double vector bundle $(D_i,V,A_i,M)$ for all $i=1,\ldots,n$, and denote their difference $\Delta^i=h_i-h_i'$. Then the two corresponding $(n+1)$-representations of $\A$ on the graded vector bundle $V[0]\oplus \underline{C}$ are isomorphic via the maps $\mu = \id_{V[0]\oplus\underline{C}} \oplus\ \Delta$ and $\mu^{-1} =  \id_{V[0]\oplus\underline{C}} \oplus\ (-\Delta)$. 
	\end{proposition}
	
	\begin{remark}\label{Transformation of adjoint representation of split Lie n-algebroid}
		Note that in the case of the double vector bundle $(\underline{TA},TM,\A,M)$ obtained by the tangent prolongation of a Lie $n$-algebroid $\A$ over $M$, decompositions correspond to $TM$-connections on the vector bundles $A_i\to M$. The result above then reads that the two representatives of the adjoint representation $\Gamma(\underline{S}(\A^*))\otimes\Gamma(\ad_\nabla)=\Gamma(\underline{S}(\A^*))\otimes\Gamma(TM[0]\oplus\A)=\Gamma(\underline{S}(\A^*))\otimes\Gamma(\ad_{\nabla'})$ of $\A$ corresponding to the two choices of $TM$-connections $\nabla$ and $\nabla'$ are related by the isomorphisms of $\A$-representations $\mu\colon\Gamma(\underline{S}(\A^*))\otimes\Gamma(\ad_{\nabla})\to\Gamma(\underline{S}(\A^*))\otimes\Gamma(\ad_{\nabla'})$ and $\mu^{-1}\colon\Gamma(\underline{S}(\A^*))\otimes\Gamma(\ad_{\nabla'})\to\Gamma(\underline{S}(\A^*))\otimes\Gamma(\ad_{\nabla})$ given by $\mu = \id_{\ad_\nabla} \oplus \Big( \nabla' - \nabla \Big)$ and $\mu^{-1} = \id_{\ad_{\nabla'}} \oplus \Big( \nabla - \nabla' \Big)$.
	\end{remark}

	\chapter{Conclusion: open problems and further research}\label{Chapter: Conclusion, open problems and further research}
	
	In this chapter, we summarise the content of the main body of this thesis and present briefly some ideas for future research topics that are related to our results.
	
	\paragraph*{Summary}
	
	Chapters \ref{Chapter: Preliminaries} and \ref{Chapter: Z- and N-graded supergeometry} laid the mathematical foundations in order for the reader to gain familiarity with our notation and conventions, and understand the structures that are important for the research results lying at the core of this thesis. This included classical differential geometric structures such as graded vector bundles, complexes of vector bundles, various notions of algebroids, double vector bundles, and sheaves on topological spaces, as well as supergeometric notions such as $\mathbb{Z}$- and $\mathbb{N}$-graded manifolds, vector bundles in the category of graded manifolds, $\Q$-structures and Lie $n$-algebroids, graded Poisson and symplectic structures on graded manifolds.
	
	In Chapter \ref{Chapter: Graded tangent and cotangent bundles}, we moved to analysing further vector bundles in the category of graded manifolds, and focused on the two most important examples that are relevant to our work, namely the tangent and the cotangent bundle of a graded manifold. We presented numerous well-known geometric constructions on these vector bundles, such as (pseudo)multivector fields and (pseudo)differential forms, and the Weil and the Poisson-Weil algebras. In addition, we reviewed homotopy Poisson structures explaining how they serve as a unified framework for Poisson brackets and $\Q$-structures on graded manifolds, and recalled the correspondence of bialgebroids with Poisson Lie 1-algebroids, and the correspondence of Poisson manifolds and Courant algebroids with symplectic Lie 1- and symplectic Lie 2-algebroids, respectively.
	
	In Chapter \ref{Chapter: Differential graded modules}, we developed the theory of differential graded modules (DG-modules) for $\Q$-manifolds and Lie $n$-algebroids, and defined the two fundamental examples, i.e.~the adjoint and the coadjoint module (Section \ref{Section: Adjoint and coadjoint modules}). In particular, Sections \ref{Section PQ-manifolds: coadjoint vs adjoint modules} and \ref{Section: PQ-manifolds: Weil vs Poisson-Weil algebras} constructed the morphism relating these two DG-modules (Theorem \ref{thm_poisson}) and the Weil with the Poisson-Weil algebras (Theorem \ref{thm_Weil_Poisson-Weil}), in the case of $\mathcal{P}\Q$-manifolds. This result was applied to Courant algebroids in order to give an alternative description in terms of its adjoint and coadjoint modules.
	
	Thereafter Chapter \ref{Chapter: Representations up to homotopy} defined representations up to homotopy of Lie $n$-algebroids and, similarly to the case of DG-modules, constructed the adjoint and the coadjoint representations (Section \ref{Adjoint representation of a Lie n-algebroid}). In the case of split Lie $2$-algebroids, 3-term representations were studied in detail (Proposition \ref{3-term_representations}) and the explicit formulae for the structure objects of the adjoint and the coadjoint representations were given (Proposition \ref{Adjoint_representation_of_Lie_2-algebroid} and and Section \ref{adjoint_module_adjoint_representation_isomorphism}). Additionally, the coordinate transformation of the adjoint representation of a split Lie $2$-algebroid was computed (Section \ref{Section: Coordinate transformation of the adjoint representation}) and the explicit map connecting the adjoint module with the adjoint representation of a general Lie $n$-algebroid was established (Section \ref{adjoint_module_adjoint_representation_isomorphism}). Section \ref{Section: The Weil algebra of a Lie n-algebroid} analysed the Weil algebra of a split Lie $n$-algebroid expressing its differentials in terms of the coadjoint representation and $n$-many double representations of the tangent Lie algebroid of the base smooth manifold, and finally Section \ref{Section: Poisson Lie algebroids of low degree} computed explicitly the map between the coadjoint and adjoint representations of a split Poisson Lie $n$-algebroid, for $n=0,1,2$.
	
	Chapters \ref{Chapter: Linear structures on vector bundles} and \ref{Chapter: Higher split VB-algebroid structures} dealt with linear structures of graded manifolds. More precisely, Chapter \ref{Chapter: Linear structures on vector bundles} used supergeometric language to put the notions of vector fields, $\Q$-structures, Poisson brackets, and homotopy Poisson structures in the context of vector bundles in the category of graded manifolds. On the other hand, Chapter \ref{Chapter: Higher split VB-algebroid structures} analysed linear $\Q$-structures in the language of classical differential geometry. In particular, it developed the notion of (split) VB-Lie $n$-algebroids as Lie $n$-algebroids that are linear over another Lie $n$-algebroid. These were described both classically using higher brackets (Section \ref{Section: Classical interpretation}) and supergeometrically using a homological vector field (Sections \ref{Section: Supergeometric interpretation} and \ref{Section: The Weil algebra of a split VB-Lie n-algebroid}). We defined the induced fat Lie $n$-algebroid of VB-Lie $n$-algebroid in Section \ref{Section: The fat Lie n-algebroid} and proved the equivalence of categories between VB-Lie $n$-algebroids with fixed side and $(n+1)$-term representations up to homotopy of its side Lie $n$-algebroid in Section \ref{Section: Split VB-Lie n-algebroids and (n+1)-representations} (Theorem \ref{Theorem RUTHS and higher VB-Lie algebroids}, Proposition \ref{Correspondence of morphisms} and Theorem \ref{Equivalence of categories}). Finally, in Section \ref{Section: Change of decomposition} we saw how different choices of decompositions of the VB-Lie $n$-algebroid lead to isomorphic representations up to homotopy of its side (Proposition \ref{Isomorphism of A-representations for decompositions of VB-Lie n-algebroids}) and applied this to express the transformation of the adjoint representation of a split Lie $n$-algebroid under different choices of linear connections (Remark \ref{Transformation of adjoint representation of split Lie n-algebroid}), extending the corresponding result of Proposition \ref{Isomorphism with change of connections} to the general case of $n\in\mathbb{N}$.
	
	\paragraph*{DG-modules and representations of Lie $n$-groupoids}
	As it was mentioned in the introduction, Lie $n$-groupoids are the global objects corresponding to Lie $n$-algebroids \cite{CaFe01,CrFe03}. The natural question that arises is whether the constructions developed in this thesis about Lie $n$-algebroids could be carried over to Lie $n$-groupoids. In the following, we give more details about this idea.
	
	\begin{definition*}
	A \textbf{groupoid}\index{groupoid} is a small category such that every morphism is an isomorphism. Explicitly, a groupoid $\mathcal{G} \rightrightarrows M$ consists of a set of \textbf{objects}\index{groupoid!objects} $M$ (with no extra structure a priori) and a set of \textbf{arrows}\index{groupoid!arrows} $\mathcal{G}$ together with the following structure maps:
	\begin{enumerate}
	\item Two maps $s,t\colon\mathcal{G}\to M$ called \textbf{source}\index{groupoid!source} and \textbf{target}\index{groupoid!target}, respectively.
	
	\item A \textbf{multiplication map}\index{groupoid!multiplication map} $m\colon\mathcal{G}_2\to \mathcal{G},(g,h)\mapsto gh$, where
	\[
	\mathcal{G}_2 := \{ (g,h)\in \mathcal{G}\times \mathcal{G}\ |\ s(g) = t(h) \},
	\]
	that satisfies
	\begin{enumerate}
	\item Compatibility with $s$ and $t$: for $(g,h)\in \mathcal{G}_2$
	\[
	s(gh) = s(h)
	\qquad \text{and} \qquad
	t(gh) = t(g)
	\]
	
	\item Associativity: for $(g,h,k)\in G_3$
	\[
	g(hk) = (gh)k,
	\]
	where in general $\mathcal{G}_0:=M,\mathcal{G}_1:=\mathcal{G}$ and for $k\geq2$
	\[
	\mathcal{G}_k := \{ (g_1,\ldots,g_k)\ |\ g_j\in \mathcal{G}_1\ \text{and}\ s(g_j) = t(g_{j+1})\ \text{for all}\ j \}.
	\]
	\end{enumerate}
	
	\item An embedding $1\colon M\to \mathcal{G},x\to 1_x$ called \textbf{identity map}\index{groupoid!identity map}, such that for all $g\in \mathcal{G}$
	\[
	g1_{s(g)} = g = 1_{t(g)}g
	\]
	and $s(1_x) = x = t(1_x)$ for all $x\in M$.
	
	\item An \textbf{inversion map}\index{groupoid!inversion map} $i\colon\mathcal{G}\to \mathcal{G},g\mapsto g^{-1}$ such that for all $g\in \mathcal{G}$
	\[
	s(g)^{-1} = t(g)
	\qquad
	t(g)^{-1} = s(g)
	\qquad
	g^{-1}g = 1_{s(g)},
	\qquad
	gg^{-1} = 1_{t(g)}.
	\]
	\end{enumerate}
	A \textbf{Lie groupoid\index{Lie groupoid}} is a groupoid $\mathcal{G}\rightrightarrows M$ such that $M$ and $\mathcal{G}$ are smooth manifolds (with $\mathcal{G}$ not necessarily Hausdorff), $M$ is a closed submanifold of $\mathcal{G}$, all the structure maps are smooth, and the source and target maps are submersions.
	\end{definition*}
	
	Every groupoid $\mathcal{G}\rightrightarrows M$ comes with the following simplicial manifold (see e.g.~\cite{Zhu09,ArCr13}) called its \textbf{nerve}\index{Lie groupoid!nerve}: the manifold of $k$-simplices is $\mathcal{G}_k$ with face maps $\diff^k_i\colon\mathcal{G}_k\to\mathcal{G}_{k-1}$
	\[
	\diff_i^k(g_1,\ldots,g_k) = \begin{cases}
		(g_2,\ldots,g_k), & \text{if}\ i=0 \\
		(g_1,\ldots,g_ig_{i+1},\ldots,g_k), & \text{if}\ 0<i<k \\
		(g_1,\ldots,g_{k-1}), & \text{if}\ i=k \\
	\end{cases}
	\]
	and degeneracy maps $s_i^k\colon\mathcal{G}_k\to\mathcal{G}_{k+1}$
	\[
	s_i^k(g_1,\ldots,g_k) = (g_1,\ldots,g_i,1,g_{i+1},\ldots,g_k)
	\]
	for $0\leq i\leq k$. The spaces $C^k(\mathcal{G}):=C^\infty(\mathcal{G}_k)$ are equipped with the differential $\diff_\mathcal{G}$ given by the alternating sum of the face maps $\diff_i$, and hence, $(C^k(\mathcal{G}),\diff_\mathcal{G})$ is a complex whose cohomology is denoted $H^\bullet(\mathcal{G})$. Moreover, the space $C^\bullet(\mathcal{G})$ carries a natural algebra structure given by the multiplication map
	\[
	(f*h)(g_1,\ldots,g_{k+\ell}) := (-1)^{k\ell}f(g_1,\ldots,g_k)h(g_{k+1},\ldots,g_{k+\ell})
	\]
	for $f\in C^k(\mathcal{G})$ and $h\in C^\ell(\mathcal{G})$. The triple $(C^k(\mathcal{G}),*,\diff_\mathcal{G})$ has the structure of a differential graded algebra, i.e.~one has the identity
	\[
	\diff_\mathcal{G}(f*h) = \diff_\mathcal{G}f*h + (-1)^kf\diff_\mathcal{G}h.
	\]
	One uses the construction above to define higher versions of Lie groupoids as follows: A \textbf{Lie $n$-groupoid}\index{Lie $n$-groupoid} is as a simplicial manifold $\mathcal{G}$ that satisfies the Kan conditions (see \cite{Zhu09b} for details).
	
	Suppose now that $\mathcal{G}\rightrightarrows M$ is a Lie groupoid. Given a vector bundle $E\to M$, we define the graded vector space $C^\bullet(\mathcal{G},E)$ which in degree $k$ is given by $C^k(\mathcal{G},E):=\Gamma(\mathcal{G}_k, t^*E)$. This graded space is a (right) graded $C^\bullet(\mathcal{G})$-module whose structure is given by
	\[
	(\eta*f)(g_1,\ldots,g_{k+\ell}) := (-1)^{k\ell} \eta(g_1,\ldots,g_k)f(g_{k+1},\ldots,g_\ell)
	\]
	for $\eta\in C^k(\mathcal{G},E)$ and $f\in C^\ell(\mathcal{G})$. Similarly as for Lie algebroids, one defines a \textbf{representation}\index{Lie $n$-groupoid!representation} of the Lie groupoid $\mathcal{G}\rightrightarrows M$ on the vector bundle $E\to M$ to be a degree $1$ operator $\D$ on $C^\bullet(\mathcal{G},E)$ that satisfies the Leibniz identity
	\[
	\D(\eta*f) = \D(\eta)*f + (-1)^{|\eta|}\eta*\diff_\mathcal{G}f.
	\]
	In the case of a graded vector bundle $\E=\bigoplus_i E_i[i]\to M$, one considers the $\E$-valued cochains of $\mathcal{G}$ with respect to the total grading given by
	\[
	C^k(\mathcal{G},\E) = \bigoplus_{i-j=k}C^i(\mathcal{G},E_{j}).
	\]
	A degree $1$ operator on this space which satisfies the above Leibniz rule is called \textbf{representation up to homotopy}\index{Lie $n$-groupoid!representation up to homotopy} of the Lie groupoid $\mathcal{G}\rightrightarrows M$ on the graded vector bundle $\E\to M$ (for more details see, e.g., \cite{ArCr13,Arias08t,GrMe17,Stefani19}).
	
	One could generalise the constructions above to the world of Lie $n$-groupoids and try to prove the results of this thesis in this setting. More precisely, one could define DG-modules over the differential graded algebra $(C^\bullet(\mathcal{G}),\diff_{\mathcal{G}})$ given by a Lie $n$-groupoid $\mathcal{G}$ and representations up to homotopy of $\mathcal{G}$, and prove that they are, in fact, equivalent. The first step for $n=1$ is done in \cite{Stefani19}, where it is proved that representations up to homotopy of a  Lie groupoid $\mathcal{G}\rightrightarrows M$ are the same as \textbf{cohesive modules}\index{cohesive module}\index{Lie groupoid!cohesive module} over the differential graded algebra $(C^\bullet(\mathcal{G}),\diff_{\mathcal{G}})$, i.e.~(bounded) graded right modules $\Em$ over $(C^\bullet(\mathcal{G}),\diff_{\mathcal{G}})$ which are finitely generated, projected, and equipped with a structure differential. One then could focus on the adjoint module and the adjoint representation of $\mathcal{G}$ and describe explicitly the isomorphism between them. Further, it would be interesting to understand when one can find a global analogue of Theorem \ref{thm_poisson}, i.e.~a connection of Poisson structures on the Lie $n$-groupoid with a natural morphism between its coadjoint and adjoint modules. Lastly, the question that arises is whether and how one can pass from the construction on Lie $n$-groupoids to the corresponding constructions on Lie $n$-algebroids and vice versa, i.e.~differentiation and integration of modules and representations up to homotopy of Lie $n$-groupoids.
	
	\begin{comment}
	\paragraph*{$\mathcal{PQ}$-structures induced by $\pi^\sharp$ and deformations of Courant algebroids}
	
	In Theorem \ref{thm_poisson}, we proved that the tripple  $(\M,\Q,\{\cdot\,,\cdot\}_k)$ is a $\mathcal{PQ}$-manifold if an only if $\pi^\sharp:\Omega^1(\M)\to\mathfrak{X}(\M)$ is an (anti-)morphism of DG-modules, provided that $\M$ is equipped with the homological vector field $\Q$ and the Poisson bracket $\{\cdot\,,\cdot\}_k$. The next question that could be adressed is the opposite; that is, given a $\Q$-manifold $(\M,\Q)$, what conditions should a map $\mu:\Omega^1(\M)\to\mathfrak{X}(\M)$ satisfy, so that $\mu = \pi^\sharp$ for a $\mathcal{PQ}$-structure $(\M,\Q,\{\cdot\,,\cdot\}_k)$? 
	\end{comment}
	
	\paragraph*{Double $\Q$-manifolds, ideals and reduction}
	
	Recall that in chapters \ref{Chapter: Linear structures on vector bundles} and \ref{Chapter: Higher split VB-algebroid structures}, we developed the notion of linear $\Q$-manifolds, or in the split setting the notion of higher VB-algebroids. These were essentially described as commutative diagrams of the form
	\[
	\begin{tikzcd}
		\underline{D} \arrow[dd] \arrow[rr] &  & \underline{A} \arrow[dd] \\
		&  &                          \\
		V \arrow[rr]                        &  & M                       
	\end{tikzcd}
	\]
	together with a homological vector field on bidegree $(0,1)$ on the graded manifold $\underline{D}[1]$ (see Theorem \ref{VB-Lie n-algebroids via the Weil algebra}). One could generalise the work of Voronov \cite{Voronov12} and consider pairs of commuting homological vector fields $(\Q_1,\Q_2)$ on the graded manifold $\underline{D}[1]$ of bidegree $(1,0)$ and $(0,1)$, respectively; that is, $[\Q_1,\Q_2] = 0$. This would lead to the notion of double structures in graded geometry, i.e.~\textbf{double $\Q$-manifolds}\index{double $\Q$-manifold} and \textbf{double Lie $n$-algebroids}\index{double Lie $n$-algebroid}. 
	
	The basic example of such double structures should be the tangent prolongation of a (split) Lie $n$-algebroid $\underline{A}\to M$. Next, one could consider subobjects of the tangent double Lie $n$-algebroid $(T\A,TM,\A,M)$ which in turn would form a generalisation of the notion of \textit{infinitesimal ideal systems} of Lie $n$-algebroids (IIS) from \cite{JoOr14}. Infinitesimal ideal systems of a Lie $n$-algebroid are expected to be linked to subrepresentations of the adjoint representation (up to homotopy) $\ad$, similarly to the work done in \cite{DrJoOr15}. The final step in this direction would be to define the quotient of the Lie $n$-algebroid $\underline{A}\to M$ with its IIS and obtain a Lie $n$-algebroid of smaller dimension over a quotient manifold, as it is proved in \cite{JoOr14} for Lie algebroids. Note that this base manifold should be the leaf space of $M$ with respect to a foliation given by the IIS. Consequently, one would have a well-defined notion of reduction of Lie $n$-algebroids for general $n\in\mathbb{N}$.

	\begin{appendices}
	\appendixpage
	\noappendicestocpagenum
	\addappheadtotoc
	
	\chapter{More on graded geometry}\label{Chapter: More on graded geometry}
	
	\section{The geometrisation of $\mathbb{N}$-graded vector bundles over $\mathbb{N}$-manifolds}\label{Appendix: The geometrisation of N-graded vector bundles}
	
	A \textbf{$\mathbb{Z}$-graded double vector bundle}\index{graded double vector budnle} is a double vector bundle of the form $(\underline{D}:=\bigoplus_{i\in\mathbb{Z}} D_i[i],V,\A:=\bigoplus_{i\in\mathbb{Z}} A_i[i],M)$:
	\begin{center}
		\begin{tikzcd}
		\underline{D} \arrow[rr] \arrow[dd] & & \A \arrow[dd] \\
		& & \\
		V \arrow[rr]      &       & M            
		\end{tikzcd}
	\end{center}
	with core $\underline{C}:=\bigoplus_{i\in\mathbb{Z}} C_i[i]$, where the (finite) direct sums are taken over the bottom arrow $V\to M$, such that each $(D_i,V,A_i,M)$ is a double vector bundle with core $C_i$. In fact, since every double vector bundle $(D,V,A,M)$ with core $C$ is non-canonically isomorphic to a decomposed double vector bundle $A\times_M V\times_M C$, it follows that every graded double vector bundle $(\underline{D},V,\A,M)$ is non-canonically isomorphic to a \textbf{decomposed graded double vector bundle}\index{graded double vector budnle!decomposed} $\A\times_M V\times_M \underline{C}$.
	
	\begin{proposition}
		There is an one-to-one correspondence between $\mathbb{N}$-graded double vector bundles of degree $n$ and vector bundles over $\n$-manifolds whose total space is also an $\n$-manifold.
	\end{proposition}
	
	\begin{proof}
		Given an $\mathbb{N}$-graded double vector bundle $\underline{D}=\bigoplus_{i=1}^nD_i[i]$ of degree $n$, one has the associated vector bundle in the category of graded manifolds $\Em_{\underline{D}}\to\M$, where the base is given by the split $\n$-manifold $\M$ with sheaf $\cin(\M)=\Gamma(\underline{S}(\A^*))$ and the sheaf of sections is given by $\Gamma(\Em_{\underline{D}})=\cin(\M)\otimes\Gamma(C_n[n]\oplus\ldots\oplus C_{1}[1]\oplus V[0])$.
		
		Conversely, suppose $\Em\to\M$ is a vector bundle in the category of graded manifolds, such that both $\Em$ and $\M$ are $\n$-manifolds. Choose splittings of $\M$ and $\Em$ as $\cin(\M)=\Gamma(\underline{S}(\A^*))$ and $\Gamma(\Em)=\Gamma(\underline{S}(\A^*))\otimes\Gamma(\E)$, for some negatively graded vector bundles $\A:=\bigoplus_{i=1}^nA_i[i]\to M$ and $\E:=\bigoplus_{i=1}^nE_i[i]\to M$, respectively.
		Then one obtains the family of decomposed double vector bundles $D_i:=A_i\times_M E_0\times_M E_i$, for $i=1,\ldots,n$. Summing over all such $i$ yields a decomposed $\mathbb{N}$-graded double vector bundle $\underline{D}_\Em:=\bigoplus_{i=1}^nD_i[i]$ of degree $n$. Since any vector bundle over an $\mathbb{N}$-manifold admits a splitting and any graded double vector bundle admits a decomposition, it follows that the two procedures above are inverse to each other.
	\end{proof}
	
	\section{Cartan calculus on $\mathbb{Z}$-graded manifolds}\label{Cartan calculus on Z-graded manifolds}
	
	In this appendix, we recall some important formulae from \cite{Mehta06} that generalise the known results from the Cartan calculus on ordinary manifolds and are used extensively in this thesis.
	
	In what follows, $\{\xi^j\}$ are local coordinates on the $\mathbb{Z}$-graded manifold $\M$, $\{\xi^j,\dot{\xi}^j\}$ are the corresponding coordinates on $T[1]\M$, and $\xi,\zeta^j\in\cin(\M),\X,\Y\in\mathfrak{X}(\M)$ are all homogeneous. Recall that $\diff\in\mathfrak{X}^1(T[1]\M)$ sends $\xi^j$ to $\diff\xi^j=\dot{\xi}^j$ and $\dot{\xi}^j$ to $0$, and thus it follows immediately that $\diff^2=0$ and $[\diff,\diff] = \diff^2 + \diff^2 = 0$.
	Recall also the vector field $i_\X\in\mathfrak{X}^{|\X|-1}(T[1]\M)$ which sends $\xi^j$ to $0$ and $\dot{\xi}^j$ to $\X(\xi^j)$, and the Lie derivative $\ldr{\X}=[i_\X,\diff]$. As we explained in Section \ref{Section: (Pseudo)differential forms and the Weil algebra of a Q-manifold}, locally these vector fields have the following form
	\[
	\diff = \dot{\xi}^j\frac{\partial}{\partial\xi^j},\qquad 
	i_\X = \zeta^j\frac{\partial}{\partial\dot{\xi}^j},\qquad 
	\ldr{\X} = \zeta^j\frac{\partial}{\partial \xi^j} + (-1)^{|\X|}\diff\zeta^j\frac{\partial}{\partial\dot{\xi}^j},
	\]
	where $\X=\zeta^j\frac{\partial}{\partial\xi^j}$. This implies that
	\[
	i_{\xi\X} = \xi\, i_\X
	\qquad \text{and} \qquad
	\ldr{\xi \X} = \xi\,\ldr{\X} + (-1)^{|\xi|+|\X|}\diff\xi\, i_\X.
	\]
	
	\begin{proposition}
		The following identities hold:
		\begin{enumerate}
			\item $[i_\X,i_\Y] = 0$,
			\item $[\ldr{\X},i_\Y] = i_{[\X,\Y]}$,
			\item $[\ldr{\X},\ldr{\Y}] = \ldr{[\X,\Y]}$,
			\item $[\diff,\ldr{\X}] = 0$.
		\end{enumerate}
	\end{proposition}    
	\begin{proof}
		Write $\X=\zeta^j\frac{\partial}{\partial\xi^j},\Y=\eta^i\frac{\partial}{\partial\xi^i}$. Using that any two vector fields on $T[1]\M$ are equal if they agree on functions of the form $\xi$ and $\diff\xi$, we compute:  
		\[
		[i_\X,i_\Y](\xi) = i_\X(i_\Y(\xi)) - (-1)^{(|\X|-1)(|\Y|-1)}i_\Y(i_\X(\xi)) = 0
		\]
		and similarly
		\[
		[i_\X,i_\Y](\diff\xi) = i_\X(\Y(\xi)) - (-1)^{(|\X|-1)(|\Y|-1)}i_\Y(\X(\xi)) = 0.
		\]
		This proves the first identity. For the second we compute
		\[
		[\ldr{\X},i_\Y](\xi) = -(-1)^{|\X|(|\Y|-1)}i_\Y(\ldr{\X}(\xi)) = 0 = i_{[\X,\Y]}(\xi) 
		\]
		and
		\begin{align*}
		[\ldr{\X},i_\Y](\diff\xi) = &\ \ldr{\X}(\Y(\xi)) - (-1)^{|\X|(|\Y|-1)}i_\Y(\ldr{\X}(\diff\xi)) \\
		= &\ \X(\Y(\xi)) - (-1)^{|\X||\Y|}\Y(\X(\xi)) \\
		= &\ [\X,\Y](\xi) \\
		= &\ i_{[\X,\Y]}(\diff\xi).
		\end{align*}
		For the third identity, clearly we have
		\[
		[\ldr{\X},\ldr{\Y}](\xi) = \X(\Y(\xi)) - (-1)^{|\X||\Y|}\Y(\X(\xi)) = \ldr{[\X,\Y]}(\xi)
		\]
		but also
		\begin{align*}
		[\ldr{\X},\ldr{\Y}](\diff\xi) = &\ -(-1)^{|\Y|-1}\ldr{\X}(\diff\Y(\xi)) + (-1)^{|\X||\Y|+|\X|-1}\ldr{\Y}(\diff\X(\xi)) \\
		= &\ (-1)^{|\X|+|\Y|}\diff(\X(\Y(\xi))) - (-1)^{|\X|+|\Y|+|\X||\Y|}\diff(\Y(\X(\xi))) \\
		= &\ (-1)^{|\X|+|\Y|}\diff([\X,\Y](\xi)) \\
		= &\ \ldr{[\X,\Y]}(\diff\xi).
		\end{align*}
		To see the last identity we compute
		\[
		[\diff,\ldr{\X}](\xi) = \diff\X(\xi) - (-1)^{|\X|}\ldr{\X}(\diff\xi) = \diff\X(\xi) - \diff\X(\xi) = 0
		\]
		and
		\[
		[\diff,\ldr{\X}](\diff\xi) = \diff(\ldr{\X}(\diff\xi)) = -(-1)^{|\X|-1}\diff^2\X(\xi) = 0.\qedhere
		\]
	\end{proof}
	
	\section{Characteristic classes of $1^{\text{st}}$ order}
	
	In this appendix, we review some constructions that are relevant to the theory of representations up to homotopy of Lie $n$-algebroids. Namely, we consider a relaxed version of representations $\D$ forgetting the condition $\D^2=0$ (called connections up to homotopy) and use them to define the notion of the Pontryagin algebra (also know as characteristic algebra) of a graded vector bundle with respect to a Lie $n$-algebroid. The case of ordinary Lie algebroids can be found in \cite{Fernandes02,Quillen85,Tu17,Bott72}.
	
	First, we briefly recall some facts from the classical $A$-Pontryagin algebra of a vector bundle $E\to M$, where $A\to M$ is a Lie algebroid with differential $\diff_A$. Given an $A$-connection $\nabla$ on the vector bundle $E\to M$, one has the curvature operator $R_\nabla\in\Omega^2(A,\End(E))$ and consequently its $i$-th powers $R_\nabla^i:=R_\nabla\circ\ldots\circ R_\nabla\in\Omega^{2i}(A,\End(E))$ for all $i\geq1$. The trace of the vector-valued $2i$-forms $R^i_\nabla$ is $\diff_A$-closed and therefore we obtain the cohomology classes $[\tr(R^i_\nabla)]\in H^i(A)$. It is a standard result proved in \cite{Fernandes02}, that the cohomology class $[\tr(R^i_\nabla)]$ does not depend on the choice of the $A$-connection $\nabla$ on $E$, for all $i\geq1$. The \textbf{$A$-Pontryagin algebra}\index{Pontryagin algebra} of $E$ is the $\mathbb{R}$-subalgebra $\Pont_A^\bullet(E)\subset H^\bullet(A)$ generated by the cohomology classes $\sigma_A^i(E) := [\tr(R^i_\nabla)]$.
	
	\begin{remark}
	The Pontryagin algebra $\Pont_A^\bullet(E)$ may be viewed as the image of a suitable $\mathbb{R}$-algebra morphism with target $H^{2\bullet}(A)$, as follows: Let $\operatorname{Sym}^\bullet(\mathfrak{gl}_k(\mathbb{R}))^{\GL_k(\mathbb{R})}$ denote the set of $\GL_k(\mathbb{R})$-invariant polynomial functions over $\mathfrak{gl}_k(\mathbb{R})$, i.e.~polynomial functions $p\colon\mathfrak{gl}_k(\mathbb{R})\to \mathbb{R}$ such that $p(gXg^{-1}) = p(X)$ for $g\in\GL_k(\mathbb{R})$ and $X\in\mathfrak{gl}_k(\mathbb{R})$. As an $\mathbb{R}$-algebra, $\operatorname{Sym}^\bullet(\mathfrak{gl}_k(\mathbb{R}))^{\GL_k(\mathbb{R})}$ is generated by the polynomial functions $P_0,P_1,P_2,\ldots$, where $P_i(X) := \tr(X^i)$ for $X\in\mathfrak{gl}_k(\mathbb{R})$
	and $\tr$ is the usual trace operator (see \cite{Bott72}). Given now a Lie algebroid $A\to M$, a vector bundle $E\to M$ and an $A$-connection $\nabla$ on $E$, each polynomial $p\in\operatorname{Sym}^\bullet(\mathfrak{gl}_k(\mathbb{R}))^{\GL_k(\mathbb{R})}$ defines a closed form $p(R_\nabla)\in\Omega^\bullet(A)$ and hence a cohomology class $[p(R_\nabla)]\in H^\bullet(A)$. This yields the \textbf{Chern-Weil morphism}\index{Chern-Weil morphism} of $\mathbb{R}$-algebras
	\[
	\operatorname{Sym}^\bullet(\mathfrak{gl}_k(\mathbb{R}))^{\GL_k(\mathbb{R})}
	\to
	H^{2\bullet}(A),\qquad p\mapsto [p(R_\nabla)],
	\]
	whose image in $H^\bullet(A)$ is the Pontryagin algebra $\Pont_A(E)$ of the vector bundle $E\to M$.
    \end{remark}
		
	Now, we proceed to generalising the construction of the Pontryagin algebra in the context of general Lie $n$-algebroids, $n\in\mathbb{N}$. In what follows, $\A=A_1[1]\oplus\ldots\oplus A_n[n]$ will always be a split Lie $n$-algebroid with homological vector field $\Q$. 
	
	Suppose that $\E$ is a graded vector bundle over $M$. The \textbf{(graded) commutator}\index{graded commutator} of two homogeneous elements $\omega_1,\omega_2\in\Gamma(\underline{S}(\A^*))\otimes\Gamma(\GEnd(\E))$ is defined by
	\[
	[\omega_1,\omega_2] = \omega_1\circ\omega_2 - (-1)^{|\omega_1||\omega_2|}\omega_2\circ\omega_1.
	\]
	The \textbf{graded} or \textbf{super trace operator}\index{graded trace operator}\index{super trace operator} $\gtr\colon\Gamma(\GEnd^0(\E))\to C^\infty(M)$ is defined by
	\[
	\gtr(\Phi) = (-1)^i\tr(\Phi)
	\]
	for $\Phi\in\Gamma(\End(E_i))$; here, $\tr$ denotes the usual trace
	operator. The graded trace operator can be viewed as an element 
	of $\Gamma(\underline{S}(\A^*))\otimes\Gamma(\underline{\Hom}(\GEnd^0(\E),\mathbb{R}))$ of degree $0$ and
	hence, due to Lemma \ref{wedge_product-operators_Correspondence_Lemma}, there is a corresponding degree-preserving graded
	$\Gamma(\underline{S}(\A^*))$-linear map
	\[
	\gtr\colon\Gamma(\underline{S}(\A^*))\otimes\Gamma(\GEnd(\E)) \to \Gamma(\underline{S}(\A^*))
	\]
	which, by definition, vanishes on
	$\Gamma(\underline{S}^k(\A^*))\otimes\Gamma(\GEnd^i(\E))$ for all $k\geq0$ and
	$i\neq0$ (i.e.~$\gtr$ is non-trivial only on
	$\Gamma(\underline{S}(\A^*))\otimes\Gamma(\GEnd^0(\E))$). The proof of the following proposition can be carried over verbatim from the Lie algebroid case in \cite{Jotz19c}.
	
	\begin{proposition}\label{gtr_of_commutator_is_0}
		The following identity holds for all $\omega_1,\omega_2\in\CM\otimes\Gamma(\GEnd(\E))$:
		\[
		\gtr[\omega_1,\omega_2] = 0.
		\]
	\end{proposition}
	\begin{comment}
	\begin{proof}
	Note that for all
	$\Phi_1\in\Hom(E_i,E_j),\Phi_2\in\Hom(E_j,E_i)$, one has the
	identity $\tr(\Phi_1\circ\Phi_2) = \tr(\Phi_2\circ\Phi_1)$
	for the usual trace operator. This easily implies that
	$\gtr[\Phi_1,\Phi_2] = 0$. Moreover, note that by the
	definition of $\gtr$, the function $\gtr[\omega_1,\omega_2]$
	may be non-trivial only if
	$\omega_1\in\CM_k\otimes\Gamma(\Hom(E_i,E_j))$ and
	$\omega_2\in\CM_\ell\otimes\Gamma(\Hom(E_j,E_i))$, for some
	$k,\ell\geq0$ and $i,j\in\mathbb{Z}$. Hence, assume that
	$\omega_1=\xi_1\otimes\Phi_1$ and
	$\omega_2=\xi_2\otimes\Phi_2$ for some
	$\xi_1\in\CM_k,\xi_2\in\CM_\ell$ and
	$\Phi_1\in\Hom(E_i,E_j),\Phi_2\in\Hom(E_j,E_i)$. Then
	\[
	\gtr[\omega_1,\omega_2]=\gtr\left((-1)^{(j-i)\ell}\xi_1\xi_2\otimes[\Phi_1,\Phi_2]\right)
	=(-1)^{(j-i)\ell}\xi_1\xi_2\gtr[\Phi_1,\Phi_2]=0.\qedhere
	\]
	\end{proof}
	\end{comment}
	
	\begin{definition}
		A \textbf{connection up to
			homotopy}\index{connection up to
			homotopy} of $(\A,\Q)$ on $\E$ is a degree 1 operator
		
		\[
		\D\colon\Gamma(\underline{S}(\A^*))\otimes\Gamma(\E)\to\Gamma(\underline{S}(\A^*))\otimes\Gamma(\E)
		\]
		which satisfies the graded product rule
		\[
		\D(\xi\wedge\omega) = \Q(\xi)\wedge\omega + (-1)^{|\xi|}\xi\wedge\D(\omega)
		\]
		for all homogeneous $\xi\in\Gamma(\underline{S}(\A^*))$ and all $\omega\in\Gamma(\underline{S}(\A^*))\otimes\Gamma(\E)$. If $\E$ is concentrated only in $n$ degrees,
		i.e.~$\E = E_0[0]\oplus\ldots\oplus E_{n-1}[n-1]$, then a
		connection up to homotopy on $\E$ is called an
		\textbf{$n$-connection}\index{$n$-connection}.
	\end{definition}
	
	\begin{remark}
		Note that if $(\A,\Q)$ is just a Lie algebroid $A$ with
		differential $\diff_A = \Q$ and $\E = E$ is a usual (non-graded)
		vector bundle, then the notion of connection up to homotopy
		agrees with the usual notion of an $A$-connection on $E$.
	\end{remark}
	
	In supergeometric terms, a connection up to homotopy is a degree 1 derivation of the vector bundle over the $\n$-manifold $\A$ whose sheaf of sections is given by $\Gamma(\underline{S}(\A^*))\otimes\Gamma(\E)$. It follows that a connection up to homotopy $\D$ on $\E$ which
	is \textbf{flat}\index{connection up to homotopy!flat}, i.e.~such that $\D^2=0$, is a representation up to homotopy of $\A$. Moreover, given connections up to homotopy $\D_{\E}$ on $\E$ and $\D_{\F}$ on
	$\F$, the same construction as for the representations up to homotopy defines connections up to homotopy on the bundles
	$\underline{S}(\E), \GHom(\E,\F)$, etc. In particular, the graded vector bundle $\GEnd(\E)$ is endowed with the
	connection $\D_{\GEnd}$ given
	by the graded commutator
	\[
	\D_{\GEnd}(\Phi) = [\D_{\E},\Phi] := \D_{\E}\circ\Phi - (-1)^{|\Phi|}\Phi\circ\D_{\E}.
	\]
	
	The following three lemmas are crucial for defining the Pontryagin
	characters. Their proofs work as the proofs of the similar results in
	\cite{Jotz19c}, but are repeated here for the convenience of the
	reader.
	
	\begin{lemma}\label{lemma_decomposition_for_D}
		A connection up to homotopy $\D$ of the split Lie $n$-algebroid
		$\A$ on $\E$ can always be written as a sum
		\[
		\D = \diff_\nabla + \Theta,
		\]
		where $\nabla$ is a usual degree-preserving connection of the (skew-symmetric)
		dull algebroid $A_1\to M$ and $\Theta$ is an element in
		
		\[
		\bigoplus_{i+j=1,j\neq 0}\Gamma(\underline{S}^i(\A^*))\otimes\Gamma(\GEnd^j(\E)).
		\]
	\end{lemma}
	\begin{proof}
		Choose a degree-preserving $A_1$-connection
		$\nabla'\colon \Gamma(A_1)\times\Gamma(\E)\to\Gamma(\E)$. The degree 1
		operator $\D-\diff_\nabla'$ on $\Gamma(\underline{S}(\A^*))\otimes\Gamma(\E)$ is graded
		$\Gamma(\underline{S}(\A^*))$-linear and hence it is given as the wedge product with a
		degree 1 element $\Theta'\in\Gamma(\underline{S}(\A^*))\otimes\Gamma(\GEnd(\E))$. Writing
		$\Theta'$ in terms of its components
		\[
		\Theta' = \sum_{k\in\mathbb{Z}} \Theta'_k
		\in\bigoplus_{k\in\mathbb{Z}} \Gamma(\underline{S}^k(\A^*))\otimes\Gamma(\GEnd^{1-k}(\E)),
		\] 
		one easily checks that
		$\diff_{\nabla} := \diff_{\nabla'} +\ \Theta'_1$ is a new
		degree-preserving $A_1$-connection on $\E$ and
		$\D = \diff_{\nabla} +(\Theta'-\Theta_1')$.  The claim is
		proved with
		$\Theta=\Theta'-\Theta_1'\in \bigoplus_{k\in\mathbb{Z}, k\neq
			1} \Gamma(\Sym^k(\A^*))\otimes\Gamma(\GEnd^{1-k}(\E))$.
	\end{proof}
	
	\begin{lemma}\label{lemma_gtr_is_flat}
		Let $A\to M$ be a dull algebroid and $\diff_{\nabla}$ a
		degree-preserving $A$-connection on a graded vector bundle $\E$.
		The $A$-connections $\diff_{\nabla^{\GEnd}}$ on $\GEnd(\E)$
		and $\diff_{A}$ on $M\times\mathbb{R}$ give rise to the
		$A$-connection $\diff_{\nabla^{\GHom}}$ on the graded vector
		bundle $\GHom(\GEnd(\E),\mathbb{R})$.
		Then
		\[
		\diff_{\nabla^{\GHom}}\gtr = 0.
		\]
	\end{lemma}
	\begin{proof}
		Consider a local frame $\{ e_1^k,\ldots,e_{r_k}^k \}$ of $E_k$ for
		each $k$ and denote its dual frame on $E_k^*$ by
		$\{ \varepsilon_1^k,\ldots,\varepsilon_{r_k}^k \}$. Then, for all
		$a\in\Gamma(A_1)$ one computes
		\begin{align*}
		(\nabla^{\GHom}_a\gtr)(e_i^k\otimes\varepsilon_j^k) = & -\gtr(\nabla^{\GEnd}_a (e_i^k\otimes\varepsilon_j^k)) \\
		= & -(-1)^{k}\sum_{p=1}^{r_k}\langle \varepsilon_p^k,\nabla^{\GEnd}_a (e_i^k\otimes\varepsilon_j^k)(e_p^k) \rangle \\
		= &\ (-1)^{k+1} \sum_{p=1}^{r_k} \langle \varepsilon_p^k,\nabla_a(\delta_{jp}e_i^k) - e_i^k\langle \varepsilon_j^k,\nabla_a e_p^k \rangle \rangle \\
		= &\ 0.\qedhere
		\end{align*}
	\end{proof}
	\begin{lemma}\label{lemma_gtr_commutes_with_differentials}
		Let $\D$ be a connection up to homotopy of $(\A,\Q)$ on $\E$. Then
		\begin{equation}\label{trace_da}
		\gtr\circ\,\D_{\GEnd} = \Q\circ\gtr.
		\end{equation}
	\end{lemma}
	\begin{proof}
		Due to the derivation rule of $\Q$ and $\D_{\GEnd}$, and the
		$\Gamma(\Sym(\A^*))$-linearity of $\gtr$, it suffices to check the identity on
		graded endomorphism $\Phi\in\Gamma(\GEnd^i(\E))$ for all
		$i\in\mathbb{Z}$. Suppose first
		that $\Phi\in\Gamma(\GEnd^0(\E))$. This means that
		$\gtr(\Phi)\in\Gamma(\Sym^0(\A^*))=C^\infty(M)$ and thus
		$\Q(\gtr(\Phi)) = \diff_{A_1}(\gtr(\Phi))\in\Gamma(A_1^*)$.
		Write  $\D=\diff_{\nabla} + \Theta$ with $\nabla$ and $\Theta$ as in Lemma \ref{lemma_decomposition_for_D}
		and compute
		\[
		\D_{\GEnd}(\Phi) = \D\circ\Phi - (-1)^{|\Phi|}\Phi\circ\D = \diff_{\nabla^{\GEnd}}\Phi + [\Theta,\Phi],
		\]
		where $\diff_{\nabla^{\GEnd}}$ is the (degree-preserving)
		$A_1$-connection on $\GEnd(\E)$ induced by $\diff_{\nabla}$ on
		$\E$. Hence, by Proposition \ref{gtr_of_commutator_is_0}:
		
		\[
		\gtr(\D_{\GEnd}(\Phi)) = \gtr(\diff_{\nabla^{\GEnd}}\Phi).
		\]
		% The $A_1$-connections $\diff_{\nabla^{\GEnd}}$ on $\GEnd(\E)$
		% and $\diff_{A_1}$ on $M\times\mathbb{R}$ give rise to the
		% $A_1$-connection $\diff_{\nabla^{\GHom}}$ on the graded vector
		% bundle $\GHom(\GEnd(\E),\mathbb{R})$.
		Thus, Lemma \ref{lemma_gtr_is_flat} and the above give
		\begin{align*}
			\mathcal Q(\gtr\Phi)-\gtr(\D_{\GEnd}\Phi)= &\ \diff_{A_1}(\gtr\Phi) - \gtr(\D_{\GEnd}\Phi) \\
			= &\ \diff_{A_1}(\gtr\Phi) - \gtr(\diff_{\nabla^{\GEnd}}\Phi) \\
			= &\ (\diff_{\nabla^{\GHom}}\gtr)(\Phi) = 0.
		\end{align*}
		
		Suppose now that $\Phi\in\Gamma(\GEnd^i(\E))$ with
		$i\neq0$. By definition, $\Q(\gtr(\Phi))=\Q(0)=0$. On
		the other hand, the decomposition
		\[\D_{\GEnd}\Phi=\diff_{\nabla^{\GEnd}}\Phi+ [\Theta,\Phi]\in \Big(\Gamma(A_1^*)\otimes\Gamma(\GEnd^i(\E))\Big)\oplus \Gamma(\GEnd^{i+1}(\E))
		\]
		yields
		\[
		\gtr(\D_{\GEnd}(\Phi)) = \gtr(\diff_{\nabla^{\GEnd}}\Phi) =0
		\]
		since $i\neq 0$ and the graded trace of $[\Theta,\Phi]$ vanishes
		by Proposition \ref{gtr_of_commutator_is_0}.
	\end{proof}
	
	Let $(\M,\Q)$ be a Lie $n$-algebroid with a connection up to homotopy
	$\D$ on $\E$. The product rule of $\D$ yields the identity
	\[
	\D^2(\xi\wedge\omega) = \D(\Q(\xi)\wedge\omega + (-1)^{|\xi|}\xi\wedge\D(\omega)) = (-1)^{2|\xi|}\xi\wedge\D^2(\omega) = \xi\wedge\D^2(\omega)
	\]
	for all homogeneous $\xi\in\CM$ and all
	$\omega\in\CM\otimes(\Gamma(\E))$. In other words, $\D^2$ is graded
	$\CM$-linear and thus $\D^2 = R_\D$, where
	$R_\D\in\CM\otimes\Gamma(\GEnd(\E))$ of total degree 2.
	
	Define now for each $k\geq1$ the elements
	$R_\D^k:=R_\D\circ\ldots\circ
	R_\D=\D^{2k}\in\CM\otimes\Gamma(\GEnd(\E))$ (k-times) of total degree
	$2k$. The \textbf{Bianchi identity}\index{Bianchi identity}
	\begin{equation}\label{bianchi}
	\D_{\GEnd}R_\D^k = [\D,\D^{2k}] = 0
	\end{equation}
	is immediate for all $k\geq1$.
	\begin{proposition}\label{pont_char_prop}
		Let $(\M,\Q)$ be a Lie $n$-algebroid and $\E$ a graded vector
		bundle over $M$.
		\begin{enumerate}
			\item For any connection up to homotopy $\D$ of $(\M,\Q)$ on
			$\E$, 
			\[
			\Q(\gtr(R_\D^k)) = 0
			\]
			for all $k\geq1$. 
			\item The cohomology classes
			$[\gtr(R_\D^k)]\in H^{2k}(\M,\Q)$ do not depend on the
			connection $\D$.
		\end{enumerate}
	\end{proposition}
	\begin{proof}
		The first part follows from the Bianchi identity and Lemma
		\ref{lemma_gtr_commutes_with_differentials}. The second part is
		proved with the standard technique from classical $TM$-Pontryagin
		classes of vector bundles \cite{Quillen85}.
	\end{proof}
	
	\begin{definition}
		Let $(\M,\Q)$ be a Lie $n$-algebroid over the smooth manifold $M$ and let $\E\to M$ be a graded vector bundle. The \textbf{$\M$-Pontryagin algebra}\index{Pontryagin algebra} of $\E$
		\[
		\Pont_\M^\bullet(\E)\subset H^\bullet(\M)
		\]
		is defined as the subalgebra generated by the cohomology classes
		\[
		\sigma_\M^i(\E):=[\gtr(R^i_{\D})]\in H^{2i}(A)
		\]
		for all $i\geq1$.
	\end{definition}
	
	\begin{remark}\label{Pullback of Pontryagin algebra}
		Note that after a choice of splitting for the Lie $n$-algebroid $\M\simeq A_1[1]\oplus\ldots\oplus A_n[n]$, one can represent the generators $\sigma_\M^i(\E)=[\gtr(R^i_{\D})]$ using a degree-preserving $A_1$-connection on $\E=\bigoplus_{i\in\mathbb Z}E_i$, i.e.~$\D=(\nabla^i)_{i\in\mathbb{Z}}$ with $\nabla^i$ an $A_1$-connection on $E_i$. A choice of $TM$-connections $\nabla^i$ on the vector bundles $E_i$ yields $A_1$-connections on $E_i$, denoted again by $\nabla^i$, by the formula $\nabla_a^i e=\nabla_{\rho(a)}^ie$ for $a\in\Gamma(A_1)$ and $e_i\in\Gamma(E_i)$. Hence, one has that
		\[
		\Pont_\M^\bullet(\E)=\rho^\star\Pont^\bullet(\E),
		\]
		where $\Pont^\bullet(\E)$ denotes the usual $TM$-Pontryagin algebra of $\E$, which implies that the $\M$-Pontryagin algebra of $\E\to M$ is completely characterised by those of $TM$. However, as \cite{Jotz19c} shows the Pontryagin classes with respect to a Lie algebroid $A\to M$ can be used in order to give obstructions for the existence of $A$-representations on a graded vector bundle $\E\to M$. In particular, \cite[Thm.~4.13]{Jotz19c} proves that there is a $2$-representation of $A\to M$ on the graded vector bundle $E_0[0]\oplus E_1[1]\to M$ only if
		\[
		\Pont_A^\bullet(E_0) = \Pont_A^\bullet(E_1) \subset H^\bullet(A),
		\]
		or alternatively, only if the pull-backs under $\rho:A\to TM$ of the $TM$-Pontryagin classes of $E_0$ and $E_1$ coincide.
	\end{remark}
	
	\begin{remark}
		Characteristic classes of $2^{\text{nd}}$ order in the context of Lie algebroids were defined by Mehta in \cite{Mehta14}. These classes are associated to graded vector bundles $\E\to M$ equipped with the differential $\D$ of an $A$-representation up to homotopy, in contrast to classes of $1^{\text{st}}$ order which vanish if $\D^2 = 0$. In particular, \cite[Thm.~7.7]{Mehta14} shows that characteristic classes of $2^{\text{nd}}$ order are invariant with respect to \textbf{gauge transformations}\index{gauge transformation}; that is, DG $A$-module automorphisms $\mu$ between the representations up to homotopy $(\E,\D)$ and $(\E,\D')$ on the same $\E$ that make the following diagram commute
		\[
		\begin{tikzcd}
			\Omega(A,\E) \arrow[dd, "\pr"'] \arrow[rr, "\mu"] & & \Omega(A,\E) \arrow[dd, "\pr"] \\
			& & \\
			\Gamma(\E) \arrow[rr, "\id"']          &             & \Gamma(\E)           
		\end{tikzcd}
		\]
		Here, the map $\pr\colon\Omega(A,\E)\to\Gamma(\E)$ is the natural projection whose kernel is $\bigoplus_{p\geq1}\Omega^p(A)\otimes\Gamma(\E)$.
	\end{remark}
	
	\begin{comment}
	\begin{proposition}
		Let $\A=A_1[1]\oplus\ldots\oplus A_n[n]\to M$ be a graded vector bundle with a vector bundle morphism $\rho:A_1\to TM$ over the identity on $M$. Then $\A$ may admit a split Lie $n$-algebroid structure over $M$ with anchor $\rho$, only if
		\[
		\rho^\star\Pont^\bullet(\A) = 0.
		\] 
	\end{proposition}
	\begin{proof}
		Suppose $\A$ is a Lie $n$-algebroid over $M$ and consider its adjoint representation $(\ad_\nabla,\D_{\ad_{\nabla}})$ constructed in Section \ref{Adjoint representation of a Lie n-algebroid}, for any choice of $TM$-connections on the vector bundle $A_i$, $i=1.\ldots,n$. Recall from Proposition \ref{pont_char_prop} that the representatives $\gtr(R_{\D}^i)$ of the generators $\sigma_{\A}^i(\A)$ do not depend on the choice of the $\D$. Hence, we may choose $\D = \D_{\ad_{\nabla}}$ which satisfies $\D_{\ad_{\nabla}}^2 = 0$. Consequently, $\gtr(R_{\D}^i) = 0$ for all $i$, and Remark \ref{Pullback of Pontryagin algebra} yields
		\[
		\rho^\star(\Pont^\bullet(\A)) = \Pont_{\A}^\bullet(\A) = 0.\qedhere
		\]
	\end{proof}
	\end{comment}
	
	\chapter{Computations}\label{Chapter: Computations}
	
	\section{Split Lie $2$-algebroids in the geometric setting}\label{Appendix: Split Lie $2$-algebroids in the geometric setting}
	
	Our goal in this appendix is to prove that Definition \ref{abstract Lie algebroids} and Definition \ref{geom split 2-alg} for split Lie $2$-algebroids are equivalent. The original source of the following proof is \cite{Jotz19b}. For convenience, we repeat the two definitions here.
	
	\begin{definition}\label{Definition 1}
		A split Lie $2$-algebroid is a split $[2]$-manifold $Q[1]\oplus B^*[2]\to M$ equipped with a homological vector field $\Q$.
	\end{definition}
	
	\begin{definition}\label{Definition 2}
		A split Lie 2-algebroid is given by a pair of
		an anchored vector bundle $(Q\to M,\rho_Q)$ and a vector bundle
		$B\to M$, together with a vector bundle map $\ell\colon B^*\to Q$, a
		skew-symmetric dull bracket
		$[\cdot\,,\cdot]\colon \Gamma(Q)\times\Gamma(Q)\to \Gamma(Q)$, a
		linear $Q$-connection $\nabla$ on $B$, and a vector valued 3-form
		$\omega\in\Omega^3(Q,B^*)$ such that
		\begin{enumerate}[(i)]
			\item $\nabla^*_{\ell(\beta_1)}\beta_2 + \nabla^*_{\ell(\beta_2)}\beta_1 = 0$, for all $\beta_1,\beta_2\in\Gamma(B^*)$,
			\item $[q,\ell(\beta)]=\ell(\nabla_q^*\beta)$ for all
			$q\in\Gamma(Q)$ and $\beta\in\Gamma(B^*)$,
			\item $\Jac_{[\cdot\,,\cdot]} = \ell\circ\omega\in\Omega^3(Q,Q)$,
			\item $R_{\nabla^*}(q_1,q_2)\beta = -\omega(q_1,q_2,\ell(\beta))$ for $q_1,q_2\in\Gamma(Q)$ and $\beta\in\Gamma(B^*)$,
			\item $\diff_{\nabla^*}\omega = 0$.
		\end{enumerate}  
	\end{definition}
	
	Note that the third condition in the last definition implies $\rho_Q\circ\ell=0$. To prove this correspondence, we will need the following elements: the bundle map $\partial_B=\ell^*:Q^*\to B$, the \textbf{Dorfman connection}\index{Dorfman connection} $\Delta:\Gamma(Q)\times\Gamma(Q^*)\to \Gamma(Q^*)$ that is dual to the (skew-symmetric) dull bracket $[\cdot\,,\cdot]$
	\[
	\langle \Delta_q\tau,q' \rangle = \rho(q)\langle \tau,q' \rangle - \langle \tau,[q,q'] \rangle
	\]
	for $q,q'\in\Gamma(Q),\tau\in\Gamma(Q^*)$, and the element $R_\omega\in\Omega^2(Q,\Hom(B,Q^*))$ defined by 
	\[
	R_\omega(q_1,q_2)b = \langle i_{q_2}i_{q_1}\omega,b \rangle.
	\]
	One then easily checks that (ii) is equivalent to $\partial_B\circ\Delta_q = \nabla_q\circ\partial_B$ and (iii) is equivalent to $R_\omega(q_1,q_2)\circ\partial_B = R_\Delta(q_1,q_2)$ for $q,q_1,q_2\in\Gamma(Q)$.
	
	Starting with the data given in Definition \ref{Definition 2}, the homological vector field $\Q$ is defined by the following equations:
	\begin{equation}\label{Equation 1}
	\Q(f) = \rho_Q^*\diff f\in\Gamma(Q^*),
	\end{equation}
	\begin{equation}\label{Equation 2}
	\Q(\tau) = \diff_Q\tau + \partial_B \tau\in\Omega^2(Q)\oplus\Gamma(B),
	\end{equation}
	\begin{equation}\label{Equation 3}
	\Q(b) = \diff_\nabla b - \langle\omega,b\rangle\in\Omega^1(Q,B)\oplus\Omega^3(Q), 
	\end{equation}
	for $f\in C^\infty(M),\tau\in\Gamma(Q^*),b\in\Gamma(B)$. Conversely, using the Leibniz identity for an arbitrary vector field of degree 1 on the [2]-manifold $\M = Q[1]\oplus B^*[2]$, one can show that it must be given by the equations above, defining an anchor map $\rho_Q:Q\to TM$, a dull bracket $[\cdot\,,\cdot]$ on $Q$, a $Q$-connection $\nabla$ on $B$, a vector valued 3-form $\omega\in\Omega^3(Q,B^*)$, and a bundle morphism $\partial_B:Q^*\to B$. It remains to show that $\Q^2 = 0$ corresponds to the various relations between these objects. For $f\in C^\infty(M)$ we have
	\[
	\Q^2(f) = \diff_Q(\rho_Q^*\diff f) + \partial_B(\rho_Q^*\diff f)\in\Omega^2(Q)\oplus\Gamma(B).
	\]
	Since for $q_1,q_2\in\Gamma(Q)$
	\begin{align*}
	\diff_Q(\rho_Q^*\diff f)(q_1,q_2) = & \rho_Q(q_1)\left(\langle \rho_Q^*\diff f,q_2 \rangle\right)-
	\rho_Q(q_2)\left(\langle \rho_Q^*\diff f,q_1 \rangle\right) -
	\langle \rho_Q^*\diff f,[q_1,q_2] \rangle\\
	= & \left(\rho_Q(q_1)\rho_Q(q_2) - \rho_Q(q_2)\rho_Q(q_1) -\rho_Q[q_1,q_2]\right)f,
	\end{align*}
	we have that $\Q^2(f) = 0$ for all $f\in C^\infty(M)$ if and only if $\partial_B\circ\rho_Q^* = 0$ and $\rho_Q[q_1,q_2] = [\rho_Q(q_1),\rho_Q(q_2)]$ for all $q_1,q_2\in\Gamma(Q)$.
	
	\begin{definition}
		Given $\partial_B:Q^*\to B$ as above, we define the vector bundle morphism $\partial_B:\Omega^k(Q)\to \Omega^{k-1}(Q,B)$ by
		\[
		\partial_B\left(\tau_1\wedge\ldots\wedge\tau_k\right) = \sum_{i=1}^{k}(-1)^{i+1}\tau_1\wedge\ldots\wedge\widehat{\tau_i}\wedge\ldots\wedge\tau_k\wedge\partial_B\tau_i.
		\]
	\end{definition}
	
	\begin{lemma}\label{Homol.VF on $Q$-forms}
		For the homological vector field $\Q$ defined on $\Gamma(Q^*)$ by equation \ref{Equation 2}, the following two identities hold:
		\begin{enumerate}[(i)]
			\item $\Q = \diff_Q + \partial_B$ on $\Omega^\bullet(Q)$.
			\item $\Q = \diff_\nabla - \langle \omega,\cdot \rangle$ on $\Gamma(S^\bullet B)$, where we define
			\[
			\langle \omega,b_1b_2\ldots b_k \rangle = \sum_{i=1}^{k} b_1b_2\ldots\langle \omega,b_i\rangle \ldots b_k,
			\]
			for $b_1,b_2,\ldots, b_k\in\Gamma(B)$.
		\end{enumerate}
	\end{lemma}
	\begin{proof}
		\begin{enumerate}[(i)]
			\item Consider an element $\theta = \tau_1\wedge\ldots\wedge\tau_k\in\Omega^k(Q)$
			with $\tau_1,\ldots,\tau_k\in\Gamma(Q^*)$. Using the Leibniz identity, one computes
			\begin{align*}
			\Q(\theta) = &\ \Q\left(\tau_1\wedge\ldots\wedge\tau_k\right)
			= \sum_{i=1}^{k} (-1)^{i+1}\tau_1\wedge\ldots\wedge\Q(\tau_i)\wedge\ldots\wedge\tau_k\\
			= &\ \sum_{i=1}^{k} (-1)^{i+1}\tau_1\wedge\ldots\wedge \diff_Q(\tau_i)\wedge\ldots\wedge\tau_k + \sum_{i=1}^{k}
			(-1)^{i+1}\tau_1\wedge\ldots\wedge\partial_B(\tau_i)\wedge\ldots\wedge\tau_k\\
			= &\ \diff_Q\theta + \partial_B\theta.
			\end{align*}
			\item Similarly, considering an element $b=b_1b_2\ldots b_k\in\Gamma(S^k B)$, we compute
			\begin{align*}
			\Q(b) = &\ \Q(b_1b_2\ldots b_k)
			= \sum_{i=1}^{k} b_1b_2\ldots\Q(b_i)\ldots b_k \\
			= &\ \sum_{i=1}^{k} b_1b_2\ldots \diff_\nabla b_i\ldots b_k - \sum_{i=1}^{k} b_1b_2\ldots\langle \omega,b_i\rangle\ldots b_k \\
			= &\ \diff_\nabla b - \langle \omega,b \rangle.\qedhere
			\end{align*}
		\end{enumerate}
	\end{proof}
	
	Let now $\tau\in\Gamma(Q^*)$. Then we have
	\[
	\Q^2(\tau) = \Q(\diff_Q\tau) + \diff_\nabla\partial_B\tau - \langle \omega,\partial_B\tau \rangle,
	\]
	and using Lemma \ref{Homol.VF on $Q$-forms} it becomes
	\[
	\Q^2(\tau) = \left(\diff_Q^2\tau - \langle \omega,\partial_B\tau \rangle\right) + \left(\partial_B\diff_Q\tau + \diff_\nabla\partial_B\tau\right)\in\Omega^3(Q)\oplus\Omega^1(Q,B).
	\]
	Since for all $q,q_1,q_2,q_3\in\Gamma(Q)$ and $\tau\in\Gamma(Q^*)$, we have $\diff_Q^2\tau(q_1,q_2,q_3) = \langle \Jac_{[\cdot,\cdot]}(q_1,q_2,q_3),\tau\rangle$ and $\langle \beta,(\partial_B\diff_Q\tau)(q) \rangle = - \langle \beta,\partial_B\Delta_q\tau \rangle$, it follows that $\Q^2(\tau) = 0$ if and only if $\Jac_{[\cdot,\cdot]}(q_1,q_2,q_3) = \ell(\omega(q_1,q_2,q_3))$ and $\partial_B\Delta_q\tau = \nabla_q(\partial_B\tau)$.
	
	Lastly, for $b\in\Gamma(B)$ we have
	\[
	\Q^2(b) = \Q(\diff_\nabla b) - \diff_Q\langle \omega,b \rangle - \partial_B\langle \omega,b \rangle.
	\]
	\begin{definition}
		For $b\in\Gamma(B)$, we define the element $\langle \omega,\diff_\nabla b \rangle\in\Omega^4(Q)$ by
		\[
		\langle \omega,\diff_\nabla b \rangle(q_1,q_2,q_3,q_4) = -\sum_{\sigma\in C_4}
		\sgn(\sigma) \left\langle \omega\left(q_{\sigma(1)},q_{\sigma(2)},q_{\sigma(3)}\right),\nabla_{q_{\sigma(4)}}b \right\rangle
		\]
		for all $q_i\in\Gamma(Q),i=1,2,3,4$, where $C_4$ is the group of 4-cyclic permutations of four elements.
	\end{definition}
	\begin{definition}
		For $b\in\Gamma(B)$, we define the element $\nabla_{\partial_B^*}b\in\Gamma(S^2B)$ given by
		\[
		(\nabla_{\partial_B^*}b)(\beta_1,\beta_2) = \left\langle \nabla_{\partial_B^*\beta_1}b,\beta_2 \right\rangle +
		\left\langle \nabla_{\partial_B^*\beta_2}b,\beta_1 \right\rangle
		\]
		for all $\beta_1,\beta_2\in\Gamma(B^*)$.
	\end{definition}
	\begin{lemma}
		Let $b\in\Gamma(B)$, then we have
		\[
		\Q(\diff_\nabla b) = \diff_\nabla^2b + \nabla_{\partial_B^*}b - \langle \omega,\diff_\nabla b \rangle
		\in\Omega^2(Q,B)\oplus\Gamma(S^2B)\oplus\Omega^4(Q).
		\]
	\end{lemma}
	\begin{proof}
		Suppose that $\diff_\nabla b = \sum \tau_i\wedge c_i$, where $\tau_i\in\Omega^1(Q),c_i\in\Gamma(B)$. Using the definition of $\Q$ together with the Leibniz identity, we compute
		\begin{align*}
		\Q(\diff_\nabla b) = &\ \Q\left(\sum_i \tau_i\wedge c_i\right)
		= \sum_i \diff_Q\tau_i\wedge c_i - \tau_i\wedge \diff_\nabla c_i + \partial_B\tau_i\wedge c_i + \tau_i\wedge\langle \omega,c_i \rangle \\
		= &\ \diff_\nabla^2b + \sum_i \partial_B\tau_i\wedge c_i + \sum_i\tau_i\wedge\langle \omega,c_i \rangle.
		\end{align*}
		For the second term, let $\beta_1,\beta_2\in\Gamma(B^*)$. Then we have
		\begin{align*}
		\left(\sum_i \partial_B\tau_i\wedge c_i\right)\left(\beta_1,\beta_2\right) = &\ \sum_i\partial_B\tau_i(\beta_1)c_i(\beta_2) + \partial_B\tau_i(\beta_2)c_i(\beta_1) \\
		= &\ \sum_i\tau_i(\partial_B^*\beta_1)c_i(\beta_2) + \tau_i(\partial_B^*\beta_2)c_i(\beta_1) \\
		= &\ \nabla_{\partial_B^*}b(\beta_1,\beta_2).
		\end{align*}
		To find the last term, consider $q_1,q_2,q_3,q_4\in\Gamma(Q)$. Then we have
		\begin{align*}
		\left(\sum_i\tau_i\wedge\langle \omega,c_i \rangle\right)\left(q_1,q_2,q_3,q_4\right)
		= &\ \sum_i\Big(\langle \omega\left(q_2,q_3,q_4\right),\tau_i(q_1)c_i \rangle
		- \langle \omega\left(q_1,q_3,q_4\right),\tau_i(q_2)c_i \rangle\\
		&\ + \langle \omega\left(q_1,q_2,q_4\right),\tau_i(q_3)c_i \rangle
		- \langle \omega\left(q_1,q_2,q_3\right),\tau_i(q_4)c_i \rangle\Big)\\
		= &\ \langle \omega\left(q_2,q_3,q_4\right),\sum_i\tau_i(q_1)c_i \rangle
		- \langle \omega\left(q_1,q_3,q_4\right),\sum_i\tau_i(q_2)c_i \rangle\\
		&\ + \langle \omega\left(q_1,q_2,q_4\right),\sum_i\tau_i(q_3)c_i \rangle
		- \langle \omega\left(q_1,q_2,q_3\right),\sum_i\tau_i(q_4)c_i \rangle\\
		= &\ \langle \omega,\diff_\nabla b \rangle\left(q_1,q_2,q_3,q_4\right).\qedhere
		\end{align*}
	\end{proof}
    \noindent
	The remaining equations follow now by grouping together the different summands of $\Q^2(b)$ and equalising them to zero.
	
	\section{Signs for the linearity of $\pi^\sharp:\Omega^1(\M)\to\mathfrak{X}(\M)$}\label{Appendix: Signs for the linearity of pi^sharp}
	
	In this appendix, we explain the choice of our signs for the map of DG-modules $\pi^\sharp:\Omega^1(\M)\to\mathfrak{X}(\M)$, where $(\M,\Q)$ is a $\Q$-manifold equipped with a Poisson bracket $\{\cdot\,,\cdot\}_k$.
	
	Recall that on exact $1$-forms $\diff\xi_1\in\Omega^1(\M)$, the map $\pi^\sharp$ is defined as $\pi^\sharp(\diff\xi_1)(\xi_2)=\{\xi_1,\xi_2\}_k$, for all $\xi_1,\xi_2\in\cin(\M)$. In order for $\pi^\sharp$ to be consistent with the properties of the de Rham differential $\diff$ on $\cin(\M)$, we need an appropriate choice of signs for the formulae
	\[
	\pi^\sharp(\xi_1\diff\xi_2) = \pm\, \xi_1\pi^\sharp(\diff\xi_2)
	\qquad \text{and} \qquad
	\pi^\sharp((\diff\xi_1)\xi_2) = \pm\, \pi^\sharp(\diff\xi_1)\,\xi_2.
	\]
	More precisely, given two homogeneous functions $\xi_1,\xi_2\in\cin(\M)$, the following must hold:
	\begin{equation}\label{linearity rules for pi map}
		\{\xi_1\xi_2,\cdot\}_k = \pi^\sharp(\diff(\xi_1\xi_2)) = \pi^\sharp((\diff\xi_1)\xi_2) + (-1)^{|\xi_1|}\pi^\sharp(\xi_1\diff\xi_2).
	\end{equation}
	This is true for the following linearity rules:
	\[
	\pi^\sharp(\xi_1\diff \xi_2) = (-1)^{|\xi_1|}\xi_1\pi^\sharp(\diff\xi_2)
	\qquad \text{and} \qquad
	\pi^\sharp((\diff\xi_1) \xi_2) = (-1)^{|\xi_2|}\pi^\sharp(\diff\xi_1)\,\xi_2.
	\]
	In order to see this, first observe that for homogeneous functions $\xi_1,\xi_2\in\cin(\M)$ and a homogeneous vector field $\X\in\mathfrak{X}^{|\X|}(\M) = (\mathfrak{X}(\M)[k])^{|\X|-k}$ we have
	\[
	(\X\cdot\xi_1)(\xi_2) = (-1)^{|\xi_1|(|\X|+1-k)}\xi_1\X(\xi_2) = (-1)^{|\xi_1|(|\xi_2|+1-k)}\X(\xi_2)\xi_1.
	\]
	Now we compute for a homogeneous $\xi_3\in\cin(\M)$
	\begin{align*}
		\{\xi_1\xi_2,\xi_3\}_k = & -(-1)^{(|\xi_1|+|\xi_2|+k)(|\xi_3|+k)}\{\xi_3,\xi_1\xi_2\}_k \\
		= & -(-1)^{(|\xi_1|+|\xi_2|+k)(|\xi_3|+k)}\{\xi_3,\xi_1\}_k\xi_2 
		- (-1)^{(|\xi_3|+k)(|\xi_2|+k)}\xi_1\{\xi_3,\xi_2\}_k \\
		= &\ (-1)^{(|\xi_3|+k)|\xi_2|}\{\xi_1,\xi_3\}_k\xi_2 
		+ \xi_1\{\xi_2,\xi_3\}_k.
	\end{align*}
	On the other hand, we have
	\begin{align*}
		\pi^\sharp((\diff\xi_1)\xi_2)(\xi_3) = &\  (-1)^{|\xi_2|}(\pi^\sharp(\diff\xi_1)\cdot\xi_2)(\xi_3) \\
		= &\  (-1)^{|\xi_2|}(-1)^{(|\xi_3|+1+k)|\xi_2|}\pi^\sharp(\diff\xi_1)(\xi_3)\cdot\xi_2 \\
		= &\ (-1)^{(|\xi_3|+k)|\xi_2|}\{\xi_1,\xi_3\}_k\xi_2
	\end{align*}
	and
	\[
	\pi^\sharp(\xi_1\diff\xi_2)(\xi_3) = (-1)^{|\xi_1|}\xi_1\pi^\sharp(\diff\xi_2)(\xi_3) =
	(-1)^{|\xi_1|}\xi_1\{\xi_2,\xi_3\}_k.
	\]
	Therefore, equation (\ref{linearity rules for pi map}) holds.
	
	\section{Adjoint representation of a split Lie $2$-algebroid}\label{Appendix: Adjoint representation of a split Lie $2$-algebroid}
	
	In order to make the computations for the equations of the adjoint
	representation that appear in Proposition \ref{Adjoint_representation_of_Lie_2-algebroid}, we will need the following identity.
	
	\begin{lemma}\label{Madeleine_calculation}
		For all $q_1,q_2,q_3\in\Gamma(Q)$ and $X\in\mathfrak{X}(M)$
		we have
		\[
		\left(\diff_{\nabla^{\text{bas},\Hom}}R_\nabla^\text{bas}\right)(q_1,q_2,q_3)X
		=
		\left(\nabla_X^{\Hom}\Jac_{[\cdot\,,\cdot]}\right)(q_1,q_2,q_3).
		\]
	\end{lemma}
	\begin{proof}
		Let $q_1,q_2,q_3\in\Gamma(Q)$ and $X\in\mathfrak{X}(M)$. We
		compute:
		\begin{align*}
		\left(\diff_{\nabla^{\text{bas},\Hom}}R_\nabla^\text{bas}\right)(q_1,q_2,q_3)X = &\ \nabla_{q_1}^\text{bas}(R_\nabla^\text{bas}(q_2,q_3)X) - R^\text{bas}_\nabla(q_2,q_3)(\nabla_{q_1}^\text{bas} X) \\
		&\ - R^\text{bas}_\nabla([q_1,q_2],q_3)X + \text{c.p.} \\
		= &\ [q_1,R_\nabla^\text{bas}(q_2,q_3)X] + \nabla_{\rho(R_\nabla^\text{bas}(q_2,q_3)X)} q_1 - R_\nabla^\text{bas}(q_2,q_3)(\nabla_{q_1}^\text{bas}X) \\
		&\ + \nabla_X [[q_1,q_2],q_3] - [\nabla_X [q_1,q_2],q_3] -[[q_1,q_2],\nabla_X q_3] \\
		&\ - \nabla_{\nabla_{q_3}^\text{bas}X} [q_1,q_2] + \nabla_{\nabla_{[q_1,q_2]}^\text{bas}X} q_3 + \text{c.p.} \\
		= &\ \nabla_X(\Jac_{[\cdot\,,\cdot]}(q_1,q_2,q_3)) - \Jac_{[\cdot\,,\cdot]}(\nabla_X q_1,q_2,q_3) \\
		&\ - \Jac_{[\cdot\,,\cdot]}(q_1,\nabla_X q_2,q_3) - \Jac_{[\cdot\,,\cdot]}(q_1,q_2,\nabla_X q_3) \\
		= &\ \left(\nabla_X^{\Hom}\Jac_{[\cdot\,,\cdot]}\right)(q_1,q_2,q_3).\qedhere
		\end{align*}
	\end{proof}
	
	\begin{proof}[Proof of Proposition \ref{Adjoint_representation_of_Lie_2-algebroid}]
		A straightforward computation shows that the elements in the
		statement are $C^\infty(M)$-linear in all their arguments
		and thus well defined. Moreover, the equation
		$\rho\circ\ell = 0$, the commutativity of $\rho$ with the
		dull bracket $[\cdot\,,\cdot]$, and equation (ii) of Definition
		\ref{geom split 2-alg} imply that the individual
		$Q$-connections on the vector bundles $B^*,\ Q$ and $TM$
		commute with $\ell$ and $\rho$, and thus define a
		$Q$-connection on the complex $B^*\to Q\to TM$. Hence, we
		only need to check that the equations in the statement of
		Proposition \ref{3-term_representations}, or equivalently
		the seven equations that are stated in the proof, hold. In the following, let
		$\beta,\beta_i\in\Gamma(B^*),\ b\in\Gamma(B),\
		q,q_i\in\Gamma(Q)$ and $X\in\mathfrak{X}(M)$, for
		$i=1,2,\ldots$
		
		\begin{enumerate}[(i)]
			\item For the first equation we have the following
			three computations
			\[
			\Big( [\partial,\omega_2](X) + \diff^2_{\nabla^{\text{bas}}}X \Big)(q_1,q_2) = \rho\left( \omega_2(q_1,q_2)X \right) +  R_{\nabla^{\text{bas}}}(q_1,q_2)X = 0
			\]
			\[
			\Big( [\partial,\omega_2](q) + \diff^2_{\nabla^{\text{bas}}}q \Big)(q_1,q_2)
			= - \ell\left( \omega_2(q_1,q_2)q \right) + \omega_2(q_1,q_2)\rho(q)  +  R_{\nabla^{\text{bas}}}(q_1,q_2)q = 0
			\]
			\[
			\Big( [\partial,\omega_2](\beta) + \diff^2_{\nabla^*}\beta \Big)(q_1,q_2) = R_{\nabla^*}(q_1,q_2)\beta - \omega_2(q_1,q_2)\ell(\beta) = 0
			\]
			where in the third equation
			we use (iv) from Definition \ref{geom split 2-alg}.
			\item For the second equation we compute
			\[
			\Big( [\partial,\phi_0](X) + \partial_B(\diff_{\nabla^{\text{bas}}}X) \Big)(\beta) = \rho(\nabla_X \ell(\beta)) - \rho(\ell(\nabla_X \beta)) + \nabla^\text{bas}_{\ell(\beta)}X  = 0
			\]
			\[
			\Big( [\partial,\phi_0](\beta_1) + \partial_B(\diff_{\nabla^*}\beta_1) \Big)(\beta_2) = - \nabla_{\ell(\beta_2)}^* \beta_1 - \phi_0(\beta_2)\ell(\beta_1) = 0
			\]
			\[
			\Big( [\partial,\phi_0](q) + \partial_B(\diff_{\nabla^{\text{bas}}}q) \Big)(\beta) =  \nabla_{\ell(\beta)}^\text{bas} q + \phi_0(\beta)\rho(q) - \ell(\phi_0(\beta)q) = 0
			\]
			where for the last two equations we used (i)
			and (ii) from Definition
			\ref*{geom split 2-alg}.
			\item The third equation splits into the following two cases
			%\begin{comment}
			\begin{align*}
			\Big( [\partial,\omega_3](X) & + [\diff_{\nabla^\text{bas}},\omega_2](X) \Big)(q_1,q_2,q_3) \\
			= &\ \ell(\nabla_X(\omega(q_1,q_2,q_3))) + \left(\nabla_X^{\Hom}\Jac_{[\cdot\,,\cdot]}\right)(q_1,q_2,q_3) - \nabla_X\left(\Jac_{[\cdot\,,\cdot]}(q_1,q_2,q_3)\right)\\
			& + \sum_{i<j}(-1)^{i+j}R^\text{bas}_\nabla([q_i,q_j],q_{k\neq i,j})X + \sum_i(-1)^{i+1}\nabla^\text{bas}_{q_i}\left(R_\nabla^\text{bas}(q_{s\neq i},q_{t\neq i})\right)X\\
			& - R_\nabla^\text{bas}(q_1,q_2)(\nabla_{q_3}^\text{bas} X)
			+R_\nabla^\text{bas}(q_1,q_3)(\nabla_{q_2}^\text{bas} X)
			-R_\nabla^\text{bas}(q_2,q_3)(\nabla_{q_1}^\text{bas} X)\\
			= &\ \ell(\nabla_X(\omega(q_1,q_2,q_3))) + \left(\nabla_X^{\Hom}\Jac_{[\cdot\,,\cdot]}\right)(q_1,q_2,q_3) - \nabla_X\left(\Jac_{[\cdot\,,\cdot]}(q_1,q_2,q_3)\right)\\
			& - \left(d_{\nabla^{\text{bas},\Hom}}R_\nabla^\text{bas}\right)(q_1,q_2,q_3)X\\
			= & \ \langle \omega(q_1,q_2,q_3),\phi_0(X) \rangle,
			\end{align*}
			\begin{align*}
			\Big( \diff_{\nabla^*}(\omega_2(q_4))\ + &\ \omega_2(\diff_{\nabla^\text{bas}} q_4) + \omega_3(\partial(q_4)) \Big) (q_1,q_2,q_3) \\
			= &\ \sum_{i<j}^3 (-1)^{i+j}\omega([q_i,q_j],\ldots,\hat{q_i},\ldots,\hat{q_j},\ldots) + \sum_{i=1}^3 (-1)^{i+1}\nabla^*_{q_i}\left( \omega(\ldots,\hat{q_i},\ldots) \right) \\
			& - \omega(q_1,q_2,\nabla^\text{bas}_{q_3} q_4) + \omega(q_1,q_3,\nabla^\text{bas}_{q_2} q_4) - \omega(q_2,q_3,\nabla^\text{bas}_{q_1} q_4)\\
			& -\nabla_{\rho(q_4)}(\omega(q_1,q_2,q_3)) + \omega(\nabla_{\rho(q_4)} q_1,q_2,q_3)\\
			& + \omega(q_1,\nabla_{\rho(q_4)} q_2,q_3) + \omega(q_1,q_2,\nabla_{\rho(q_4)} q_3)\\
			= &\ \nabla^*_{q_4}(\omega(q_1,q_2,q_3)) - \nabla_{\rho(q_4)}(\omega(q_1,q_2,q_3))\\
			= &\ \langle \omega(q_1,q_2,q_3),\phi_0(q_4) \rangle,
			\end{align*}
			where for the first equation we used Lemma
			\ref{Madeleine_calculation} together with
			(iii) from Definition \ref{geom split 2-alg}
			and for the second equation we used (v) from
			Definition \ref{geom split 2-alg}.
			\item For the fourth equation we compute
			\begin{align*}
			\Big( \diff_{\overline{\nabla}}\phi_0(X) &\ + [\partial,\phi_1](X) +\partial_B(\omega_2(X) \Big)(\beta,q) \\ 
			& = \nabla_q^\text{bas}(\phi_0(\beta)X) - \phi_0(\nabla_q^*\beta)X - \phi_0(\beta)(\nabla^\text{bas}_q X)\\ & - \ell(\phi_1(\beta,q)X) + R_\nabla^\text{bas}(q,\ell(\beta))X = 0,
			\end{align*}
			\begin{align*}
			\Big( \diff_{\overline{\nabla}}\phi_0(q_1) &\ + [\partial,\phi_1](q_1) +\partial_B(\omega_2(q_1) \Big)(\beta,q_2)\\
			& = \nabla_{q_2}^\text{bas}(\phi_0(\beta)q_1) - \phi_0(\nabla_{q_2}^*\beta)q_1 - \phi_0(\beta)(\nabla^\text{bas}_{q_2} q_1)\\
			& + \omega_2(\ell(\beta),q_2)q_1 + \phi_1(\beta,q_2)(\rho(q_1)) = 0
			\end{align*}
			where in the first calculation we used (ii)
			from Definition \ref{geom split 2-alg} and
			$\rho\circ\ell = 0$, and in the second
			calculation we used (iv) from Definition
			\ref{geom split 2-alg}.
			\item For the fifth equation we have
			\begin{align*}
			\Big( \diff_{\nabla^*}(\omega_3(X))\ + &\ \omega_2(\omega_2(X)) + \omega_3(\diff_{\nabla^\text{bas}} X) \Big)(q_1,q_2,q_3,q_4) \\
			= &\ \sum_{i<j}^4 (-1)^{i+j}\omega_3([q_i,q_j],\ldots,\hat{q_i},\ldots,\hat{q_j},\ldots) +
			\sum_{i=1}^{4} (-1)^{i+1}\nabla^*_{q_i}(\omega_3(\ldots,\hat{q_i},\ldots)) \\
			&\ + \omega_2(q_1,q_2)(\omega_2(q_3,q_4)X)
			- \omega_2(q_1,q_3)(\omega_2(q_2,q_4)X) \\
			&\ + \omega_2(q_1,q_4)(\omega_2(q_2,q_3)X)
			+ \omega_2(q_2,q_3)(\omega_2(q_1,q_4)X) \\
			&\ - \omega_2(q_2,q_4)(\omega_2(q_1,q_3)X)
			+ \omega_2(q_3,q_4)(\omega_2(q_1,q_2)X) \\
			&\ - \omega_3(q_2,q_3,q_4)(\nabla^\text{bas}_{q_1} X)
			+ \omega_3(q_1,q_3,q_4)(\nabla^\text{bas}_{q_2} X) \\
			&\ - \omega_3(q_1,q_2,q_4)(\nabla^\text{bas}_{q_3} X)
			+ \omega_3(q_1,q_2,q_3)(\nabla^\text{bas}_{q_4} X) \\
			&\ - \phi_1(\omega(q_2,q_3,q_4),q_1)X 
			+ \phi_1(\omega(q_1,q_3,q_4),q_2)X \\
			&\ - \phi_1(\omega(q_1,q_2,q_4),q_3)X
			+ \phi_1(\omega(q_1,q_2,q_3),q_4)X \\
			= &\ \langle \omega,\phi(X) \rangle(q_1,q_2,q_3,q_4),
			\end{align*}
			where we used (v) from Definition
			\ref{geom split 2-alg}.
			\item The sixth equation becomes
			\begin{align*}
			\Big( \diff_{\overline{\nabla}}\phi_1 (X) + \omega_2(\phi_0(X))\ + &\ \phi_0(\omega_2(X)) + \partial_B(\omega_3(X)) \Big)(\beta,q_1,q_2)\\
			= &\ - \phi_1(\beta,[q_1,q_2])X + \phi_1(\nabla^*_{q_2}\beta,q_1)X - \phi_1(\nabla^*_{q_1}\beta,q_2)X \\
			&\ - \phi_1(\beta,q_2)(\nabla^*_{q_1}X) + \phi_1(\beta,q_1)(\nabla^*_{q_2}X) + \nabla_{q_1}^*(\phi_1(\beta,q_2)X) \\
			&\ - \nabla_{q_2}^*(\phi_1(\beta,q_1)X) + \omega(q_1,q_2,\ell(\nabla_X\beta)) - \omega(q_1,q_2,\nabla_X\ell(\beta)) \\
			&\ - \phi_0(\beta)(R_\nabla^\text{bas}(q_1,q_2)X) + \omega_3(q_1,q_2,\ell(\beta))X \\
			= &\ 0,
			\end{align*}
			where we used (iv) from Definition
			\ref{geom split 2-alg} and that
			$\rho(R_\nabla^\text{bas}(q_1,q_2)X) =
			R_{\nabla^\text{bas}}(q_1,q_2)X$.
			\item For the last equation we compute
			\begin{align*}
			\Big( \phi_0(\phi_0(X)) + \partial_B(\phi_1(X)) \Big)(\beta_1,\beta_2) = &\ \phi_0(\phi_0(\beta_1)X,\beta_2) + \phi_0(\phi_0(\beta_2)X,\beta_1)\\
			& + \phi_1(\beta_1,\ell(\beta_2))X + \phi_1(\beta_2,\ell(\beta_1))X\\
			= &\ \nabla^*_{\ell(\nabla_X \beta_1)} \beta_2 + \nabla^*_{\ell(\nabla_X \beta_2)} \beta_1 + \nabla^*_{\ell(\beta_2)}\nabla_X\beta_1 \\
			& -\nabla_X\nabla^*_{\ell(\beta_2)}\beta_1 +\nabla^*_{\ell(\beta_1)}\nabla_X\beta_2 - \nabla_X\nabla^*_{\ell(\beta_1)}\beta_2 \\
			= &\ 0,
			\end{align*}
			where in the last line we used (i) from Definition \ref{geom split 2-alg}.\qedhere
		\end{enumerate}
	\end{proof} 

\end{appendices}

\def\cprime{$'$} \def\polhk#1{\setbox0=\hbox{#1}{\ooalign{\hidewidth
			\lower1.5ex\hbox{`}\hidewidth\crcr\unhbox0}}} \def\cprime{$'$}
\def\cprime{$'$}

%\bibliographystyle{apa}
%\bibliography{biblio}

\begin{thebibliography}{}
	
	\bibitem[\protect\astroncite{{Abouzaid} and {Boyarchenko}}{2006}]{AbBo06}
	{Abouzaid}, M. and {Boyarchenko}, M. (2006).
	\newblock {Local structure of generalized complex manifolds.}
	\newblock {\em {J. Symplectic Geom.}}, 4(1):43--62.
	
	\bibitem[\protect\astroncite{Arias~Abad}{2008}]{Arias08t}
	Arias~Abad, C. (2008).
	\newblock Representations up to homotopy and cohomology of classifying spaces.
	\newblock {\em Ph. D. Thesis Utrecht University}.
	
	\bibitem[\protect\astroncite{Arias~Abad and Crainic}{2011}]{ArCr11}
	Arias~Abad, C. and Crainic, M. (2011).
	\newblock The {W}eil algebra and the {V}an {E}st isomorphism.
	\newblock {\em Ann. Inst. Fourier (Grenoble)}, 61(3):927--970.
	
	\bibitem[\protect\astroncite{Arias~Abad and Crainic}{2012}]{ArCr12}
	Arias~Abad, C. and Crainic, M. (2012).
	\newblock Representations up to homotopy of {L}ie algebroids.
	\newblock {\em J. Reine Angew. Math.}, 663:91--126.
	
	\bibitem[\protect\astroncite{Arias~Abad and Crainic}{2013}]{ArCr13}
	Arias~Abad, C. and Crainic, M. (2013).
	\newblock Representations up to homotopy and {B}ott's spectral sequence for
	{L}ie groupoids.
	\newblock {\em Adv. Math.}, 248:416--452.
	
	\bibitem[\protect\astroncite{Arias~Abad et~al.}{2011}]{ArCrDh11}
	Arias~Abad, C., Crainic, M., and Dherin, B. (2011).
	\newblock Tensor products of representations up to homotopy.
	\newblock {\em J. Homotopy Relat. Struct.}, 6(2):239--288.
	
	\bibitem[\protect\astroncite{Arias~Abad and Sch\"{a}tz}{2011}]{ArSc11}
	Arias~Abad, C. and Sch\"{a}tz, F. (2011).
	\newblock Deformations of {L}ie brackets and representations up to homotopy.
	\newblock {\em Indag. Math. (N.S.)}, 22(1-2):27--54.
	
	\bibitem[\protect\astroncite{Arias~Abad and Sch\"{a}tz}{2013}]{ArSc13}
	Arias~Abad, C. and Sch\"{a}tz, F. (2013).
	\newblock The {$\textbf{A}_\infty$} de {R}ham theorem and integration of
	representations up to homotopy.
	\newblock {\em Int. Math. Res. Not. IMRN}, (16):3790--3855.
	
	\bibitem[\protect\astroncite{Baez and Crans}{2004}]{BaCr04}
	Baez, J.~C. and Crans, A.~S. (2004).
	\newblock Higher-dimensional algebra. {VI}. {L}ie 2-algebras.
	\newblock {\em Theory Appl. Categ.}, 12:492--538.
	
	\bibitem[\protect\astroncite{Berezin}{1966}]{Berezin66}
	Berezin, F. (1966).
	\newblock {\em The Method of Second Quantization}.
	\newblock Translated from the Russian by Nobumichi Mugibayashi and Alan
	Jeffrey. Pure and Applied Physics, Vol. 24. Academic Press, New York-London.
	
	\bibitem[\protect\astroncite{Berezin}{1987}]{Berezin87}
	Berezin, F.~A. (1987).
	\newblock {\em Introduction to superanalysis}, volume~9 of {\em Mathematical
		Physics and Applied Mathematics}.
	\newblock D. Reidel Publishing Co., Dordrecht.
	\newblock Edited and with a foreword by A. A. Kirillov, With an appendix by V.
	I. Ogievetsky, Translated from the Russian by J. Niederle and R. Koteck\'y,
	Translation edited by Dimitri Le\u\i tes.
	
	\bibitem[\protect\astroncite{{Bischoff} et~al.}{2020}]{BiBuMe20}
	{Bischoff}, F., {Bursztyn}, H., {Lima}, H., and {Meinrenken}, E. (2020).
	\newblock {Deformation spaces and normal forms around transversals.}
	\newblock {\em {Compos. Math.}}, 156(4):697--732.
	
	\bibitem[\protect\astroncite{{Blohmann}}{2017}]{Blohmann17}
	{Blohmann}, C. (2017).
	\newblock {Removable presymplectic singularities and the local splitting of
		Dirac structures.}
	\newblock {\em {Int. Math. Res. Not.}}, 2017(23):7344--7374.
	
	\bibitem[\protect\astroncite{Bonavolont{\`a} and Poncin}{2013}]{BoPo13}
	Bonavolont{\`a}, G. and Poncin, N. (2013).
	\newblock On the category of {L}ie {$n$}-algebroids.
	\newblock {\em J. Geom. Phys.}, 73:70--90.
	
	\bibitem[\protect\astroncite{Bott}{1972}]{Bott72}
	Bott, R. (1972).
	\newblock {Lectures on characteristic classes and foliations. Notes by Lawrence
		Conlon. Appendices by J. Stasheff.}
	\newblock {Lectures algebraic diff. Topology, Lect. Notes Math. 279, 1-94
		(1972).}
	
	\bibitem[\protect\astroncite{Brahic and Ortiz}{2019}]{BrOr19}
	Brahic, O. and Ortiz, C. (2019).
	\newblock Integration of {$2$}-term representations up to homotopy via
	{$2$}-functors.
	\newblock {\em Trans. Amer. Math. Soc.}, 372(1):503--543.
	
	\bibitem[\protect\astroncite{Bruce}{2010}]{Bruce10}
	Bruce, A.~J. (2010).
	\newblock On higher poisson and koszul--schouten brackets.
	\newblock {\em {arXiv e-prints}}, page arXiv:0910.1992.
	
	\bibitem[\protect\astroncite{Bursztyn et~al.}{2009}]{BuCaOr09}
	Bursztyn, H., Cabrera, A., and Ortiz, C. (2009).
	\newblock Linear and multiplicative 2-forms.
	\newblock {\em Lett. Math. Phys.}, 90(1-3):59--83.
	
	\bibitem[\protect\astroncite{Cabrera et~al.}{2018}]{CaBrOr18}
	Cabrera, A., Brahic, O., and Ortiz, C. (2018).
	\newblock Obstructions to the integrability of {$\mathcal{VB}$}-algebroids.
	\newblock {\em J. Symplectic Geom.}, 16(2):439--483.
	
	\bibitem[\protect\astroncite{{Carmeli} et~al.}{2011}]{CaCaFi11}
	{Carmeli}, C., {Caston}, L., and {Fioresi}, R. (2011).
	\newblock {\em {Mathematical foundations of supersymmetry.}}
	\newblock Z\"urich: European Mathematical Society (EMS).
	
	\bibitem[\protect\astroncite{{Cattaneo}}{2008}]{Cattaneo08}
	{Cattaneo}, A.~S. (2008).
	\newblock {Deformation quantization and reduction}.
	\newblock In {\em {Poisson geometry in mathematics and physics. Proceedings of
			the international conference, Tokyo, Japan, June 5--9, 2006}}, pages 79--101.
	Providence, RI: American Mathematical Society (AMS).
	
	\bibitem[\protect\astroncite{Cattaneo and Felder}{2001}]{CaFe01}
	Cattaneo, A.~S. and Felder, G. (2001).
	\newblock Poisson sigma models and symplectic groupoids.
	\newblock In {\em Quantization of singular symplectic quotients}, volume 198 of
	{\em Progr. Math.}, pages 61--93. Birkh\"auser, Basel.
	
	\bibitem[\protect\astroncite{{Cattaneo} and {Felder}}{2004}]{CaFe04}
	{Cattaneo}, A.~S. and {Felder}, G. (2004).
	\newblock {Coisotropic submanifolds in Poisson geometry and branes in the
		Poisson sigma model}.
	\newblock {\em {Lett. Math. Phys.}}, 69:157--175.
	
	\bibitem[\protect\astroncite{{Cattaneo} and {Felder}}{2007}]{CaFe07}
	{Cattaneo}, A.~S. and {Felder}, G. (2007).
	\newblock {Relative formality theorem and quantisation of coisotropic
		submanifolds}.
	\newblock {\em {Adv. Math.}}, 208(2):521--548.
	
	\bibitem[\protect\astroncite{{Courant}}{1990}]{CourantTheo90}
	{Courant}, T.~J. (1990).
	\newblock {Dirac manifolds.}
	\newblock {\em {Trans. Am. Math. Soc.}}, 319(2):631--661.
	
	\bibitem[\protect\astroncite{Crainic and Fernandes}{2003}]{CrFe03}
	Crainic, M. and Fernandes, R.~L. (2003).
	\newblock Integrability of {L}ie brackets.
	\newblock {\em Ann. of Math. (2)}, 157(2):575--620.
	
	\bibitem[\protect\astroncite{{Crainic} et~al.}{2014}]{CrScSt14}
	{Crainic}, M., {Sch\"atz}, F., and {Struchiner}, I. (2014).
	\newblock {A survey on stability and rigidity results for Lie algebras}.
	\newblock {\em {Indag. Math., New Ser.}}, 25(5):957--976.
	
	\bibitem[\protect\astroncite{{Cueca}}{2019}]{Cueca19a}
	{Cueca}, M. (2019).
	\newblock {\em Applications of graded manifolds to Poisson geometry}.
	\newblock PhD thesis, IMPA, available at
	\url{www.impa.br/wp-content/uploads/2019/10/tese_dout_Miquel-Cueca.pdf}, Rio
	de Janeiro.
	
	\bibitem[\protect\astroncite{Cueca}{2019}]{Cueca19}
	Cueca, M. (2019).
	\newblock {The geometry of graded cotangent bundles}.
	\newblock {\em arXiv e-prints}, page arXiv:1905.13245.
	
	\bibitem[\protect\astroncite{Cueca and Mehta}{2021}]{CuMe21}
	Cueca, M. and Mehta, R.~A. (2021).
	\newblock Courant cohomology, cartan calculus, connections, curvature,
	characteristic classes.
	\newblock {\em Communications in Mathematical Physics}, 381(3):1091--1113.
	
	\bibitem[\protect\astroncite{da~Silva et~al.}{1999}]{daSiWe99}
	da~Silva, A., Weinstein, A., Society, A.~M., for Pure, B.~C., and Mathematics,
	A. (1999).
	\newblock {\em Geometric Models for Noncommutative Algebras}.
	\newblock Berkeley mathematics lecture notes. American Mathematical Society.
	
	\bibitem[\protect\astroncite{del Carpio-Marek}{2015}]{delCarpio-Marek15}
	del Carpio-Marek, F. (2015).
	\newblock {\em Geometric structures on degree 2 manifolds}.
	\newblock PhD thesis, IMPA, available at
	\url{www.impa.br/wp-content/uploads/2017/05/Fernando\_Del\_Carpio.pdf}, Rio
	de Janeiro.
	
	\bibitem[\protect\astroncite{Deligne and Morgan}{1999}]{DeMo99}
	Deligne, P. and Morgan, J.~W. (1999).
	\newblock Notes on supersymmetry (following {J}oseph {B}ernstein).
	\newblock In {\em Quantum fields and strings: a course for mathematicians,
		{V}ol. 1, 2 ({P}rinceton, {NJ}, 1996/1997)}, pages 41--97. Amer. Math. Soc.,
	Providence, RI.
	
	\bibitem[\protect\astroncite{Dorfman}{1993}]{Dorfman93}
	Dorfman, I. (1993).
	\newblock {\em Dirac {S}tructures and {I}ntegrability of {N}onlinear
		{E}volution {E}quations}.
	\newblock Nonlinear Science: Theory and Applications. John Wiley \& Sons Ltd.,
	Chichester.
	
	\bibitem[\protect\astroncite{Drummond et~al.}{2015}]{DrJoOr15}
	Drummond, T., Jotz, M., and Ortiz, C. (2015).
	\newblock {VB}-algebroid morphisms and representations up to homotopy.
	\newblock {\em Differential Geometry and its Applications}, 40:332--357.
	
	\bibitem[\protect\astroncite{{Dufour}}{2001}]{Dufour01}
	{Dufour}, J.-P. (2001).
	\newblock {Normal forms for Lie algebroids.}
	\newblock In {\em {Lie algebroids and related topics in differential geometry.
			Proceedings of the conference, Warsaw, Poland, June 12--18, 2000}}, pages
	35--41. Warsaw: Polish Academy of Sciences, Institute of Mathematics.
	
	\bibitem[\protect\astroncite{{Dufour} and {Nguyen Tien Zung}}{2005}]{DuZu05}
	{Dufour}, J.-P. and {Nguyen Tien Zung} (2005).
	\newblock {\em {Poisson structures and their normal forms.}}, volume 242.
	\newblock Basel: Birkh\"auser.
	
	\bibitem[\protect\astroncite{Evens et~al.}{1999}]{EvLuWe99}
	Evens, S., Lu, J.-H., and Weinstein, A. (1999).
	\newblock Transverse measures, the modular class and a cohomology pairing for
	{L}ie algebroids.
	\newblock {\em Quart. J. Math. Oxford Ser. (2)}, 50(200):417--436.
	
	\bibitem[\protect\astroncite{{Fernandes}}{2002}]{Fernandes02}
	{Fernandes}, R.~L. (2002).
	\newblock {Lie algebroids, holonomy and characteristic classes.}
	\newblock {\em {Adv. Math.}}, 170(1):119--179.
	
	\bibitem[\protect\astroncite{Grabowski and Rotkiewicz}{2009}]{GrRo09}
	Grabowski, J. and Rotkiewicz, M. (2009).
	\newblock Higher vector bundles and multi-graded symplectic manifolds.
	\newblock {\em J. Geom. Phys.}, 59(9):1285--1305.
	
	\bibitem[\protect\astroncite{Gracia-Saz et~al.}{2018}]{GrJoMaMe18}
	Gracia-Saz, A., Jotz~Lean, M., Mackenzie, K. C.~H., and Mehta, R.~A. (2018).
	\newblock Double {L}ie algebroids and representations up to homotopy.
	\newblock {\em J. Homotopy Relat. Struct.}, 13(2):287--319.
	
	\bibitem[\protect\astroncite{{Gracia-Saz} and {Mehta}}{2010}]{GrMe10}
	{Gracia-Saz}, A. and {Mehta}, R.~A. (2010).
	\newblock {Lie algebroid structures on double vector bundles and representation
		theory of Lie algebroids.}
	\newblock {\em {Adv. Math.}}, 223(4):1236--1275.
	
	\bibitem[\protect\astroncite{Gracia-Saz and Mehta}{2017}]{GrMe17}
	Gracia-Saz, A. and Mehta, R.~A. (2017).
	\newblock {$\mathcal{VB}$}-groupoids and representation theory of {L}ie
	groupoids.
	\newblock {\em J. Symplectic Geom.}, 15(3):741--783.
	
	\bibitem[\protect\astroncite{{Hamilton}}{1835}]{Hamilton1835}
	{Hamilton}, W.~R. (1835).
	\newblock {Second essay on a general method in dynamics.}
	\newblock {\em {Phil. Trans. R. Soc.}}, 125:95--144.
	
	\bibitem[\protect\astroncite{{Hartshorne}}{1977}]{Hartshorne77}
	{Hartshorne}, R. (1977).
	\newblock {\em {Algebraic geometry}}, volume~52.
	\newblock Springer, New York, NY.
	
	\bibitem[\protect\astroncite{Heuer and Lean}{2018}]{HeJo18}
	Heuer, M. and Lean, M.~J. (2018).
	\newblock Multiple vector bundles: cores, splittings and decompositions.
	
	\bibitem[\protect\astroncite{{Jacobi}}{1884}]{Jacobi1884}
	{Jacobi}, C. G.~J. (1884).
	\newblock {C. G. J. Jacobi's Gesammelte Werke. Herausgegeben auf Veranlassung
		der K\"oniglich Preussischen Akademie der Wissenschaften. Supplementband.
		Herausgegeben von E. Lottner. Vorlesungen \"uber Dynamik.}
	\newblock {Berlin. G. Reimer (1884,1891).}
	
	\bibitem[\protect\astroncite{Jotz~Lean}{2018a}]{Jotz18a}
	Jotz~Lean, M. (2018a).
	\newblock Dorfman connections and {C}ourant algebroids.
	\newblock {\em J. Math. Pures Appl. (9)}, 116:1--39.
	
	\bibitem[\protect\astroncite{Jotz~Lean}{2018b}]{Jotz18b}
	Jotz~Lean, M. (2018b).
	\newblock The geometrization of {$\mathbb N$}-manifolds of degree 2.
	\newblock {\em Journal of Geometry and Physics}, 133:113 -- 140.
	
	\bibitem[\protect\astroncite{{Jotz Lean}}{2019}]{Jotz19b}
	{Jotz Lean}, M. (2019).
	\newblock {Lie 2-algebroids and matched pairs of 2-representations: a geometric
		approach}.
	\newblock {\em {Pac. J. Math.}}, 301(1):143--188.
	
	\bibitem[\protect\astroncite{Jotz~Lean}{2019}]{Jotz19c}
	Jotz~Lean, M. (2019).
	\newblock Obstructions to representations up to homotopy and ideals.
	\newblock {\em arXiv e-prints}, page arXiv:1905.10237.
	
	\bibitem[\protect\astroncite{{Jotz Lean}}{2020}]{Jotz18d}
	{Jotz Lean}, M. (2020).
	\newblock {On LA-Courant algebroids and Poisson Lie 2-algebroids}.
	\newblock {\em {Math. Phys. Anal. Geom.}}, 23(3):40.
	\newblock Id/No 31.
	
	\bibitem[\protect\astroncite{{Jotz Lean} et~al.}{2020}]{JoMePa19}
	{Jotz Lean}, M., {Amit Mehta}, R., and {Papantonis}, T. (2020).
	\newblock {Modules and representations up to homotopy of Lie $n$-algebroids}.
	\newblock {\em arXiv e-prints}, page arXiv:2001.01101.
	
	\bibitem[\protect\astroncite{Jotz~Lean and Ortiz}{2014}]{JoOr14}
	Jotz~Lean, M. and Ortiz, C. (2014).
	\newblock Foliated groupoids and infinitesimal ideal systems.
	\newblock {\em Indag. Math. (N.S.)}, 25(5):1019--1053.
	
	\bibitem[\protect\astroncite{Kajiura and Stasheff}{2006}]{KaSt06}
	Kajiura, H. and Stasheff, J. (2006).
	\newblock Homotopy algebras inspired by classical open-closed string field
	theory.
	\newblock {\em Comm. Math. Phys.}, 263(3):553--581.
	
	\bibitem[\protect\astroncite{Keller and Waldmann}{2015}]{KeWa15}
	Keller, F. and Waldmann, S. (2015).
	\newblock Deformation theory of {C}ourant algebroids via the {R}othstein
	algebra.
	\newblock {\em J. Pure Appl. Algebra}, 219(8):3391--3426.
	
	\bibitem[\protect\astroncite{{Kontsevich}}{2003}]{Kontsevich03}
	{Kontsevich}, M. (2003).
	\newblock {Deformation quantization of Poisson manifolds.}
	\newblock {\em {Lett. Math. Phys.}}, 66(3):157--216.
	
	\bibitem[\protect\astroncite{Kosmann-Schwarzbach}{1995}]{Kosmann95}
	Kosmann-Schwarzbach, Y. (1995).
	\newblock Exact {G}erstenhaber algebras and {L}ie bialgebroids.
	\newblock {\em Acta Appl. Math.}, 41(1-3):153--165.
	\newblock Geometric and algebraic structures in differential equations.
	
	\bibitem[\protect\astroncite{Kosmann-Schwarzbach}{2005}]{Ko-Sc05}
	Kosmann-Schwarzbach, Y. (2005).
	\newblock {\em Quasi, twisted, and all that... in Poisson geometry and Lie
		algebroid theory}, pages 363--389.
	\newblock Birkh{\"a}user Boston, Boston, MA.
	
	\bibitem[\protect\astroncite{Kostant}{1977}]{Kostant77}
	Kostant, B. (1977).
	\newblock Graded manifolds, graded {L}ie theory, and prequantization.
	\newblock In {\em Differential geometrical methods in mathematical physics
		({P}roc. {S}ympos., {U}niv. {B}onn, {B}onn, 1975)}, pages 177--306. Lecture
	Notes in Math., Vol. 570. Springer, Berlin.
	
	\bibitem[\protect\astroncite{{La Pastina}}{2020}]{LaPastina20}
	{La Pastina}, P.~P. (2020).
	\newblock {Deformations of vector bundles in the categories of Lie algebroids
		and groupoids}.
	\newblock {\em arXiv e-prints}, page arXiv:2001.07559.
	
	\bibitem[\protect\astroncite{{La Pastina} and {Vitagliano}}{2019a}]{LaVi18}
	{La Pastina}, P.~P. and {Vitagliano}, L. (2019a).
	\newblock {Deformations of linear Lie brackets}.
	\newblock {\em {Pac. J. Math.}}, 303(1):265--298.
	
	\bibitem[\protect\astroncite{{La Pastina} and {Vitagliano}}{2019b}]{LaVi19}
	{La Pastina}, P.~P. and {Vitagliano}, L. (2019b).
	\newblock {Deformations of Vector Bundles over Lie Groupoids}.
	\newblock {\em arXiv e-prints}, page arXiv:1907.05670.
	
	\bibitem[\protect\astroncite{{Laurent-Gengoux} et~al.}{2012}]{GePiVa12}
	{Laurent-Gengoux}, C., {Pichereau}, A., and {Vanhaecke}, P. (2012).
	\newblock {\em {Poisson structures.}}, volume 347.
	\newblock Berlin: Springer.
	
	\bibitem[\protect\astroncite{Li-Bland}{2012}]{Li-Bland12}
	Li-Bland, D. (2012).
	\newblock Phd thesis: {LA}-{C}ourant {A}lgebroids and their {A}pplications.
	\newblock {\em arXiv:1204.2796}.
	
	\bibitem[\protect\astroncite{Liu et~al.}{1997}]{LiWeXu97}
	Liu, Z.-J., Weinstein, A., and Xu, P. (1997).
	\newblock Manin triples for {L}ie bialgebroids.
	\newblock {\em J. Differential Geom.}, 45(3):547--574.
	
	\bibitem[\protect\astroncite{Mackenzie}{2005}]{Mackenzie05}
	Mackenzie, K.~C.~H. (2005).
	\newblock {\em General {T}heory of {L}ie {G}roupoids and {L}ie {A}lgebroids},
	volume 213 of {\em London Mathematical Society Lecture Note Series}.
	\newblock Cambridge University Press, Cambridge.
	
	\bibitem[\protect\astroncite{Mackenzie and Xu}{1994}]{MaXu94}
	Mackenzie, K.~C.~H. and Xu, P. (1994).
	\newblock {Lie bialgebroids and Poisson groupoids.}
	\newblock {\em Duke Math. J.}, 73(2):415--452.
	
	\bibitem[\protect\astroncite{Mackenzie and Xu}{2000}]{MaXu00}
	Mackenzie, K.~C.~H. and Xu, P. (2000).
	\newblock {Integration of Lie bialgebroids.}
	\newblock {\em Topology}, 39(3):445--467.
	
	\bibitem[\protect\astroncite{{Manin}}{1997}]{Manin97}
	{Manin}, Y.~I. (1997).
	\newblock {\em {Gauge field theory and complex geometry. Transl. from the
			Russian by N. Koblitz and J. R. King. With an appendix by S. Merkulov. 2nd
			ed.}}, volume 289.
	\newblock Berlin: Springer, 2nd ed. edition.
	
	\bibitem[\protect\astroncite{Mehta}{2006}]{Mehta06}
	Mehta, R. (2006).
	\newblock Supergroupoids, double structures, and equivariant cohomology.
	\newblock {\em arXiv:math/0605356}.
	
	\bibitem[\protect\astroncite{Mehta}{2009}]{Mehta09a}
	Mehta, R.~A. (2009).
	\newblock {$Q$}-groupoids and their cohomology.
	\newblock {\em Pacific J. Math.}, 242(2):311--332.
	
	\bibitem[\protect\astroncite{{Mehta}}{2009}]{Mehta09}
	{Mehta}, R.~A. (2009).
	\newblock {\(Q\)-algebroids and their cohomology.}
	\newblock {\em {J. Symplectic Geom.}}, 7(3):263--293.
	
	\bibitem[\protect\astroncite{{Mehta}}{2011}]{Mehta11}
	{Mehta}, R.~A. (2011).
	\newblock {On homotopy Poisson actions and reduction of symplectic
		\(Q\)-manifolds}.
	\newblock {\em {Differ. Geom. Appl.}}, 29(3):319--328.
	
	\bibitem[\protect\astroncite{{Mehta}}{2014}]{Mehta14}
	{Mehta}, R.~A. (2014).
	\newblock {Lie algebroid modules and representations up to homotopy.}
	\newblock {\em {Indag. Math., New Ser.}}, 25(5):1122--1134.
	
	\bibitem[\protect\astroncite{Mehta}{2015}]{Mehta15}
	Mehta, R.~A. (2015).
	\newblock Modular classes of {L}ie groupoid representations up to homotopy.
	\newblock {\em SIGMA Symmetry Integrability Geom. Methods Appl.}, 11:Paper 058,
	10.
	
	\bibitem[\protect\astroncite{Meinrenken and Pike}{2020}]{MePi19}
	Meinrenken, E. and Pike, J. (2020).
	\newblock {The Weil Algebra of a Double Lie Algebroid}.
	\newblock {\em International Mathematics Research Notices}.
	\newblock rnz361.
	
	\bibitem[\protect\astroncite{{Poisson}}{1809}]{Poisson1809}
	{Poisson}, S.~D. (1809).
	\newblock {Sur la variation des constantes arbitraires dans les questions de
		m\'{e}canique.}
	\newblock {\em {J. \'{E}cole Polytechnique}}, 8:265--344.
	
	\bibitem[\protect\astroncite{{Popescu} and {Popescu}}{2019}]{PoPo19}
	{Popescu}, M. and {Popescu}, P. (2019).
	\newblock {Almost Lie algebroids and characteristic classes}.
	\newblock {\em {SIGMA, Symmetry Integrability Geom. Methods Appl.}}, 15:paper
	021, 12.
	
	\bibitem[\protect\astroncite{Pradines}{1977}]{Pradines77}
	Pradines, J. (1977).
	\newblock {\em Fibr\'es vectoriels doubles et calcul des jets non holonomes},
	volume~29 of {\em Esquisses Math\'ematiques [Mathematical Sketches]}.
	\newblock Universit\'e d'Amiens U.E.R. de Math\'ematiques, Amiens.
	
	\bibitem[\protect\astroncite{Quillen}{1985}]{Quillen85}
	Quillen, D. (1985).
	\newblock Superconnections and the {C}hern character.
	\newblock {\em Topology}, 24(1):89--95.
	
	\bibitem[\protect\astroncite{Roytenberg}{1999}]{Roytenberg99}
	Roytenberg, D. (1999).
	\newblock {\em Courant algebroids, derived brackets and even symplectic
		supermanifolds}.
	\newblock ProQuest LLC, Ann Arbor, MI.
	\newblock Thesis (Ph.D.)--University of California, Berkeley.
	
	\bibitem[\protect\astroncite{Roytenberg}{2002}]{Roytenberg02}
	Roytenberg, D. (2002).
	\newblock On the structure of graded symplectic supermanifolds and {C}ourant
	algebroids.
	\newblock In {\em Quantization, {P}oisson brackets and beyond ({M}anchester,
		2001)}, volume 315 of {\em Contemp. Math.}, pages 169--185. Amer. Math. Soc.,
	Providence, RI.
	
	\bibitem[\protect\astroncite{Salam and Strathdee}{1974}]{SaSt74}
	Salam, A. and Strathdee, J. (1974).
	\newblock Super-gauge transformations.
	\newblock {\em Nuclear Physics B}, 76(3):477 -- 482.
	
	\bibitem[\protect\astroncite{Sch\"{a}tz}{2009}]{Schatz09}
	Sch\"{a}tz, F. (2009).
	\newblock B{FV}-complex and higher homotopy structures.
	\newblock {\em Comm. Math. Phys.}, 286(2):399--443.
	
	\bibitem[\protect\astroncite{{Sch\"atz}}{2009}]{Schaetz09}
	{Sch\"atz}, F. (2009).
	\newblock {\em Coisotropic submanifolds and the BFV-complex}.
	\newblock PhD thesis, Universit\"at Z\"urich.
	
	\bibitem[\protect\astroncite{{\v{S}}evera}{2005}]{Severa05}
	{\v{S}}evera, P. (2005).
	\newblock Some title containing the words ``homotopy'' and ``symplectic'', e.g.
	this one.
	\newblock In {\em Travaux math\'ematiques. {F}asc. {XVI}}, Trav. Math., XVI,
	pages 121--137. Univ. Luxemb., Luxembourg.
	
	\bibitem[\protect\astroncite{Sheng and Zhu}{2017}]{ShZh17}
	Sheng, Y. and Zhu, C. (2017).
	\newblock Higher extensions of {L}ie algebroids.
	\newblock {\em Commun. Contemp. Math.}, 19(3):1650034, 41.
	
	\bibitem[\protect\astroncite{Stefani}{2019}]{Stefani19}
	Stefani, D. (2019).
	\newblock Representations up to homotopy and perfect complexes over
	differentiable stacks.
	\newblock {\em PhD thesis}.
	
	\bibitem[\protect\astroncite{Trentinaglia and Zhu}{2016}]{TrZh16}
	Trentinaglia, G. and Zhu, C. (2016).
	\newblock Some remarks on representations up to homotopy.
	\newblock {\em Int. J. Geom. Methods Mod. Phys.}, 13(3):1650024, 15.
	
	\bibitem[\protect\astroncite{Tu}{2017}]{Tu17}
	Tu, L.~W. (2017).
	\newblock {\em Differential geometry}, volume 275 of {\em Graduate Texts in
		Mathematics}.
	\newblock Springer, Cham.
	\newblock Connections, curvature, and characteristic classes.
	
	\bibitem[\protect\astroncite{{Tuynman}}{2004}]{Tuynman04}
	{Tuynman}, G.~M. (2004).
	\newblock {\em {Supermanifolds and supergroups. Basic theory.}}
	\newblock Dordrecht: Kluwer Academic Publishers.
	
	\bibitem[\protect\astroncite{Va{\u\i}ntrob}{1997}]{Vaintrob97}
	Va{\u\i}ntrob, A.~Y. (1997).
	\newblock Lie algebroids and homological vector fields.
	\newblock {\em Uspekhi Mat. Nauk}, 52(2(314)):161--162.
	
	\bibitem[\protect\astroncite{{Vaisman}}{1994}]{Vaisman94}
	{Vaisman}, I. (1994).
	\newblock {\em {Lectures on the geometry of Poisson manifolds.}}, volume 118.
	\newblock Basel: Birkh\"auser.
	
	\bibitem[\protect\astroncite{van Est}{1962a}]{vEs62a}
	van Est, W.~T. (1962a).
	\newblock Local and global groups. {I}.
	\newblock {\em Nederl. Akad. Wetensch. Proc. Ser. A 65 = Indag. Math.},
	24:391--408.
	
	\bibitem[\protect\astroncite{van Est}{1962b}]{vEs62b}
	van Est, W.~T. (1962b).
	\newblock Local and global groups. {II}.
	\newblock {\em Nederl. Akad. Wetensch. Proc. Ser. A 65 = Indag. Math.},
	24:409--425.
	
	\bibitem[\protect\astroncite{Varadarajan}{2004}]{Varadarajan04}
	Varadarajan, V.~S. (2004).
	\newblock {\em Supersymmetry for mathematicians: an introduction}, volume~11 of
	{\em Courant Lecture Notes in Mathematics}.
	\newblock New York University, Courant Institute of Mathematical Sciences, New
	York; American Mathematical Society, Providence, RI.
	
	\bibitem[\protect\astroncite{Vitagliano}{2015}]{Vitagliano15b}
	Vitagliano, L. (2015).
	\newblock Representations of homotopy {L}ie-{R}inehart algebras.
	\newblock {\em Math. Proc. Cambridge Philos. Soc.}, 158(1):155--191.
	
	\bibitem[\protect\astroncite{Voronov}{2005}]{Voronov05}
	Voronov, T. (2005).
	\newblock Higher derived brackets and homotopy algebras.
	\newblock {\em J. Pure Appl. Algebra}, 202(1-3):133--153.
	
	\bibitem[\protect\astroncite{Voronov}{2002}]{Voronov02}
	Voronov, T.~T. (2002).
	\newblock Graded manifolds and {D}rinfeld doubles for {L}ie bialgebroids.
	\newblock In {\em Quantization, {P}oisson brackets and beyond ({M}anchester,
		2001)}, volume 315 of {\em Contemp. Math.}, pages 131--168. Amer. Math. Soc.,
	Providence, RI.
	
	\bibitem[\protect\astroncite{{Voronov}}{2005}]{Voronov05a}
	{Voronov}, T.~T. (2005).
	\newblock {Higher derived brackets for arbitrary derivations}.
	\newblock In {\em {Proceedings of the 4th conference on Poisson geometry,
			Luxembourg, June 7--11, 2004}}, pages 163--186. Luxembourg: Universit\'e du
	Luxembourg.
	
	\bibitem[\protect\astroncite{Voronov}{2012}]{Voronov12}
	Voronov, T.~T. (2012).
	\newblock Q-{M}anifolds and {M}ackenzie {T}heory.
	\newblock {\em Comm. Math. Phys.}, 315(2):279--310.
	
	\bibitem[\protect\astroncite{{Waldmann}}{2007}]{Waldmann07}
	{Waldmann}, S. (2007).
	\newblock {\em {Poisson-Geometrie und Deformationsquantisierung. Eine
			Einf\"uhrung.}}
	\newblock Berlin: Springer.
	
	\bibitem[\protect\astroncite{{Weinstein}}{1983}]{Weinstein83}
	{Weinstein}, A. (1983).
	\newblock {The local structure of Poisson manifolds.}
	\newblock {\em {J. Differ. Geom.}}, 18:523--557.
	
	\bibitem[\protect\astroncite{{Weinstein}}{2000}]{Weinstein00}
	{Weinstein}, A. (2000).
	\newblock {Almost invariant submanifolds for compact group actions.}
	\newblock {\em {J. Eur. Math. Soc. (JEMS)}}, 2(1):53--86.
	
	\bibitem[\protect\astroncite{Weinstein}{2001}]{Weinstein01}
	Weinstein, A. (2001).
	\newblock Groupoids: unifying internal and external symmetry. {A} tour through
	some examples.
	\newblock In {\em Groupoids in analysis, geometry, and physics ({B}oulder,
		{CO}, 1999)}, volume 282 of {\em Contemp. Math.}, pages 1--19. Amer. Math.
	Soc., Providence, RI.
	
	\bibitem[\protect\astroncite{{Zhu}}{2009a}]{Zhu09b}
	{Zhu}, C. (2009a).
	\newblock {Kan replacement of simplicial manifolds}.
	\newblock {\em {Lett. Math. Phys.}}, 90(1-3):383--405.
	
	\bibitem[\protect\astroncite{{Zhu}}{2009b}]{Zhu09}
	{Zhu}, C. (2009b).
	\newblock {\(n\)-groupoids and stacky groupoids}.
	\newblock {\em {Int. Math. Res. Not.}}, 2009(21):4087--4141.
	
\end{thebibliography}

\clearpage
\phantomsection
\addcontentsline{toc}{chapter}{Alphabetical Index}
\begin{theindex}
	
	\noindent\textbf{Symbols}\par\nopagebreak
	
	\item $\Q$-manifold, \hyperpage{6}, \hyperpage{35}
	\subitem linear, \hyperpage{9}, \hyperpage{93}
	\subitem morphism, \hyperpage{35}
	\subitem product, \hyperpage{35}
	\item $\mathbb{N}$-graded manifold, \hyperpage{5}, \hyperpage{29}
	\subitem split, \hyperpage{31}
	\item $\mathbb{N}$-manifold, \hyperpage{29}
	\item $\mathbb{N}\Q$-manifold, \hyperpage{6}, \hyperpage{35}
	\item $\mathbb{N}\mathcal{P}$-manifold, \hyperpage{41}
	\subitem morphism, \hyperpage{41}
	\item $\mathbb{N}\mathcal{P}_k$-manifold, \hyperpage{41}
	\item $\mathbb{Z}$-graded manifold, \hyperpage{5}, \hyperpage{29}
	\subitem split, \hyperpage{31}
	\item $\mathbb{Z}$-manifold, \hyperpage{29}
	\item $\mathbb{Z}_2$-graded supermanifold, \hyperpage{5}
	\item $\mathcal{PQ}$-manifold, \hyperpage{6}, \hyperpage{42}
	\subitem morphism, \hyperpage{42}
	\item $\mathcal{P}$-manifold, \hyperpage{6}, \hyperpage{41}
	\subitem linear, \hyperpage{9}
	\subitem morphism, \hyperpage{41}
	\item $\mathcal{P}_k$-manifold, \hyperpage{41}
	\item $\n$-manifold, \hyperpage{29}
	\item $n$-connection, \hyperpage{123}
	
	\indexspace
	
	\noindent\textbf{A}\par\nopagebreak
	
	\item adjoint representation, \hyperpage{6}, \hyperpage{78}, 
	\hyperpage{82}
	\item almost Lie algebroid, \hyperpage{17}
	\item anchor map, \hyperpage{3}, \hyperpage{16}
	
	\indexspace
	
	\noindent\textbf{B}\par\nopagebreak
	
	\item basic connections, \hyperpage{18}
	\item basic curvature, \hyperpage{18}
	\item Bianchi identity, \hyperpage{125}
	\item bracket, \hyperpage{16}
	
	\indexspace
	
	\noindent\textbf{C}\par\nopagebreak
	
	\item canonical involution, \hyperpage{21}
	\item Chern-Weil morphism, \hyperpage{122}
	\item coadjoint representation, \hyperpage{78}, \hyperpage{82}
	\item cochain complex of vector bundles, \hyperpage{16}
	\subitem $k$-morphism, \hyperpage{16}
	\subitem degree $k$ morphism, \hyperpage{16}
	\subitem differential, \hyperpage{16}
	\item cohesive module, \hyperpage{114}
	\item connection up to homotopy, \hyperpage{122}
	\subitem flat, \hyperpage{123}
	\item Courant algebroid, \hyperpage{4}
	\subitem $k$-cochain, \hyperpage{61}
	\subitem adjoint representation, \hyperpage{77}
	\subitem Keller-Waldmann algebra, \hyperpage{60}
	\subitem symbol of a $k$-cochain, \hyperpage{61}
	\item Courant bracket, \hyperpage{4}
	\item curvature of a connection, \hyperpage{18}
	
	\indexspace
	
	\noindent\textbf{D}\par\nopagebreak
	
	\item de Rham cohomology, \hyperpage{46}
	\item de Rham complex, \hyperpage{46}
	\item de Rham differential, \hyperpage{45}
	\item decomposed double vector bundle, \hyperpage{20}
	\subitem linear horizontal lift, \hyperpage{20}
	\subitem morphism, \hyperpage{20}
	\item degree-homogeneous element, \hyperpage{15}
	\item derivation Lie 2-algebra, \hyperpage{38}
	\item derivation Lie algebroid, \hyperpage{7}
	\item derived bracket formula, \hyperpage{47, 48}
	\item DG-manifold, \hyperpage{6}
	\item differential graded manifold, \hyperpage{6}
	\item Dirac structure, \hyperpage{4}
	\item direct limit, \hyperpage{26}
	\item direct system, \hyperpage{26}
	\item Dorfman connection, \hyperpage{127}
	\item double $\Q$-manifold, \hyperpage{115}
	\item double Lie $n$-algebroid, \hyperpage{115}
	\item double tangent bundle, \hyperpage{21}
	\item double vector bundle, \hyperpage{19}
	\subitem core, \hyperpage{19}
	\subitem core section, \hyperpage{19}
	\subitem core-linear section, \hyperpage{20}
	\subitem decomposition, \hyperpage{21}
	\subitem dual linear splitting, \hyperpage{22}
	\subitem duals, \hyperpage{22}
	\subitem linear horizontal lift, \hyperpage{22}
	\subitem linear section, \hyperpage{19}
	\subitem linear splitting, \hyperpage{21}
	\subitem morphism, \hyperpage{19}
	\subitem sides, \hyperpage{19}
	\subitem Weil algebra, \hyperpage{104}
	\item dull algebroid, \hyperpage{16}
	\subitem connection on a complex of vector bundles, \hyperpage{17}
	\subitem connection on a graded vector bundle, \hyperpage{17}
	\subitem connection on a vector bundle, \hyperpage{17}
	\subitem dual connection, \hyperpage{17}
	\subitem flat connection, \hyperpage{18}
	\subitem forms, \hyperpage{17}
	
	\indexspace
	
	\noindent\textbf{E}\par\nopagebreak
	
	\item equaliser, \hyperpage{24}
	
	\indexspace
	
	\noindent\textbf{G}\par\nopagebreak
	
	\item gauge transformation, \hyperpage{126}
	\item generalised wedge product, \hyperpage{39}
	\item germ, \hyperpage{26}
	\item graded commutator, \hyperpage{122}
	\item graded double vector budnle, \hyperpage{119}
	\subitem decomposed, \hyperpage{119}
	\item graded Lie algebroid, \hyperpage{95}
	\item graded manifold, \hyperpage{29}
	\subitem base manifold, \hyperpage{29}
	\subitem body, \hyperpage{29}
	\subitem linear manifold, \hyperpage{32}
	\subitem local coordinates, \hyperpage{29}
	\subitem morphism, \hyperpage{29}
	\subitem product, \hyperpage{32}
	\item graded Poisson bracket, \hyperpage{41}
	\subitem linear, \hyperpage{94}
	\item graded skew symmetry, \hyperpage{36}
	\item graded trace operator, \hyperpage{122}
	\item graded vector bundle over a smooth manifold, \hyperpage{15}
	\item groupoid, \hyperpage{112}
	\subitem arrows, \hyperpage{112}
	\subitem identity map, \hyperpage{113}
	\subitem inversion map, \hyperpage{113}
	\subitem multiplication map, \hyperpage{112}
	\subitem objects, \hyperpage{112}
	\subitem source, \hyperpage{112}
	\subitem target, \hyperpage{112}
	
	\indexspace
	
	\noindent\textbf{H}\par\nopagebreak
	
	\item Hamilton's equations of motion, \hyperpage{2}
	\item Hamiltonian function, \hyperpage{2}
	\item Hamiltonian vector field, \hyperpage{2}, \hyperpage{41}, 
	\hyperpage{57}
	\item homological vector field, \hyperpage{35}
	\subitem linear, \hyperpage{93}
	\item homotopy Poisson structure, \hyperpage{51}
	\subitem deformation, \hyperpage{52}
	\subitem formal deformation, \hyperpage{52}
	\subitem infinitesimal deformation, \hyperpage{52}
	\subitem linear, \hyperpage{94}
	
	\indexspace
	
	\noindent\textbf{J}\par\nopagebreak
	
	\item Jacobi identity, \hyperpage{2}
	
	\indexspace
	
	\noindent\textbf{L}\par\nopagebreak
	
	\item Leibniz identity, \hyperpage{3}
	\item Lichnerowicz complex, \hyperpage{50}
	\item Lie $n$-algebroid, \hyperpage{6}, \hyperpage{35}
	\subitem $k$-ismorphism of DG-modules, \hyperpage{65}
	\subitem $k$-morphism of DG-modules, \hyperpage{65}
	\subitem $k$-representation, \hyperpage{71}
	\subitem  $k$-term representation, \hyperpage{71}
	\subitem adjoint module, \hyperpage{66}
	\subitem anti-symmetric powers module, \hyperpage{65}
	\subitem category of DG-modules, \hyperpage{65}
	\subitem coadjoint module, \hyperpage{66}
	\subitem cohomology of a DG-module, \hyperpage{64}
	\subitem degree $0$-morphism of DG-modules, \hyperpage{64}
	\subitem degree $k$-morphism of DG-modules, \hyperpage{65}
	\subitem differential graded bimodule, \hyperpage{64}
	\subitem direct sum module, \hyperpage{65}
	\subitem dual module, \hyperpage{64}
	\subitem fat, \hyperpage{105}
	\subitem generalised functions, \hyperpage{39}
	\subitem left DG-module, \hyperpage{64}
	\subitem left differential graded module, \hyperpage{63}
	\subitem left representation up to homotopy, \hyperpage{71}
	\subitem linear, \hyperpage{100}
	\subitem module of homomorphisms, \hyperpage{64}
	\subitem morphism, \hyperpage{35}
	\subitem morphism of DG-modules, \hyperpage{64}
	\subitem right DG-module, \hyperpage{64}
	\subitem right differential graded module, \hyperpage{63}
	\subitem right representation up to homotopy, \hyperpage{71}
	\subitem shifted module, \hyperpage{65}
	\subitem split, \hyperpage{6}, \hyperpage{35}
	\subitem symmetric powers module, \hyperpage{65}
	\subitem tangent prolongation, \hyperpage{101}
	\subitem tensor product module, \hyperpage{64}
	\subitem vector valued functions, \hyperpage{39}
	\item Lie $n$-groupoid, \hyperpage{114}
	\subitem representation, \hyperpage{114}
	\subitem representation up to homotopy, \hyperpage{114}
	\item Lie 2-algebra, \hyperpage{37}
	\item Lie algebroid, \hyperpage{3}, \hyperpage{17}
	\subitem differential, \hyperpage{4}
	\subitem Lie bracket, \hyperpage{3}
	\subitem module, \hyperpage{8}, \hyperpage{63}
	\subitem representation up to homotopy, \hyperpage{71}
	\subitem self-dual representation, \hyperpage{87}
	\item Lie bialgebroid, \hyperpage{53}
	\item Lie bracket of vector fields, \hyperpage{34}
	\item Lie cohomology, \hyperpage{46}
	\item Lie complex, \hyperpage{46}
	\item Lie complex of (pseudo)multivector fields, \hyperpage{50}
	\item Lie derivative, \hyperpage{46}, \hyperpage{49}
	\item Lie groupoid, \hyperpage{3}, \hyperpage{113}
	\subitem cohesive module, \hyperpage{114}
	\subitem nerve, \hyperpage{113}
	\item Lie quasi-bialgebroid, \hyperpage{55}
	\item linear vector field, \hyperpage{20}
	
	\indexspace
	
	\noindent\textbf{M}\par\nopagebreak
	
	\item Master equation, \hyperpage{60}
	\item Maurer-Cartan element, \hyperpage{53}
	
	\indexspace
	
	\noindent\textbf{P}\par\nopagebreak
	
	\item Poisson bivector field, \hyperpage{2}
	\item Poisson bracket, \hyperpage{2}
	\item Poisson complex, \hyperpage{50}
	\item Poisson homological vector field, \hyperpage{42}
	\item Poisson Lie $n$-algebroid, \hyperpage{6}, \hyperpage{42}
	\item Poisson vector field, \hyperpage{41}
	\item Poisson-Weil algebra, \hyperpage{50}
	\item Poisson-Weil double complex, \hyperpage{50}
	\item Pontryagin algebra, \hyperpage{121}, \hyperpage{125}
	\item presheaf, \hyperpage{23}
	\subitem isomorphism, \hyperpage{24}
	\subitem morphism, \hyperpage{24}
	\subitem restriction maps, \hyperpage{24}
	\subitem sections, \hyperpage{24}
	\subitem stalk, \hyperpage{26}
	\item principal part of morphism of split manifolds, \hyperpage{33}
	\item pseudomultivector fields, \hyperpage{47}
	
	\indexspace
	
	\noindent\textbf{Q}\par\nopagebreak
	
	\item Quasi-Lie bialgebroid, \hyperpage{54}
	
	\indexspace
	
	\noindent\textbf{R}\par\nopagebreak
	
	\item representation, \hyperpage{6}
	\subitem of a Lie algebra, \hyperpage{6}
	\subitem of a Lie algebroid, \hyperpage{7}
	\subitem of a Lie group, \hyperpage{6}
	\subitem of a Lie groupoid, \hyperpage{7}
	\subitem up to homotopy, \hyperpage{7}
	\item representation theory, \hyperpage{6}
	\item ringed space, \hyperpage{27}
	\subitem morphism, \hyperpage{27}
	
	\indexspace
	
	\noindent\textbf{S}\par\nopagebreak
	
	\item Schouten bracket of bidegree $(-1,k)$, \hyperpage{47}
	\item sheaf, \hyperpage{25}
	\subitem associated to a presheaf, \hyperpage{27}
	\subitem direct image, \hyperpage{27}
	\subitem inverse image, \hyperpage{27}
	\subitem morphism, \hyperpage{25}
	\subitem morphism of modules, \hyperpage{28}
	\subitem of modules, \hyperpage{27}
	\subitem pull-back, \hyperpage{28}
	\subitem tensor product, \hyperpage{28}
	\item shift functor $[k]$ for complexes of vector bundles, 
	\hyperpage{16}
	\item shifted cotangent bundle, \hyperpage{47}
	\item shifted tangent bundle, \hyperpage{45}
	\item split Lie $2$-algebroid, \hyperpage{37}
	\subitem 3-term representation, \hyperpage{73}
	\subitem adjoint complex, \hyperpage{76}
	\subitem adjoint representation, \hyperpage{78}
	\subitem coadjoint complex, \hyperpage{78}
	\subitem coadjoint representation, \hyperpage{78}
	\subitem representation, \hyperpage{72}
	\item split Lie $n$-algebroid
	\subitem adjoint representation, \hyperpage{82}
	\subitem Weil algebra, \hyperpage{83}
	\item split symplectic Lie 2-algebroid, \hyperpage{38}
	\item strong homotopy Jacobi identity, \hyperpage{36}
	\item super trace operator, \hyperpage{122}
	\item symplectic form, \hyperpage{2}, \hyperpage{57}
	\item symplectic Lie $n$-algebroid, \hyperpage{42}
	\item symplectic manifold, \hyperpage{3}
	\item symplectic structure, \hyperpage{2}
	\item symplectic vector field, \hyperpage{57}
	
	\indexspace
	
	\noindent\textbf{T}\par\nopagebreak
	
	\item tangent $\mathcal  {Q}$-manifold, \hyperpage{93}
	\item tangent bundle of a graded manifold, \hyperpage{44}
	\item tangent prolongation of a vector bundle, \hyperpage{20}
	
	\indexspace
	
	\noindent\textbf{V}\par\nopagebreak
	
	\item VB-Lie $n$-algebroid, \hyperpage{9}, \hyperpage{100}
	\subitem decomposed, \hyperpage{101}
	\subitem morphism, \hyperpage{101}
	\subitem split, \hyperpage{9}
	\item vector bundle over a graded manifold, \hyperpage{43}
	\subitem basic coordinates, \hyperpage{44}
	\subitem basic functions, \hyperpage{44}
	\subitem derivation, \hyperpage{91}
	\subitem flat derivation, \hyperpage{92}
	\subitem linear coordinates, \hyperpage{44}
	\subitem linear functions, \hyperpage{44}
	\subitem linear multivector field, \hyperpage{92}
	\subitem linear vector field, \hyperpage{91}
	\subitem sections, \hyperpage{43}
	\item vector field, \hyperpage{34}
	\item vector valued differential form, \hyperpage{39}
	\item vertical lift, \hyperpage{20}
	
	\indexspace
	
	\noindent\textbf{W}\par\nopagebreak
	
	\item Weil algebra, \hyperpage{46}
	\item Weil cohomology, \hyperpage{47}
	\item Weil complex, \hyperpage{47}
	
\end{theindex}

%\printindex

\end{document}